\documentclass[a4paper,12pt]{article}

\usepackage[top=1in,bottom=1in,left=1in,right=1in]{geometry}
\usepackage{amstext,amssymb,amsmath,latexsym}
\usepackage{bm,enumerate}
\usepackage{amsthm}
\usepackage{mathrsfs}
\usepackage{amscd}
\usepackage{paralist}
\usepackage{array,multirow}
\usepackage[all]{xy}

\numberwithin{equation}{section}

\newtheorem{theorem}{Theorem}[section]
\newtheorem{definition}[theorem]{Definition}
\newtheorem{lemma}[theorem]{Lemma}
\newtheorem{corollary}[theorem]{Corollary}

\newlength\saveparindent
\saveparindent=\parindent

\DeclareMathAlphabet{\mathsfsl}{OT1}{cmss}{m}{sl}
\newcommand{\tens}[1]{\mathsfsl{#1}}

\DeclareMathOperator*{\supp}{Supp}
\DeclareMathOperator*{\conv}{Conv}
\DeclareMathOperator*{\face}{Face}
\DeclareMathOperator*{\facet}{Facet}
\DeclareMathOperator*{\ord}{ord}

\DeclareMathOperator*{\rank}{rank}

\newcommand{\BB}[1]{{\mathbb{#1}}}
\newcommand{\hs}[1]{\hspace{#1}}
\newcommand{\mc}[3]{\multicolumn{#1}{#2}{#3}}
\newcommand{\mr}[2]{\multirow{#1}{*}{#2}}
\newcommand{\ol}[1]{\overline{#1}}
\newcommand{\wh}[1]{\widehat{#1}}
\newcommand{\wt}[1]{{\widetilde{#1}}}

\newcommand{\AM}{\wt\Z_*(\wt\Y_*)}
\newcommand{\ald}{{\Ga_\D}}
\newcommand{\alg}{{\Ga'_{\D'}}}

\newcommand{\Bt}[2]{b_{#2}({#1})}
\newcommand{\Btb}[1]{\Bt{\partial\SO}{#1}}
\newcommand{\Bti}[1]{\Bt{\SO^\circ}{#1}}
\newcommand{\Bts}[1]{\Bt{\SO}{#1}}
\newcommand{\bz}{\mathbf{0}}
\newcommand{\CP}[3]{{#1}_{#2}\wedge\,\cdots\,\wedge{#1}_{#3}}
\newcommand{\DE}[1]{\Delta_{#1}}
\newcommand{\DF}[1]{\BB F^{#1}}
\newcommand{\DFs}[1]{(\BB F^*)^{#1}}
\newcommand{\DN}[1]{\BB N^{#1}\setminus\{\bz\}}
\newcommand{\DR}{\Rp^n\setminus\{\bz\}}
\newcommand{\dotc}[2]{\{{#1},\dotsc,{#2}\}}
\newcommand{\EI}[2]{\mathcal{E}_{\iI{#2}}(#1)}
\newcommand{\Ec}[2]{{#1}_{\hs{-1pt}\iI{#2}}}

\newcommand{\Ev}[2]{{#1}_{\hs{-1pt}\iI{#2}}}
\newcommand{\Evv}[2]{{#1}_{\hs{-1pt}\iI{#2}\times\iI{#2}}}
\newcommand{\Evz}[3]{{#1}_{\hs{-1pt}\iI{#2}\times\iI{#3}}}
\newcommand{\es}[2]{{\supp\nolimits_{\iI{#2}}\hs{-2pt}(#1)}}
\newcommand{\esc}[2]{{\supp\nolimits_{\iI{#2}}^c\hs{-2pt}(#1)}}
\newcommand{\FFs}{\mathfrak{F}^\star}
\newcommand{\FV}[1]{\CF{\CV{#1}}}
\newcommand{\fnd}{\mathfrak{g}}
\newcommand{\gfv}[1]{\CG{\FV{#1}}}
\newcommand{\gv}[1]{\CG{\CV{#1}}}

\newcommand{\LE}[3]{{#1}_{#2},\dotsc,{#1}_{#3}}
\newcommand{\LEs}[5]{{#1}_{#2}^{#4},\dotsc,{#1}_{#3}^{#5}}
\newcommand{\Le}[1]{{#1}_1,{#1}_2,{#1}_3}
\newcommand{\lalg}{\Al''_{\D'}}
\newcommand{\lt}{{^t\hs{-1pt}}}
\newcommand{\ME}[3]{{#1}_{#2}~\cdots~{#1}_{#3}}
\newcommand{\MEs}[5]{{#1}_{#2}^{#4}~\cdots~{#1}_{#3}^{#5}}
\newcommand{\MP}[1]{\TP{(\overline{#1})}}
\newcommand{\Me}[1]{{#1}_1~{#1}_2~{#1}_3}
\newcommand{\Mnorm}[1]{\|{#1}\|_{\max}}
\newcommand{\mnorm}[1]{\|{#1}\|_{\min}}
\newcommand{\NP}{\mathsf{NP}}
\newcommand{\npp}[1]{\NP(\varphi_{#1})}
\newcommand{\OM}[1]{\mathrm{\Omega}_{#1}}
\newcommand{\ON}{\OM\NP}
\newcommand{\ONs}{\ON^\star}
\newcommand{\oDE}[1]{\ol\Delta_{#1}}
\newcommand{\oSV}[1]{\ol{\Delta}_{\CV{#1}}}
\newcommand{\PI}[1]{\Pi_{#1}}
\newcommand{\pAl}{\perm\hs{-0.8pt}\Al}
\newcommand{\RI}[2]{\mathcal{R}_{\iI{#2}}\hs{-1pt}(#1)}
\newcommand{\Rp}{\BB R_{\ge 0}}
\newcommand{\ro}{{(\rho)}}
\newcommand{\SF}[1]{\mathsf{F}_{#1}}
\newcommand{\SO}{\Sec\ONs}
\newcommand{\SV}[1]{\DE{\CV{#1}}}
\newcommand{\Sec}[1]{\Gamma_{#1}}

\newcommand{\succn}{\succ_*}

\newcommand{\succx}{\succ_0}
\newcommand{\sumsep}{\hs{-3pt}}
\newcommand{\sumsepm}{\hs{-4pt}}
\newcommand{\Top}{{\T'_{\sz}}}

\newcommand{\tlZ}{{\_\wt\Z}}
\newcommand{\VF}[1]{\CV{\SF{#1}}}
\newcommand{\vft}[2]{{#1}_{\wt z_{#2}}}
\newcommand{\vftp}[2]{{#1}_{\wt z'_{#2}}}
\newcommand{\vfz}[2]{{#1}_{z_{#2}}}
\newcommand{\vfze}[1]{{#1}_{\Z'_\star}}
\newcommand{\vfzp}[2]{{#1}_{z'_{#2}}}

\newcommand{\XIp}{{\XI'}}
\newcommand{\XIq}{{\XI''}}
\newcommand{\Xop}{{\X'_{\sz}}}
\newcommand{\YV}[1]{\Y_{\CV{#1}}}
\newcommand{\Ya}{{\Y_*}}
\newcommand{\Yo}{{\Y_{\sz}}}
\newcommand{\Yop}{{\Y'_{\sz}}}
\newcommand{\Za}{{\Z_*}}
\newcommand{\Zap}{{\Z'_*}}
\newcommand{\Zo}{{\Z_{\sz}}}
\newcommand{\Zop}{{\Z'_{\sz}}}
\newcommand{\Zoq}{{\Z''_{\sz}}}

\newcommand{\CC}{\mathcal{C}}
\newcommand{\CF}[1]{\mathcal{F}_{#1}}
\newcommand{\cF}{\mathcal{F}}
\newcommand{\CG}[1]{\mathcal{G}_{#1}}
\newcommand{\CK}{\mathcal{K}}
\newcommand{\CI}{\mathcal{I}}
\newcommand{\CL}[1]{\mathcal{L}_{#1}}
\newcommand{\CT}[1]{\mathcal{T}_{#1}}
\newcommand{\CV}[1]{\mathcal{V}_{#1}}

\newcommand{\iI}[1]{\textit{I}_{#1}}
\newcommand{\iJ}{\textit{J}}

\newcommand{\lS}{\_\bs}

\newcommand{\lX}{{\_\X}}

\newcommand{\lZ}{{\_\Z}}

\newcommand{\lXp}{{\_\Xp}}

\newcommand{\lAl}{\_\hs{-0.5pt}\Al}

\newcommand{\lTM}{{\_\TM}}

\newcommand{\lTQ}{{\_\hs{-1.2pt}\TQ}}

\newcommand{\Xp}{{\X'}}
\newcommand{\Xq}{{\X''}}
\newcommand{\Yp}{{\Y'}}
\newcommand{\Zp}{{\Z'}}
\newcommand{\Zq}{{\Z''}}

\newcommand{\sF}{\mathsf{F}}
\newcommand{\sH}{\mathsf{H}}
\newcommand{\sP}{\mathsf{P}}
\newcommand{\sU}{\mathsf{U}}

\newcommand{\TA}{\tens{A}}
\newcommand{\TB}{\tens{B}}
\newcommand{\TC}{\tens{C}}
\newcommand{\TD}{\tens{D}}
\newcommand{\TE}{\tens{E}}
\newcommand{\TF}{\tens{F}}
\newcommand{\TG}{\tens{G}}
\newcommand{\Th}{\tens{H}}
\newcommand{\TM}{\tens{M}}
\newcommand{\TN}{\tens{N}}
\newcommand{\TO}{\tens{0}}
\newcommand{\TP}{\tens{P}}
\newcommand{\TQ}{\tens{Q}}
\newcommand{\TS}{\tens{S}}
\newcommand{\TT}{\tens{T}}

\newcommand{\LAm}{\tens{\Lambda}}
\newcommand{\TOm}{\tens{\Omega}}

\newcommand{\Al}{{\bm{\alpha}}}
\newcommand{\Be}{{\bm{\beta}}}
\newcommand{\Ga}{{\bm{\gamma}}}
\newcommand{\La}{{\bm{\lambda}}}
\newcommand{\Mu}{{\bm{\mu}}}
\newcommand{\XI}{{\bm{\xi}}}
\newcommand{\Ze}{{\bm{\zeta}}}

\newcommand{\A}{{\mathbf{a}}}
\newcommand{\B}{{\bm{b}}}
\newcommand{\D}{{\bm{d}}}
\newcommand{\E}{{\bm{e}}}
\newcommand{\F}{{\bm{f}}}
\newcommand{\bl}{{\bm{l}}}
\newcommand{\M}{{\bm{m}}}
\newcommand{\N}{{\bm{n}}}
\newcommand{\R}{{\bm{r}}}
\newcommand{\bs}{{\bm{s}}}
\newcommand{\T}{{\bm{t}}}
\newcommand{\U}{{\bm{u}}}
\newcommand{\V}{{\bm{v}}}
\newcommand{\W}{{\bm{w}}}
\newcommand{\X}{{\bm{x}}}
\newcommand{\Y}{{\bm{y}}}
\newcommand{\Z}{{\bm{z}}}

\newcommand{\crD}{\mathscr{D}}
\newcommand{\crC}{\mathscr{C}}

\newsavebox{\Dlu}
\sbox{\Dlu}{$\Delta\hs{-1pt}$\_\hs{-1pt}}
\newcommand{\Dl}{\usebox{\Dlu}}

\newsavebox{\permt}
\sbox{\permt}{$^\sigma\hs{-2.5pt}$}
\newcommand{\per}{^\sigma\hs{-2.5pt}}
\newcommand{\perl}{\usebox{\permt}\hs{-1mm}}
\newcommand{\perm}{\usebox{\permt}}

\newsavebox{\ssz}
\sbox{\ssz}{\raisebox{0.3pt}[0pt]{$\scriptscriptstyle{0}$}}
\newcommand{\sz}{\usebox{\ssz}}

\begin{document}

\title{A Newtonian and Weierstrassian Approach\\
to Local Resolution of Singularities\\
in Characteristic Zero\\
\footnotetext{\textit{Email:} \texttt{masm.math@163.com}}
\footnotetext{\textit{Address:} LMIB \& School of Mathematics,
Beihang University, No. 37 Xue Yuan Road, Hai Dian District,
Beijing 100191, China.} \footnotetext{2000 \textit{Mathematics
Subject Classification.} Primary 14B05; 14J17; 14E15; 32S05.}
\footnotetext{\textit{Key words and phrases:} resolution of
singularities; Newton polyhedron; monomial transformation; branch
point; Weierstrass preparation theorem; partition of unity.}}
\author{Sheng-Ming Ma}
\date{}
\maketitle

\begin{abstract}
This paper formulates an elementary algorithm for resolution of singularities in a neighborhood of a singular point over a field of characteristic zero.
The algorithm is composed of finite sequences of Newton polyhedra and monomial transformations and based on Weierstrass preparation theorem.
This approach entails such new methods as canonical reduction and synthesis of monomial transformations as well as latency and revival of primary variables.
The orders of primary variables serve as the decreasing singularity invariants for the algorithm albeit with some temporary increases.
A finite partition of unity in a neighborhood of the singular point is constructed in an inductive way depending on the topological constraint imposed by Euler characteristic of the normal vector set of Newton polyhedron.
\end{abstract}

\tableofcontents

\section{Introduction}

It is well known that an assemblage of algebraic geometers have made preeminent contributions to the problem of resolution of singularities.
The seminal works of Zariski, Walker, Abhyankar and Lipman culminated in the monumental achievement of Hironaka \cite{Hi} who first solved the problem over a field of characteristic zero around 1964.
The first constructive and simplified proofs were formulated by Bierstone and Milman \cite{BM1} and Villamayor \cite{V}.
The subsequent simplifications were accomplished by Bierstone and Milman \cite{BM2}, Encinas and Villamayor \cite{EV}, Encinas and Hauser \cite{EH}, Cutkosky \cite{C}, W\l odarczyk \cite{W} and Koll\'ar \cite{K}.
This paper formulates a resolution algorithm that is composed of finite sequences of Newton polyhedra and monomial transformations in a neighborhood of a singular point over a field of characteristic zero.
The algorithm is elementary and local in nature and not written in the modern language of algebraic geometry.
The motivation of the paper is to demonstrate that the original insight of Newton and ingenious preparation of Weierstrass can lead to another approach to the problem.
Nonetheless the guidance built upon the endeavors of the aforementioned algebraic geometers, especially that of Hironaka, can never be overestimated when the author was entangled in technical intricacies.

Each step of the resolution algorithm in the paper starts with a procedure called the preliminary and Weierstrass reductions.
The preliminary reduction amounts to a linear modification of the function under investigation, which is ensued by the Weierstrass reduction entailing an invocation of Weierstrass preparation theorem and completion of perfect power.
After these reductions the function or a part of it is represented in a Weierstrass form with an apex such as \eqref{PrimaryForm}, \eqref{MonoWeierstrass} or \eqref{LatentWeierstrassForm}.
The order of the apex represents that of the function and is naively defined as the singularity invariant called the singularity height in the paper.
The termination of the resolution algorithm depends on the ultimate strict decrease of the singularity height albeit with some temporary increases.

Associated with each refined vertex cone of the Newton polyhedron of a Weierstrass form as in \eqref{PrimaryForm}, a monomial transformation is defined that partially factorizes the form into exceptional factors and a partial transform.
Instead of studying the improvement of singularities of the partial transform directly, we reduce the monomial transformation into a canonical form so that it is an identity map on the exponents of the non-exceptional variables.
In this way the singularity height can measure the improvement of singularities at a reduced branch point directly based on such exponential identities as \eqref{ExponIdentity}.

The singularities of branch points can be classified into regular and irregular ones.
The former one can be reduced to examining the order of a univariate polynomial in terms of the primary variable based on the canonical reduction, which is elaborated in Section \ref{Section:RegularSing};
whereas the later one represents a temporary but essential increase of the singularity height.
In resolution of irregular singularities in Section \ref{Section:SecondarySing}, a method called latency implementation is contrived to define a latent variable and thus reduce the dimension of the problem.
This approach is similar to the blowup of the axis of the primary variable in the traditional approach but is from the perspective of Newton polyhedron.
The involved subtlety is the method to apply Weierstrass preparation theorem in the presence of the latent variable such that an effective Weierstrass form can be established.
This is achieved by a partial gradation of the function into a reducible part and remainder one with the former one being reduced to a Weierstrass form.
Moreover, the residual order strictly decreases from the prior singularity height after the revival of the latent variable.
This is due to the invariance of the exponent of the latent variable throughout the resolution process for irregular singularities.

The singularity height might have a temporary but essential increase in the case of inconsistent singularities, which is defined in Definition \ref{Def:Consistency-1} and addressed in Section \ref{Section:Inconsistency}.
In this case the primary variable becomes latent and consequently it is ineffectual to invoke the canonical reduction like in Section \ref{Section:RegularSing}.
A new primary variable is defined instead in Section \ref{Section:Inconsistency} through the procedure of latent preliminary and Weierstrass reductions.
Weierstrass preparation theorem is invoked through a procedure of partial gradation called the latent gradation.
The intricate method in this section is the decomposition and synthesis of reduced exponential matrices into a synthetic one such as \eqref{CstnSyntheticMatrix} or \eqref{IncstnSyntheticMatrix} that acts on the exponents of the latent primary variables directly.
In this way either the latent primary variables sustain their latency, in which case the degree of nested latency increases, or one of them is revived through a canonical reduction, in which case resumes the strict decrease of the singularity height associated with the latent primary variable.
In Section \ref{Section:Inconsistency} it is shown that the former case eventually leads to a resumed strict decrease of the singularity height associated with one of the latent primary variables as well.

Section \ref{Section:Generic} serves to elucidate the case when the irregular and inconsistent singularities are intermingled.
Moreover, the above resolution algorithm for the hypersurface case can be applied to the generators of an ideal simultaneously in the generic case since the traditional resolution centers are selected by Newton polyhedra ``automatically".

In Section \ref{Section:Partition} a finite partition of unity is constructed in a dominant neighborhood of each branch point.
The proof involves a calculation of Euler characteristic of the sectional polyhedra of the vertex cones of the pertinent Newton polyhedron.
As a result, it suffices to study the part of an exceptional branch inside a compact set because with each common facet of two refined vertex cones there associated a pair of conjugate variables that are reciprocal to each other.
The compactness ensures that a finite number of dominant neighborhoods of branch points can cover all the exceptional branches.

In Section \ref{Section:SurfaceExample} we investigate the simple case of analytic surfaces that exemplifies the algorithm for resolution of singularities in the paper.

Let us use $\BB F$ to denote a field of characteristic zero and $\DF n$ the $n$-dimensional affine space over $\BB F$ throughout the paper.
Let us denote as usual the sets of complex numbers, real numbers, rational
numbers, integers and natural numbers as $\BB C$,
$\BB R$, $\BB Q$, $\BB Z$ and $\BB N$
respectively.
In particular, we adopt the convention that $0\in\BB N$, $\BB N^*=\BB N\setminus\{0\}$ and $\BB F^*=\BB F\setminus\{0\}$ as well as $\Rp:=\{r\in\BB R\colon r\ge 0\}$.
In this paper the vectors are typeset as lowercase letters in boldface and matrices as slanted uppercase letters.
In this way a matrix $\TM$ written as $[\ME\V 1n]$ means that it has column vectors $\LE\V 1n$.
The determinant of a matrix $\TM$ is denoted as $\det\TM$.
The variables and their exponents are typeset as lowercase English and Greek letters in boldface respectively such as $\X=(\LE x 1n)$ and $\Al=(\LE\alpha 1n)\in\BB N^n$.
Let us use the notations $\BB F[\X]=\BB F[\LE x1n]$ and $\BB F\{\X\}=\BB F\{\LE x1n\}$ to denote the respective algebras of polynomials and convergent power series in the paper.
Let us also use the notation $\sigma$ to denote a permutation of the index set $\dotc 1n$.
A $k$ by $k$ unit matrix is usually written as $\TE_k$, or simply as $\TE$ when its dimensions are clear from the context.
The Euclidean inner product of two vectors $\V$ and $\W$ is denoted as $\langle\V,\W\rangle$.
The notation $|A|$ stands for the number of elements in a set $A$.
The maximal number of linearly independent vectors in a vector set $V$ is given the usual notation $\rank (V)$.
The vector set $\{\LE\E 1n\}:=\dotc{(1,\dotsc,0)}{(0,\dotsc,1)}$ denotes the standard basis of $\BB R^n$.
The zero vector is usually written as $\bz$ whose dimension should be either clear from the context or stated otherwise.
For simplicity of notations in this paper, two functions with different variables are deemed as different functions albeit they might share the same name.

Please note that a few terminologies like ``consistency" and ``localization" in the paper are different from those well known in the literature either because they stand for different objects of study or due to the elementary nature of the algorithm in the paper.
In order to avoid unaesthetic page breaks, large matrices are not at their optimum places occasionally and readers should refer to them by their tag numbers.
Moreover, the author strives to make a reference list that is both complete and relevant and were there any regretful omissions, the authors of the existing literature have the absolute priorities.

\section{Resolution of singularities via canonical reduction}

\subsection{Partial resolution of singularities}

\begin{definition}\label{Def:NewtonPolyhedron}
{\upshape ($\X^\Al$; $\X^\TM$; $\supp(f)$; $\ord (f)$; $\deg (f)$; Newton polyhedron $\NP(f)$; normal vector set $\PI\NP$; $\facet (\V)$; $\face (\W)$)}

With $\Al:=(\LE\alpha 1n)\in\BB N^n$ and $\X:=(\LE x 1n)$, let us denote $\X^\Al:=x_1^{\alpha_1}\dotsm x_n^{\alpha_n}$ as a monomial henceforth.

Let $\TM$ be an $n$ by $n$ matrix with elements in $\BB N$.
If its $j$-th row vector is denoted as $\M_j$ for $1\le j\le n$, then $\X^\TM:=(\LEs \X{}{}{\M_1}{\M_n})$.

For a function $f(\X):=\sum_\Al c_\Al\X^\Al\in\BB F\{\X\}$ with $\Al\in\BB N^n$ and $c_\Al\in\BB F$, let us define $\supp(f):=\{\Al\in\BB N^n\colon c_\Al\ne 0\}$.
Further, $\ord (f):=\min\{\sum_{j=1}^n\alpha_j\colon\Al\in\supp(f)\}$ and $\deg (f):=\max\{\sum_{j=1}^n\alpha_j\colon\Al\in\supp(f)\}$ when $f\in\BB F[\X]$.

The \emph{Newton polyhedron} of $f$, denoted as $\NP(f)$, is defined as the convex hull of the set $\supp(f)+\Rp^n$, i.e., $\NP(f):=\conv(\supp(f))+\Rp^n$.

Every facet of a Newton polyhedron $\NP$ has a normal vector in $\DN n$ whose nonzero components are coprime.
The set of all such \emph{normal} vectors of $\NP$ is denoted as $\PI\NP$.
The \emph{facet} of a normal vector $\V$ is denoted as $\facet (\V)$.

For $\forall\W\in\DR$, let us denote the \emph{face} of $\NP$ associated with the vector $\W$ as follows.
\begin{equation}\label{FaceEq}
\face (\W):=\{\Al\in\NP\colon\langle\W,\Al\rangle\le
\langle\W,\Be\rangle~\text{for}~\forall\Be\in\NP\}.
\end{equation}
\end{definition}

\begin{definition}\label{Def:VertexCone}
{\upshape (Convex polyhedral cone $\CC (V)$; $\CC^\circ(V)$; simplicial cone; facial cone $\CV\sF$ and its generator set $\gv\sF$; vertex cone $\CV\A$ and its generator set $\gv\A$)}

A cone $\CC$ is defined to be a \emph{convex polyhedral cone} generated by a vector set $V=\{\LE\V 1m\}\subseteq\DR$ if $\CC=\bigl\{\sum_{j=1}^m\lambda_j\V_j\colon\lambda_j\in\Rp\bigr\}\setminus\{\bz\}:=\CC (V)$.
If we substitute $\lambda_j\in\BB R_{>0}$ for $\lambda_j\in\Rp$ in this definition, we obtain the definition for $\CC^\circ(V)$.
In particular, $\CC (V)$ is called a \emph{$k$-dimensional} cone if $\rank (V)=k$, and $k$-dimensional \emph{simplicial} cone if $|V|=\rank (V)=k$ with $1\le k\le n$.

let $\sF$ be a face of a Newton polyhedron $\NP$ such that $\sF=\face (\W)$ as in \eqref{FaceEq} with $\W\in\DR$.
The normal vector set $\gv\sF:=\{\V\in\PI\NP\colon\sF\subseteq\facet (\V)\}$ is called the \emph{generator set} of the \emph{facial cone} $\CV\sF:=\CC (\gv\sF)$ associated with the face $\sF$.
In particular, when the face $\sF$ is a vertex $\A$ of the Newton polyhedron $\NP$, which is the intersection of at least $n$ facets of $\NP$, the normal vector set $\gv\A:=\{\V\in\PI\NP\colon\A\in\facet (\V)\}$ is called the \emph{generator set} of the \emph{vertex cone} $\CV\A:=\CC (\gv\A)$ associated with the vertex $\A$.
\end{definition}

\begin{definition}\label{Def:Simplicity}
{\upshape (Projective cone $[\CC]$; cone section $\Sec\CC$; polyhedral cone; cone facet and face; minimal generator set; polyhedral and simplicial fan; convex and simplicial subdivision; simplicial refinement; auxiliary vector)}

The \emph{projective cone} $[\CC]$ associated with a cone $\CC\subseteq\Rp^n$ is defined as the set of vectors $\V\in\CC$ modulo the equivalence relation $\V\sim\lambda\V$ for $\lambda\in\BB R_{>0}$.
Equivalently $[\CC]$ can be defined as the corresponding set of rays in $\CC$ or corresponding subset of the real projective space $\BB R\BB P^{n-1}$.
The \emph{section} $\Sec\CC$ of a cone $\CC$ is defined as the transverse cross section of $\CC$ and the $(n-1)$-dimensional hyperplane with normal vector $\bm{1}:=(1,\dotsc,1)$ such that the equivalent classes of the elements in $\Sec\CC$ modulo the above equivalence relation constitute the projective cone $[\CC]$.
A cone $\CC$ is called a \emph{polyhedral} cone if its section $\Sec\CC$ is a polyhedron.
A subcone $\CC'\subseteq\CC$ is called a \emph{facet} or \emph{face} of a polyhedral cone $\CC$ if its section $\Sec{\CC'}$ is a facet or face of the polyhedron $\Sec\CC$.
The generator set $V$ of a convex polyhedral cone $\CC (V)$ is said to be \emph{minimal} if $\V\notin\CC (V\setminus\{\V\})$ for $\forall\V\in V$.

A $k$-dimensional \emph{polyhedral} (or \emph{simplicial}) \emph{fan} $\CK$ is defined as a finite union of $k$-dimensional convex (or simplicial) polyhedral cones $\CK=\bigcup_{j=1}^l\CC_j$ such that $\CC_i\cap\CC_j$ is a common facet or face of both $\CC_i$ and $\CC_j$ when $\CC_i\cap\CC_j\ne\emptyset$ for $1\le i<j\le l$.
The procedure to partition a polyhedral cone into a polyhedral (or simplicial) fan is called a \emph{convex} (or \emph{simplicial}) \emph{subdivision} of the cone.

Let $V$ be the minimal generator set of a convex polyhedral cone $\CC (V)$.
A \emph{simplicial refinement} of $\CC (V)$ via a finite set of \emph{auxiliary} vectors $W\subset\CC (V)$ is defined as a simplicial subdivision $\CC (V)=\bigcup_{j=1}^l\CC (V_j)$ such that $\bigcup_{j=1}^lV_j=V\cup W$.
\end{definition}

The simplicial subdivision is well known in toric geometry whose terminologies are not a prerequisite to understand the paper.
Moreover, the following lemma is a classical result and presented here solely for the purpose of completeness.

\begin{lemma}\label{Lemma:Refinement}
Let $\CV\A$ be a vertex cone associated with a vertex $\A$ of a Newton polyhedron $\NP$.
Then there exists a simplicial refinement of $\CV\A$ via a finite set $\Lambda_\A\subseteq\CV\A$ of auxiliary vectors in $\DN n$ such that if $\CC (\LE\V 1n)$ is a simplicial cone in the simplicial refinement, then $|\det [\ME\V 1n]|=1$.
\end{lemma}
\begin{proof}
Let us construct the auxiliary vector set $\Lambda_\A$ one by one inductively starting with a simplicial subdivision of $\CV\A$.
If $\CC (\LE\V 1n)$ is a simplicial cone in the simplicial subdivision that satisfies $|\det [\ME\V 1n]|>1$, then the generator set $\{\LE\V 1n\}\subset\CV\A$ is not a basis for the integral lattice $\BB Z^n$.
Hence there exists a vector $\W\in\DN n$ that satisfies $\W=\sum_{j=1}^n\lambda_j\V_j$ with $\lambda_j\in [0,1)$ for $1\le j\le n$.
We add $\W$ into $\Lambda_\A$ and make a simplicial refinement of $\CC (\LE\V 1n)$ via $\{\W\}$.
Each simplicial cone in the simplicial refinement, e.g., $\CC (\W,\LE\V 2n)$, satisfies
\[
1\le |\det [\W~\ME\V 2n]|<|\det [\ME\V 1n]|.
\]
The conclusion of the lemma follows from a decreasing
induction on the absolute values of the integer-valued determinants.
\end{proof}

The simplicial refinement of $\CV\A$ in Lemma \ref{Lemma:Refinement} assumes implicitly that its generator set $\gv\A$ is minimal, which is required in Definition \ref{Def:Simplicity} and shall be proved in Lemma \ref{Lemma:IntersectLemma} \eqref{item:Minimality}.

\begin{definition}\label{Def:ExponentialMatrix}
{\upshape (Refined vertex cone $\CV\A$; exponential matrix $\TM$; monomial transformation $\CT\TM$)}

A simplicial cone $\CC (\LE\V 1n)$ obtained through a simplicial refinement of a vertex cone $\CV\A$ as in Lemma \ref{Lemma:Refinement} and satisfying $|\det [\ME\V 1n]|=1$ is called a \emph{refined} vertex cone of $\CV\A$.
Let us abuse the notations a bit and still denote a refined vertex cone of $\CV\A$ as $\CV\A$ henceforth.

Let $\CV\A=\CC (\LE\V 1n)$ be a refined vertex cone of $\NP$.
The $n$ by $n$ matrix $\TM:=[\ME\V 1n]$ satisfying $\det\TM=1$ with column vectors $\{\LE\V 1n\}=\gv\A$ is called the \emph{exponential matrix} associated with $\CV\A$.
The \emph{monomial transformation} $\CT\TM$ is defined as a coordinate change
$\X=\CT\TM(\Y):=\Y^\TM$ as in Definition \ref{Def:NewtonPolyhedron} with $\X=(\LE x 1n)$ and $\Y=(\LE y 1n)$.
\end{definition}

Let $\M_j$ and $\V_j$ denote the $j$-th row and column vectors of the exponential matrix $\TM$ for $1\le j\le n$.
Then we say that $\M_j$ and $\V_j$ correspond to the variables $x_j$ and $y_j$ respectively for $1\le j\le n$ since $x_j=\Y^{\M_j}$ and the components of $\V_j$ are the exponents of $y_j$ in $\Y^\TM$.

Let us start with a hypersurface $f(\X)=0$ with $\ord (f)>1$ that is represented in the following form
\begin{equation}\label{PrePrimaryForm}
f(\X)=c_\D x_1^d+\sum_{\Al\in\supp(f)\setminus\D}\sumsep c_\Al\X^\Al=0
\end{equation}
via a non-degenerate linear transformation $\CL\X$ such that $d=\ord (f)>1$.
Here $\D:=(d,\bz)$ denotes the exponent of the term $c_\D x_1^d$ with $c_\D\in\BB F^*$.
After invoking Weierstrass preparation theorem and completing perfect power, we can abuse the names of the variables $\X$ a bit and further represent $f(\X)$ in the following form:
\begin{equation}\label{PrimaryForm}
f(\X)=\biggl[x_1^d+\sum_{j=2}^dc_j(\LE x2n)x_1^{d-j}\biggr]
(c+r(\X)):=w(\X)(c+r(\X)),
\end{equation}
where $c\in\BB F^*$ and $c_j(\bz)=r(\bz)=0$ for $1<j\le d$.

\begin{definition}
{\upshape (Apex $\D$; apex form; singularity height $d$; Weierstrass form; Weierstrass polynomial; redundant function; primary variable $x_1$; linear modification $\CL\X$; preliminary and Weierstrass reductions)}

The exponent $\D=(d,\bz)$ of the term $c_\D x_1^d$ in \eqref{PrePrimaryForm} that satisfies $d=\ord (f)$ is called the \emph{apex} of $f(\X)$.
The representation form \eqref{PrePrimaryForm} is called an \emph{apex} form of $f(\X)$.
The order $d=\ord (f)$ is called the singularity \emph{height} of the apex form $f(\X)$.

The representation form \eqref{PrimaryForm} is called the \emph{Weierstrass} form of $f(\X)$ with the functions $w(\X)$ and $r(\X)$ in \eqref{PrimaryForm} called its \emph{Weierstrass} polynomial and \emph{redundant} function respectively.

The \emph{primary} variable refers to the variable $x_1$ of the Weierstrass polynomial $w(\X)$ in \eqref{PrimaryForm}.
The non-degenerate linear transformation $\CL\X$ to set up the apex $\D$ and primary variable $x_1$ in \eqref{PrePrimaryForm} is called a linear \emph{modification}.

The linear modification $\CL\X$ constitutes the procedure of \emph{preliminary} reduction; the invocation of Weierstrass preparation theorem and completion of perfect power constitute the \emph{Weierstrass} reduction.
\end{definition}

For the Newton polyhedron $\NP(w(\X))$ of the Weierstrass polynomial $w(\X)$ in \eqref{PrimaryForm}, let $\CV\A=\CC (\LE\V 1n)$ be a refined vertex cone of $\NP(w(\X))$ associated with a vertex $\A$ and $\TM=[\ME\V 1n]$ the exponential matrix associated with $\CV\A$.
Let us define $\supp(\TM):=\bigcup_{j=1}^n(\supp(w)\cap\face (\V_j))$.
For convenience, we reorganize the terms of $w(\X)$ in \eqref{PrimaryForm} as follows.
\begin{equation}\label{VertexForm}
w(\X)=c_\A\X^\A+\sum_{\Al\in\supp(\TM)\setminus\{\A\}}
\sumsep c_\Al\X^\Al
+\sum_{\Be\in\supp^c(\TM)}\sumsep c_\Be\X^\Be,
\end{equation}
where $c_\A\in\BB F^*$ since $\A\in\supp(w)$ is the vertex of $\NP(w(\X))$ associated with $\CV\A$ as above.
The complementary support $\supp^c(\TM):=\supp(w)\setminus\supp(\TM)$.

The monomial transformation
\begin{equation}\label{MonomialTransfmtn}
\X=\CT\TM(\Y)=\Y^\TM
\end{equation}
transforms $w(\X)$ in \eqref{VertexForm} into a new function $w_\TM(\Y):=w(\CT\TM(\Y))$ as following:
\begin{equation}\label{TotalTransform}
\begin{aligned}
w_\TM(\Y)&=\Y^{\A\cdot\TM}\cdot
\biggl(c_\A+\sum_{\Al\in \supp(\TM)\setminus\{\A\}}
\sumsep c_\Al\Y^{(\Al-\A)\cdot\TM}
+\sum_{\Be\in\supp^c(\TM)}\sumsep c_\Be
\Y^{(\Be-\A)\cdot\TM}\biggr)\\
&:=\Y^{\A\cdot\TM}\cdot p(\Y),
\end{aligned}
\end{equation}
where $\A\cdot\TM$ denotes the multiplication of the row vector $\A$ and matrix $\TM$, and the same for $(\Al-\A)\cdot\TM$ and $(\Be-\A)\cdot\TM$.

\begin{definition}\label{Def:ProperTransform}
{\upshape (Vertex form; total transform; partial transform $p(\Y)$; exceptional factors $\Y^{\A\cdot\TM}$; partial factorization; exceptional and non-exceptional variables $\Y_0$ and $\Y_*$; exceptional and non-exceptional index sets $\iI\Yo$ and $\iI{\Y_*}$; codimension and canonical form of $\iI\Yo$; primary row vector $\M_1$; remnant exponential submatrix $\lTM$)}

The representation form \eqref{VertexForm} is called the \emph{vertex} form of $w(\X)$ associated with a refined vertex cone $\CV\A$.
The function $w_\TM(\Y)$ and factor $p(\Y)$ in \eqref{TotalTransform} are called the \emph{total} and \emph{partial} transforms of $w(\X)$ under the monomial transformation $\CT\TM$ in \eqref{MonomialTransfmtn} respectively.
The factors $\Y^{\A\cdot\TM}$ in \eqref{TotalTransform} are called the \emph{exceptional} factors.
The factorization of the total transform $w_\TM(\Y)$ into the partial transform $p(\Y)$ and exceptional factors $\Y^{\A\cdot\TM}$ in \eqref{TotalTransform} is called a \emph{partial} factorization of the total transform $w_\TM(\Y)$.

We discriminate between two complementary subsets of variables among the variables $\Y=(\LE y1n)$ and call them the \emph{exceptional} and \emph{non-exceptional} variables that are denoted as $\Y_0$ and $\Y_*$ respectively.
Their corresponding index sets are called the \emph{exceptional} and \emph{non-exceptional} index sets and denoted as $\iI\Yo$ and $\iI{\Y_*}$ respectively.

In particular, the parameter $k:=|\iI{\Y_*}|=n-|\iI\Yo|$ satisfying $0\le k<n$ is called the \emph{codimension} of $\iI\Yo$.
An exceptional index set $\iI\Yo$ with codimension $k$ can always be reduced to a \emph{canonical} form $\dotc{k+1}n$ by a permutation and relabeling of the variables $\Y$ and their corresponding column vectors of the exponential matrix $\TM$ as below Definition \ref{Def:ExponentialMatrix}.

The first row vector $\M_1$ of an exponential matrix $\TM$ corresponding to the primary variable $x_1$ is called the \emph{primary} row vector of $\TM$ since $x_1=\Y^{\M_1}$ as per the monomial transformation $\CT\TM$ in \eqref{MonomialTransfmtn}.
The other row vectors of $\TM$ constitute the \emph{remnant} exponential submatrix of $\TM$ that is denoted as $\lTM$.
\end{definition}

\begin{definition}\label{Def:ExceptionalSupport}
{\upshape (Submatrix $\TA_{\iI{}\times\iJ}$; column submatrix $\Ev\TA{}$; exceptional submatrix $\Evv\TM\Yo$; exceptional and non-exceptional column submatrices $\Ev\TM\Yo$ and $\Ec\TM\Ya$; proper transform $p(\Y_*,\bz)$; exceptional branch $\EI\TM\Yo$; branch point $\R=(\R_*,\bz)$; exceptional support $\es\TM\Yo$)}

The submatrix $\TA_{\iI{}\times\iJ}$ of a matrix $\TA$ is composed of the elements of $\TA$ whose indices are in $\iI{}\times\iJ$.
The \emph{column} submatrix $\Ev\TA{}$ of $\TA$ is composed of the column vectors of $\TA$ whose indices are in $\iI{}$.

In particular, the submatrix $\Evv\TM\Yo$ of an exponential matrix $\TM$ is called the \emph{exceptional} submatrix of $\TM$ with respect to an exceptional index set $\iI\Yo$.
The column submatrices $\Ev\TM\Yo$ and $\Ev\TM\Ya$ are called the \emph{exceptional} and \emph{non-exceptional} column submatrices of $\TM$ with respect to $\iI\Yo$ and $\iI\Ya$ respectively.

For the partial transform $p(\Y)$ in \eqref{TotalTransform}, the function $p(\Y_*,\bz)$ is called the \emph{proper} transform of $w(\X)$ under $\CT\TM$ with respect to $\iI\Yo$.
The following set
\begin{equation}\label{ExceptionalBranch}
\EI\TM\Yo:=\{(\Y_*,\bz)\colon p(\Y_*,\bz)=0,\,\Y_*\in\DFs k\}
\end{equation}
is called the \emph{exceptional branch} associated with $\CT\TM$ and $\iI\Yo$.
A point $\R:=(\R_*,\bz)\in\EI\TM\Yo$ is called a \emph{branch point} of the origin $\X=\bz$.

The set of exponents $\es\TM\Yo:=\bigcap_{j\in\iI\Yo}(\supp(w)\cap\face (\V_j))$ is called the \emph{exceptional} support of $w(\X)$ with respect to $\TM=[\ME\V 1n]$ and $\iI\Yo$.
\end{definition}

Let us reorganize the terms in the partial transform $p(\Y)$ in \eqref{TotalTransform} as per the exceptional support $\es\TM\Yo$ and it is easy to deduce that the proper transform $p(\Y_*,\bz)$ as in Definition \ref{Def:ExceptionalSupport} is as follows.
\begin{equation}\label{ProperTransform-M}
p(\Y_*,\bz)=c_\A+\sum_{\Al\in \es\TM\Yo\setminus\{\A\}}\sumsep c_\Al\Y_*^{(\Al-\A)\cdot\TM_{\iI{\Y_*}}},
\end{equation}
where $\TM_{\iI{\Y_*}}$ denotes the non-exceptional column submatrix of $\TM$ with respect to $\iI{\Y_*}$ as in Definition \ref{Def:ExceptionalSupport}.

\begin{definition}\label{Def:PartialResolution}
{\upshape (Localized partial and proper transforms $\wt p(\wt\Y)$ and $p_*(\wt\Y_*)$; partial resolution of singularities)}

The partial transform $p(\Y)$ in \eqref{TotalTransform} can be localized as $p(\wt\Y_*+\R_*,\wt\Y_0):=\wt p(\wt\Y)$ with $\wt\Y:=(\wt\Y_*,\wt\Y_0):=(\Y_*-\R_*,\Y_0)$ in a neighborhood of a branch point $\R=(\R_*,\bz)$.
Accordingly the proper transform $p(\Y_*,\bz)$ in \eqref{ProperTransform-M} can be \emph{localized} as $p_*(\wt\Y_*):=\wt p(\wt\Y_*,\bz)$.

The preliminary and Weierstrass reductions, monomial transformation, partial factorization as well as localization at a branch point constitute the procedure of \emph{partial} resolution of singularities.
\end{definition}

\subsection{Canonical reduction of monomial transformations}
\label{Subsection:CanonicalReduction}

It readily follows that when the exceptional index set $\iI\Yo$ has codimension $k=0$, i.e., $\iI\Yo=\dotc 1n$, the exceptional variables $\Y_0=(\LE y1n)$ and the proper transform equals $p(\bz)=c_\A\in\BB F^*$ in \eqref{ProperTransform-M}.
Thus the exceptional branch $\EI\TM\Yo=\emptyset$ and the proper transform is nonsingular in this case.

When the exceptional index set $\iI\Yo$ has positive codimension $k>0$, that is, $\iI\Yo=\dotc{k+1}n$, let us consider a monomial transformation $\CT\TM$ associated with an exponential matrix $\TM=[\ME\V 1n]$ as in the form of \eqref{MonomialMatrix} whose exceptional submatrix $\Evv\TM\Yo$ is non-degenerate and satisfies $\det\Evv\TM\Yo\ne 0$.
Here $\V_j=(\LEs vjj1n)$ for $1\le j\le n$.

\begin{align}\label{MonomialMatrix}
\TM_{n\times n}&=\left[
        \begin{array}{ccccc@{}cccc}
        v_1^{1}     &       \cdots      &       v_j^1       &
        \cdots      &       v_k^1       &       \vline      &
        v_{k+1}^1   &       \cdots      &       v_n^1               \\
        \vdots      &       \ddots      &       \vdots      &
        \ddots      &       \vdots      &       \vline      &
        \vdots      &       \ddots      &       \vdots              \\
        \cline{3-3}\cline{7-9}
        v_1^i       &       \cdots      &
        \multicolumn{1}{|c|}{v_j^{i}}   &       \cdots      &
        v_k^i       &       \vline      &
        \multicolumn{1}{|c}{v_{k+1}^{i}}        &           \cdots
        &           \multicolumn{1}{c|}{v_n^i}                      \\
        \cline{3-3}\cline{7-9}
        \vdots      &       \ddots      &       \vdots      &
        \ddots      &       \vdots      &       \vline      &
        \vdots      &       \ddots      &       \vdots              \\
        v_1^k       &       \cdots      &       v_j^k       &
        \cdots      &       v_k^k       &       \vline      &
        v_{k+1}^k   &       \cdots      &       v_n^k               \\
        \hline      &       &           &       &           &
        \vline      &       &           &                           \\ [-2ex]
        \cline{3-3} \cline{7-9}
        v_1^{k+1}   &       \cdots      &
        \multicolumn{1}{|c|}{v_j^{k+1}} &       \cdots      &
        v_k^{k+1}   &       \vline      &
        \multicolumn{1}{|c}{v_{k+1}^{k+1}}      &           \cdots
        &           \multicolumn{1}{c|}{v_n^{k+1}}                  \\
        \vdots      &       \ddots      &
        \multicolumn{1}{|c|}{\vdots}    &       \ddots      &
        \vdots      &       \vline      &
        \multicolumn{1}{|c}{\vdots}     &       \ddots      &
        \multicolumn{1}{c|}{\vdots}                                 \\
        v_1^n       &       \cdots      &
        \multicolumn{1}{|c|}{v_j^n}     &       \cdots      &
        v_k^n       &       \vline      &
        \multicolumn{1}{|c}{v_{k+1}^n}  &       \cdots      &
        \multicolumn{1}{c|}{v_n^n}                                  \\
        \cline{3-3} \cline{7-9}
        \end{array}
    \right]\allowdisplaybreaks[4]\\ \label{SplitMonomialMatrix}
&:=\left[
        \begin{array}{l|l}
            \TA_{k\times k}&\TB_{k\times{(n-k)}}\\
            \hline
            \TC_{{(n-k)}\times k}&\TD_{{(n-k)}\times{(n-k)}}
        \end{array}
    \right]
\end{align}

Based on the partition of the exponential matrix $\TM$ in \eqref{SplitMonomialMatrix}, let us define a new exponential matrix $\TN$ as in \eqref{CanonicalMatrix}, in which $\E_j$ denotes the $j$-th standard basis vector of $\BB R^n$ for $1\le j\le k$.
The submatrices $\TE_{k\times k}$ and $\TO_{{(n-k)}\times k}$ of $\TN$ denote the unit matrix and zero matrix respectively.
In particular, the submatrix $\TD_{{(n-k)}\times{(n-k)}}$ denotes the exceptional submatrix $\Evv\TM\Yo$ and satisfies $\det\TD\ne 0$.

\begin{equation}\label{CanonicalMatrix}
\TN_{n\times n}:=[\ME\E 1k~\ME\V{k+1}n]=
    \left[
        \begin{array}{l|l}
            \TE_{k\times k}&\TB_{k\times{(n-k)}}\\
            \hline
            \TO_{{(n-k)}\times k}&\TD_{{(n-k)}\times{(n-k)}}
        \end{array}
    \right].
\end{equation}

The monomial transformation $\CT\TN$ associated with the new exponential matrix $\TN$ in \eqref{CanonicalMatrix} is defined as
\begin{equation}\label{NormalTransformation}
\X=\CT\TN(\Z):=\Z^\TN.
\end{equation}

The exceptional index set $\iI\Zo$ of the new variables $\Z$ in \eqref{NormalTransformation} satisfies $\iI\Zo=\iI\Yo$ such that the non-exceptional and exceptional variables of $\Z$ are written as $\Z_*:=(\LE z1k)$ and $\Z_0:=(\LE z{k+1}n)$ respectively.
Let $\Ga_*$ and $\Ga_0$ denote the respective exponents of $\Z_*$ and $\Z_0$.
Accordingly let us write $\X=(\X_*,\X_0)$ as per $\iI\Zo$ whose respective exponents are denoted as $\Al=(\Al_*,\Al_0)$.
Then the form of the exponential matrix $\TN$ in \eqref{CanonicalMatrix} indicates that $\CT\TN$ is a linear exponential transformation as follows.
\begin{equation}\label{ExponIdentity}
\Ga_*=\Al_*;\quad\Ga_0=\Al\cdot\Ev\TN\Zo
\end{equation}
with $\Ev\TN\Zo=\Ev\TM\Yo$ being the exceptional column submatrix as in Definition \ref{Def:ExceptionalSupport}.

Let $\TF_{ij}$ denote the $(n-k+1)$ by $(n-k+1)$ submatrix of the exponential matrix $\TM$ consisting of the elements in the framed boxes in \eqref{MonomialMatrix} with $1\le i,j\le k$.
Let us further define:
\begin{equation}\label{ReductionMatrix}
\TF_{k\times k}:=\left[
        \begin{array}{ccc}
            \det\TF_{11}&\cdots&\det\TF_{1k}\\
            \vdots&\ddots&\vdots\\
            \det\TF_{k1}&\cdots&\det\TF_{kk}
        \end{array}
    \right].
\end{equation}

\begin{definition}\label{Def:ReductionMatrix}
{\upshape (Reduced exponential matrix $\TN$ and its canonical form; reduced monomial transformation $\CT\TN$; reduction matrix $\TF$)}

For an exponential matrix $\TM$ being partitioned as in \eqref{SplitMonomialMatrix} with respect to the exceptional index set $\iI\Yo$, the matrix $\TN$ as defined in \eqref{CanonicalMatrix} is called a \emph{canonical reduction} of $\TM$ with respect to $\iI\Yo$, or simply the \emph{reduced} exponential matrix with respect to $\TM$ and $\iI\Yo$.
In particular, the form of $\TN$ in \eqref{CanonicalMatrix} satisfying $\det\TD\ne 0$ is called its \emph{canonical} form.
Accordingly we have the \emph{reduced} monomial transformation $\CT\TN$ associated with $\TN$ as defined in \eqref{NormalTransformation}.
The matrix $\TF$ in \eqref{ReductionMatrix} is called the \emph{reduction} matrix from $\CT\TM$ to $\CT\TN$.
\end{definition}

\begin{lemma}\label{Lemma:CanonicalReduction}
Let the exceptional submatrix $\Evv\TM\Yo=\TD$ satisfy $\det\TD\ne 0$ as in the partition of $\TM$ in \eqref{SplitMonomialMatrix}.
\begin{inparaenum}[(a)]
\item\label{item:NondegNormal}
The reduction matrix $\TF$ in \eqref{ReductionMatrix} is non-degenerate, i.e., $\det\TF\ne 0$;
\item\label{item:NormalIdentity} The monomial transformation $\X=\CT\TM(\Y)=\Y^\TM$ in \eqref{MonomialTransfmtn} and its canonical reduction $\X=\CT\TN(\Z)=\Z^\TN$ in \eqref{NormalTransformation} indicate that
\begin{equation}\label{CanonicalReduction}
\Z_*^{\det\TD\cdot\TE_k}=\Y_*^{\TF},
\end{equation}
where $\Z_*=(\LE z1k)$ and $\Y_*=(\LE y1k)$ are the non-exceptional variables with respect to $\iI\Zo=\iI\Yo$.
The matrix $\TE_k$ denotes the $k$ by $k$ unit matrix.
\end{inparaenum}
\end{lemma}
\begin{proof}
\begin{asparaenum}[(a)]
{\parindent=\saveparindent
\item Let us prove by contradiction.
Let $\F_i$ denote the $i$-th row vector of the reduction matrix $\TF$ in \eqref{ReductionMatrix} for $1\le i\le k$.
If $\det\TF=0$, then $\exists\La=(\LE\lambda 1k)\in\BB Q^k\setminus\{\bz\}$ such that $\sum_{i=1}^k\lambda_i\F_i=\bz$, that is, $\sum_{i=1}^k\lambda_i\det\TF_{ij}=0$ for $1\le j\le k$.
According to the definition of the element $\det\TF_{ij}$ of $\TF$ above \eqref{ReductionMatrix}, this is equivalent to the following condition in \eqref{LinearDepency} for $1\le j\le k$.
\begin{equation}\label{LinearDepency}
\det\left[
    \begin{array}{cccc}
    \sum_{i=1}^k\lambda_iv_j^i      ~\,     &       \,~
    \sum_{i=1}^k\lambda_iv_{k+1}^i  ~\,     &       \,~
    \cdots                          ~\,     &       \,~
    \sum_{i=1}^k\lambda_iv_n^i                                  \\
    v_j^{k+1}                       ~\,     &       \,~
    v_{k+1}^{k+1}                   ~\,     &       \,~
    \cdots                          ~\,     &       \,~
    v_n^{k+1}                                                   \\
    \vdots                          ~\,     &       \,~
    \vdots                          ~\,     &       \,~
    \ddots                          ~\,     &       \,~
    \vdots                                                      \\
    v_j^n                           ~\,     &       \,~
    v_{k+1}^n                       ~\,     &       \,~
    \cdots                          ~\,     &       \,~
    v_n^n
    \end{array}
    \right]=0
\end{equation}
As per $\det\TD\ne 0$, the above condition in \eqref{LinearDepency} indicates that $\exists\Mu_j:=(\LEs\mu jj{k+1}n)\in\BB Q^{n-k}$ for $1\le j\le k$ such that
\begin{equation}\label{BaseEq}
\sum_{i=1}^k\lambda_i(v_j^i,\LEs v{k+1}nii)
=\sum_{l=k+1}^n\mu_j^l(v_j^l,\LEs v{k+1}nll),
\end{equation}
from which we can deduce that $\Mu_1=\dotsb=\Mu_k:=\Mu:=(\LE\mu {k+1}n)\in\BB Q^{n-k}$ since $\det\TD\ne 0$.
Thus the equalities in \eqref{BaseEq} can be summarized into:
\[
\sum_{i=1}^k\lambda_i(\LEs v1nii)
=\sum_{l=k+1}^n\mu_l(\LEs v1nll).
\]
Evidently this leads to the contradiction that $\det\TM=0$.

\item The monomial transformation and its canonical reduction $\X=\Y^\TM=\Z^\TN$ indicate that
\begin{equation}\label{PreliminaryIdentity}
\Z_0^\TD=\Y^{[\TC~\TD]}
\end{equation}
with $[\TC~\TD]$ being the submatrix of the exponential matrix $\TM$ in \eqref{SplitMonomialMatrix}.
If we denote the adjoint matrix of $\TD$ as $\TD^*$ such that $\TD^*\cdot\TD=\det\TD\cdot\TE_{n-k}$ and hence $(\Z_0^\TD)^{\TD^*}=\Z_0^{\TD^*\cdot\TD}=\Z_0^{\det\TD\cdot\TE_{n-k}}$ with $\TE_{n-k}$ being the $(n-k)$ by $(n-k)$ unit matrix, then we have
\begin{equation}\label{AdjointConversion}
\Z_0^{\det\TD\cdot\TE_{n-k}}
=\Y^{[\TD^*\cdot\TC~\det\TD\cdot\TE_{n-k}]}
=\Y_*^{\TD^*\cdot\TC}\circ\Y_0^{\det\TD\cdot\TE_{n-k}},
\end{equation}
where the notation $\circ$ denotes the Hadamard (or Schur) product defined as $\V\circ\W:=(v_1\cdot w_1,\dotsc,v_l\cdot w_l)$ with $\V=(\LE v1l)$ and $\W=(\LE w1l)$.
Hence
\begin{equation}\label{NonExceptionalConstraint}
\Z_0^{\det\TD\cdot\TB}=\Y^{[\TB\cdot\TD^*\cdot\TC~\det\TD\cdot\TB]}
=\Y_*^{\TB\cdot\TD^*\cdot\TC}\circ\Y_0^{\det\TD\cdot\TB}.
\end{equation}

The monomial transformation and its canonical reduction $\X=\Y^\TM=\Z^\TN$ also indicate that $\Z^{[\TE_k~\TB]}=\Y^{[\TA~\TB]}$, which amounts to $\Z_*\circ\Z_0^\TB=\Y_*^\TA\circ\Y_0^\TB$.
Hence $\Z_*^{\det\TD\cdot\TE_k}=\Y_*^{\det\TD\cdot\TA}\circ\Y_0^{\det\TD\cdot\TB}
\circ\Z_0^{-\det\TD\cdot\TB}$.
Together with \eqref{NonExceptionalConstraint}, we obtain:
\begin{equation*}
\Z_*^{\det\TD\cdot\TE_k}=\Y_*^{\det\TD\cdot\TA-\TB\cdot\TD^*\cdot\TC}=\Y_*^\TF
\end{equation*}
based on the following straightforward identity:
\begin{equation*}
\TF=\det\TD\cdot\TA-\TB\cdot\TD^*\cdot\TC
\end{equation*}
with $\TF$ being the reduction matrix in \eqref{ReductionMatrix}.
\qedhere}
\end{asparaenum}
\end{proof}

\begin{definition}\label{Def:Consistency-1}
{\upshape (Consistency and inconsistency; canonical consistency form)}

The primary variable $x_1$ is said to be \emph{consistent} with the exponential matrix $\TM$ and exceptional index set $\iI\Yo$ if the exceptional column submatrix of the remnant exponential submatrix $\Ev{\lTM}\Yo$ as in Definition \ref{Def:ExceptionalSupport} and Definition \ref{Def:ProperTransform} satisfies $\rank(\Ev{\lTM}\Yo)=|\iI\Yo|$.
The primary variable $x_1$ is said to be \emph{inconsistent} with $\TM$ and $\iI\Yo$ otherwise.

If the primary variable $x_1$ is consistent with the exponential matrix $\TM$ and exceptional index set $\iI\Yo$ and the exceptional submatrix $\Evv\TM\Yo$ as in Definition \ref{Def:ExceptionalSupport} is non-degenerate as $\det\Evv\TM\Yo\ne 0$ with $\iI\Yo$ being in canonical form as in Definition \ref{Def:ProperTransform}, then we say that $\TM$ is in a \emph{canonical form} of consistency, or simply a \emph{canonical consistency form}, with respect to $\iI\Yo$.
Accordingly its canonical reduction $\TN$ is said to be in canonical consistency form with respect to the exceptional index set $\iI\Zo$ as well if the primary variable $x_1$ is consistent with $\TN$ and $\iI\Zo$ and $\det\Evv\TN\Zo\ne 0$ with $\iI\Zo$ being in canonical form.
\end{definition}

In the case when the primary variable $x_1$ is consistent with the exponential matrix $\TM$ and exceptional index set $\iI\Yo$ whereas the exceptional submatrix $\Evv\TM\Yo$ is degenerate as $\det\Evv\TM\Yo=0$, we make a permutation and relabeling of the row vectors of the remnant exponential matrix $\lTM$ and their corresponding variables $\lX:=(\LE x2n)$ of the Weierstrass polynomial $w(\X)$ in \eqref{PrimaryForm} simultaneously such that $\det\Evv\TM\Yo\ne 0$.
In this way we can always assume that the exponential matrix $\TM$ is in a canonical consistency form as in Definition \ref{Def:Consistency-1} whenever $x_1$ is consistent with $\TM$ and $\iI\Yo$.

In what follows and Section \ref{Section:RegularSing} let us assume that the exponential matrix $\TM$ is in a canonical consistency form with respect to $\iI\Yo$ as in Definition \ref{Def:Consistency-1}.
The inconsistency case shall be addressed in Section \ref{Section:Inconsistency}.

Let $\Ev\TM\Yo$ and $\Ev\TN\Zo$ denote the exceptional column submatrices of $\TM$ and $\TN$ with respect to $\iI\Yo=\iI\Zo$ respectively.
Evidently $\Ev\TN\Zo=\Ev\TM\Yo$ and hence the exceptional supports $\es\TN\Zo=\es\TM\Yo$, which are as in Definition \ref{Def:ExceptionalSupport}.
Let us rewrite the vertex form \eqref{VertexForm} of $w(\X)$ by organizing its terms with respect to $\es\TN\Zo$ as follows.
\begin{equation}\label{NormalForm}
w(\X)=c_\A\X^\A+\sum_{\Al\in\es\TN\Zo\setminus\{\A\}}\sumsep c_\Al\X^\Al
+\sum_{\Be\in\esc\TN\Zo}\sumsep c_\Be\X^\Be
\end{equation}
with the complementary support $\esc\TN\Zo:=\supp(w)\setminus\es\TN\Zo$.

The reduced monomial transformation $\CT\TN$ in \eqref{NormalTransformation} transforms $w(\X)$ in \eqref{NormalForm} into a total transform $w_\TN(\Z):=w(\CT\TN(\Z))$ that can be partially factorized into a partial transform $q(\Z)$ and the exceptional factors as follows.
\begin{equation}\label{NormalTransform}
\begin{aligned}
w_\TN(\Z)&=\Z_0^{\A\cdot\Ev\TN\Zo}\cdot
\biggl(c_\A\Z_*^{\A_*}+\sumsep\sum_{\Al\in\es\TN\Zo\setminus\{\A\}}
\sumsep c_\Al\Z_*^{\Al_*}+\sumsep\sum_{\Be\in\esc\TN\Zo}
\sumsep c_\Be\Z_*^{\Be_*}\Z_0^{(\Be-\A)\cdot\Ev\TN\Zo}\biggr)\\
&:=\Z_0^{\A\cdot\Ev\TN\Zo}\cdot q(\Z)
\end{aligned}
\end{equation}
with the exponents $\A=(\A_*,\A_0)$, $\Al=(\Al_*,\Al_0)$ and $\Be=(\Be_*,\Be_0)$ being defined as per the exceptional index set $\iI\Zo$.
It is easy to deduce that the proper transform $q(\Z_*,\bz)$ of $w(\X)$ under $\CT\TN$ is as follows.
\begin{equation}\label{ProperTransform-N}
q(\Z_*,\bz)=c_\A\Z_*^{\A_*}+\sum_{\Al\in \es\TN\Zo\setminus\{\A\}}\sumsep c_\Al\Z_*^{\Al_*}.
\end{equation}

The reduced monomial transformation $\CT\TN$ in \eqref{NormalTransformation} transforms the redundant function $r(\X)$ in the Weierstrass form \eqref{PrimaryForm} into a new function $r_\TN(\Z):=r(\CT\TN(\Z))$ as follows.
\begin{equation}\label{RedundantTransform}
r_\TN(\Z)=\sum_{\Al\in\supp(r)}\sumsep c_\Al\Z^{\Al\cdot\TN}=\sum_{\Al\in\supp(r)}\sumsep c_\Al\Z_*^{\Al_*}\Z_0^{\Al\cdot\Ev\TN\Zo}
\end{equation}
with $\Ev\TN\Zo$ being the exceptional column submatrix as in \eqref{NormalTransform}.
Please note that there is no partial factorization involved for $r_\TN(\Z)$ in \eqref{RedundantTransform} during the partial resolution of singularities based on $\CT\TN$.

\begin{definition}\label{Def:Localization}
{\upshape (Transformed apex $\ald$; redundant transform $r_\TN(\Z)$; reduced branch point $\bs=(\bs_*,\bz)$; localized partial and proper transforms $\wt q(\wt\Z)$ and $q_*(\wt\Z_*)$; regular and irregular branch points; partial resolution of singularities via canonical reduction)}

In accordance with the reduced monomial transformation $\CT\TN$ in \eqref{NormalTransformation}, if $\D$ is the apex in \eqref{PrimaryForm}, then the \emph{transformed} apex $\ald:=(d,\bz,(\D-\A)\cdot\Ev\TN\Zo)\in\supp (q(\Z))$ with $q(\Z)$ being the partial transform as in \eqref{NormalTransform}.

The function $r_\TN(\Z)$ in \eqref{RedundantTransform} is called the \emph{redundant} transform of the redundant function $r(\X)$ under the reduced monomial transformation $\CT\TN$.

For a branch point $\R=(\R_*,\bz)\in\EI\TM\Yo$ as in \eqref{ExceptionalBranch}, the set of points
\[
\{\bs\in\DF n\colon\bs=(\bs_*,\bz),\bs_*^{\det\TD\cdot\TE_k}=\R_*^{\TF}\}
\]
as per \eqref{CanonicalReduction} are called the \emph{canonically reduced} branch points of $\R=(\R_*,\bz)$, or simply the \emph{reduced} branch points or branch points.

The partial transform $q(\Z)$ in \eqref{NormalTransform} and proper transform $q(\Z_*,\bz)$ in \eqref{ProperTransform-N} can be \emph{localized} at a reduced branch point $(\bs_*,\bz)$ as $\wt q(\wt\Z):=q(\wt\Z_*+\bs_*,\wt\Z_0)$ and $q_*(\wt\Z_*):=\wt q(\wt\Z_*,\bz)$ respectively with $\wt\Z:=(\wt\Z_*,\wt\Z_0):=(\Z_*-\bs_*,\Z_0)$.

A reduced branch point $\bs=(\bs_*,\bz)$ is said to be \emph{regular} if the localized proper transform $q_*(\wt\Z_*)$ at $\bs$ with $\wt\Z_*=(\LE{\wt z}1k)$ satisfies $q_*(\wt z_1,\bz)\ne 0$; otherwise it is said to be an \emph{irregular} branch point.

The procedure of partial resolution of singularities based on the reduced monomial transformation $\CT\TN$ is called the partial resolution of singularities via \emph{canonical reduction}.
\end{definition}

\begin{lemma}\label{Lemma:TransformedExponents}
\begin{inparaenum}[(a)]
\item\label{item:UniqueMax} With $q(\Z)$ being the partial transform in \eqref{NormalTransform}, the transformed apex $\ald\in\supp(q)$ as in Definition \ref{Def:Localization} has the unique maximal degree $d$ in the variable $z_1$;
\item\label{item:Devoidance} For $\forall\Al=(\LE\alpha 1n)\in\supp(q)$, we have $\alpha_1\ne d-1$.
\end{inparaenum}
\end{lemma}
\begin{proof}
These conclusions are evident based on the exponential identity \eqref{ExponIdentity}.
\end{proof}

The following relation between the localized non-exceptional variables $\wt\Y_*$ and $\wt\Z_*$ can be readily deduced from the canonical reduction \eqref{CanonicalReduction}:
\begin{equation}\label{LocalConversion}
(\wt\Z_*+\bs_*)^{\det\TD\cdot\TE_k}=(\wt\Y_*+\R_*)^{\TF}
\end{equation}
with $\bs_*,\R_*\in (\BB F^*)^k$.
Moreover, based on \eqref{AdjointConversion} we have the following straightforward relation between the exceptional variables $\Y_0$ and $\Z_0$:
\begin{equation}\label{ExceptionalRelation}
\Z_0^{\det\TD\cdot\TE_{n-k}}=(\wt\Y_*+\R_*)^{\TD^*\cdot\TC}
\circ\Y_0^{\det\TD\cdot\TE_{n-k}}.
\end{equation}

\begin{lemma}\label{Lemma:NonSingMap}
Let $\BB F=\BB C$ and $p_*(\wt\Y_*)$ and $q_*(\wt\Z_*)$ denote the localized proper transforms around a branch point $\R=(\R_*,\bz)$ and one of its canonical reductions $\bs=(\bs_*,\bz)$ respectively.
\begin{inparaenum}[(a)]
\item There exist respective neighborhoods $U$ and $V$ of $\R_*$ and $\bs_*\in\DFs k$ that are biholomorphically equivalent, i.e., there exists a unique biholomorphic map $\AM$ from $U$ to $V$ satisfying \eqref{LocalConversion} such that $\wt\Z_*(U)=V$;
\item The orders of the localized proper transforms $p_*(\wt\Y_*)$ and $q_*(\wt\Z_*)$ are invariant under the biholomorphic map $\AM$, that is, $\ord (q_*)=\ord (p_*)$;
\item\label{item:NonDegMap} There exist a neighborhood $U$ of $\R=(\R_*,\bz)$ and a finite number of neighborhoods $\{V_l\}$ of $\bs=(\bs_*,\bz)$ constituting a branched covering of $U$ such that each $V_l$ is biholomorphically equivalent to $U$, i.e., for each $V_l$ there exists a biholomorphic map $\wt\Z(\wt\Y)$ from $U$ to $V_l$ satisfying \eqref{LocalConversion} and \eqref{ExceptionalRelation} such that $\wt\Z(U)=V_l$.
\end{inparaenum}
\end{lemma}
\begin{proof}
The conclusions are easy consequences of Lemma \ref{Lemma:CanonicalReduction} \eqref{item:NondegNormal} and the inverse function theorem.
\end{proof}

It is easy to see that the biholomorphic maps in Lemma \ref{Lemma:NonSingMap} can be similarly defined for a generic field $\BB F$ of characteristic zero such that the conclusions of Lemma \ref{Lemma:NonSingMap} still hold.

\subsection{Resolution of regular and consistent singularities}
\label{Section:RegularSing}

\begin{lemma}\label{Lemma:UnivariablePolynomial}
For $\ell\in\BB N^*$, every nonzero root of a univariate polynomial with $\ell$ terms has multiplicity at most $\ell-1$.
\end{lemma}
\begin{proof}
The conclusion can be easily proved by an induction on the number of terms of the univariate polynomial.
\end{proof}

\begin{lemma}\label{Lemma:SimpleDecrease}
After a reduced monomial transformation $\CT\TN$ in canonical consistency form as \eqref{CanonicalMatrix} is applied to a Weierstrass polynomial $w(\X)$ in \eqref{PrimaryForm} whose singularity height equals $d$, the localized proper transform $q_*(\wt\Z_*)$ satisfies $\ord (q_*)<d$ at a regular reduced branch point $\bs=(\bs_*,\bz)$.
\end{lemma}
\begin{proof}
The conclusion is an easy consequence of Lemma \ref{Lemma:TransformedExponents} \eqref{item:UniqueMax} and \eqref{item:Devoidance} based on Lemma \ref{Lemma:UnivariablePolynomial}.
\end{proof}

After the above partial resolution of singularities via canonical reduction at a regular reduced branch point to obtain the localized proper transform $q_*(\wt\Z_*)$, we proceed with the resolution algorithm by making preliminary and Weierstrass reductions of the localized partial transform $\wt q(\wt\Z)$ to acquire a Weierstrass form $f(\Xp)$.
More specifically, through a linear modification $\CL\Xp$ of the localized variables $\wt\Z$ such that $\Xp=\CL\Xp(\wt\Z)$, an apex form $f(\Xp)$ similar to that in \eqref{PrePrimaryForm} is acquired whose singularity height $d'$ equals $\ord (\wt q(\wt\Z))$.
This is followed by an invocation of Weierstrass preparation theorem and completion of perfect power such that $f(\Xp)$ is reduced into a Weierstrass form similar to that in \eqref{PrimaryForm}, which is still denoted as $f(\Xp)$ with $\Xp$ denoting the new variables and $x'_1$ the new primary variable.

After the above preliminary and Weierstrass reductions, let us continue with the partial resolution of singularities of the Weierstrass polynomial of the above Weierstrass form $f(\Xp)$, which is denoted as $w(\Xp)$ like in \eqref{PrimaryForm}.
More specifically, let us first define an $n$ by $n$ exponential matrix $\TM'$ via a refined vertex cone of the Newton polyhedron of $w(\Xp)$.
Then let us consider the monomial transformation
\begin{equation}\label{2ndMonomialTransform}
\Xp=\CT{\TM'}(\Yp):=\Yp^{\TM'}
\end{equation}
with the new variables $\Yp:=(\LEs y1n\prime\prime)$.
For an arbitrary set of exceptional variables among the variables $\Yp$, in what follows let us assume that the primary variable $x'_1$ is consistent with the exponential matrix $\TM'$ and exceptional index set $\iI\Yop$, and postpone the discussion on the inconsistency case until Section \ref{Section:Inconsistency}.
Without loss of generality, let us also assume that $\TM'$ is in a canonical consistency form as in Definition \ref{Def:Consistency-1}.
Suppose that the exceptional index set $\iI\Yop$ has codimension $l=n-|\iI\Yop|$.
The non-exceptional and exceptional variables of $\Yp$ are written as $\Y'_*:=(y'_1,\dotsc,y'_l)$ and $\Y'_0:=(y'_{l+1},\dotsc,y'_n)$ respectively with respect to $\iI\Yop$.
Hereafter the exceptional index set $\iI\Xop=\iI\Zo$ shall be referred to as the \emph{prior} exceptional index set of $\TM'$.

Similar to the matrix partition in \eqref{SplitMonomialMatrix}, we partition the exponential matrix $\TM'$ in \eqref{2ndMonomialTransform} into four submatrices $\TA',\TB',\TC'$ and $\TD'$ such that $\TD'$ is the exceptional submatrix $\Evv\TM\Yop'$ satisfying $\det\TD'\ne 0$.
Let us define a reduced exponential matrix $\TN'$ in the same way as the definition of $\TN$ in \eqref{CanonicalMatrix}, whose canonical form comprises the submatrices $\TB'$ and $\TD'$ as well as the unit matrix $\TE$ and zero matrix $\TO$.
The reduced monomial transformation $\CT{\TN'}$ associated with $\TN'$ can be written as:
\begin{equation}\label{2ndCanonicalMonomialTransform}
\Xp=\CT{\TN'}(\Zp):=\Zp^{\TN'}
\end{equation}
with $\Zp:=(z'_1,\dotsc,z'_n)$ being the new variables whose exceptional index set $\iI\Zop=\iI\Yop$.
The reduction matrix from $\CT{\TM'}$ to $\CT{\TN'}$ is similar to the reduction matrix $\TF$ in Definition \ref{Def:ReductionMatrix} and satisfies the same properties as in Lemma \ref{Lemma:CanonicalReduction}.
Furthermore, a conclusion similar to Lemma \ref{Lemma:SimpleDecrease} still holds here.
More specifically, let us write $q_*(\wt\Z_*')$ as the localized proper transform of the Weierstrass polynomial $w(\Xp)$ under $\CT{\TN'}$.
Then $\ord (q_*(\wt\Z_*'))<d'=\ord (\wt q(\wt\Z))$ at a regular reduced branch point with $\wt q(\wt\Z)$ being the localized partial transform in the prior resolution step associated with $\CT\TN$.

The resolution of singularities at a regular branch point is called the resolution of \emph{regular} singularities; otherwise it is called the resolution of \emph{irregular} singularities.

\begin{lemma}\label{Lemma:Prerequisite}
If the apex $\D=(d,\bz)$ in \eqref{PrimaryForm} satisfies $\D\in\es\TN\Zo$, then every reduced branch point is regular.
That is, a necessary condition for the irregular singularities is that $\D\notin\es\TN\Zo$.
\end{lemma}
\begin{proof}
When the apex $\D=(d,\bz)$ in \eqref{PrimaryForm} satisfies $\D\in\es\TN\Zo$, the formula for the proper transform $q(\Z_*,\bz)$ as in \eqref{ProperTransform-N} indicates that the transformed apex $\ald$ as in Definition \ref{Def:Localization} equals $(d,\bz)\in\supp (q(\Z_*,\bz))$.
Moreover, Lemma \ref{Lemma:TransformedExponents} \eqref{item:UniqueMax} indicates that $\ald$ has the unique maximal degree $d$ in the variable $z_1$ in $q(\Z_*,\bz)$.
Thus every reduced branch point $\bs=(\bs_*,\bz)$ is regular in this case.
\end{proof}

\begin{definition}\label{Def:Deficiency}
{\upshape ($\npp\bs$; dominant neighborhood; deficient index set $\iI\bz(\TA)$ and its complement $\iI*(\TA)$; $\iI*(\Al)$; deficiency; deficient support $\supp_{\iI\bz}(r)$; deficient function $r_{\iI\bz}(\X)$; deficient variety $\RI\TN\Zo$)}

The Newton polyhedron of the localization $\wt \varphi(\X-\bs):=\varphi((\X-\bs)+\bs)$ of a function $\varphi(\X)$ in the variables $\X-\bs$ is defined as the Newton polyhedron of $\varphi(\X)$ at $\bs\in\DF n$ and denoted as $\npp\bs$.
By default the notation $\NP(\varphi)$ refers to the Newton polyhedron of $\varphi(\X)$ at the origin $\bz$.
A neighborhood $U$ of a point $\bs$ is called its \emph{dominant} neighborhood if $\npp\bs\subseteq\npp\X$ for $\forall\X\in U$.

The index set of the zero row vectors of a matrix $\TA$ is called the \emph{deficient} index set of $\TA$ and denoted as $\iI\bz(\TA)$ whose complementary set is denoted as $\iI*(\TA)$.
Similarly the index set of the nonzero components of an exponential vector $\Al$ is denoted as $\iI*(\Al)$.
Let $\Ev\TN\Zo$ denote the exceptional column submatrix of a reduced exponential matrix $\TN$ in canonical form.
When the deficient index set $\iI\bz(\Ev\TN\Zo)\ne\emptyset$, we say that $\Ev\TN\Zo$ is \emph{deficient}, or $\TN$ is \emph{deficient} with respect to $\iI\Zo$.

When the exceptional column submatrix $\Ev\TN\Zo$ is deficient, the \emph{deficient} support of the redundant function $r(\X)$ in \eqref{PrimaryForm} is defined as
\begin{equation}\label{DeficientSupport}
\supp\nolimits_{\iI\bz}(r):=\{\Al\in\supp(r)\colon\iI*(\Al)
\subseteq\iI\bz(\Ev\TN\Zo)\}
\end{equation}
and accordingly, the \emph{deficient} function is defined as
\begin{equation}\label{DeficientFunction}
r_{\iI\bz}(\X):=\sum_{\Al\in\supp_{\iI\bz}(r)}c_\Al\X^\Al
\end{equation}
whose terms conform with those in $r(\X)$ in \eqref{PrimaryForm}.
The variety
\begin{equation}\label{DeficientVariety}
\RI\TN\Zo:=\{\X\in\DF n\colon c+r_{\iI\bz}(\X)=0\}
\end{equation}
is called the \emph{deficient} variety of the deficient function $r_{\iI\bz}(\X)$, where the constant $c\in\BB F^*$ is the same as in \eqref{PrimaryForm}.
\end{definition}

It is evident that when the deficient support $\supp_{\iI\bz}(r)=\emptyset$, which includes the case when the exceptional column submatrix $\Ev\TN\Zo$ is not deficient, the deficient variety $\RI\TN\Zo=\emptyset$.

\begin{lemma}\label{Lemma:DeficientProperties}
\begin{inparaenum}[(a)]
\item\label{item:DeficientScope} For a reduced exponential matrix $\TN$ in canonical consistency form, the deficient index set of the exceptional column submatrix $\Ev\TN\Zo$ is a subset of the non-exceptional index set, that is, $\iI\bz(\Ev\TN\Zo)\subseteq\iI{\Z_*}$;
\item\label{item:DeficientVariables} We have $z_j=x_j$ for $\forall j\in\iI\bz(\Ev\TN\Zo)$ under the reduced monomial transformation $\CT\TN$.
\end{inparaenum}
\end{lemma}
\begin{proof}
\begin{asparaenum}[(a)]
{\parindent=\saveparindent
\item The inclusion $\iI\bz(\Ev\TN\Zo)\subseteq\iI{\Z_*}$ directly follows from the consistency assumption, i.e., the non-degeneracy assumption on the exceptional submatrix as $\det\Evv\TN\Zo\ne 0$.

\item The conclusion readily follows from \eqref{item:DeficientScope} as well as the form of $\TN$ in \eqref{CanonicalMatrix}.
\qedhere}
\end{asparaenum}
\end{proof}

\begin{definition}\label{Def:Dominance}
{\upshape (Dominance $\Be\succeq\Al$; strict dominance $\Be\succ\Al$; exceptional dominance $\Be\succx\Al$; non-exceptional dominance $\Be\succn\Al$)}

For two exponential vectors $\Al,\Be\in\BB N^n$, the \emph{dominance} $\Be\succeq\Al$, that is, $\Be$ is \emph{dominated} by $\Al$, means that the difference $\Be-\Al\in\BB N^n$; whereas the \emph{strict} dominance $\Be\succ\Al$ means that $\Be-\Al\in\DN n$.
In particular, the \emph{exceptional} dominance $\Be\succx\Al$ with respect to an exceptional index set $\iJ$ means that their respective exceptional parts $\Be_0$ and $\Al_0$ in $\Be=(\Be_*,\Be_0)$ and $\Al=(\Al_*,\Al_0)$ with respect to $\iJ$ satisfy $\Be_0\succ\Al_0$.
The \emph{non-exceptional} dominance $\Be\succn\Al$ means that $\Be_0=\Al_0$ and $\Be_*\succ\Al_*$.
\end{definition}

\begin{lemma}\label{Lemma:RedundantDominance}
For $\forall\Al=(\LE\alpha 1n)\in\supp(r)$ in \eqref{RedundantTransform}, we have $\Al\cdot\TN\succx\bz$ unless $\Al\in\supp_{\iI\bz}(r)$ in which case $\Al\cdot\TN\succn\bz$.
\end{lemma}
\begin{proof}
The conclusion is evident based on Definition \ref{Def:Deficiency} and Lemma \ref{Lemma:DeficientProperties} \eqref{item:DeficientScope}.
\end{proof}

According to \eqref{RedundantTransform} and the definition of the deficient function $r_{\iI\bz}(\X)$ in \eqref{DeficientFunction}, as well as Lemma \ref{Lemma:DeficientProperties} and Lemma \ref{Lemma:RedundantDominance}, it is easy to deduce the following identity:
\begin{equation}\label{DeficientIdentity}
r_\TN(\Z_*,\bz)=r_{\iI\bz}(\Z_*)=r_{\iI\bz}(\X_*)
\end{equation}
with the variables $\X=(\X_*,\X_0)$ defined as per the exceptional index set $\iI\Zo$.

\begin{definition}\label{Def:DeficientContraction}
{\upshape (Deficient identity; deficient contraction)}

The identity \eqref{DeficientIdentity} is called the \emph{deficient} identity hereafter.
A \emph{deficient} contraction is an appropriate contraction of the dominant neighborhood of the origin $\X=\bz$ as in Definition \ref{Def:Deficiency} such that it has no intersection with the deficient variety $\RI\TN\Zo$ as in \eqref{DeficientVariety}.
\end{definition}

Let $p(\Z)$ denote the partial transform of the Weierstrass form $f(\X)$ in \eqref{PrimaryForm} under the reduced monomial transformation $\CT\TN$ in \eqref{NormalTransformation} and evidently
\begin{equation}\label{PartialTransformTotal}
p(\Z)=q(\Z)(c+r_\TN(\Z))
\end{equation}
with $r_\TN(\Z)$ being the redundant transform in \eqref{RedundantTransform} and satisfying $r_\TN(\bz)=0$ as per Lemma \ref{Lemma:RedundantDominance}.
The constant $c\in\BB F^*$ in \eqref{PartialTransformTotal} denotes the constant $c$ in \eqref{PrimaryForm}.
The function $q(\Z)$ is the partial transform of the Weierstrass polynomial $w(\X)$ as in \eqref{NormalTransform}.

The proper transform $p(\Z_*,\bz)$ of $f(\X)$ under $\CT\TN$ bears the following form as per \eqref{PartialTransformTotal} and the deficient identity \eqref{DeficientIdentity}.
\begin{equation}\label{TotalProperTransform}
p(\Z_*,\bz)=q(\Z_*,\bz)(c+r_{\iI\bz}(\X_*)).
\end{equation}

\begin{lemma}\label{Lemma:TotalTransformDecrease}
Let $p_*(\wt\Z_*):=p(\wt\Z_*+\bs_*,\bz)$ denote the localization of the proper transform $p(\Z_*,\bz)$ and $q_*(\wt\Z_*)$ that of $q(\Z_*,\bz)$ at a regular reduced branch point $\bs=(\bs_*,\bz)$ in \eqref{TotalProperTransform}.
After an appropriate deficient contraction, we have $\ord (p_*)=\ord (q_*)<d$, i.e., the conclusion of Lemma \ref{Lemma:SimpleDecrease} still holds for $p_*(\wt\Z_*)$.
\end{lemma}
\begin{proof}
When the exceptional column submatrix $\Ev\TN\Zo$ is deficient, we need to make an appropriate deficient contraction as in Definition \ref{Def:DeficientContraction} such that the deficient function $r_{\iI\bz}(\X_*)$ satisfies $c+r_{\iI\bz}(\X_*)\ne 0$.
Then \eqref{TotalProperTransform} indicates that the order of the localized proper transform $p_*(\wt\Z_*)$ satisfies $\ord (p_*)=\ord (q_*)<d$ as per Lemma \ref{Lemma:SimpleDecrease}.
When the exceptional column submatrix $\Ev\TN\Zo$ is not deficient, the conclusion holds without any deficient contraction.
\end{proof}

\subsection{Resolution of irregular singularities via latent variables}
\label{Section:SecondarySing}

In this section let us study the irregular singularities as in Definition \ref{Def:Localization} under the consistency assumption as in Definition \ref{Def:Consistency-1}, that is, resolution of irregular singularities when the primary variable is consistent with the exponential matrix and exceptional index set.

Let $k$ be the codimension of the exceptional index set $\iI\Zo$ as under \eqref{NormalTransformation} and $\wt\Z_*$ the localized non-exceptional variables at a reduced branch point $\bs=(\bs_*,\bz)$.
For $1<k<n$, let us define new variables $\tlZ_*:=(\LE{\wt z}2k)$ such that $\wt\Z_*=(\wt z_1,\tlZ_*)$.
Suppose that the localized proper transform $q_*(\wt\Z_*)$ of the Weierstrass polynomial $w(\X)$ as in Definition \ref{Def:Localization} satisfies $q_*(\wt z_1,\bz)=0$ when $\tlZ_*=\bz$, i.e., $\bs=(\bs_*,\bz)$ is an irregular branch point as in Definition \ref{Def:Localization}.
Let $q(\Z)=q(\Z_*,\Z_0)$ be the partial transform in \eqref{NormalTransform}.
With $\Z_*=(z_1,\lZ_*)$, let us postpone the localization of the variable $z_1$ and study a new function as follows.
\begin{equation}\label{PreVirtualization}
\vfz q1(\tlZ_*,\Z_0):=q(z_1,\tlZ_*+\lS_*,\Z_0)=q(z_1,\lZ_*,\Z_0)
\end{equation}
with $\lS_*:=(\LE s2k)$ such that $\bs_*=(s_1,\lS_*)\in\DFs k$.
Here $\vfz q1$ is treated as a function in the algebra $\BB F[z_1]\{\tlZ_*,\Z_0\}$ based on Lemma \ref{Lemma:TransformedExponents}.

Let us define a new set of variables $\Xp:=(\LEs x2n\prime\prime)$ that are partitioned into $\X'_*:=(\LEs x2k\prime\prime)$ and $\X'_0:=(\LEs x{k+1}n\prime\prime)$ according to the prior exceptional index set $\iI\Zo=\dotc{k+1}n$.
Let us consider a non-degenerate linear transformation $\CL{\X'_*}$ defined as
\begin{equation}\label{LinearModification}
\X'_*=\X'_*(\tlZ_*),\qquad
\X'_0=\Z_0
\end{equation}
that transforms the function $\vfz q1(\tlZ_*,\Z_0)$ in \eqref{PreVirtualization} into the following apex form:
\begin{equation}\label{InitialForm}
\vfz f1(\Xp):=c_{\D'}(z_1){x'_2}^{d'}
+\sum_{\Al'\in\supp(\vfz f1)\setminus\D'}c_{\Al'}(z_1)\Xp^{\Al'}
\end{equation}
such that the singularity height $d'=\ord (\vfz q1^*)$ with $\vfz q1^*(\tlZ_*):=\vfz q1(\tlZ_*,\bz)$ treated as a function in the algebra $\BB F[z_1]\{\tlZ_*\}$.
That is, $\ord (\vfz q1^*)$ refers to the order of $\vfz q1^*$ in the variables $\tlZ_*$.
Here $\Al'=(\LEs\alpha 2n\prime\prime)\in\DN{n-1}$.
The coefficients $c_{\D'}(z_1),c_{\Al'}(z_1)\in\BB F[z_1]$ and the apex $\D':=(d',\bz)\in\DN{n-1}$.

Let us reorganize the terms of $\vfz f1(\Xp)$ in \eqref{InitialForm} as follows.
\begin{equation}\label{PreWeierstrass}
\begin{aligned}
\vfz f1(\Xp)&=z_1^mf(\Xp)+\bar c_{\D'}(z_1){x'_2}^{d'}+\sum_{\Al'\in\supp(\vfz f1)\setminus\D'}\bar c_{\Al'}(z_1)\Xp^{\Al'}\\
&:=z_1^mf(\Xp)+\vfz\phi 1(\Xp)
\end{aligned}
\end{equation}
with $m:=\deg (c_{\D'}(z_1))\in\BB N$.
Neither the coefficient $\bar c_{\D'}(z_1)$ nor $\bar c_{\Al'}(z_1)$ in the function $\vfz\phi 1(\Xp)$ in \eqref{PreWeierstrass} contains a term with the monomial factor $z_1^m$.
The function $f(\Xp)$ in \eqref{PreWeierstrass} bears an apex form:
\begin{equation}\label{ParadigmFunction}
f(\Xp):=\tilde c_{\D'}{x'_2}^{d'}+\sum_{\Al'\in\supp(\vfz f1)\setminus\D'}\tilde c_{\Al'}\Xp^{\Al'}
\end{equation}
such that $\tilde c_{\D'} z_1^m+\bar c_{\D'}(z_1)=c_{\D'}(z_1)$ and $\tilde c_{\Al'}z_1^m+\bar c_{\Al'}(z_1)=c_{\Al'}(z_1)$ for $\forall\Al'\in\supp(\vfz f1)\setminus\D'$ with $\tilde c_{\D'}\in\BB F^*$ and $\tilde c_{\Al'}\in\BB F$.

The reorganization of the terms of $\vfz f1(\Xp)$ in \eqref{PreWeierstrass} amounts to a partial gradation of $\vfz f1(\Xp)$ by the degree $m$ of the variable $z_1$.

After invoking Weierstrass preparation theorem and completing perfect power, we can abuse the names of the variables $\Xp$ a bit and represent the function $f(\Xp)$ in \eqref{ParadigmFunction} into the following Weierstrass form:
\begin{equation}\label{MonoWeierstrass}
\begin{aligned}
f(\Xp)&=\biggl[{x'_2}^{d'}+\sum_{j=2}^{d'}c_j(\LEs x3n\prime\prime){x'_2}^{d'-j}\biggr]
(c+r(\Xp))\\
&:=w(\Xp)(c+r(\Xp)),
\end{aligned}
\end{equation}
where $c_j(\bz)=0$ for $1<j\le d'$ and $c\in\BB F^*$.
The Weierstrass polynomial is denoted as $w(\Xp)$ and the redundant function $r(\Xp)$ satisfies $r(\bz)=0$.

Let us abuse the notations a bit and assume that the function $f(\Xp)$ in \eqref{PreWeierstrass} bears the form in \eqref{MonoWeierstrass}, with which the terms of $\vfz f1(\Xp)$ in \eqref{InitialForm} conform as well.
The Newton polyhedron of the apex form $\vfz f1(\Xp)$ in \eqref{InitialForm} in terms of the variables $\Xp$ is denoted as $\NP(\vfz f1(\Xp))$.

\begin{definition}\label{Def:Virtualization}
{\upshape (Partial localization; latent variable $z_1$; active variables $\X'_*$; primary variable $x'_2$; linear modification $\CL{\X'_*}$; latency implementation; latent reducible and remainder functions $f(\Xp)$ and $\vfz\phi 1(\Xp)$; latent gradation; gradation index $m$; latent preliminary and Weierstrass reductions; latent apex $\Al'_\D$)}

The localization of the non-exceptional variables $\lZ_*=(\LE z2k)$ without the variable $z_1$ as in \eqref{PreVirtualization} is called a \emph{partial} localization of the partial transform $q(\Z)$ in \eqref{NormalTransform}.
The variable $z_1$ in the coefficients of $\vfz f1(\Xp)\in\BB F[z_1]\{\Xp\}$ in \eqref{InitialForm} is called the \emph{latent} variable of $\vfz f1(\Xp)$; whereas the variables $\X'_*=(\LEs x2k\prime\prime)$ are called the \emph{active} variables of $\vfz f1(\Xp)$ among which the variable $x'_2$ is called the \emph{primary} variable.
The non-degenerate linear transformation $\CL{\X'_*}$ to set up the apex $\D'$ and primary variable $x'_2$ in \eqref{LinearModification} is called a linear \emph{modification} of the localized non-exceptional variables $\tlZ_*$.
The treatment of the variable $z_1$ as a latent variable like in \eqref{PreVirtualization} and \eqref{InitialForm} is called a \emph{latency implementation} of the variable $z_1$.

The functions $f(\Xp)$ and $\vfz\phi 1(\Xp)$ in \eqref{PreWeierstrass} are called the latent \emph{reducible} and \emph{remainder} functions of the apex form $\vfz f1(\Xp)$ respectively.
The partial gradation of $\vfz f1(\Xp)$ into the latent reducible and remainder functions by the latent variable $z_1$ as in \eqref{PreWeierstrass} is called a \emph{latent gradation} of $\vfz f1(\Xp)$.
The degree $m$ of $z_1$ in \eqref{PreWeierstrass} is called the \emph{gradation index}.

The latency implementation, partial localization, linear modification and latent gradation constitute the procedure of \emph{latent} preliminary reduction; whereas the invocation of Weierstrass preparation theorem and completion of perfect power on the latent reducible function $f(\Xp)$ as in \eqref{MonoWeierstrass} constitute the \emph{latent} Weierstrass reduction.

With the transformed apex $\ald$ being defined as in Definition \ref{Def:Localization}, the \emph{latent} apex $\Al'_\D$ is defined as $\Al'_\D:=(\bz,(\D-\A)\cdot\Ev\TN\Zo)\in\DN{n-1}$ based on the irregularity assumption and Lemma \ref{Lemma:Prerequisite} such that $\ald=(d,\Al'_\D)$.
\end{definition}

\begin{lemma}\label{Lemma:PrimaryDegBnd}
Let $\D=(d,\bz)$ denote the apex of the Weierstrass polynomial $w(\X)$ in \eqref{PrimaryForm} and $\Al'_\D$ the latent apex as in Definition \ref{Def:Virtualization}.
\begin{inparaenum}[(a)]
\item\label{item:InitialDeg} The gradation index $m=\deg (c_{\D'}(z_1))$ as in \eqref{PreWeierstrass} satisfies $m<d-1$;
\item\label{item:TransformedApex} The latent remainder function $\vfz\phi 1(\Xp)$ contains the latent apex term $c_\D z_1^d\Xp^{\Al'_\D}$, i.e., $\Al'_\D\in\supp(\vfz\phi 1)$ and $c_\D\in\BB F^*$;
\item\label{item:StrictDegBnd} For $\forall\Al'\in\supp(\vfz f1)\setminus\{\Al'_\D\}$, its coefficient $c_{\Al'}(z_1)$ satisfies $\deg (c_{\Al'}(z_1))<d-1$.
\end{inparaenum}
\end{lemma}
\begin{proof}
\begin{asparaenum}[(a)]
{\parindent=\saveparindent
\item The irregularity assumption and Lemma \ref{Lemma:Prerequisite} indicate that the apex $\D\notin\es\TN\Zo$.
Hence the latent apex $\Al'_\D\notin\supp(\vfz q1^*)$ with the function $\vfz q1^*(\tlZ_*)$ being defined as under \eqref{InitialForm}.
This shows that $m<d$ as per Lemma \ref{Lemma:TransformedExponents} \eqref{item:UniqueMax}.
Moreover, Lemma \ref{Lemma:TransformedExponents} \eqref{item:Devoidance} and the invariance of the exponent of the latent variable $z_1$ through the latent preliminary and Weierstrass reductions indicate that $m<d-1$.

\item Evidently neither the partial localization in \eqref{PreVirtualization} as $\lZ_*=\tlZ_*+\lS_*$ nor the linear modification $\CL{\X'_*}$ in \eqref{LinearModification} has impact on the variables $(z_1,\Z_0)=(z_1,\X'_0)$.
Moreover, the conclusion \eqref{item:InitialDeg} and Lemma \ref{Lemma:TransformedExponents} \eqref{item:UniqueMax} indicate that the Weierstrass reduction in \eqref{MonoWeierstrass} has no impact on the latent apex term $c_\D z_1^d\Xp^{\Al'_\D}$.
And $\Al'_\D\in\supp(\vfz\phi 1)$ readily follows from the conclusion \eqref{item:InitialDeg}.

\item A direct consequence of Lemma \ref{Lemma:TransformedExponents} and the invariance of the exponent of the latent variable $z_1$ through the procedure of latent preliminary and Weierstrass reductions.
\qedhere}
\end{asparaenum}
\end{proof}

Let $\vfz p1(\Xp)$ and $\vfz r1(\Xp)$ denote the respective functions obtained from the partial transform $p(\Z)$ and redundant transform $r_\TN(\Z)$ in \eqref{PartialTransformTotal} through the same procedure of latent preliminary and Weierstrass reductions as that to acquire the apex form $\vfz f1(\Xp)$ in \eqref{PreWeierstrass}.
According to the identity \eqref{PartialTransformTotal}, we have the following identity:
\begin{equation}\label{TotalApexForm}
\vfz p1(\Xp)=\vfz f1(\Xp)(c+\vfz r1(\Xp)),
\end{equation}
where $c\in\BB F^*$ is the constant $c$ in \eqref{PartialTransformTotal} and $c+\vfz r1(\bz)\ne 0$ due to the deficient contraction in Lemma \ref{Lemma:TotalTransformDecrease}.
Hence it follows that the function $c+\vfz r1(\Xp)$ in \eqref{TotalApexForm} can be disregarded and it suffices to study the apex form $\vfz f1(\Xp)$ as in \eqref{PreWeierstrass}.

Now we are ready for the next step of resolution of irregular singularities.
Based on a refined vertex cone of the Newton polyhedron $\NP(\vfz f1(\Xp))$ with generators $\LEs\W 2n\prime\prime$, let us define an $(n-1)$ by $(n-1)$ exponential matrix $\TM'=[\MEs\W 2n\prime\prime]$ whose associated monomial transformation $\CT{\TM'}$ is defined as:
\begin{equation}\label{VirtualMonomialTransform}
\Xp=\CT{\TM'}(\Yp):=\Yp^{\TM'}
\end{equation}
with the new variables $\Yp:=(\LEs y2n\prime\prime)$.
Suppose that the exceptional index set $\iI\Yop$ has codimension $l$ and is in canonical form $\dotc{l+1}n$ as in Definition \ref{Def:ProperTransform}.
Here we only consider the case when the primary variable $x'_2$ is consistent with the exponential matrix $\TM'$ and exceptional index set $\iI\Yop$ and leave the inconsistency case to Section \ref{Section:Inconsistency}.

Let $\TN'$ denote the reduced exponential matrix with respect to $\TM'$ and $\iI\Yop$, which is similar to the reduced exponential matrix $\TN$ in \eqref{CanonicalMatrix}.
The reduced monomial transformation $\CT{\TN'}$ associated with $\TN'$ is defined as
\begin{equation}\label{VirtualResolution}
\Xp=\CT{\TN'}(\Zp):=\Zp^{\TN'}
\end{equation}
with the new variables $\Zp:=(\LEs z2n\prime\prime)$ whose exceptional index set $\iI\Zop=\iI\Yop$.
The non-exceptional and exceptional variables of $\Zp$ are denoted as $\Z'_*:=(\LEs z2l\prime\prime)$ and $\Z'_0:=(\LEs z{l+1}n\prime\prime)$ respectively.
The exceptional support of the apex form $\vfz f1(\Xp)$ in \eqref{InitialForm} with respect to $\TN'$ and $\iI\Zop$ is defined as:
\begin{equation}\label{SecondarySupport}
\es{\TN'}\Zop:=\bigcap\nolimits_{j\in\iI\Zop}(\supp(\vfz f1)\cap\face (\W'_j))
\end{equation}
with $\W'_j$ being the column vector of the exponential matrix $\TM'$ in \eqref{VirtualMonomialTransform}.
Let $\vfz q1(\Zp)$ denote the partial transform of the apex form $\vfz f1(\Xp)$ in \eqref{InitialForm} under $\CT{\TN'}$ in \eqref{VirtualResolution}.
Let $\Ga'_*$ and $\Ga'_0$ denote the respective exponents of the non-exceptional and exceptional variables $\Z'_*$ and $\Z'_0$.
Let us also abuse the notations a bit and denote $\Al'=(\Al'_*,\Al'_0)$ as the exponents of the variables $\Xp=(\X'_*,\X'_0)$ as per $\iI\Zop$.
Then the reduced monomial transformation $\CT{\TN'}$ in \eqref{VirtualResolution} is a linear exponential transformation similar to \eqref{ExponIdentity} as follows.
\begin{equation}\label{2ndExponIdentity}
\Ga'_*=\Al'_*,\qquad
\Ga'_0=\Al'\cdot\Ev\TN\Zop'
\end{equation}
with $\Ev\TN\Zop'$ being the exceptional column submatrix of $\TN'$ with respect to $\iI\Zop$.

The identity \eqref{PreWeierstrass} is transformed into the following identity by $\CT{\TN'}$ in \eqref{VirtualResolution}:
\begin{equation}\label{ParadigmRemainder}
\vfz q1(\Zp)=z_1^mp(\Zp)+\vfz\psi 1(\Zp),
\end{equation}
where $p(\Zp)$ and $\vfz\psi 1(\Zp)$ denote the partial transforms of $f(\Xp)$ and $\vfz\phi 1(\Xp)$ respectively.
In particular, based on the Weierstrass form of the latent reducible function $f(\Xp)$ in \eqref{MonoWeierstrass}, we have
\begin{equation}\label{2ndPartialTransformTotal}
p(\Zp)=q(\Zp)(c+r_{\TN'}(\Zp)),
\end{equation}
which is similar to \eqref{PartialTransformTotal}.
Here $q(\Zp)$ and $r_{\TN'}(\Zp)$ denote the partial and redundant transforms of the Weierstrass polynomial $w(\Xp)$ and redundant function $r(\Xp)$ in \eqref{MonoWeierstrass} under $\CT{\TN'}$ respectively.
The definition of $r_{\TN'}(\Zp)$ is similar to that of the redundant transform $r_\TN(\Z)$ in \eqref{RedundantTransform}.
And similar to Lemma \ref{Lemma:RedundantDominance}, we have $r_{\TN'}(\bz)=0$.

\begin{lemma}\label{Lemma:2ndSimpleDecrease}
Let $q_*(\wt\Z'_*)$ and $p_*(\wt\Z'_*)$ denote the localized proper transforms of $q(\Zp)$ and $p(\Zp)$ in \eqref{2ndPartialTransformTotal} at a reduced branch point respectively.
If the branch point is regular, then similar to Lemma \ref{Lemma:SimpleDecrease}, we have $\ord (q_*(\wt\Z'_*))<d'$ and similar to Lemma \ref{Lemma:TotalTransformDecrease}, we have $\ord (p_*(\wt\Z'_*))=\ord (q_*(\wt\Z'_*))<d'$ after an appropriate deficient contraction.
\end{lemma}
\begin{proof}
It is evident that the conclusions in Lemma \ref{Lemma:TransformedExponents} for the partial transform $q(\Z)$ can be repeated verbatim here for the partial transform $q(\Zp)$ in \eqref{2ndPartialTransformTotal}.
The arguments for Lemma \ref{Lemma:SimpleDecrease} and Lemma \ref{Lemma:TotalTransformDecrease} can be repeated verbatim here as well.
\end{proof}

\begin{lemma}\label{Lemma:StrictApexDominance}
\begin{inparaenum}[(a)]
\item\label{item:SeparateBnd} Let $\vfz{\wt q}1(\wt\Z')$ and $\vfz\psi 1^*(\wt\Z'_*)$ denote the localized partial and proper transforms of $\vfz q1(\Zp)$ and $\vfz\psi 1(\Zp)$ in \eqref{ParadigmRemainder} respectively.
Then we have
\[
\ord (\vfz{\wt q}1(\wt\Z'))\le\min\{\ord (p_*(\wt\Z'_*)),\ord (\vfz\psi 1^*(\wt\Z'_*))\}
\]
with $p_*(\wt\Z'_*)$ denoting the localized proper transform of $p(\Zp)$ in \eqref{2ndPartialTransformTotal};
\item\label{item:StrictDominance} For $\forall\Ga'=(\LEs\gamma 2n\prime\prime)\in\supp(\vfz q1(\Zp))\setminus\{\alg\}$ with $\gamma'_2\ge d'$ and $\vfz q1(\Zp)$ being the partial transform in \eqref{ParadigmRemainder}, the exponents $\Ga'$ is strictly dominated by the transformed apex $\alg$ of the apex $\D'$ in \eqref{InitialForm} under $\CT{\TN'}$ in \eqref{VirtualResolution}, i.e., $\Ga'\succ\alg$.
\end{inparaenum}
\end{lemma}
\begin{proof}
\begin{asparaenum}[(a)]
{\parindent=\saveparindent
\item The coefficients of the latent remainder function $\vfz\phi 1(\Xp)$ in \eqref{PreWeierstrass} do not contain a term with the monomial factor $z_1^m$ and hence neither the coefficients of its localized proper transform $\vfz\psi 1^*(\wt\Z'_*)$.
Thus follows the conclusion as per \eqref{ParadigmRemainder}.

\item The conclusion is easy to corroborate in a way similar to Lemma \ref{Lemma:RedundantDominance}.
\qedhere}
\end{asparaenum}
\end{proof}

\begin{lemma}\label{Lemma:TwoCases}
With the same notations as in Lemma \ref{Lemma:2ndSimpleDecrease} and Lemma \ref{Lemma:StrictApexDominance}, we have the following conclusions.
\begin{inparaenum}[(a)]
\item\label{item:CaseOne} When the reduced branch point is regular for the localized proper transform $q_*(\wt\Z'_*)$, we have $\ord (\vfz{\wt q}1(\wt\Z'))<d'$ after an appropriate deficient contraction in case of necessity;
\item\label{item:CaseTwo} When the reduced branch point is irregular for $q_*(\wt\Z'_*)$ whereas regular for the localized proper transform $\vfz\psi 1^*(\wt\Z'_*)$, we have $\ord (\vfz{\wt q}1(\wt\Z'))<d'$ as well.
\end{inparaenum}
\end{lemma}
\begin{proof}
\begin{asparaenum}[(a)]
{\parindent=\saveparindent
\item The conclusion readily follows from Lemma \ref{Lemma:2ndSimpleDecrease} and Lemma \ref{Lemma:StrictApexDominance} \eqref{item:SeparateBnd}.

\item Similar to Lemma \ref{Lemma:Prerequisite}, it follows from the reduced branch point being irregular for $q_*(\wt\Z'_*)$ that the apex $\D'$ in \eqref{InitialForm} satisfies $\D'\notin\es{\TN'}\Zop$.
Then Lemma \ref{Lemma:StrictApexDominance} \eqref{item:StrictDominance} and the convexity of Newton polyhedron indicate that the maximal degree of the variable $z'_2$ in the proper transform $\vfz q1(\Z'_*,\bz)$ is strictly less than $d'$.
Hence $\ord (\vfz\psi 1^*(\wt\Z'_*))<d'$ since the reduced branch point is regular for the localized proper transform $\vfz\psi 1^*(\wt\Z'_*)$.
The conclusion readily follows from Lemma \ref{Lemma:StrictApexDominance} \eqref{item:SeparateBnd}.
\qedhere}
\end{asparaenum}
\end{proof}

\begin{lemma}\label{Lemma:StarFalling}
Let $\vfz q1^*(\wt\Z'_*)$ denote the localized proper transform of $\vfz q1(\Zp)$ in \eqref{ParadigmRemainder}.
If $\vfz q1^*(\bz)$ is a nonzero univariate polynomial in $z_1$, then its order after a localization of $z_1$ at $s_1\in\BB F^*$ is strictly less than the prior singularity height $d$ of the Weierstrass polynomial $w(\X)$ in \eqref{PrimaryForm}, i.e., $\vft\ord 1(\vfz q1^*(\bz))<d$ with $\wt z_1:=z_1-s_1$.
In particular, this is the case when the latent apex $\Al'_\D$ as in Definition \ref{Def:Virtualization} satisfies $\Al'_\D\in\es{\TN'}\Zop$.
\end{lemma}
\begin{proof}
When the latent apex $\Al'_\D$ satisfies $\Al'_\D\in\es{\TN'}\Zop$, it is evident that the latent apex term $c_\D z_1^d\Xp^{\Al'_\D}=c_\D z_1^d{\X'_0}^{(\D-\A)\cdot\Ev\TN\Zo}$ in $\vfz f1(\Xp)$ as per Lemma \ref{Lemma:PrimaryDegBnd} \eqref{item:TransformedApex} is transformed into a term with coefficient $c_\D z_1^d$ in the proper transform $\vfz q1(\Z'_*,\bz)$ under the monomial transformation $\CT{\TN'}$ in \eqref{VirtualResolution}.
According to Lemma \ref{Lemma:PrimaryDegBnd} \eqref{item:StrictDegBnd}, the coefficient $c_{\Al'}(z_1)$ satisfies $\deg (c_{\Al'}(z_1))<d-1$ for $\forall\Al'\in\es{\TN'}\Zop\setminus\{\Al'_\D\}\subseteq\supp(\vfz f1)\setminus\{\Al'_\D\}$.
Hence the invariance of the exponent of the latent variable $z_1$ through $\CT{\TN'}$ and the localization of the non-exceptional variables $\Z'_*$ indicates that $\vfz q1^*(\bz)$ is a nonzero univariate polynomial in $z_1$ and let us denote it as $p(z_1)$.
We localize the latent variable $z_1$ at $s_1\in\BB F^*$ as following:
\begin{equation}\label{VirtualLocalization}
\wt p(\wt z_1):=p(\wt z_1+s_1),
\end{equation}
which had been postponed until now after the partial localization of the partial transform $q(\Z)$ in \eqref{PreVirtualization}.
It is easy to see that $\ord (\wt p(\wt z_1))<d$ as per Lemma \ref{Lemma:PrimaryDegBnd} and Lemma \ref{Lemma:UnivariablePolynomial}.

When the latent apex $\Al'_\D\notin\es{\TN'}\Zop$ but $\vfz q1^*(\bz)$ is a nonzero univariate polynomial in $z_1$, the conclusion $\vft\ord 1(\vfz q1^*(\bz))<d-1$ readily follows from Lemma \ref{Lemma:PrimaryDegBnd} \eqref{item:StrictDegBnd} and the invariance of the exponent of $z_1$ through $\CT{\TN'}$ and the localization of $\Z'_*$.
\end{proof}

After the localization of $z_1$ at $s_1\in\BB F^*$ as in \eqref{VirtualLocalization}, let us resume the latent variable $z_1$ and define:
\begin{equation}\label{Resumption}
\wt q_+(\wt z_1,\wt\Z'):=q_+(\wt z_1+s_1,\wt\Z')
\end{equation}
with $q_+(z_1,\wt\Z'):=\vfz{\wt q}1(\wt\Z')$, the localization of the partial transform $\vfz q1(\Zp)$ in \eqref{ParadigmRemainder}.
Evidently $\ord (\wt q_+)<d$ since $\ord (\wt p(\wt z_1))<d$ in \eqref{VirtualLocalization}.
We reduce $\wt q_+(\wt z_1,\wt\Z')$ to a Weierstrass form similar to \eqref{PrimaryForm} through preliminary and Weierstrass reductions and then repeat the resolution algorithm via the canonical reduction.

\begin{definition}\label{Def:ResidualOrder}
{\upshape (Revived localization; revival; residual polynomial; residual order)}

The localization of the latent variable $z_1$ in \eqref{VirtualLocalization} is called the \emph{revived} localization of $z_1$.
The process of resuming the latent variable $z_1$ like in \eqref{Resumption} is called the \emph{revival} of the latent variable $z_1$, or we say that the latent variable $z_1$ is \emph{revived}.
The nonzero univariate polynomial $p(z_1)$ in \eqref{VirtualLocalization} is called the \emph{residual} polynomial whose order after the revived localization is called the \emph{residual} order.
\end{definition}

It is easy to see that a sufficient condition for the revival of the latent variable is that the residual polynomial is a nonzero univariate polynomial.
In what follows let us assume that, contrary to the case in Lemma \ref{Lemma:StarFalling}, the latent variable $z_1$ is not revived, from which we can deduce that the latent apex $\Al'_\D$ satisfies $\Al'_\D\notin\es{\TN'}\Zop$.

In the cases of Lemma \ref{Lemma:TwoCases} \eqref{item:CaseOne}, the strict decrease of the singularity height leads to the ultimate scenario when the singularity height $d'=1$ in \eqref{InitialForm}.
It is trivial and evident when the Newton polyhedron $\NP(\vfz f1(\Xp))$ in \eqref{InitialForm} is an orthant $\D'+\Rp^{n-1}$ with $d'=1$.
In the generic case we have the following conclusion.

\begin{lemma}\label{Lemma:VirtualRecovery}
Suppose that the Newton polyhedron $\NP(\vfz f1(\Xp))$ is not an orthant $\D'+\Rp^{n-1}$ when the singularity height $d'=1$ in \eqref{InitialForm}.
If the primary variable $x_2'$ is consistent with the reduced exponential matrix $\TN'$ and exceptional index set $\iI\Zop$ in \eqref{VirtualResolution}, then the apex $\D'$ satisfies $\D'\in\es{\TN'}\Zop$.
\end{lemma}
\begin{proof}
Let us denote $\PI\NP$ as the normal vector set of the Newton polyhedron $\NP(\vfz f1(\Xp))$ and $\{\LE\E 2n\}$ as the standard basis of $\BB R^{n-1}$ such that $\E_2=(1,\bz)\in\BB N^{n-1}$.
We first show that the apex $\D'\notin\facet(\E_2)$ in the case of $d'=1$.
In fact, the apex $\D'\in\facet(\E_2)$ would indicate that for $\forall\Al'=(\LEs\alpha 2n\prime\prime)\in\NP(\vfz f1(\Xp))$, we would have $\alpha'_2=\langle\E_2,\Al'\rangle\ge\langle\E_2,\D'\rangle=1$ by the convexity of $\NP(\vfz f1(\Xp))$.
Hence $\Al'\succeq\D'$ for $\forall\Al'\in\NP(\vfz f1(\Xp))$ and $\NP(\vfz f1(\Xp))$ would be an orthant with a unique vertex $\D'$.
This would contradict the assumption on $\NP(\vfz f1(\Xp))$ in the lemma.

We prove next that for $\forall\W'=(\LEs w2n\prime\prime)\in\PI\NP\setminus\{\E_2\}$, the apex $\D'\in\facet (\W')$.
In fact, suppose that the $\facet (\W')$ has equation $\langle\W',\Al'\rangle=p\in\BB N$ with $\Al'\in\DN{n-1}$.
In the case when $p=0$, evidently $\W'\in\{\LE\E 3n\}$ and hence $\langle\W',\D'\rangle=0$, which means that $\D'\in\facet (\W')$.
In the case when $p>0$, the condition $d'=1$ as well as the convexity of Newton polyhedron $\NP(\vfz f1(\Xp))$ require that $w'_2=\langle\W',\D'\rangle\ge\langle\W',\Al'\rangle=p>0$ for $\forall\Al'\in\facet (\W')$ as in \eqref{FaceEq}.
This indicates that the hyperplane of $\facet (\W')$ has a positive intercept $p/w'_2>0$.
Hence $\exists\Al'_0=(\LEs\alpha 00{\prime 2}{\prime n})\in\supp(\vfz f1(\Xp))$ with $\alpha_0^{\prime 2}\ge 1$ that satisfies $\Al'_0\in\facet (\W')$ as per the definition of Newton polyhedron.
On the other hand, $\langle\W',\D'\rangle=w'_2\le w'_2\cdot\alpha_0^{\prime 2}\le\langle\W',\Al'_0\rangle=p$.
Consequently $\langle\W',\D'\rangle=p$ and $\D'\in\facet (\W')$ for $\forall\W'\in\PI\NP\setminus\{\E_2\}$.

Now let $\CV{\A'}$ be a refined vertex cone of $\NP(\vfz f1(\Xp))$ and $\U'\in\gv{\A'}\subseteq\DN n$ an auxiliary vector as in the proof of Lemma \ref{Lemma:Refinement}.
Then $\exists V\subseteq\gv{\A'}\setminus\{\E_2\}\subseteq\PI\NP\setminus\{\E_2\}$ with $|V|>1$ such that $\U'\in\CC^\circ(V)$ since $\E_2$ is a standard basis vector.
Thus $\face (\U')=\bigcap_{\W'\in V}\facet (\W')$ as per Lemma \ref{Lemma:IntersectLemma} \eqref{item:FaceIn} and the apex $\D'\in\face (\U')$ by an induction on the auxiliary vector set $\Lambda_\A$ as in the proof of Lemma \ref{Lemma:Refinement}.

Finally, suppose that the exceptional column submatrix $\Ev\TN\Zop'=[\MEs\W{l+1}n\prime\prime]$ with $\iI\Zop:=\dotc{l+1}n$.
For $\forall j\in\iI\Zop$, the normal or auxiliary vector $\W'_j$ satisfies $\W'_j\ne\E_2$ since the exceptional submatrix $\Evv\TN\Zop'$ would satisfy $\det\Evv\TN\Zop'=0$ otherwise, contradicting the consistency assumption of the lemma.
\end{proof}

When the singularity height $d'=1$ in the Weierstrass form \eqref{InitialForm}, it is easy to see that the latent reducible function in \eqref{MonoWeierstrass} bears the form
$f(\Xp)=x'_2(c+r(\Xp))$, i.e., the Weierstrass polynomial $w(\Xp)=x'_2$ after the completion of perfect power in \eqref{MonoWeierstrass}.
As per Lemma \ref{Lemma:VirtualRecovery} and the exponential identity \eqref{2ndExponIdentity}, the partial transform $q(\Zp)=z'_2$ in \eqref{2ndPartialTransformTotal}.
Thus the localized proper transform $q_*(\wt\Z'_*)$ satisfies $q_*(\bz)\in\BB F^*$.
Since the coefficients of the localized proper transform $\vfz\psi 1^*(\wt\Z'_*)$ of $\vfz\psi 1(\Zp)$ in \eqref{ParadigmRemainder} do not contain a term with the monomial factor $z_1^m$, it follows that $\vfz q1^*(\bz)$ as in Lemma \ref{Lemma:StarFalling} is a nonzero residual polynomial in $z_1$ whose residual order satisfies $\vft\ord 1(\vfz q1^*(\bz))<d$.
Subsequently we localize and then revive the latent variable $z_1$ as in \eqref{VirtualLocalization} and \eqref{Resumption}.

If the reduced branch point is irregular for both the localized proper transforms $q_*(\wt\Z'_*)$ and $\vfz\psi 1^*(\wt\Z'_*)$, which is contrary to the assumptions in Lemma \ref{Lemma:TwoCases} \eqref{item:CaseOne} and \eqref{item:CaseTwo}, then we render the variable $z'_2$ latent and localize the non-exceptional variables $\lZ'_*:=(\LEs z3l\prime\prime)$, which is similar to the latency implementation and partial localization in \eqref{PreVirtualization}.
Let us denote the function obtained from $\vfz q1(\Zp)$ in \eqref{ParadigmRemainder} in this way as $\vfze q(\tlZ'_*,\Z'_0)$ with $\Z'_\star:=(z_1,z'_2)$ being the latent variables and $\tlZ'_*$ the localization of the non-exceptional variables $\lZ'_*$ as above.
Similar to \eqref{LinearModification}, let us make a linear modification of the function $\vfze q\in\BB F\{\Z'_\star\}\{\tlZ'_*,\Z'_0\}$ that yields an apex form resembling \eqref{InitialForm} as follows.
\begin{equation}\label{3rdApexForm}
\vfze f(\Xq):=c_{\D''}(\Z'_\star){x''_3}^{d''}+\sum_{\Al''\in\supp(\vfze f)\setminus\D''}c_{\Al''}(\Z'_\star){\Xq}^{\Al''}
\end{equation}
with the apex $\D'':=(d'',\bz)$ such that $d''=\ord (\vfze q(\tlZ'_*,\bz))$, i.e., the order in the variables $\tlZ'_*$.
The new variables $\Xq$ are defined as $(\LEs x3n{\prime\prime}{\prime\prime})$ whose exponents $\Al'':=(\LEs\alpha 3n{\prime\prime}{\prime\prime})$.

Similar to the gradation index $m$ in \eqref{PreWeierstrass}, suppose that we have gradation indices $\M'_\star:=(m_1,m'_2)\in\supp (c_{\D''}(\Z'_\star))$.
In the case when $m_1\ne m$ with $m$ being the gradation index in \eqref{PreWeierstrass}, a latent gradation of the apex form $\vfze f(\Xq)$ in \eqref{3rdApexForm} into latent reducible and remainder functions by the gradation indices $\M'_\star$ is as follows.
\begin{equation}\label{2ndSeparation}
\vfze f(\Xq)={\Z'_\star}^{\M'_\star}f(\Xq)+\vfze\phi(\Xq),
\end{equation}
which is similar to \eqref{PreWeierstrass}.
Here the latent reducible function $f(\Xq)$ is defined in the same way as $f(\Xp)$ in \eqref{ParadigmFunction} such that the coefficients of the latent remainder function $\vfze\phi(\Xq)$ do not contain a term with the monomial factor ${\Z'_\star}^{\M'_\star}$, which resembles $\vfz\phi 1(\Xp)$ in \eqref{PreWeierstrass}.
Weierstrass preparation theorem and a completion of perfect power can be applied to the latent reducible function $f(\Xq)$ in \eqref{2ndSeparation} in a way similar to that in \eqref{MonoWeierstrass}.
After a reduced monomial transformation $\CT{\TN''}$ defined as $\Xq=\CT{\TN''}(\Zq):={\Zq}^{\TN''}$, which is similar to $\CT{\TN'}$ in \eqref{VirtualResolution}, the identity \eqref{2ndSeparation} is transformed into:
\begin{equation}\label{3rdPartialTransform}
\vfze q(\Zq)={\Z'_\star}^{\M'_\star}p(\Zq)+\vfze\psi (\Zq)
\end{equation}
resembling \eqref{ParadigmRemainder}, where $p(\Zq)$ and $\vfze\psi(\Zq)$ denote the partial transforms of $f(\Xq)$ and $\vfze\phi(\Xq)$ in \eqref{2ndSeparation} respectively.
It is easy to corroborate that the conclusions of Lemma \ref{Lemma:TwoCases} still hold for the localization $\vfze{\wt q}(\wt\Z'')$ of the partial transform $\vfze q(\Zq)$ in \eqref{3rdPartialTransform}.
That is, if a reduced branch point is regular for the localized proper transform of either $p(\Zq)$ or $\vfze\psi(\Zq)$ in \eqref{3rdPartialTransform}, then we have $\ord (\vfze{\wt q}(\wt\Z''))<d''$ up to an appropriate deficient contraction as in Lemma \ref{Lemma:2ndSimpleDecrease} with $d''$ being the singularity height in \eqref{3rdApexForm}.

Let $\vfze q^*(\wt\Z''_*)$ denote the localized proper transform of the partial transform $\vfze q(\Zq)$ in \eqref{3rdPartialTransform}.
Then the conclusion of Lemma \ref{Lemma:StarFalling} on residual order still applies here.
More specifically, if the residual polynomial $\vfze q^*(\bz)$ in $\Z'_\star$ is nonzero, then its residual order in $\wt z'_2$ is strictly less than the prior singularity height $d'$ of the apex form $\vfz f1(\Xp)$ in \eqref{InitialForm}, i.e., $\vftp\ord 2(\vfze q^*(\bz))<d'$.
In particular, let $\alg$ and $\lalg$ denote the transformed and latent apexes associated with the apex $\D'$ in \eqref{InitialForm} under $\CT{\TN'}$ in \eqref{VirtualResolution}, which resembles $\ald$ in Definition \ref{Def:Localization} and $\Al'_\D$ in Definition \ref{Def:Virtualization} respectively, such that $\alg:=(d',\lalg)$.
Then the above conclusion holds when the latent apex $\lalg\in\es{\TN''}\Zoq$.
In fact, according to the assumption in \eqref{2ndSeparation} that the gradation index $m_1\ne m$ with $m$ being the gradation index in \eqref{PreWeierstrass}, the latent remainder function $\vfze\phi(\Xq)$ in \eqref{2ndSeparation} can be written as:
\begin{equation}\label{PriorPerspective}
\vfze\phi(\Xq)=z_1^m\vfzp f2(\Xq)(c+\vfzp r2(\Xq))+\vfze{\ol\phi}(\Xq)
\end{equation}
as per \eqref{ParadigmRemainder} and \eqref{2ndPartialTransformTotal} with $\vfzp f2(\Xq)$ and $\vfzp r2(\Xq)$ denoting the latent preliminary and Weierstrass reductions of the partial and redundant transforms $q(\Zp)$ and $r_{\TN'}(\Zp)$ in \eqref{2ndPartialTransformTotal} respectively and $\vfze{\ol\phi}(\Xq)$ denoting the remaining terms.
Similar to Lemma \ref{Lemma:PrimaryDegBnd} \eqref{item:TransformedApex}, we can prove that the latent apex $\lalg\in\supp(\vfzp f2(\Xq))$.
Let $\vfzp q2^*(\wt\Z''_*)$ denote the localized proper transform of $\vfzp f2(\Xq)$ in \eqref{PriorPerspective} under $\CT{\TN''}$.
When $\lalg\in\es{\TN''}\Zoq$, it is easy to see that $\vfzp q2^*(\bz)$ is a nonzero residual polynomial in the latent variable $z'_2$ whose residual order is strictly less than $d'$ after the revived localization of $z'_2$ according to Lemma \ref{Lemma:UnivariablePolynomial}, which can be similarly proved as Lemma \ref{Lemma:StarFalling}.
Moreover, we have $c+\vfzp r2(\X''_*,\bz)\ne 0$ due to the deficient contraction in Lemma \ref{Lemma:2ndSimpleDecrease}.
Hence follows the conclusion $\vftp\ord 2(\vfze\psi^*(\bz))<d'$ with $\vfze\psi^*(\wt\Z''_*)$ denoting the localized proper transform of $\vfze\phi(\Xq)$ in \eqref{PriorPerspective} under $\CT{\TN''}$ and thus $\vftp\ord 2(\vfze q^*(\bz))<d'$ as per \eqref{3rdPartialTransform}.

According to Lemma \ref{Lemma:StrictApexDominance} \eqref{item:StrictDominance}, it is easy to deduce that for $\forall\Al''\in\supp(\vfze\phi(\Xq))$, if its coefficient $c_{\Al''}(\Z'_\star)$ satisfies $\vfzp\deg 2(c_{\Al''}(\Z'_\star))\ge d'$, then $\Al''$ is dominated by the latent apex $\lalg$, i.e., $\Al''\succeq\lalg$.
When the latent apex $\lalg\notin\es{\TN''}\Zoq$ but $\vfze\psi^*(\bz)$ is a nonzero residual polynomial, the conclusion $\vftp\ord 2(\vfze\psi^*(\bz))<d'$ readily follows from the above dominance of $\lalg$ due to the invariance of the exponent of the latent variable $z'_2$ through both $\CT{\TN''}$ and the localization of the non-exceptional variables $\Z''_*$.

In the case when the gradation index $m_1$ in \eqref{2ndSeparation} satisfies $m_1=m$ with $m$ being the gradation index in \eqref{PreWeierstrass}, let us denote $\vfzp f2(\Xq)$ as the latent preliminary and Weierstrass reductions of the partial transform $q(\Zp)$ in \eqref{2ndPartialTransformTotal} like in \eqref{PriorPerspective}.
The latent gradation of $\vfzp f2(\Xq)$ by the gradation index $m'_2$ into latent reducible and remainder functions resembling \eqref{PreWeierstrass} is as follows.
\begin{equation}\label{OldSeparation}
\vfzp f2(\Xq)={z'_2}^{m'_2}f(\Xq)+\vfzp\phi 2(\Xq),
\end{equation}
where the latent reducible function $f(\Xq)$ is the same as in \eqref{2ndSeparation} to which Weierstrass preparation theorem and a completion of perfect power can be applied.
It is easy to verify that the latent apex $\lalg\in\supp(\vfzp f2(\Xq))$ like in \eqref{PriorPerspective}.
The rest of the argument for the conclusion $\vftp\ord 2(\vfze q^*(\bz))<d'$ entails a discussion on whether or not the latent apex $\lalg\in\es{\TN''}{\Zoq}$, which is similar to the above case when the gradation index $m_1\ne m$.

When the reduced branch point is irregular for the localized proper transform $\vfze q^*(\wt\Z''_*)$ of the partial transform $\vfze q(\Zq)$ in \eqref{3rdPartialTransform}, we make the variable $z_3$ latent and repeat the above discussions in a similar fashion.
By a dimensional induction, it is evident that such incessant latency implementations without revival shall lead to a Weierstrass form with a single active variable in which case every branch point is regular.
It is easy to see that the singularity height strictly decreases in this case.
Furthermore, the revival of a latent variable is triggered when the singularity height is reduced to one like in Lemma \ref{Lemma:VirtualRecovery} under the consistency assumption.
And the residual order of the latent variable strictly decreases from the prior singularity height like in Lemma \ref{Lemma:StarFalling}.

\subsection{Resolution of inconsistent singularities via synthesis\\
of monomial transformations}
\label{Section:Inconsistency}

The canonical reduction of monomial transformation in Section \ref{Subsection:CanonicalReduction} and resolution of regular and irregular singularities in Section \ref{Section:RegularSing} and Section \ref{Section:SecondarySing} are based on the consistency assumption on the primary variable as in Definition \ref{Def:Consistency-1}.
The inconsistent singularities is addressed in this section and referred to as  the resolution of \emph{inconsistent} singularities.

With the same notations as in Definition \ref{Def:Consistency-1}, suppose that the primary variable $x_1$ is inconsistent with the exponential matrix $\TM$ and exceptional index set $\iI\Yo$, that is, $\rank(\Ev{\lTM}\Yo)<|\iI\Yo|$ with $\lTM$ being the remnant exponential submatrix as in Definition \ref{Def:ProperTransform} and $\iI\Yo$ in canonical form as $\dotc{k+1}n$.
In this case let us permute and relabel the row vectors of $\TM$ and their corresponding variables $\X$ simultaneously to acquire a new exponential matrix denoted as $\perm\TM$ such that the last row vector of $\perm\TM$ equals the primary row vector of $\TM$ and the exceptional submatrix of $\perm\TM$ is non-degenerate as $\det\hs{-1pt}\Evv{\perm\TM}\Yo\ne 0$.
Accordingly the canonical reduction of $\perm\TM$ with respect to $\iI\Yo$ is denoted as $\perm\TN$ that bears a canonical form as follows.
\begin{equation}\begin{aligned}\label{IncstnMnmlTrnsfrm}
    \perm\TN:=
    \begin{array}{r @{} c @{\hs{-0.5pt}} c @{\hs{-3pt}} p{4mm}
    @{\hs{-1mm}} *{3}c @{\hs{2pt}} c @{\hs{2pt}} c @{\hs{1pt}}
    c @{\hs{3pt}} *{3}c @{\hs{-0.8mm}} p{4mm} @{\hs{-2.5mm}}
    l}
    &           &           &           &           \E_1        &
    \cdots      &           \E_k        &           &           \V_{k+1}
    &           &           \V_{k+2}    &           \cdots      &
    \V_n        &           &                                               \\
    \mr{10}{$\rule{0mm}{24mm}$}         &           x_2         &
    \mr{10}{$\rule{0mm}{24mm}$}         &
    \raisebox{3pt}[0pt]{\mr{9}{$\left[\rule{0mm}{66.7pt}\right.$}}
    &           \mc{3}{c}{\mr{4}{$\TE$}}            &           \vline
    &           \mc{5}{c}{\mr{3}{$\TOm\TB$}}        &
    \raisebox{3pt}[0pt]{\mr{9}{$\left.\rule{0mm}{66.7pt}\right]$}}
    &           \mr{10}{$\rule{0mm}{24mm}$}                                 \\
    &           \vdots      &           &           &           &
    &           &           \vline      &           &           &
    &           &           &           &                                   \\
    &           x_k         &           &           &           &
    &           &           \vline      &           &           &
    &           &           &           &                                   \\
    \cline{2-2}\cline{8-14}
    &           x_{k+1}     &           &           &           &
    &           &           \vline      &
    \mc{5}{c}{\raisebox{-1pt}[0pt]{$\La\TB$}}       &           &           \\
    \cline{2-14}
    &           x_{k+2}     &           &           &
    \mc{3}{c}{\mr{4}{$\TO$}}            &           \vline      &
    \mr{3}{$\TD\bl$}        &           \vline      &
    \mc{3}{c}{\mr{3}{$\TD$}}            &           &                       \\
    &           \vdots      &           &           &           &
    &           &           \vline      &           &           \vline
    &           &           &           &           &                       \\
    &           x_n         &           &           &           &
    &           &           \vline      &           &           \vline
    &           &           &           &           &                       \\
    \cline{2-2}\cline{8-14}
    &           x_1         &           &           &           &
    &           &           \vline      &
    \mc{5}{c}{\raisebox{-1.5pt}[0pt]{$\N_1$}}       &           &           \\
    &           &           &           &           z_1         &
    \cdots      &           z_k         &           \vline      &
    z_{k+1}     &           \vline      &           z_{k+2}     &
    \cdots      &           z_n         &           &                       \\
    \end{array}
\end{aligned}\end{equation}
with the matrix $\TB$ denoting the submatrix $[\TD\bl~\TD]$.
The $(n-k-1)$ by $(n-k-1)$ submatrix $\TD$ satisfies $\det\TD\ne 0$ and exceptional index set $\iI\Zo=\iI\Yo$.
The vectors $\La$ and $\bl$ in \eqref{IncstnMnmlTrnsfrm} are $(n-k-1)$-dimensional row and column vectors in $\BB Q^{n-k-1}$ respectively.
The matrix $\TOm$ is of dimensions $(k-1)$ by $(n-k-1)$ with elements in $\BB Q$ as well.
The submatrices $\TE$ and $\TO$ in \eqref{IncstnMnmlTrnsfrm} denote the same unit and zero submatrices as in \eqref{CanonicalMatrix} respectively.
The exceptional submatrix $\Evv{\perm\TN}\Zo$ is non-degenerate as $\det\Evv{\perm\TN}\Zo\ne 0$ due to the non-degeneracy $\det\hs{-1pt}\Evv{\perm\TM}\Yo\ne 0$.
The row vector $\N_1\in\DN{n-k}$ is part of the primary row vector $\M_1$ of the original exponential matrix $\TM$.
Please note that we abuse the notations a bit here and use the same notations $\TB$ and $\TD$ for different submatrices of $\TN$ in \eqref{CanonicalMatrix} and $\perm\TN$ in \eqref{IncstnMnmlTrnsfrm} since the consistency and inconsistency cases are mutually exclusive.

The reduced exponential matrix $\perm\TN$ in \eqref{IncstnMnmlTrnsfrm} yields a reduced monomial transformation $\CT{\per\TN}$ on the Weierstrass polynomial $w(\perm\X)$ in \eqref{PrimaryForm} with $\perm\X$ denoting the new variables corresponding to the canonical form in \eqref{IncstnMnmlTrnsfrm}.
Similar to $w_\TN(\Z)$ in \eqref{NormalTransform}, let us denote the total transform of $w(\perm\X)$ under $\CT{\per\TN}$ as $w_{\per\TN}(\Z)$.
For $\forall\pAl\in\supp(w(\perm\X))$, the exponents $\Ga_0$ of the exceptional variables $\Z_0$ of $w_{\per\TN}(\Z)$ bear the following form:
\begin{equation}\label{LatentIdentity}
\Ga_0:=\pAl\cdot\Ev{\perm\TN}\Zo=
\bigg(\lAl_*\cdot\TOm+\Al_0\cdot
\begin{bmatrix}\La\\ \TE\end{bmatrix}\biggr)\cdot\TD\cdot [\,\bl~\TE]
+\alpha_1\N_1,
\end{equation}
where $\Ev{\perm\TN}\Zo$ denotes the exceptional column submatrix of $\perm\TN$ with respect to $\iI\Zo$.
And $\Al_0$ represent the exponents of the variables $\X_0$ whereas $\lAl_*:=(\LE\alpha 2k)$ those of the variables $\lX_*:=(\LE x2k)$ as in \eqref{ExponIdentity}.
The $(n-k)$-dimensional vector $\N_1$ is as in \eqref{IncstnMnmlTrnsfrm}.
Both the unit matrices $\TE$ in \eqref{LatentIdentity} are of dimensions $(n-k-1)$ by $(n-k-1)$.

It is easy to see that the non-degeneracy condition $\det\Evv{\perm\TN}\Zo\ne 0$ indicates the necessary condition:
\begin{equation}\label{PrimaryVector}
\N_1\cdot\begin{bmatrix}-1\\ \bl\end{bmatrix}\ne 0.
\end{equation}

\begin{lemma}\label{Lemma:ConstantPrimaryExp}
Let $\perm\TN$ denote the reduced exponential matrix as in \eqref{IncstnMnmlTrnsfrm}, and $\es{\perm\TN}\Zo$ the exceptional support with respect to $\perm\TN$ and $\iI\Zo$ as in Definition \ref{Def:ExceptionalSupport}.
For $\forall\pAl=(\lAl_*,\Al_0,\alpha_1)\in\es{\perm\TN}\Zo$ as in \eqref{LatentIdentity}, the exponent $\alpha_1$ of the primary variable $x_1$ is a constant.
\end{lemma}
\begin{proof}
For $\forall\pAl_1,\pAl_2\in\es{\perm\TN}\Zo$, let us denote $\Delta\pAl:=\pAl_1-\pAl_2$.
Based upon $\Delta\pAl\cdot\Ev{\perm\TN}\Zo=\bz$ and \eqref{LatentIdentity}, we have the following straightforward identity:
\begin{equation}\label{OrthogonalIdentity}
\biggl(\Dl\Al_*\cdot\TOm+\Delta\Al_0\cdot
\begin{bmatrix}\La\\ \TE\end{bmatrix}\biggr)\cdot\TD\cdot [\,\bl~\TE]
+\Delta\alpha_1\N_1=\bm{0},
\end{equation}
where $\Delta\pAl=(\Dl\Al_*,\Delta\Al_0,\Delta\alpha_1)$.
The conclusion is an easy consequence of \eqref{PrimaryVector} and \eqref{OrthogonalIdentity}.
\end{proof}

\begin{definition}\label{Def:LatentPrimaryCompnt}
{\upshape (Latent primary variable $x_1$ and its exponent $\alpha_1$; exceptional and non-exceptional exponents; latent primary component $\wt\alpha_1$)}

The primary variable $x_1$ of the Weierstrass polynomial $w(\perm\X)$ is called the \emph{latent} primary variable with which the decreasing process of the singularity height associated becomes latent.
Accordingly its exponent $\alpha_1$ is called the \emph{latent} primary exponent.

Let $q(\Z)$ be the partial transform of the Weierstrass polynomial $w(\perm\X)$ as in \eqref{NormalTransform} under the reduced monomial transformation $\CT{\per\TN}$.
For $\forall\Ga=(\Ga_*,\Ga_0)\in\supp(q(\Z))$, the exponents $\Ga_0$ of the exceptional variables $\Z_0$ are called the \emph{exceptional} exponents of $q(\Z)$ whereas $\Ga_*$ of the non-exceptional variables $\Z_*$ are called the \emph{non-exceptional} exponents.
Due to the partial factorization as in \eqref{NormalTransform}, the exceptional exponents $\Ga_0=(\pAl-\perm\A)\cdot\Ev{\perm\TN}\Zo$.
The component $\wt\alpha_1:=\alpha_1-a_1$ of $\pAl-\perm\A$ is called the latent primary \emph{component} of the exceptional exponents $\Ga_0$.
\end{definition}

Let $q(\Z_*,\bz)$ be the proper transform of the Weierstrass polynomial $w(\perm\X)$ as in \eqref{ProperTransform-N} under the reduced monomial transformation $\CT{\per\TN}$ and partial factorization.
Lemma \ref{Lemma:ConstantPrimaryExp} indicates that all the latent primary components of the exceptional exponents of $q(\Z_*,\bz)$ equal zero, i.e., $\wt\alpha_1=0$.
Let us reorganize the terms of the partial transform $q(\Z)$ in \eqref{NormalTransform} with respect to the latent primary component $\wt\alpha_1$ as follows.
\begin{equation}\label{PreReorganize}
q(\Z)=q(\Z_*,\bz)+\ol q_0(\Z)+q_1(\Z):=q_0(\Z)+q_1(\Z)
\end{equation}
with the latent primary component $\wt\alpha_1$ of $\ol q_0(\Z)$ and $q_0(\Z)$ equaling zero whereas none of those of $q_1(\Z)$ equaling zero.

We proceed with the resolution algorithm by localizing the partial transform $q(\Z)$ at a reduced branch point $(\bs_*,\bz)$.
Let $\wt q(\wt\Z)$ and $q_*(\wt\Z_*)$ denote the localized partial and proper transforms respectively as in Definition \ref{Def:Localization}.
For a new set of variables $\Xp=(\LEs x1n\prime\prime)$ that are partitioned as $\X'_*$ and $\X'_0$ as per the prior exceptional index set $\iI\Zo=\dotc{k+1}n$, let us make a non-degenerate linear modification $\CL\Xp$ defined as
\begin{equation}\label{LatentLinearModif}
\X_*'=\X_*'(\wt\Z_*);\qquad
\X_0'=\Z_0,
\end{equation}
which transforms $\wt q(\wt\Z)$ into the following apex form denoted as $f(\Xp)$:
\begin{equation}\label{LtnApexForm}
f(\Xp)=c_{\D'}{x'_1}^{d'}+\sum_{\Al'\in\supp(f)\setminus\D'}c_{\Al'}\Xp^{\Al'}.
\end{equation}
The primary variable is $x'_1$ whose associated singularity height $d'=\ord (q_*(\wt\Z_*))$.
Here $\D'$ denotes the apex $\D'=(d',\bz)$ and $c_{\D'}\in\BB F^*$.
The singularity height might have a temporary increase at a regular reduced branch point due to the inconsistency in \eqref{IncstnMnmlTrnsfrm}, that is, $d'\ge d$ with $d$ being the singularity height associated with the primary variable $x_1$ in \eqref{PrimaryForm}.

It is easy to see that neither the localization of $q(\Z)$ nor linear modification $\CL\Xp$ in \eqref{LatentLinearModif} alters the exceptional exponents $\Ga_0$ of $q(\Z)$.
Hence for simplicity the exponents $\Al'_0$ of the variables $\X'_0$ as in \eqref{LatentLinearModif} are also called the exceptional exponents of $f(\Xp)$ in \eqref{LtnApexForm} hereafter.
Similar to \eqref{PreReorganize}, let us reorganize the terms of the apex form $f(\Xp)$ in \eqref{LtnApexForm} by the latent primary component $\wt\alpha_1$ of their exceptional exponents as in Definition \ref{Def:LatentPrimaryCompnt} according to \eqref{LatentIdentity} as follows.
\begin{equation}\label{ReorganizedFunc}
f(\Xp)=c_{\D'}{x'_1}^{d'}+\ol f_0(\Xp)+\phi(\Xp):=f_0(\Xp)+\phi(\Xp),
\end{equation}
where all the latent primary components of the exceptional exponents of $f_0(\Xp)$ equal zero whereas none of those of $\phi(\Xp)$ equal zero.

In order to leave the latent primary components of the exceptional exponents of all the functions in \eqref{ReorganizedFunc} unaltered through the Weierstrass reduction, let us invoke Weierstrass preparation theorem and then complete perfect power on the function $f_0(\Xp)$ the latent primary components of whose exceptional exponents equal zero.
In this way we can represent $f_0(\Xp)$ in \eqref{ReorganizedFunc} into a Weierstrass form as follows.
\begin{equation}\label{LatentWeierstrassForm}
\begin{aligned}
f_0(\Xp)&=\biggl[{x'_1}^{d'}+\sum_{j=2}^{d'}c_j(\lXp){x'_1}^{d'-j}\biggr]
(c+r(\Xp))\\
&:=w(\Xp)(c+r(\Xp)),
\end{aligned}
\end{equation}
where $c_j(\bz)=0$ for $1<j\le d'$ and $c\in\BB F^*$.
The variables $\lXp:=(\LEs x2n\prime\prime)$ and the redundant function $r(\Xp)$ satisfies $r(\bz)=0$.

Henceforth we abuse the notations a bit and assume that the function $f_0(\Xp)$ in \eqref{ReorganizedFunc} bears the Weierstrass form in \eqref{LatentWeierstrassForm} and $\phi(\Xp)$ in \eqref{ReorganizedFunc} conforms with the completion of perfect power in \eqref{LatentWeierstrassForm}.

\begin{definition}\label{Def:LatentReductions}
{\upshape (Latent reducible and remainder functions; latent gradation; latent preliminary and Weierstrass reductions)}

The functions $f_0(\Xp)$ and $\phi(\Xp)$ in \eqref{ReorganizedFunc} are called the latent \emph{reducible} and \emph{remainder} functions of the apex form $f(\Xp)$ respectively.
The partial gradation of $f(\Xp)$ into the latent reducible and remainder functions as in \eqref{ReorganizedFunc} by the latent primary component $\wt\alpha_1$ is called a \emph{latent gradation} of $f(\Xp)$.

The partial factorization in \eqref{NormalTransform}, localization of $q(\Z)$, linear modification $\CL\Xp$ in \eqref{LatentLinearModif} and latent gradation in \eqref{ReorganizedFunc} constitute the procedure of \emph{latent} preliminary reduction; whereas the invocation of Weierstrass preparation theorem and completion of perfect power on the latent reducible function $f_0(\Xp)$ in \eqref{LatentWeierstrassForm} constitute the \emph{latent} Weierstrass reduction.
\end{definition}

It is easy to see that after the latent preliminary and Weierstrass reductions, we have the following trivial identity for the exceptional variables and their exponents:
\begin{equation}\label{VariableIdentity}
\X_0'=\Z_0;\qquad\Al'_0=\Ga_0.
\end{equation}

Now we are ready for the next step of partial resolution of inconsistent singularities.
Let $f(\Xp)$ be the apex form in \eqref{ReorganizedFunc} after the latent preliminary and Weierstrass reductions with the latent reducible function $f_0(\Xp)$ bearing the Weierstrass form in \eqref{LatentWeierstrassForm}.
Based on a refined vertex cone of the Newton polyhedron $\NP(f(\Xp))$ whose generators are denoted as $\LEs\V1n\prime\prime$, let us define an $n$ by $n$ exponential matrix $\TM'=[\MEs\V1n\prime\prime]$ whose associated monomial transformation $\CT{\TM'}$ is defined as $\Xp=\CT{\TM'}(\Yp):=\Yp^{\TM'}$ with new variables $\Yp:=(\LEs y1n\prime\prime)$.
Let us assume that the exceptional index set $\iI\Yop$ is in a canonical form $\dotc{l+1}n$ satisfying $|\iI\Yop|=n-l<|\iI\Xop|=|\iI\Zo|=n-k$.

When the primary variable $x'_1$ is consistent with the exponential matrix $\TM'$ and exceptional index set $\iI\Yop$, suppose that $\TM'$ is in a canonical consistency form as in Definition \ref{Def:Consistency-1}, i.e., the exceptional submatrix $\Evv\TM\Yop'$ satisfies $\det\Evv\TM\Yop'\ne 0$.
Let $\TN'$ denote the canonical reduction of $\TM'$ with respect to $\iI\Yop$, which is similar to the reduced exponential matrix $\TN$ in \eqref{CanonicalMatrix} and bears the following canonical form with the exceptional index set $\iI\Zop=\iI\Yop$:
\begin{equation}\label{CstnExpntlMtrx}\begin{aligned}
\setlength{\arraycolsep}{10pt}
    \TN':=
    \begin{array}{r @{} c @{\hs{-0.5pt}} c @{\hs{-3.6pt}} p{4mm}
    @{\hs{-1mm}} *{3}c @{\hs{4pt}} c @{\hs{4pt}} *{3}c
    @{\hs{-0.8mm}} p{4mm} @{\hs{-2.5mm}} l}
    &           &           &           &           \E_1        &
    \cdots      &           \E_l        &           &
    \V'_{l+1}   &           \cdots      &           \V'_n       &
    &                                                                       \\
    \mr{12}{$\rule{0mm}{24mm}$}         &           x'_1        &
    \mr{12}{$\rule{0mm}{24mm}$}         &
    \raisebox{5mm}[0pt]{\mr{12}{$\left[\rule{0mm}{80pt}\right.$}}
    &           \mc{3}{c}{\mr{6}{$\TE$}}            &           \vline
    &
    \mc{3}{c}{\mr{3}{\raisebox{-4mm}[0pt]{$\TOm'_1\TD'$}}}
    &
    \raisebox{5mm}[0pt]{\mr{12}{$\left.\rule{0mm}{80pt}\right]$}}
    &           \mr{12}{$\rule{0mm}{24mm}$}                                 \\
    &           \vdots      &           &           &           &
    &           &           \vline      &           &           &
    &           &                                                           \\
    &           x'_k        &           &           &           &
    &           &           \vline      &           &           &
    &           &                                                           \\
    \cline{2-2}\cline{8-12}
    &           x'_{k+1}    &           &           &           &
    &           &           \vline      &
    \mc{3}{c}{\mr{3}{$\TOm'_2\TD'$}}    &           &                       \\
    &           \vdots      &           &           &           &
    &           &           \vline      &           &           &
    &           &                                                           \\
    &           x'_l        &           &           &           &
    &           &           \vline      &           &           &
    &           &                                                           \\
    \cline{2-12}
    &           x'_{l+1}    &           &           &
    \mc{3}{c}{\mr{4}{$\TO$}}            &           \vline      &
    \mc{3}{c}{\mr{4}{$\TD'$}}           &           &                       \\
    &           \vdots      &           &           &           &
    &           &           \vline      &           &           &
    &           &                                                           \\
    &           x'_n        &           &           &           &
    &           &           \vline      &           &           &
    &           &                                                           \\
    &           &           &           &           z'_1        &
    \cdots      &           z'_l        &           \vline      &
    z'_{l+1}    &           \cdots      &           z'_n        &
    &                                                                       \\
    \end{array}.
\end{aligned}\end{equation}
Here $\TD'$ denotes the $(n-l)$ by $(n-l)$ exceptional submatrix $\Evv\TN\Zop'=\Evv\TM\Yop'$ such that $\det\TD'\ne 0$.
The submatrices $\TOm'_1$ and $\TOm'_2$ whose elements are in $\BB Q$ are of dimensions $k$ by $(n-l)$ and $(l-k)$ by $(n-l)$ respectively.

When the primary variable $x'_1$ is inconsistent with the exponential matrix $\TM'$ and exceptional index set $\iI\Yop$, suppose that $\iI\Yop$ is in canonical form $\dotc{l+1}n$ and satisfies $|\iI\Yop|=n-l\le |\iI\Xop|=|\iI\Zo|=n-k$.
Same as the canonical reduction procedure to obtain the reduced exponential matrix $\perm\TN$ in \eqref{IncstnMnmlTrnsfrm}, we make a permutation and relabeling of the row vectors of $\TM'$ and their corresponding variables $\Xp$ to acquire a new exponential matrix $\perm\TM'$ whose exceptional submatrix satisfies $\det\hs{-1pt}\Evv{\perm\TM}\Yop'\ne 0$ and whose last row vector equals the primary row vector $\M'_1$ of $\TM'$.
The canonical reduction of $\perm\TM'$ with respect to $\iI\Yop$ is denoted as $\perm\TN'$ and bears the canonical form in \eqref{IncstnExpntlMtrx} with the exceptional index set $\iI\Zop=\iI\Yop$.
In particular, the $(n-l-1)$ by $(n-l-1)$ submatrix $\TD'$ in \eqref{IncstnExpntlMtrx} satisfies $\det\TD'\ne 0$.
The matrix $\TB'$ denotes the submatrix $[\TD'\bl'~\TD']$ in \eqref{IncstnExpntlMtrx}.
The vector $\bl'$ in \eqref{IncstnExpntlMtrx} is an $(n-l-1)$-dimensional column vector in $\BB Q^{n-l-1}$.
The $(k-1)$ by $(n-l-1)$ submatrix $\TOm'_1$ and $(l-k+1)$ by $(n-l-1)$ submatrix $\TOm'_2$ are composed of elements in $\BB Q$.
The vector $\N'_1\in\DN{n-l}$ resembles the vector $\N_1$ in \eqref{IncstnMnmlTrnsfrm} and is part of the primary row vector $\M'_1$ of the original exponential matrix $\TM'$.
Please note that we abuse the notations a bit here and use the same notations $\TOm'_1$ and $\TOm'_2$ for different matrices of $\TN'$ in \eqref{CstnExpntlMtrx} and $\perm\TN'$ in \eqref{IncstnExpntlMtrx} since the consistency and inconsistency cases are mutually exclusive.

The exceptional submatrix of $\perm\TN'$ is non-degenerate as $\det\hs{-1pt}\Evv{\perm\TN}\Zop'=\det\hs{-1pt}\Evv{\perm\TM}\Yop'\ne 0$, which, similar to \eqref{PrimaryVector}, indicates that
\begin{equation}\label{NonDegVector}
\N'_1\cdot\begin{bmatrix}-1\\ \bl'\end{bmatrix}\ne 0.
\end{equation}

\begin{equation}\label{IncstnExpntlMtrx}\begin{aligned}
\setlength{\arraycolsep}{8pt}
    \perm\TN':=
    \begin{array}{r @{} c @{\hs{-0.5pt}} c @{\hs{-4pt}} p{4mm}
    @{\hs{-1mm}} *{3}c @{\hs{2pt}} c @{\hs{2pt}} c @{} c
    @{\hs{4pt}} *{3}c @{\hs{-0.8mm}} p{4mm} @{\hs{-2.5mm}} l}
    &           &           &           &           \E_1        &
    \cdots      &           \E_l        &           &           \V'_{l+1}
    &           &           \V'_{l+2}   &           \cdots      &
    \V'_n       &           &                                               \\
    \mr{12}{$\rule{0mm}{24mm}$}         &           x'_2        &
    \mr{12}{$\rule{0mm}{24mm}$}         &
    \mr{11}{$\left[\rule{0mm}{83.75pt}\right.$}     &
    \mc{3}{c}{\mr{7}{$\TE$}}            &           \vline      &
    \mc{5}{c}{\mr{3}{\quad $\TOm'_1\TB'$}}          &
    \mr{11}{$\left.\rule{0mm}{83.75pt}\right]$}     &
    \mr{12}{$\rule{0mm}{24mm}$}                                             \\
    &           \vdots      &           &           &           &
    &           &           \vline      &           &           &
    &           &           &           &                                   \\
    &           x'_k        &           &           &           &
    &           &           \vline      &           &           &
    &           &           &           &                                   \\
    \cline{2-2}\cline{8-14}
    &           x'_{k+1}    &           &           &           &
    &           &           \vline      &
    \mc{5}{c}{\mr{3}{\quad $\TOm'_2\TB'$}}          &           &           \\
    &           \vdots      &           &           &           &
    &           &           \vline      &           &           &
    &           &           &           &                                   \\
    &           x'_{l+1}    &           &           &           &
    &           &           \vline      &           &           &
    &           &           &           &                                   \\
    \cline{2-14}
    &           x'_{l+2}    &           &           &
    \mc{3}{c}{\mr{4}{$\TO$}}            &           \vline      &
    \mr{3}{$\TD'\bl'$}      &           \vline      &
    \mc{3}{c}{\mr{3}{$\TD'$}}           &           &                       \\
    &           \vdots      &           &           &           &
    &           &           \vline      &           &           \vline
    &           &           &           &           &                       \\
    &           x'_n        &           &           &           &
    &           &           \vline      &           &           \vline
    &           &           &           &           &                       \\
    \cline{2-2}\cline{8-14}
    &           x'_1\rule{0mm}{9.5pt}   &           &           &
    &           &           &           \vline      &
    \mc{5}{c}{\N'_1}        &           &                                   \\
    &           &           &           &           z'_1        &
    \cdots      &           z'_l        &           \vline      &
    z'_{l+1}    &           \vline      &           z'_{l+2}    &
    \cdots      &           z'_n        &           &                       \\
    \end{array}
\end{aligned}\end{equation}

Let $\CT{\TN'}$ and $\CT{\per\TN'}$ denote the reduced monomial transformations associated with the reduced exponential matrices $\TN'$ in \eqref{CstnExpntlMtrx} and $\perm\TN'$ in \eqref{IncstnExpntlMtrx} respectively.
The exceptional index set $\iI\Zop=\iI\Yop$ with the non-exceptional and exceptional variables of $\Zp$ denoted as $\Z_*':=(\LEs z1l\prime\prime)$ and $\Z_0':=(\LEs z{l+1}n\prime\prime)$ respectively.
The total transform of the apex form $f(\Xp)$ in \eqref{ReorganizedFunc} under $\CT{\TN'}$ or $\CT{\per\TN'}$ is denoted as $f_{\TN'}(\Zp)$ or $f_{\per\TN'}(\Zp)$.

\begin{align}\label{FirstInterimMatrix}
    \perm\TS&:=
    \begin{array}{r @{} c @{\hs{-0.5pt}} c @{\hs{-3pt}} p{4mm}
    @{} *{3}c @{\hs{2pt}} c @{\hs{4pt}} *{3}c @{\hs{2pt}}
    c @{\hs{4pt}} c @{} p{4mm} @{\hs{-2.5mm}} l}
    &           &           &           &           \E_1        &
    \cdots      &           \E_k        &           &           &
    &           &           &           \E_n        &           &           \\
    \mr{10}{$\rule{0mm}{24mm}$}         &           x_2         &
    \mr{10}{$\rule{0mm}{24mm}$}         &
    \raisebox{1pt}[0pt]{\mr{9}{$\left[\rule{0mm}{67pt}\right.$}}
    &           \mc{3}{c}{\mr{4}{$\TE$}}            &           \vline
    &           \mc{3}{c}{\mr{3}{$\TOm\TD$}}        &           \vline
    &           \mr{7}{$\TO$}           &
    \raisebox{1pt}[0pt]{\mr{9}{$\left.\rule{0mm}{67pt}\right]$}}
    &           \mr{10}{$\rule{0mm}{24mm}$}                                 \\
    &           \vdots      &           &           &           &
    &           &           \vline      &           &           &
    &           \vline      &           &           &                       \\
    &           x_k         &           &           &           &
    &           &           \vline      &           &           &
    &           \vline      &           &           &                       \\
    \cline{2-2}\cline{8-11}
    &           x_{k+1}     &           &           &           &
    &           &           \vline      &
    \mc{3}{c}{\raisebox{-1.8pt}[0pt]{$\La\TD$}}     &           \vline
    &           &           &                                               \\
    \cline{2-11}
    &           x_{k+2}     &           &           &
    \mc{3}{c}{\mr{3}{$\TO$}}            &           \vline      &
    \mc{3}{c}{\mr{3}{$\TD$}}            &           \vline      &
    &           &                                                           \\
    &           \vdots      &           &           &           &
    &           &           \vline      &           &           &
    &           \vline      &           &           &                       \\
    &           x_n         &           &           &           &
    &           &           \vline      &           &           &
    &           \vline      &           &           &                       \\
    \cline{2-14}
    &           x_1         &           &           &
    \mc{7}{c}{\raisebox{-2pt}[0pt]{$\TO$}}          &           \vline
    &           \raisebox{-2pt}[0pt]{$1$}           &           &           \\
    &           &           &           &           z_1         &
    \cdots      &           z_k         &           \vline      &
    \zeta_{k+2} &           \cdots      &           \zeta_n     &
    \vline      &           x_1         &           &                       \\
    \end{array}
\end{align}

There is another perspective that can elucidate the above partial resolution of singularities.
The reduced exponential matrix $\perm\TN$ in \eqref{IncstnMnmlTrnsfrm} can be decomposed as:
\begin{equation}\label{MatrixDecomposition}
\perm\TN=\perm\TS\cdot\perm\TT
\end{equation}
with the interim exponential matrices $\perm\TS$ and $\perm\TT$ being defined as in \eqref{FirstInterimMatrix} and \eqref{SecondInterimMatrix}, in which the interim variables $(\LE\zeta{k+2}n):=\Ze$ are associated with the monomial transformations $\CT{\per\TS}$ and $\CT{\per\TT}$ in accordance with \eqref{MatrixDecomposition}.

The decomposition in \eqref{MatrixDecomposition} is essentially a decomposition of the exceptional column submatrix $\Ev{\perm\TN}\Zo$ in \eqref{IncstnMnmlTrnsfrm} and it adheres to the following principles.
The first one is that the latent primary exponent $\alpha_1$ of the latent primary variable $x_1$ should be unaltered by the exponential matrix $\perm\TS$ in \eqref{FirstInterimMatrix};
The second one is that the non-exceptional column submatrix $\Ec{\perm\TS}\Za$ in \eqref{FirstInterimMatrix} as in Definition \ref{Def:ExceptionalSupport} should be identical to $\Ec{\perm\TN}\Za$ in \eqref{IncstnMnmlTrnsfrm};
The third one is that the difference between the exponents of the interim variables $\Ze$ in \eqref{FirstInterimMatrix} and those of the exceptional variables $(\LE z{k+2}n)$ in \eqref{IncstnMnmlTrnsfrm} or \eqref{SecondInterimMatrix} should be a multiple of the latent primary exponent $\alpha_1$, which ensures that the decomposition in \eqref{MatrixDecomposition} is compatible with the latent gradation in \eqref{ReorganizedFunc}.

\begin{align}\label{SecondInterimMatrix}
    \perm\TT&:=
    \begin{array}{r @{} c @{\hs{-0.5pt}} c @{\hs{-3pt}} p{4mm} @{}
    *{3}c @{\hs{2pt}} c @{\hs{4pt}} c @{\hs{2pt}} c @{\hs{3pt}}
    *{3}c @{} p{4mm} @{\hs{-2.5mm}} l}
    &           &           &           &           \E_1        &
    \cdots      &           \E_k        &           &           &
    &           &           &           &           &                       \\
    \mr{8}{$\rule{0mm}{20mm}$}          &           z_1         &
    \mr{8}{$\rule{0mm}{20mm}$}          &
    \raisebox{1pt}[0pt]{\mr{8}{$\left[\rule{0mm}{60pt}\right.$}}
    &           \mc{3}{c}{\mr{3}{$\TE$}}            &           \vline
    &           \mc{5}{c}{\mr{3}{$\TO$}}            &
    \raisebox{1pt}[0pt]{\mr{8}{$\left.\rule{0mm}{60pt}\right]$}}
    &           \mr{8}{$\rule{0mm}{20mm}$}                                 \\
    &           \vdots      &           &           &           &
    &           &           \vline      &           &           &
    &           &           &           &                                   \\
    &           z_k         &           &           &           &
    &           &           \vline      &           &           &
    &           &           &           &                                   \\
    \cline{2-14}
    &           \zeta_{k+2} &           &           &
    \mc{3}{c}{\mr{4}{$\TO$}}            &           \vline      &
    \mr{3}{$\bl$}           &           \vline      &
    \mc{3}{c}{\mr{3}{$\TE$}}            &           &                       \\
    &           \vdots      &           &           &           &
    &           &           \vline      &           &           \vline
    &           &           &           &           &                       \\
    &           \zeta_n     &           &           &           &
    &           &           \vline      &           &           \vline
    &           &           &           &           &                       \\
    \cline{2-2}\cline{8-14}
    &           x_1\rule{0mm}{9.5pt}    &           &           &
    &           &           &           \vline      &
    \mc{5}{c}{\raisebox{-2pt}[0pt]{$\N_1$}}         &           &           \\
    &           &           &           &           z_1         &
    \cdots      &           z_k         &           \vline      &
    z_{k+1}     &           \vline      &           z_{k+2}     &
    \cdots      &           z_n         &           &                       \\
    \end{array}
\end{align}

The new perspective is to consider the above monomial transformation $\CT{\per\TS}$ in \eqref{FirstInterimMatrix} on the Weierstrass polynomial $w(\perm\X)$ in \eqref{NormalForm} with $\perm\X$ denoting the new variables corresponding to the canonical form $\perm\TN$ in \eqref{IncstnMnmlTrnsfrm}, from which a total transform $p(\Z_*,\Ze,x_1)$ is acquired as follows.
\begin{equation}\label{SyntheticTransform}
p(\Z_*,\Ze,x_1):=w(\CT{\per\TS}(\Z_*,\Ze,x_1)).
\end{equation}
Lemma \ref{Lemma:ConstantPrimaryExp} shows that the latent primary exponent $\alpha_1$ is a constant on the exceptional support $\es{\perm\TN}\Zo$.
More specifically, let $\A=(\LE a1n)\in\es{\perm\TN}\Zo$ be the vertex associated with the original exponential matrix $\TM$.
Then we have $\alpha_1=a_1$ for $\forall\pAl=(\LE\alpha 2n,\alpha_1)\in\es{\perm\TN}\Zo$.
Hence similar to \eqref{PreReorganize}, we make a partial gradation of $p(\Z_*,\Ze,x_1)$ in \eqref{SyntheticTransform} by the latent primary exponent $\alpha_1$ and as a result, the terms of $p(\Z_*,\Ze,x_1)$ are reorganized as follows.
\begin{equation}\label{InterimGradation}
p(\Z_*,\Ze,x_1)=x_1^{a_1}p_0(\Z_*,\Ze)+p_1(\Z_*,\Ze,x_1)
\end{equation}
such that $\alpha_1\ne a_1$ for $\forall (\Ga_*,\Al_\Ze,\alpha_1)\in\supp(p_1)$ with $\Ga_*$ and $\Al_\Ze$ denoting the exponents of the variables $\Z_*$ and $\Ze$ respectively.
Let $\iI\Ze$ denote the index set $\dotc{k+1}{n-1}$ whose corresponding column submatrix of $\perm\TS$ in \eqref{FirstInterimMatrix} is denoted as $\Ev{\perm\TS}\Ze$ as in Definition \ref{Def:ExceptionalSupport}.
Based on the partial factorization yielding the exceptional factors $\Z_0^{\A\cdot\Ev{\per\TN}\Zo}$ in \eqref{NormalTransform} as well as the above third principle for the decomposition in \eqref{MatrixDecomposition}, the function $p_0(\Z_*,\Ze)$ in \eqref{InterimGradation} can be partially factorized as follows.
\begin{equation}\label{InterimTransform}
\begin{aligned}
p_0&=\Ze^{\A\cdot\Ev{\per\TS}\Ze}\cdot
\biggl(c_\A\Z_*^{\A_*}+\sumsepm\sum_{\Al\in\es{\per\TN}\Zo\setminus\{\A\}}
\sumsepm c_\Al\Z_*^{\Al_*}+\sumsep\sum_{\Be\in\esc{\per\TN}\Zo}
\sumsepm c_\Be\Z_*^{\Be_*}\Ze^{(\Be-\A)\cdot\Ev{\per\TS}\Ze}\biggr)\\
&:=\Ze^{\A\cdot\Ev{\per\TS}\Ze}\cdot q_0(\Z_*,\Ze),
\end{aligned}
\end{equation}
where the complementary exponential set $\esc{\perm\TN}\Zo$ is defined as:
\[
\esc{\perm\TN}\Zo:=\{\Be=(\LE\beta 1n)\in\supp(w)\setminus\es{\perm\TN}\Zo
\colon\beta_1=a_1\}
\]
with $w$ denoting the Weierstrass polynomial $w(\perm\X)$ in \eqref{NormalForm} and $a_1$ the first component of the vertex $\A$ in \eqref{InterimGradation}.

From \eqref{InterimTransform} it is easy to see that the function $q_0(\Z_*,\bz)$ equals the proper transform $q(\Z_*,\bz)$ of $w(\perm\X)$ under $\CT{\per\TN}$ in \eqref{ProperTransform-N}.
Consequently, we can invoke the same localization of the non-exceptional variables $\Z_*$ and essentially the same linear modification as those in \eqref{LatentLinearModif} with $\X_*'=\X_*'(\wt\Z_*)$ that transform the above total transform $p(\Z_*,\Ze,x_1)$ in \eqref{InterimGradation} into a new apex form $f(\X'_*,\Ze,x_1)$.
According to the partial gradation in \eqref{InterimGradation} and partial factorization in \eqref{InterimTransform}, the apex form $f(\X'_*,\Ze,x_1)$ can be organized into the following form:
\begin{equation}\label{InterimReorgFunc}
f(\X'_*,\Ze,x_1)=x_1^{a_1}\Ze^{\A\cdot\Ev{\per\TS}\Ze}\cdot f_0(\X'_*,\Ze)
+\phi(\X'_*,\Ze,x_1)
\end{equation}
such that the latent primary exponent $\alpha_1$ in $\supp (\phi)$ satisfies $\alpha_1\ne a_1$, which is similar to the latent gradation in \eqref{ReorganizedFunc}.
In particular, the functions $f_0(\X'_*,\Ze)$ and $\phi(\X'_*,\Ze,x_1)$ in \eqref{InterimReorgFunc} are the respective localizations and linear modifications of $q_0(\Z_*,\Ze)$ in \eqref{InterimTransform} and $p_1(\Z_*,\Ze,x_1)$ in \eqref{InterimGradation}.
Similar to \eqref{LatentWeierstrassForm}, we invoke Weierstrass preparation theorem and then complete perfect power on $f_0(\X'_*,\Ze)$ so as to represent it in a Weierstrass form as follows.
\begin{equation}\label{InterimWeierstrassForm}
\begin{aligned}
f_0(\XIp)&=\biggl[{x'_1}^{d'}+\sum_{j=2}^{d'}c_j(\_\XIp){x'_1}^{d'-j}\biggr]
(c+r(\XIp))\\
&:=w(\XIp)(c+r(\XIp))
\end{aligned}
\end{equation}
with $c_j(\bz)=0$ for $1<j\le d'$ and $c\in\BB F^*$.
Here for brevity of notations, we define new variables $\XIp:=(\X'_*,\Ze)$ and write $f_0(\X'_*,\Ze)$ in \eqref{InterimReorgFunc} as $f_0(\XIp)$ in \eqref{InterimWeierstrassForm}.
The variables $\_\XIp:=(\LEs\xi 2n\prime\prime):=(\lX'_*,\Ze)$ with $\lX'_*:=(\LEs x2k\prime\prime)$ and the redundant function $r(\XIp)$ satisfies $r(\bz)=0$.
In particular, a deficient contraction as in Definition \ref{Def:DeficientContraction} might be necessary  so that the redundant function $r(\XIp)$ can be disregarded when the interim exponential matrix $\perm\TS$ in \eqref{FirstInterimMatrix} is deficient with respect to $\iI\Zo$.
Henceforth we just assume that the function $f_0(\XIp)$ in \eqref{InterimReorgFunc} bears the Weierstrass form in \eqref{InterimWeierstrassForm}.

Based on the identity \eqref{VariableIdentity} and interim exponential matrix $\perm\TT$ in \eqref{SecondInterimMatrix}, let us define a new interim exponential matrix $\perm\wh\TT$ as in \eqref{ModifiedInterimMatrix} resembling $\perm\TT$.

\begin{equation}\label{ModifiedInterimMatrix}\begin{aligned}
    \perm\wh\TT:=
    \begin{array}{r @{} c @{\hs{-0.5pt}} c @{\hs{-3pt}} p{4mm} @{}
    *{3}c @{\hs{2pt}} c @{\hs{4pt}} c @{} c @{\hs{3pt}} *{3}c @{}
    p{4mm} @{\hs{-2.5mm}} l}
    &           &           &           &           \E_1        &
    \cdots      &           \E_k        &           &           &
    &           &           &           &           &                       \\
    \mr{8}{$\rule{0mm}{20mm}$}          &           x'_1        &
    \mr{8}{$\rule{0mm}{20mm}$}          &
    \raisebox{1pt}[0pt]{\mr{8}{$\left[\rule{0mm}{60pt}\right.$}}
    &           \mc{3}{c}{\mr{3}{$\TE$}}            &           \vline
    &           \mc{5}{c}{\mr{3}{$\TO$}}            &
    \raisebox{1pt}[0pt]{\mr{8}{$\left.\rule{0mm}{60pt}\right]$}}
    &           \mr{8}{$\rule{0mm}{20mm}$}                                 \\
    &           \vdots      &           &           &           &
    &           &           \vline      &           &           &
    &           &           &           &                                   \\
    &           x'_k        &           &           &           &
    &           &           \vline      &           &           &
    &           &           &           &                                   \\
    \cline{2-14}
    &           \zeta_{k+2} &           &           &
    \mc{3}{c}{\mr{4}{$\TO$}}            &           \vline      &
    \mr{3}{$\bl$}           &           \vline      &
    \mc{3}{c}{\mr{3}{$\TE$}}            &           &                       \\
    &           \vdots      &           &           &           &
    &           &           \vline      &           &           \vline
    &           &           &           &           &                       \\
    &           \zeta_n     &           &           &           &
    &           &           \vline      &           &           \vline
    &           &           &           &           &                       \\
    \cline{2-2}\cline{8-14}
    &           x_1\rule{0mm}{9.5pt}    &           &           &
    &           &           &           \vline      &
    \mc{5}{c}{\raisebox{-2pt}[0pt]{$\N_1$}}         &           &           \\
    &           &           &           &           x'_1        &
    \cdots      &           x'_k        &           \vline      &
    x'_{k+1}    &           \vline      &           x'_{k+2}    &
    \cdots      &           x'_n        &           &                       \\
    \end{array}
\end{aligned}\end{equation}

\begin{definition}\label{Def:InterimReductions}
{\upshape (Interim preliminary and Weierstrass reductions; interim transform $f(\XIp,x_1)$; interim reducible and remainder functions $f_0(\XIp)$ and $\phi(\XIp,x_1)$)}

The partial gradation in \eqref{InterimGradation}, partial factorization in \eqref{InterimTransform} as well as the localization and linear modification leading to \eqref{InterimReorgFunc} constitute the procedure of \emph{interim} preliminary reduction; whereas the invocation of Weierstrass preparation theorem and completion of perfect power in \eqref{InterimWeierstrassForm} constitute the \emph{interim} Weierstrass reduction.

After the interim monomial transformation $\CT{\per\TS}$ in \eqref{FirstInterimMatrix} and interim preliminary and Weierstrass reductions, the apex form $f(\XIp,x_1)$ in \eqref{InterimReorgFunc} is called the \emph{interim} transform of the Weierstrass polynomial $w(\perm\X)$ in \eqref{NormalForm}.
The functions $f_0(\XIp)$ and $\phi(\XIp,x_1)$ in \eqref{InterimReorgFunc} are called the interim \emph{reducible} and \emph{remainder} functions of the apex form $f(\XIp,x_1)$ respectively.
\end{definition}

The relation between the interim transform $f(\XIp,x_1)$ in \eqref{InterimReorgFunc} and apex form $f(\Xp)$ in \eqref{ReorganizedFunc} can be summarized into the following commutative diagram.
\begin{equation}\label{InterimCommut}
\begin{aligned}
\xymatrix{
w(\perm\X)  \ar[0,3]^-{\CT{\per\TN}~\text{and reductions}}
\ar[2,0]_(0.35){\CT{\per\TS}}^(0.35){\text{and reductions}}     &
&           &           f(\Xp)                                              \\
&           &           &                                                   \\
f(\XIp,x_1) \ar[-2,3]|-(0.47)\hole
_-(0.6){\text{and partial factorization}}^-(0.36){\CT{\per\wh\TT}} &
&           &
}
\end{aligned}
\end{equation}
The reductions ensuing $\CT{\per\TS}$ in the diagram \eqref{InterimCommut} refer to the interim preliminary and Weierstrass reductions in Definition \ref{Def:InterimReductions} whereas those ensuing $\CT{\per\TN}$ in \eqref{InterimCommut} refer to the latent preliminary and Weierstrass reductions in Definition \ref{Def:LatentReductions}.
The partial factorization in \eqref{InterimCommut} resembles that in \eqref{NormalTransform} as per the identity $\X_0'=\Z_0$ in \eqref{VariableIdentity}.
The commutativity of the above diagram follows from the fact that the invocation of Weierstrass preparation theorem and subsequent completion of perfect power involve only exponential additions and exponential scalar multiplications whereas the interim transformations $\CT{\per\TT}$ and $\CT{\per\wh\TT}$ are linear exponential transformations.
Please refer to \cite{ZS} for the details of the proof of Weierstrass preparation theorem.

The monomial transformations $\CT{\per\wh\TT}$ in \eqref{ModifiedInterimMatrix} and $\CT{\TN'}$ in \eqref{CstnExpntlMtrx} can be synthesized into a new monomial transformation $\CT{\TQ'}$ such that the synthetic exponential matrix $\TQ'=\perm\wh\TT\cdot\TN'$ bears the form in \eqref{CstnSyntheticMatrix}.
The $(n-k)$ by $(n-l)$ submatrix $\TOm'_3:=\bigl[\begin{smallmatrix}\TOm'_2\\ \TE\end{smallmatrix}\bigr]$ is as in \eqref{CstnExpntlMtrx}.
The submatrix $\TA:=[\,\bl~\TE]$ is of dimensions $(n-k-1)$ by $(n-k)$ and is a submatrix of $\perm\wh\TT$ in \eqref{ModifiedInterimMatrix}.
The $(n-k)$ by $(l-k)$ submatrix $\TF':=\bigl[\begin{smallmatrix}\TE\\ \TO\end{smallmatrix}\bigr]$ with $\TE$ being the $(l-k)$ by $(l-k)$ unit matrix.
The $(n-l)$ by $(n-l)$ submatrix $\TD'$ is as in \eqref{CstnExpntlMtrx}.
The $(n-k)$-dimensional row vector $\N_1$ is as in \eqref{ModifiedInterimMatrix}.

\begin{equation}\label{CstnSyntheticMatrix}\begin{aligned}
    \TQ':=
    \begin{array}{r @{} c @{\hs{-0.5pt}} c @{\hs{-3.6pt}} p{4mm}
    @{\hs{-1mm}} *{3}c @{\hs{2pt}} c @{\hs{2pt}} *{3}c @{\hs{2pt}}
    c @{\hs{2pt}} *{3}c @{\hs{-0.8mm}} p{4mm} @{\hs{-2.5mm}} l}
    &           &           &           &           \E_1        &
    \cdots      &           \E_k        &           &           &
    &           &           &                                               \\
    \mr{12}{$\rule{0mm}{24mm}$}         &           x'_1        &
    \mr{12}{$\rule{0mm}{24mm}$}         &
    \mr{12}{$\left[\rule{0mm}{90pt}\right.$}        &
    \mc{3}{c}{\mr{4}{$\TE$}}            &           \vline      &
    \mc{3}{c}{\mr{4}{$\TO$}}            &           \vline      &
    \mc{3}{c}{\mr{4}{$\TOm'_1\TD'$}}    &
    \mr{12}{$\left.\rule{0mm}{90pt}\right]$}        &
    \mr{12}{$\rule{0mm}{24mm}$}                                             \\
    &           x'_2        &           &           &           &
    &           &           \vline      &           &           &
    &           \vline      &           &           &           &
    &                                                                       \\
    &           \vdots      &           &           &           &
    &           &           \vline      &           &           &
    &           \vline      &           &           &           &
    &                                                                       \\
    &           x'_k        &           &           &           &
    &           &           \vline      &           &           &
    &           \vline      &           &           &           &
    &                                                                       \\
    \cline{2-16}
    &           \zeta_{k+2} &           &           &
    \mc{3}{c}{\mr{7}{$\TO$}}            &           \vline      &
    \mc{3}{c}{\mr{6}{$\TA\TF'$}}        &           \vline      &
    \mc{3}{c}{\mr{6}{$\TA\TOm'_3\TD'$}} &           &                       \\
    &           \vdots      &           &           &           &
    &           &           \vline      &           &           &
    &           \vline      &           &           &           &
    &                                                                       \\
    &           \zeta_l     &           &           &           &
    &           &           \vline      &           &           &
    &           \vline      &           &           &           &
    &                                                                       \\
    \cline{2-2}
    &           \zeta_{l+1} &           &           &           &
    &           &           \vline      &           &           &
    &           \vline      &           &           &           &
    &                                                                       \\
    &           \vdots      &           &           &           &
    &           &           \vline      &           &           &
    &           \vline      &           &           &           &
    &                                                                       \\
    &           \zeta_n     &           &           &           &
    &           &           \vline      &           &           &
    &           \vline      &           &           &           &
    &                                                                       \\
    \cline{2-2}\cline{8-16}
    &           x_1         &           &           &           &
    &           &           \vline      &
    \mc{3}{c}{\raisebox{-0.8pt}[0pt]{$\N_1\TF'$}}   &           \vline
    &
    \mc{3}{c}{\raisebox{-0.8pt}[0pt]{$\N_1\TOm'_3\TD'$}}        &
    &                                                                       \\
    &           &           &           &           z'_1        &
    \cdots      &           z'_k        &           \vline      &
    z'_{k+1}    &           \cdots      &           z'_l        &
    \vline      &           z'_{l+1}    &           \cdots      &
    z'_n        &           &                                               \\
    \end{array}
\end{aligned}\end{equation}

Let $\Evv\TQ\Zo'$ denote the submatrix of the synthetic exponential matrix $\TQ'$ in \eqref{CstnSyntheticMatrix} as in Definition \ref{Def:ExceptionalSupport} with the prior exceptional index set $\iI\Zo=\dotc{k+1}n$.
The non-degeneracy of $\TQ'$ is evident as per the following matrix identity:
\begin{equation}\label{NondegSynthetic-1}
\Evv\TQ\Zo'=\Evv{\perm\wh\TT}\Zo\cdot\TN'_{\iI\Zo\times\iI\Zo},
\end{equation}
where $\perm\wh\TT$ and $\TN'$ are as in \eqref{ModifiedInterimMatrix} and \eqref{CstnExpntlMtrx} respectively.

Similarly the monomial transformations $\CT{\per\wh\TT}$ in \eqref{ModifiedInterimMatrix} and $\CT{\per\TN'}$ in \eqref{IncstnExpntlMtrx} can be synthesized into a new monomial transformation $\CT{\per\TQ'}$ such that the synthetic exponential matrix $\perm\TQ'=\perm\wh\TT\cdot\perm\TN'$ bears the following form in \eqref{IncstnSyntheticMatrix}:
\begin{equation}\label{IncstnSyntheticMatrix}\begin{aligned}
    \perm\TQ':=
    \begin{array}{r @{} c @{\hs{-0.5pt}} c @{\hs{-3.6pt}} p{4mm}
    @{\hs{-1mm}} *{3}c @{\hs{2pt}} c @{\hs{3pt}} *{3}c @{\hs{2pt}}
    c @{\hs{2pt}} *{3}c @{\hs{-0.8mm}} p{4mm} @{\hs{-2.5mm}} l}
    &           &           &           &           \E_1        &
    \cdots      &           \E_{k-1}    &           &           &
    &           &           &                                               \\
    \mr{12}{$\rule{0mm}{24mm}$}         &           x'_2        &
    \mr{12}{$\rule{0mm}{24mm}$}         &
    \mr{12}{$\left[\rule{0mm}{90pt}\right.$}        &
    \mc{3}{c}{\mr{3}{$\TE$}}            &           \vline      &
    \mc{3}{c}{\mr{3}{$\TO$}}            &           \vline      &
    \mc{3}{c}{\mr{3}{$\TOm'_1\TB'$}}    &
    \mr{12}{$\left.\rule{0mm}{90pt}\right]$}        &
    \mr{12}{$\rule{0mm}{24mm}$}                                             \\
    &           \vdots      &           &           &           &
    &           &           \vline      &           &           &
    &           \vline      &           &           &           &
    &                                                                       \\
    &           x'_k        &           &           &           &
    &           &           \vline      &           &           &
    &           \vline      &           &           &           &
    &                                                                       \\
    \cline{2-16}
    &           \zeta_{k+2} &           &           &
    \mc{3}{c}{\mr{8}{$\TO$}}            &           \vline      &
    \mc{3}{c}{\mr{6}{$\TA\TF'$}}        &           \vline      &
    \mc{3}{c}{\mr{6}{$\TA\TOm'_3\TB'$}} &           &                       \\
    &           \vdots      &           &           &           &
    &           &           \vline      &           &           &
    &           \vline      &           &           &           &
    &                                                                       \\
    &           \zeta_{l+1} &           &           &           &
    &           &           \vline      &           &           &
    &           \vline      &           &           &           &
    &                                                                       \\
    \cline{2-2}
    &           \zeta_{l+2}             &           &           &
    &           &           &           \vline      &           &
    &           &           \vline      &           &           &
    &           &                                                           \\
    &           \vdots      &           &           &           &
    &           &           \vline      &           &           &
    &           \vline      &           &           &           &
    &                                                                       \\
    &           \zeta_n     &           &           &           &
    &           &           \vline      &           &           &
    &           \vline      &           &           &           &
    &                                                                       \\
    \cline{2-2}\cline{8-16}
    &           x_1         &           &           &
    &           &           &           \vline      &
    \mc{3}{c}{\raisebox{-0.8pt}[0pt]{$\N_1\TF'$}}   &           \vline
    &
    \mc{3}{c}{\raisebox{-0.8pt}[0pt]{$\N_1\TOm'_3\TB'$}}        &
    &                                                                       \\
    \cline{2-2}\cline{8-16}
    &           x'_1        &           &           &           &
    &           &           \vline      &
    \mc{3}{c}{\raisebox{-2pt}[0pt]{$\TO$}}          &           \vline
    &           \mc{3}{c}{\raisebox{-1.5pt}[0pt]{$\N'_1$}}
    &           &                                                           \\
    &           &           &           &           z'_1        &
    \cdots      &           z'_{k-1}    &           \vline      &
    z'_k        &           \cdots      &           z'_l        &
    \vline      &           z'_{l+1}    &           \cdots      &
    z'_n        &           &                                               \\
    \end{array},
\end{aligned}\end{equation}
where the $(n-k)$ by $(n-l-1)$ submatrix $\TOm'_3:=\bigl[\begin{smallmatrix}\TOm'_2\\ \TE\end{smallmatrix}\bigr]$ is as in \eqref{IncstnExpntlMtrx}.
The submatrix $\TA:=[\,\bl~\TE]$ is of dimensions $(n-k-1)$ by $(n-k)$ and is a submatrix of $\perm\wh\TT$ in \eqref{ModifiedInterimMatrix}.
The $(n-k)$ by $(l-k+1)$ submatrix $\TF':=\bigl[\begin{smallmatrix}\TE\\ \TO\end{smallmatrix}\bigr]$ with $\TE$ being the $(l-k+1)$ by $(l-k+1)$ unit matrix.
The $(n-l-1)$ by $(n-l)$ submatrix $\TB'$ denotes the submatrix $[\TD'\bl'~\TD']$ in \eqref{IncstnExpntlMtrx}.
Please note that we use the same notations $\TF'$ and $\TOm'_3$ in \eqref{CstnSyntheticMatrix} and \eqref{IncstnSyntheticMatrix} for different matrices since the cases of $\TQ'$ and $\perm\TQ'$ are mutually exclusive.

Let the index sets $\iJ_\Z:=\dotc k{n-1}$ and $\iJ_1:=\dotc kl\cup\dotc{l+2}n$.
It is evident that the non-degeneracy of the synthetic exponential matrix $\perm\TQ'$ in \eqref{IncstnSyntheticMatrix} follows from that of its submatrix $\Evv{\perm\TQ}k^\prime$ with $\iI k=\dotc kn$, which can further be corroborated by the following matrix identity:
\begin{equation}\label{NondegSynthetic-2}
\perm\TQ'_{\iJ_\Z\times\iJ_1}=\Evv{\perm\wh\TT}\Zo
\cdot\perm\TN'_{\iJ_\Z\times\iJ_1}
\end{equation}
with $\perm\wh\TT$ and $\perm\TN'$ being as in \eqref{ModifiedInterimMatrix} and \eqref{IncstnExpntlMtrx} respectively.
The prior exceptional index set $\iI\Zo$ equals $\dotc{k+1}n$ as before.

The construction per se of the exponential matrices $\TQ'$ in \eqref{CstnSyntheticMatrix} and $\perm\TQ'$ in \eqref{IncstnSyntheticMatrix} indicates the commutativity of the following diagram.
\begin{equation}\label{SynCommut}
\begin{aligned}
\xymatrix{
&       &       &       &   f(\Xp)
\ar[2,0]_-{\CT{\TN'}}^-{\text{or~}\CT{\per\TN'}}                            \\
&       &       &       &                                                   \\
f(\XIp,x_1) \ar[-2,4]|-(0.51)\hole
^-(0.36){\CT{\per\wh\TT}\text{~and partial factorization}\hs{-10mm}}
\ar[0,4]^-{\CT{\TQ'}\text{~or~}\CT{\per\TQ'}}_-{\text{~and partial
factorization}} &       &       &       &
f_{\TN'}(\Zp)\text{~or~}f_{\per\TN'}(\Zp)
}
\end{aligned}
\end{equation}

Please note that the partial factorization ensuing $\CT{\TQ'}$ or $\CT{\per\TQ'}$ in the above diagram in \eqref{SynCommut} corresponds to the one ensuing $\CT{\per\wh\TT}$ so as to make the diagram commutative and is not part of the resolution algorithm.

\begin{definition}\label{Def:SyntheticDef}
{\upshape (Synthetic and interim exponential matrices and monomial transformations; interim variables $\Ze$ and their exponents $\Al_\Ze$; synthetic variables $(\XIp,x_1)$)}

The exponential matrices $\TQ'$ and $\perm\TQ'$ in \eqref{CstnSyntheticMatrix} and \eqref{IncstnSyntheticMatrix} are called the \emph{synthetic} exponential matrices whose associated monomial transformations $\CT{\TQ'}$ and $\CT{\per\TQ'}$ are called the \emph{synthetic} monomial transformations.
The $\perm\TS$ in \eqref{FirstInterimMatrix}, $\perm\TT$ in \eqref{SecondInterimMatrix} and $\perm\wh\TT$ in \eqref{ModifiedInterimMatrix} are called the \emph{interim} exponential matrices whose associated \emph{interim} monomial transformations are $\CT{\per\TS}$, $\CT{\per\TT}$ and $\CT{\per\wh\TT}$ respectively.

The variables $\Ze=(\LE\zeta{k+2}n)$ in \eqref{FirstInterimMatrix} and \eqref{SecondInterimMatrix} are called the \emph{interim} variables whose exponents are denoted as $\Al_\Ze$;
whereas the variables $(\X'_*,\Ze,x_1):=(\XIp,x_1)$ in \eqref{InterimReorgFunc} are called the \emph{synthetic} variables.
\end{definition}

Let us elucidate the change of singularity height as per the new perspective of synthetic exponential matrix.
Let $\M_{x_1}$ denote the row vector of the synthetic exponential matrix $\TQ'$ in \eqref{CstnSyntheticMatrix} or $\perm\TQ'$ in \eqref{IncstnSyntheticMatrix} associated with the latent primary variable $x_1$.
Similar to Definition \ref{Def:Consistency-1}, the latent primary variable $x_1$ is said to be consistent with $\TQ'$ and $\iI\Zop$, or $\perm\TQ'$ and $\iI\Zop$, if $\rank (\Ev{\lTQ}{\Zop}')=|\iI\Zop|$ or $\rank (\Ev{\perl\lTQ}{\Zop}')=|\iI\Zop|$ with $\lTQ'$ or $\perl\lTQ'$ denoting the remnant submatrix of $\TQ'$ or $\perm\TQ'$ excluding the latent primary row vector $\M_{x_1}$.
According to \eqref{NonDegVector}, it is evident that the primary variable $x'_1$ is inconsistent with $\perm\TQ'$ and $\iI\Zop$ as in \eqref{IncstnSyntheticMatrix}.

When the latent primary variable $x_1$ is consistent with the synthetic exponential matrix $\TQ'$ and $\iI\Zop$ in \eqref{CstnSyntheticMatrix}, we assume that $\TQ'$ is in a canonical consistency form as in Definition \ref{Def:Consistency-1} whose exceptional submatrix $\Evv\TQ\Zop'$ is non-degenerate, based on which $\TQ'$ has a canonical reduction in the form of $\TN_{\TQ'}$ in \eqref{CstnNrmlzdMtrx}.
In \eqref{CstnNrmlzdMtrx} the vector $\U_1$ denotes $\N_1\TOm'_3\TD'$ which is the exceptional part of the latent primary row vector $\M_{x_1}$ of $\TQ'$ in \eqref{CstnSyntheticMatrix}.
The submatrix $\TG'$ of $\TN_{\TQ'}$ in \eqref{CstnNrmlzdMtrx} denotes the non-degenerate exceptional submatrix $\Evv{\TQ}\Zop'$ in \eqref{CstnSyntheticMatrix}.
The matrix $\TOm'$ is of dimensions $(l-1)$ by $(n-l)$ with elements in $\BB Q$.

\begin{equation}\label{CstnNrmlzdMtrx}\begin{aligned}
    \TN_{\TQ'}=
    \begin{array}{r @{} c @{\hs{-0.5pt}} c @{\hs{-3.2pt}} p{4mm}
    @{\hs{-1mm}} *{3}c @{\hs{2pt}} c @{\hs{2pt}} *{3}c
    @{\hs{-0.8mm}} p{4mm} @{\hs{-2.5mm}} l}
    &           &           &           &           \E_1        &
    \cdots      &           \E_l        &           &           &
    &           &           &                                               \\
    \mr{12}{$\rule{0mm}{24mm}$}         &           x_1         &
    \mr{12}{$\rule{0mm}{24mm}$}         &
    \raisebox{5pt}[0pt]{\mr{12}{$\left[\rule{0mm}{82pt}\right.$}}
    &           \mc{3}{c}{\mr{7}{$\TE$}}            &           \vline
    &           \mc{3}{c}{\U_1}         &
    \raisebox{5pt}[0pt]{\mr{12}{$\left.\rule{0mm}{82pt}\right]$}}
    &           \mr{12}{$\rule{0mm}{24mm}$}                                 \\
    \cline{2-2}\cline{8-12}
    &           \xi'_1      &           &           &           &
    &           &           \vline      &
    \mc{3}{c}{\mr{6}{$\TOm'\TG'$}}      &           &                       \\
    &           \vdots      &           &           &           &
    &           &           \vline      &           &           &
    &           &                                                           \\
    &           \xi'_k      &           &           &           &
    &           &           \vline      &           &           &
    &           &                                                           \\
    \cline{2-2}
    &           \xi'_{k+1}  &           &           &           &
    &           &           \vline      &           &           &
    &           &                                                           \\
    &           \vdots      &           &           &           &
    &           &           \vline      &           &           &
    &           &                                                           \\
    &           \xi'_{l-1}  &           &           &           &
    &           &           \vline      &           &           &
    &           &                                                           \\
    \cline{2-12}
    &           \xi'_l      &           &           &
    \mc{3}{c}{\mr{3}{\raisebox{-2mm}[0pt]{$\TO$}}}
    &           \vline      &
    \mc{3}{c}{\mr{3}{\raisebox{-2mm}[0pt]{$\TG'$}}} &           &           \\
    &           \vdots      &           &           &           &
    &           &           \vline      &           &           &
    &           &                                                           \\
    &           \xi'_{n-1}  &           &           &           &
    &           &           \vline      &           &           &
    &           &                                                           \\
    &           &           &           &           t'_1        &
    \cdots      &           t'_l        &           \vline      &
    t'_{l+1}    &           \cdots      &           t'_n        &
    &                                                                       \\
    \end{array}
\end{aligned}\end{equation}

The above canonical form in \eqref{CstnNrmlzdMtrx} shows that if we take $t'_1$ as the primary variable, then similar to the discussion in Lemma \ref{Lemma:SimpleDecrease}, its singularity height shall strictly decrease from $d$, the prior singularity height associated with the latent primary variable $x_1$ as in \eqref{PrimaryForm}, at a regular reduced branch point.
This is also due to the fact that the interim preliminary and Weierstrass reductions as in Definition \ref{Def:InterimReductions} do not alter the latent primary exponent $\alpha_1$ of the latent primary variable $x_1$.

\begin{definition}\label{Def:Revival}
{\upshape (Revived primary variable)}

The primary variable like $t'_1$ as above that resumes the decreasing process of the singularity height associated with the latent primary variable $x_1$ is called the \emph{revived} primary variable associated with $x_1$.
We also say that the latent primary variable $x_1$ is \emph{revived} as $t'_1$ by the synthetic exponential matrix $\TQ'$ henceforth.
\end{definition}

\begin{equation}\label{IncstnNrmlzdMtrx}\begin{aligned}
    \TN_{\per\TQ'}=
    \begin{array}{r @{} c @{\hs{-0.5pt}} c @{\hs{-3.2pt}} p{4mm}
    @{\hs{-1mm}} *{3}c @{\hs{2pt}} c @{\hs{2pt}} *{3}c
    @{\hs{-0.8mm}} p{4mm} @{\hs{-2.5mm}} l}
    &           &           &           &           \E_1        &
    \cdots      &           \E_l        &           &           &
    &           &           &                                               \\
    \mr{12}{$\rule{0mm}{24mm}$}         &           x_1         &
    \mr{12}{$\rule{0mm}{24mm}$}         &
    \raisebox{-1pt}[0pt]{\mr{12}{$\left[\rule{0mm}{90pt}\right.$}}
    &           \mc{3}{c}{\mr{7}{$\TE$}}            &           \vline
    &           \mc{3}{c}{\U_1}         &
    \raisebox{-1pt}[0pt]{\mr{12}{$\left.\rule{0mm}{90pt}\right]$}}
    &           \mr{12}{$\rule{0mm}{24mm}$}                                 \\
    \cline{2-2}\cline{8-12}
    &           \xi'_2      &           &           &           &
    &           &           \vline      &
    \mc{3}{c}{\mr{6}{$\TOm'\TG'\TB'$}}  &           &                       \\
    &           \vdots      &           &           &           &
    &           &           \vline      &           &           &
    &           &                                                           \\
    &           \xi'_k      &           &           &           &
    &           &           \vline      &           &           &
    &           &                                                           \\
    \cline{2-2}
    &           \xi'_{k+1}  &           &           &           &
    &           &           \vline      &           &           &
    &           &                                                           \\
    &           \vdots      &           &           &           &
    &           &           \vline      &           &           &
    &           &                                                           \\
    &           \xi'_l      &           &           &           &
    &           &           \vline      &           &           &
    &           &                                                           \\
    \cline{2-12}
    &           \xi'_{l+1}  &           &           &
    \mc{3}{c}{\mr{4}{\raisebox{-2mm}[0pt]{$\TO$}}}
    &           \vline      &
    \mc{3}{c}{\mr{3}{\raisebox{-2mm}[0pt]{$\TG'\TB'$}}}         &
    &                                                                       \\
    &           \vdots      &           &           &           &
    &           &           \vline      &           &           &
    &           &                                                           \\
    &           \xi'_{n-1}  &           &           &           &
    &           &           \vline      &           &           &
    &           &                                                           \\
    \cline{2-2}\cline{8-12}
    &           x'_1        &           &           &           &
    &           &           \vline      &           \mc{3}{c}{\N'_1}
    &           &                                                           \\
    &           &           &           &           t'_1        &
    \cdots      &           t'_l        &           \vline      &
    t'_{l+1}    &           \cdots      &           t'_n        &
    &                                                                       \\
    \end{array}
\end{aligned}\end{equation}

Let $\M_{x_1}$ and $\M_{x'_1}$ denote the row vectors of the synthetic exponential matrix $\perm\TQ'$ corresponding to the latent primary variables $x_1$ and $x'_1$ respectively.
From $\TB'=[\TD'\bl'~\TD']$ in \eqref{IncstnSyntheticMatrix}, it is easy to deduce that the primary variable $x'_1$ is always inconsistent with $\perm\TQ'$ and exceptional index set $\iI\Zop$ in \eqref{IncstnSyntheticMatrix}.
Hence similar to the form of $\perm\TN$ in \eqref{IncstnMnmlTrnsfrm}, we take $\M_{x'_1}$ as the last row vector of $\perm\TQ'$.
When the latent primary variable $x_1$ is consistent with $\perm\TQ'$ and $\iI\Zop$ in \eqref{IncstnSyntheticMatrix}, via a permutation and relabeling of the row vectors of $\perm\TQ'$, we take $\M_{x_1}$ as the first row vector of $\perm\TQ'$ under the assumption that its exceptional submatrix $\Evv{\perm\TQ}\Zop'$ is non-degenerate.
In this way $\perm\TQ'$ has a canonical reduction in the form of $\TN_{\per\TQ'}$ in \eqref{IncstnNrmlzdMtrx}.
The $(n-l)$-dimensional row vector $\N'_1$ in \eqref{IncstnNrmlzdMtrx} is the same as in \eqref{IncstnSyntheticMatrix}.
The $(n-l-1)$ by $(n-l-1)$ dimensional submatrix $\TG'$ in \eqref{IncstnNrmlzdMtrx} is non-degenerate as $\det\TG'\ne 0$.
Similar to the case in \eqref{CstnNrmlzdMtrx}, the canonical form in \eqref{IncstnNrmlzdMtrx} shows that the latent primary variable $x_1$ is revived as $t'_1$ whose associated singularity height strictly decreases from the prior singularity height $d$ at a regular reduced branch point.

Now let us suppose that the latent primary variable $x_1$ is inconsistent with the synthetic exponential matrix $\TQ'$ and exceptional index set $\iI\Zop$ in \eqref{CstnSyntheticMatrix} or $\perm\TQ'$ and $\iI\Zop$ in \eqref{IncstnSyntheticMatrix}.
From $\rank (\TA)=n-k-1$ and $\rank (\TOm'_3)=n-l$ in \eqref{CstnSyntheticMatrix}, or $\rank (\TOm'_3)=n-l-1$ in \eqref{IncstnSyntheticMatrix}, it follows that $\rank (\TA\TOm'_3)\ge n-l-1$ for the submatrix $\TA\TOm'_3$ of $\TQ'$ in \eqref{CstnSyntheticMatrix}, or $\rank (\TA\TOm'_3)\ge n-l-2$ for $\perm\TQ'$ in \eqref{IncstnSyntheticMatrix}.
Hence we can make a permutation and relabeling of the interim variables $\Ze$ as well as their corresponding row vectors of $\TQ'$ in \eqref{CstnSyntheticMatrix} or $\perm\TQ'$ in \eqref{IncstnSyntheticMatrix} such that $\TQ'$ or $\perm\TQ'$ bears the form in \eqref{PermCstnSynMtrx} or \eqref{SmplIncnSynMtrx}.
The submatrix $\LAm'$ is a permutation of the submatrix $\TA\TF'$ in \eqref{CstnSyntheticMatrix} or \eqref{IncstnSyntheticMatrix} and the matrix $\TOm'$ is of dimensions $l$ by $(n-l-1)$ in \eqref{PermCstnSynMtrx} or $l$ by $(n-l-2)$ in \eqref{SmplIncnSynMtrx} with elements in $\BB Q$.
The exceptional submatrix $\Evv\TQ\Zop'$ in \eqref{PermCstnSynMtrx} is non-degenerate, i.e., $\det\Evv\TQ\Zop'\ne 0$, which amounts to the condition $\U_1\cdot\bigl[\begin{smallmatrix}-1\\ \bl'\end{smallmatrix}\bigr]\ne 0$.
The column vector $\bl'\in\BB Q^{n-l-1}$ and the submatrix $\TB'=[\TG'\bl'~\TG']$ in \eqref{PermCstnSynMtrx} is of dimensions $(n-l-1)$ by $(n-l)$ with elements in $\BB Q$.

The matrix $\TB'_2$ in \eqref{SmplIncnSynMtrx} denotes the $(n-l-2)$ by $(n-l)$ submatrix $[\TB'_1\bl'_2~\TB'_1]$ with elements in $\BB Q$ and $\TB'_1:=[\TG'_1\bl'_1~\TG'_1]$ such that $\bl'_1$ and $\bl'_2$ are $(n-l-2)$-dimensional and $(n-l-1)$-dimensional column vectors in $\BB Q^{n-l-2}$ and $\BB Q^{n-l-1}$ respectively.
The submatrix $\TG'_2$ in \eqref{SmplIncnSynMtrx} denotes the $(n-l-1)$ by $(n-l-1)$ submatrix $\bigl[\begin{smallmatrix}\TB'_1\\ \U'_1\end{smallmatrix}\bigr]$.
It is evident that we have the following nested non-degeneracies:
\begin{equation}\label{NestedNonDeg}
\det\TG'_1\cdot\det\TG'_2\cdot\det\Evv{\perm\TQ}\Zop'\ne 0
\end{equation}
in \eqref{SmplIncnSynMtrx} corresponding to the nested latency of the latent primary variables $x_1$ and $x'_1$.
The row vectors $\U'_2=\N'_1$ and $(\U'_1\cdot\bl'_2,\U'_1)=\N_1\TOm'_3\TB'$ are as in \eqref{IncstnSyntheticMatrix}, both of which are in $\DN{n-l}$.

\begin{equation}\label{PermCstnSynMtrx}\begin{aligned}
    \TQ':=
    \begin{array}{r @{} c @{\hs{-0.5pt}} c @{\hs{-3.6pt}} p{4mm}
    @{\hs{-1mm}} *{3}c @{\hs{2pt}} c @{\hs{3pt}} *{3}c @{\hs{2pt}}
    c @{\hs{3pt}} c @{} c @{\hs{3pt}} *{3}c @{\hs{-0.8mm}}
    p{4mm} @{\hs{-2.5mm}} l}
    &           &           &           &           \E_1        &
    \cdots      &           \E_k        &           &           &
    &           &           &           &                                   \\
    \mr{12}{$\rule{0mm}{24mm}$}         &           x'_1        &
    \mr{12}{$\rule{0mm}{24mm}$}         &
    \raisebox{7pt}[0pt]{\mr{12}{$\left[\rule{0mm}{83pt}\right.$}}
    &           \mc{3}{c}{\mr{3}{$\TE$}}            &           \vline
    &           \mc{3}{c}{\mr{3}{$\TO$}}            &           \vline
    &           \mc{5}{c}{\mr{6}{$\TOm'\TB'$}}      &
    \raisebox{7pt}[0pt]{\mr{12}{$\left.\rule{0mm}{83pt}\right]$}}
    &           \mr{12}{$\rule{0mm}{24mm}$}                                 \\
    &           \vdots      &           &           &           &
    &           &           \vline      &           &           &
    &           \vline      &           &           &           &
    &           &                                                           \\
    &           x'_k        &           &           &           &
    &           &           \vline      &           &           &
    &           \vline      &           &           &           &
    &           &                                                           \\
    \cline{2-11}
    &           \zeta_{k+2} &           &           &
    \mc{3}{c}{\mr{7}{$\TO$}}            &           \vline      &
    \mc{3}{c}{\mr{6}{$\LAm'$}}          &           \vline      &
    &           &           &           &           &                       \\
    &           \vdots      &           &           &           &
    &           &           \vline      &           &           &
    &           \vline      &           &           &           &
    &           &                                                           \\
    &           \zeta_{l+1} &           &           &           &
    &           &           \vline      &           &           &
    &           \vline      &           &           &           &
    &                                                                       \\
    \cline{2-2}\cline{12-18}
    &           \zeta_{l+2} &           &           &           &
    &           &           \vline      &           &           &
    &           \vline      &           \mr{3}{$\TG'\bl'$}      &
    \vline      &           \mc{3}{c}{\mr{3}{$\TG'$}}                       \\
    &           \vdots      &           &           &           &
    &           &           \vline      &           &           &
    &           \vline      &           &           \vline      &
    &           &           &                                               \\
    &           \zeta_n     &           &           &           &
    &           &           \vline      &           &           &
    &           \vline      &           &           \vline      &
    &           &           &                                               \\
    \cline{2-2}\cline{8-18}
    &           x_1         &           &           &           &
    &           &           \vline      &
    \mc{3}{c}{\N_1\TF'}     &           \vline      &
    \mc{5}{c}{\U_1}         &           &                                   \\
    &           &           &           &           z'_1        &
    \cdots      &           z'_k        &           \vline      &
    z'_{k+1}    &           \cdots      &           z'_l        &
    \vline      &           z'_{l+1}    &           \vline      &
    z'_{l+2}    &           \cdots      &           z'_n        &
    &                                                                       \\
    \end{array}
\end{aligned}\end{equation}

\begin{equation}\label{SmplIncnSynMtrx}\begin{aligned}
    \perm\TQ':=
    \begin{array}{r @{} c @{\hs{-0.5pt}} c @{\hs{-3.6pt}} p{4mm}
    @{\hs{-0.8mm}} *{3}c @{\hs{0.6mm}} c @{\hs{1mm}} *{3}c @{\hs{1mm}}
    c @{\hs{1mm}} c @{\hs{0.3mm}} c @{\hs{0.6mm}} c @{\hs{0.3mm}}
    c @{\hs{1mm}} *{3}c @{\hs{-0.8mm}} p{4mm} @{\hs{-2.5mm}} l}
    &           &           &           &           \E_1        &
    \cdots      &           \E_{k-1}    &           &           &
    &           &           &           &           &           &
    &           &           &           &           &                       \\
    \mr{12}{$\rule{0mm}{24mm}$}         &           x'_2        &
    \mr{12}{$\rule{0mm}{24mm}$}         &
    \mr{12}{$\left[\rule{0mm}{90pt}\right.$}        &
    \mc{3}{c}{\mr{3}{$\TE$}}            &           \vline      &
    \mc{3}{c}{\mr{3}{$\TO$}}            &           \vline      &
    \mc{7}{c}{\mr{6}{$\TOm'\hs{-1.5pt}\TB'_2$}}     &
    \mr{12}{$\left.\rule{0mm}{90pt}\right]$}        &
    \mr{12}{$\rule{0mm}{24mm}$}                                             \\
    &           \vdots      &           &           &           &
    &           &           \vline      &           &           &
    &           \vline      &           &           &           &
    &           &           &           &           &                       \\
    &           x'_k        &           &           &           &
    &           &           \vline      &           &           &
    &           \vline      &           &           &           &
    &           &           &           &           &                       \\
    \cline{2-11}
    &           \zeta_{k+2} &           &           &
    \mc{3}{c}{\mr{8}{$\TO$}}            &           \vline      &
    \mc{3}{c}{\mr{6}{$\LAm'$}}          &           \vline      &
    &           &           &           &           &           &
    &           &                                                           \\
    &           \vdots      &           &           &           &
    &           &           \vline      &           &           &
    &           \vline      &           &           &           &
    &           &           &           &           &                       \\
    &           \zeta_{l+2} &           &           &           &
    &           &           \vline      &           &           &
    &           \vline      &           &           &           &
    &           &           &           &           &                       \\
    \cline{2-2}\cline{12-20}
    &           \zeta_{l+3} &           &           &           &
    &           &           \vline      &           &           &
    &           \vline      &
    \mr{4}{\raisebox{-1mm}[0pt]{$\TG'_2\bl'_2$}}    &
    \vline      &           \mr{3}{$\TG'_1\bl'_1$}  &           \vline
    &           \mc{3}{c}{\mr{3}{$\TG'_1$}}         &           &           \\
    &           \vdots      &           &           &           &
    &           &           \vline      &           &           &
    &           \vline      &           &           \vline      &
    &           \vline      &           &           &           &
    &                                                                       \\
    &           \zeta_n     &           &           &           &
    &           &           \vline      &           &           &
    &           \vline      &           &           \vline      &
    &           \vline      &           &           &           &
    &                                                                       \\
    \cline{2-2}\cline{8-11}\cline{14-20}
    &           x_1         &           &           &           &
    &           &           \vline      &
    \mc{3}{c}{\raisebox{-1pt}[0pt]{$\N_1\TF'$}}     &           \vline
    &           &           \vline      &
    \mc{5}{c}{\raisebox{-0.6pt}[0pt]{$\U'_1$}}      &           &           \\
    \cline{2-2}\cline{8-20}
    &           x'_1        &           &           &           &
    &           &           \vline      &
    \mc{3}{c}{\raisebox{-1.8pt}[0pt]{$\TO$}}        &           \vline
    &           \mc{7}{c}{\raisebox{-1pt}[0pt]{$\U'_2$}}        &
    &                                                                       \\
    &           &           &           &           z'_1        &
    \cdots      &           z'_{k-1}    &           \vline      &
    z'_k        &           \cdots      &           z'_l        &
    \vline      &           z'_{l+1}    &           \vline      &
    z'_{l+2}    &           \vline      &           z'_{l+3}    &
    \cdots      &           z'_n        &           &                       \\
    \end{array}
\end{aligned}\end{equation}

\begin{definition}\label{Def:NestLatency}
{\upshape (Sustained latency; nested latency; nesting degree $\fnd$; nesting level)}

In the case when the latent primary variable $x_1$ is inconsistent with the synthetic exponential matrix $\TQ'$ and exceptional index set $\iI\Zop$ as in \eqref{PermCstnSynMtrx} or $\perm\TQ'$ and $\iI\Zop$ as in \eqref{SmplIncnSynMtrx}, we say that $x_1$ \emph{sustains} its latency with respect to $\TQ'$ and $\iI\Zop$ or $\perm\TQ'$ and $\iI\Zop$.

For the synthetic exponential matrix $\perm\TQ'$ in \eqref{SmplIncnSynMtrx}, the latent primary variables $x_1$ and $x'_1$ constitute the \emph{nested latency} of $\perm\TQ'$ with their number defined as the \emph{nesting degree} and denoted as $\fnd=2$.
In particular, $x_1$ and $x'_1$ are called the latent primary variables of \emph{nesting level} $1$ and $2$ respectively according to their order of appearance.
\end{definition}

Based on the non-degeneracy of the exceptional submatrix $\Evv\TQ\Zop'$ in \eqref{PermCstnSynMtrx}, the synthetic exponential matrix $\TQ'$ in \eqref{PermCstnSynMtrx} has a canonical reduction $\TN_{\TQ'}$ as in \eqref{FinalNrmlzdMtrx} resembling the canonical reduction $\perm\TN$ in \eqref{IncstnMnmlTrnsfrm}, both of which have $x_1$ as the latent primary variable.
Thus the discussions ensuing \eqref{IncstnMnmlTrnsfrm} can be repeated in a similar fashion here.
Moreover, the following conclusion holds for the new singularity height after the monomial transformation associated with $\TN_{\TQ'}$.

\begin{equation}\label{FinalNrmlzdMtrx}\begin{aligned}
    \TN_{\TQ'}:=
    \begin{array}{r @{} c @{\hs{-0.5pt}} c @{\hs{-3.6pt}} p{4mm}
    @{\hs{-1mm}} *{3}c @{\hs{2pt}} c @{\hs{3pt}} c @{} c
    @{\hs{3pt}} *{3}c @{\hs{-0.8mm}} p{4mm}
    @{\hs{-2.5mm}} l}
    &           &           &           &           \E_1        &
    \cdots      &           \E_l        &           &           &           \\
    \mr{12}{$\rule{0mm}{24mm}$}         &           x'_1        &
    \mr{12}{$\rule{0mm}{24mm}$}         &
    \raisebox{7pt}[0pt]{\mr{12}{$\left[\rule{0mm}{83pt}\right.$}}
    &           \mc{3}{c}{\mr{6}{$\TE$}}            &           \vline
    &           \mc{5}{c}{\mr{6}{$\TOm'\TB'$}}      &
    \raisebox{7pt}[0pt]{\mr{12}{$\left.\rule{0mm}{83pt}\right]$}}
    &           \mr{12}{$\rule{0mm}{24mm}$}                                 \\
    &           \vdots      &           &           &           &
    &           &           \vline      &           &           &
    &           &           &                                               \\
    &           x'_k        &           &           &           &
    &           &           \vline      &           &           &
    &           &           &                                               \\
    \cline{2-2}
    &           \zeta_{k+2} &           &           &           &
    &           &           \vline      &           &           &
    &           &           &                                               \\
    &           \vdots      &           &           &           &
    &           &           \vline      &           &           &
    &           &           &                                               \\
    &           \zeta_{l+1} &           &           &           &
    &           &           \vline      &           &           &
    &           &           &           &                                   \\
    \cline{2-14}
    &           \zeta_{l+2} &           &           &
    \mc{3}{c}{\mr{4}{$\TO$}}            &           \vline      &
    \mr{3}{$\TG'\bl'$}      &           \vline      &
    \mc{3}{c}{\mr{3}{$\TG'$}}                                               \\
    &           \vdots      &           &           &           &
    &           &           \vline      &           &           \vline
    &           &           &           &                                   \\
    &           \zeta_n     &           &           &           &
    &           &           \vline      &           &           \vline
    &           &           &           &                                   \\
    \cline{2-2}\cline{8-14}
    &           x_1         &           &           &           &
    &           &           \vline      &           \mc{5}{c}{\U_1}
    &           &                                                           \\
    &           &           &           &           t'_1        &
    \cdots      &           t'_l        &           \vline      &
    t'_{l+1}    &           \vline      &           t'_{l+2}    &
    \cdots      &           t'_n        &           &                       \\
    \end{array}
\end{aligned}\end{equation}

\begin{lemma}\label{Lemma:MiddleReviveDecrease}
After the monomial transformation associated with the canonical reduction $\TN_{\TQ'}$ in \eqref{FinalNrmlzdMtrx}, the singularity height associated with the primary variable $t'_1$ strictly decreases from the prior singularity height $d'$ in \eqref{ReorganizedFunc} and \eqref{InterimReorgFunc}.
\end{lemma}
\begin{proof}
When the exceptional support $\es{\TN_{\TQ'}}\Top=\es{\TQ'}\Zop$ corresponds to the latent primary exponent $\alpha_1=a_1$ as in \eqref{InterimReorgFunc} according to Lemma \ref{Lemma:ConstantPrimaryExp}, an argument similar to Lemma \ref{Lemma:SimpleDecrease} based on the Weierstrass form $f_0(\XIp)$ in \eqref{InterimWeierstrassForm} yields the conclusion.

The interim exponential matrix $\perm\wh\TT$ in \eqref{ModifiedInterimMatrix} and the ensuing partial factorization constitute a non-degenerate linear transformation from $\supp (f(\XIp,x_1))$ to $\supp (f(\Xp))$, which is denoted as $\CL{}$ here.
Suppose that the latent primary exponent $\alpha_1$ of the exceptional support $\es{\TN_{\TQ'}}\Top$ satisfies $\alpha_1\ne a_1$.
For $\forall (\Al_\XIp,\alpha_1)\in\supp (f(\XIp,x_1))$ in \eqref{InterimReorgFunc} with $\alpha'_1\ge d'$ such that $\CL{}(\Al_\XIp,\alpha_1)\ne\D'$, we have $(\Al_\XIp,\alpha_1)\in\supp (\phi(\XIp,x_1))$ in \eqref{InterimReorgFunc}, i.e., $\alpha_1\ne a_1$, since the apex $(d',\bz)$ in \eqref{InterimWeierstrassForm} has the unique maximal degree $d'$ in the primary variable $x'_1$ in $\supp(f_0(\XIp))$ with the redundant function $r(\XIp)$ disregarded.
Thus similar to Lemma \ref{Lemma:ConstantPrimaryExp}, $\alpha_1\ne a_1$ indicates that $\Ga'_0(\Al')\succ\Ga'_0(\D')$ with $\Ga'_0$ denoting the exceptional exponents of the exceptional variables $\T'_0$ in \eqref{FinalNrmlzdMtrx}.
Hence for $\forall (\Al_\XIp,\alpha_1)\in\es{\TN_{\TQ'}}\Top$, we have $\alpha'_1<d'$.
\end{proof}

When the reduced exponential matrix $\TN'$ in \eqref{CstnExpntlMtrx} is deficient with respect to $\iI\Zop$ as in Definition \ref{Def:Deficiency}, we need to make an appropriate deficient contraction similar to Lemma \ref{Lemma:TotalTransformDecrease} so as to ensure the strict decrease of the singularity height as above.

In the case when the latent primary variable $x_1$ is inconsistent with $\perm\TQ'$ and $\iI\Zop$ in \eqref{SmplIncnSynMtrx}, the singularity height of the proper transform $q(\Z_*',\bz)$ of the apex form $f(\XIp,x_1)$ in \eqref{InterimReorgFunc} under $\CT{\per\TQ'}$ might have a temporary increase from the prior singularity height $d'$ in \eqref{InterimReorgFunc}.
Nonetheless in this case the nesting degree of $\perm\TQ'$ in \eqref{SmplIncnSynMtrx} satisfies $\fnd=2$, which is strictly more than that of $\perm\TN$ in \eqref{IncstnMnmlTrnsfrm} satisfying $\fnd=1$.
Please note that in this case we do not make a canonical reduction of $\perm\TQ'$ like $\TN_{\TQ'}$ in \eqref{FinalNrmlzdMtrx}.

Let $q(\Zp)$ denote the partial transform of the apex form $f(\XIp,x_1)$ in \eqref{InterimReorgFunc} under $\CT{\per\TQ'}$.
After the latent preliminary and Weierstrass reductions as in Definition \ref{Def:LatentReductions}, suppose that $q(\Zp)$ is reduced to an apex form with singularity height $d''$ as following:
\begin{equation}\label{NewApexForm}
f(\Xq)=f_0(\Xq)+\phi(\Xq)
\end{equation}
resembling \eqref{ReorganizedFunc} with $f_0(\Xq)$ and $\phi(\Xq)$ being the latent reducible and remainder functions respectively.
Here similar to the latent gradation in \eqref{ReorganizedFunc}, both the latent primary components $\wt\alpha_1$ and $\wt\alpha'_1$ of $f_0(\Xq)$ equal zero whereas not both of $\phi(\Xq)$ zero.

Suppose that in the next resolution step the reduced exponential matrix bears the form $\TN''$ in \eqref{MoreCstnExptlMtrx} or $\perm\TN''$ in \eqref{MoreIncnExptlMtrx} resembling $\TN'$ in \eqref{CstnExpntlMtrx} or $\perm\TN'$ in \eqref{IncstnExpntlMtrx} and depending on the consistency of the primary variable $x''_1$ with $\TN''$ and the exceptional index set $\iI\Zoq=\dotc{m+1}n$ with $m>l$ or $\perm\TN''$ and $\iI\Zoq$ with $m\ge l$ .
In particular, the matrix $\TB''$ in \eqref{MoreIncnExptlMtrx} denotes the submatrix $[\TD''\bl''~\TD'']$ of $\perm\TN''$ in \eqref{MoreIncnExptlMtrx}.

\begin{equation}\label{MoreCstnExptlMtrx}\begin{aligned}
\setlength{\arraycolsep}{10pt}
    \TN'':=
    \begin{array}{r @{} c @{\hs{-0.5pt}} c @{\hs{-3.6pt}} p{4mm}
    @{\hs{-1mm}} *{3}c @{\hs{4pt}} c @{\hs{4pt}} *{3}c
    @{\hs{-0.8mm}} p{4mm} @{\hs{-2.5mm}} l}
    &           &           &           &           \E_1        &
    \cdots      &           \E_m        &           &
    \V''_{m+1}  &           \cdots      &           \V''_n      &
    &                                                                       \\
    \mr{12}{$\rule{0mm}{24mm}$}         &           x''_1       &
    \mr{12}{$\rule{0mm}{24mm}$}         &
    \raisebox{5mm}[0pt]{\mr{12}{$\left[\rule{0mm}{80pt}\right.$}}
    &           \mc{3}{c}{\mr{6}{$\TE$}}            &           \vline
    &           \mc{3}{c}{\mr{3}{\raisebox{-4mm}[0pt]{$\TOm''_1\TD''$}}}
    &
    \raisebox{5mm}[0pt]{\mr{12}{$\left.\rule{0mm}{80pt}\right]$}}
    &           \mr{12}{$\rule{0mm}{24mm}$}                                 \\
    &           \vdots      &           &           &           &
    &           &           \vline      &           &           &
    &           &                                                           \\
    &           x''_l       &           &           &           &
    &           &           \vline      &           &           &
    &           &                                                           \\
    \cline{2-2}\cline{8-12}
    &           x''_{l+1}   &           &           &           &
    &           &           \vline      &
    \mc{3}{c}{\mr{3}{$\TOm''_2\TD''$}}  &           &                       \\
    &           \vdots      &           &           &           &
    &           &           \vline      &           &           &
    &           &                                                           \\
    &           x''_m       &           &           &           &
    &           &           \vline      &           &           &
    &           &                                                           \\
    \cline{2-12}
    &           x''_{m+1}   &           &           &
    \mc{3}{c}{\mr{4}{$\TO$}}            &           \vline      &
    \mc{3}{c}{\mr{4}{$\TD''$}}          &           &                       \\
    &           \vdots      &           &           &           &
    &           &           \vline      &           &           &
    &           &                                                           \\
    &           x''_n       &           &           &           &
    &           &           \vline      &           &           &
    &           &                                                           \\
    &           &           &           &           z''_1       &
    \cdots      &           z''_m       &           \vline      &
    z''_{m+1}   &           \cdots      &           z''_n       &
    &                                                                       \\
    \end{array}
\end{aligned},\end{equation}

\begin{equation}\label{MoreIncnExptlMtrx}\begin{aligned}
\setlength{\arraycolsep}{8pt}
    \perm\TN'':=
    \begin{array}{r @{} c @{\hs{-0.5pt}} c @{\hs{-4pt}} p{4mm}
    @{\hs{-1mm}} *{3}c @{\hs{2pt}} c @{\hs{2pt}} c @{} c
    @{\hs{4pt}} *{3}c @{\hs{-0.8mm}} p{4mm} @{\hs{-2.5mm}} l}
    &           &           &           &           \E_1        &
    \cdots      &           \E_m        &           &
    \V''_{m+1}  &           &           \V''_{m+2}  &           \cdots
    &           \V''_n      &           &                                   \\
    \mr{12}{$\rule{0mm}{24mm}$}         &           x''_2        &
    \mr{12}{$\rule{0mm}{24mm}$}         &
    \mr{11}{$\left[\rule{0mm}{83.75pt}\right.$}     &
    \mc{3}{c}{\mr{7}{$\TE$}}            &           \vline      &
    \mc{5}{c}{\mr{3}{$\TOm''_1\TB''$}}  &
    \mr{11}{$\left.\rule{0mm}{83.75pt}\right]$}     &
    \mr{12}{$\rule{0mm}{24mm}$}                                             \\
    &           \vdots      &           &           &           &
    &           &           \vline      &           &           &
    &           &           &           &                                   \\
    &           x''_l       &           &           &           &
    &           &           \vline      &           &           &
    &           &           &           &                                   \\
    \cline{2-2}\cline{8-14}
    &           x''_{l+1}   &           &           &           &
    &           &           \vline      &
    \mc{5}{c}{\mr{3}{$\TOm''_2\TB''$}}  &           &                       \\
    &           \vdots      &           &           &           &
    &           &           \vline      &           &           &
    &           &           &           &                                   \\
    &           x''_{m+1}   &           &           &           &
    &           &           \vline      &           &           &
    &           &           &           &                                   \\
    \cline{2-14}
    &           x''_{m+2}   &           &           &
    \mc{3}{c}{\mr{4}{\raisebox{-1mm}[0pt]{$\TO$}}}  &           \vline
    &           \mr{3}{$\TD''\bl''$}    &           \vline      &
    \mc{3}{c}{\mr{3}{$\TD''$}}          &           &                       \\
    &           \vdots      &           &           &           &
    &           &           \vline      &           &           \vline
    &           &           &           &           &                       \\
    &           x''_n       &           &           &           &
    &           &           \vline      &           &           \vline
    &           &           &           &           &                       \\
    \cline{2-2}\cline{8-14}
    &           \raisebox{-1pt}[0pt]{$x''_1$}       &           &
    &           &           &           &           \vline      &
    \mc{5}{c}{\raisebox{-1pt}[0pt]{$\N''_1$}}       &           &           \\
    &           &           &           &           z''_1       &
    \cdots      &           z''_m       &           \vline      &
    z''_{m+1}   &           \vline      &           z''_{m+2}   &
    \cdots      &           z''_n       &           &                       \\
    \end{array}.
\end{aligned}\end{equation}

Similar to the decomposition in \eqref{MatrixDecomposition}, we decompose the synthetic exponential matrix $\perm\TQ'$ in \eqref{SmplIncnSynMtrx} as follows.
\begin{equation}\label{NewMatrixDecomp}
\perm\TQ'=\perm\TS'\cdot\perm\TT'
\end{equation}
with $\perm\TS'$ and $\perm\TT'$ being defined as in \eqref{FirstSynInterimMtrx} and \eqref{SecondSynInterimMtrx} respectively.
The exponents of the new interim variables $\Ze'=(\LEs\zeta{l+3}n\prime\prime)$ are not fractional since none of those of the variables $(\LEs z{l+3}n\prime\prime)$ in \eqref{SmplIncnSynMtrx} are fractional and moreover, the row vectors $\U'_1\in\DN{n-l-1}$ and $\U'_2\in\DN{n-l}$ in \eqref{SmplIncnSynMtrx} and we have identical submatrices $\Evz{\perm\TS}12'=\Evz{\perm\TQ}13'$ with the index sets $\iI 1:=\dotc 1{n-2}$, $\iI 2:=\dotc{l+1}{n-2}$ and $\iI 3:=\dotc{l+3}n$.
The $(n-l)$-dimensional row vector $\ol\U'_1$ of $\perm\TT'$ is defined as $\ol\U'_1:=(\U'_1\cdot\bl'_2,\U'_1)\in\DN{n-l}$, same as in \eqref{SmplIncnSynMtrx}.
The matrix $\TA'_1$ of $\perm\TT'$ denotes the $(n-l-2)$ by $(n-l-1)$ submatrix $[\,\bl'_1~\TE]$ of $\perm\TT'$ in \eqref{SecondSynInterimMtrx}.
The exponent of the new variable $\ol x_1$ in $\perm\TS'$ equals the latent primary exponent $\alpha_1$ of the latent primary variable $x_1$.
For simplicity the new variable $\ol x_1$ and its exponent are still called the latent primary variable and exponent henceforth whose nesting level is designated as one.
The other notations are the same as in \eqref{SmplIncnSynMtrx}.

\begin{align}
    \perm\TS'&:=
    \begin{array}{r @{} c @{\hs{-0.5pt}} c @{\hs{-3.6pt}} p{4mm}
    @{\hs{-0.2mm}} *{3}c @{\hs{1mm}} c @{\hs{1mm}} *{3}c
    @{\hs{1mm}} c @{\hs{1mm}} *{3}c @{\hs{1mm}} c @{} c
    @{\hs{2pt}} c @{\hs{-0.2mm}} p{4mm} @{\hs{-2.5mm}} l}
    &           &           &           &           \E_1        &
    \cdots      &           \E_{k-1}    &           &           &
    &           &           &           &           &           &
    &           \E_{n-1}    &           \E_n        &           &           \\
    \mr{12}{$\rule{0mm}{24mm}$}         &           x'_2        &
    \mr{12}{$\rule{0mm}{24mm}$}         &
    \mr{12}{$\left[\rule{0mm}{90pt}\right.$}        &
    \mc{3}{c}{\mr{3}{$\TE$}}            &           \vline      &
    \mc{3}{c}{\mr{3}{$\TO$}}            &           \vline      &
    \mc{3}{c}{\mr{6}{\raisebox{-2.5mm}[0pt]{$\TOm'\TG'_1$}}}    &
    \vline      &
    \mc{2}{c}{\mr{9}{\raisebox{-2mm}[0pt]{$\TO$}}}  &
    \mr{12}{$\left.\rule{0mm}{90pt}\right]$}        &
    \mr{12}{$\rule{0mm}{24mm}$}                                             \\
    &           \vdots      &           &           &           &
    &           &           \vline      &           &           &
    &           \vline      &           &           &           &
    \vline      &           &           &           &                       \\
    &           x'_k        &           &           &           &
    &           &           \vline      &           &           &
    &           \vline      &           &           &           &
    \vline      &           &           &           &                       \\
    \cline{2-11}
    &           \zeta_{k+2} &           &           &
    \mc{3}{c}{\mr{8}{\raisebox{-3mm}[0pt]{$\TO$}}}  &           \vline
    &           \mc{3}{c}{\mr{6}{$\LAm'$}}          &           \vline
    &           &           &           &           \vline      &
    &           &           &                                               \\
    &           \vdots      &           &           &           &
    &           &           \vline      &           &           &
    &           \vline      &           &           &           &
    \vline      &           &           &           &                       \\
    &           \zeta_{l+2} &           &           &           &
    &           &           \vline      &           &           &
    &           \vline      &           &           &           &
    \vline      &           &           &           &                       \\
    \cline{2-2}\cline{12-15}
    &           \zeta_{l+3} &           &           &           &
    &           &           \vline      &           &           &
    &           \vline      &
    \mc{3}{c}{\mr{3}{\raisebox{-2mm}[0pt]{$\TG'_1$}}}           &
    \vline      &           &           &           &                       \\
    &           \vdots      &           &           &           &
    &           &           \vline      &           &           &
    &           \vline      &           &           &           &
    \vline      &           &           &           &                       \\
    &           \zeta_n     &           &           &           &
    &           &           \vline      &           &           &
    &           \vline      &           &           &           &
    \vline      &           &           &           &                       \\
    \cline{2-2}\cline{8-19}
    &           x_1         &           &           &           &
    &           &           \vline      &
    \mc{3}{c}{\N_1\TF'}     &           \vline      &
    \mc{3}{c}{\mr{2}{\raisebox{-1mm}[0pt]{$\TO$}}}  &           \vline
    &           \mc{2}{c}{\mr{2}{\raisebox{-1mm}[0pt]{$\TE$}}}  &
    &                                                                       \\
    \cline{2-2}\cline{8-11}
    &           x'_1        &           &           &           &
    &           &           \vline      &
    \mc{3}{c}{\raisebox{-1.5pt}[0pt]{$\TO$}}        &           \vline
    &           &           &           &           \vline      &
    &           &           &                                               \\
    &           &           &           &           z'_1        &
    \cdots      &           z'_{k-1}    &           \vline      &
    z'_k        &           \cdots      &           z'_l        &
    \vline      &           \zeta'_{l+3}            &           \cdots
    &           \zeta'_n    &           \vline      &           \ol x_1
    &           x'_1        &           &                                   \\
    \end{array};\label{FirstSynInterimMtrx}\allowdisplaybreaks[4]\\
    \perm\TT'&:=
    \begin{array}{r @{} c @{\hs{-0.5pt}} c @{\hs{-2.8pt}} p{4mm}
    @{\hs{-0.8mm}} *{3}c @{\hs{1mm}} c @{\hs{1mm}} c @{\hs{0.3mm}}
    c @{\hs{1mm}} c @{\hs{0.3mm}} c @{\hs{1mm}} *{3}c @{\hs{-0.8mm}}
    p{4mm} @{\hs{-2.5mm}} l}
    &           &           &           &           \E_1        &
    \cdots      &           \E_l        &           &           &
    &           &           &           &           &           &
    &                                                                       \\
    \mr{8}{$\rule{0mm}{24mm}$}          &           z'_1        &
    \mr{8}{$\rule{0mm}{24mm}$}          &
    \raisebox{2.5pt}[0pt]{\mr{9}{$\left[\rule{0mm}{66pt}\right.$}}
    &           \mc{3}{c}{\mr{3}{$\TE$}}            &           \vline
    &           \mc{7}{c}{\mr{3}{$\TO$}}            &
    \raisebox{2.5pt}[0pt]{\mr{9}{$\left.\rule{0mm}{66pt}\right]$}}
    &           \mr{8}{$\rule{0mm}{24mm}$}                                  \\
    &           \vdots      &           &           &           &
    &           &           \vline      &           &           &
    &           &           &           &           &           &           \\
    &           z'_l        &           &           &           &
    &           &           \vline      &           &           &
    &           &           &           &           &           &           \\
    \cline{2-16}
    &           \zeta'_{l+3}            &           &           &
    \mc{3}{c}{\mr{5}{$\TO$}}            &           \vline      &
    \mr{3}{$\TA'_1\bl'_2$}  &           \vline      &
    \mr{3}{$\bl'_1$}        &           \vline      &
    \mc{3}{c}{\mr{3}{$\TE$}}            &           &                       \\
    &           \vdots      &           &           &           &
    &           &           \vline      &           &           \vline
    &           &           \vline      &           &           &
    &           &                                                           \\
    &           \zeta'_n    &           &           &           &
    &           &           \vline      &           &           \vline
    &           &           \vline      &           &           &
    &           &                                                           \\
    \cline{2-2}\cline{8-16}
    &           \ol x_1     &           &           &           &
    &           &           \vline      &
    \mc{7}{c}{\raisebox{-1pt}[0pt]{$\ol\U'_1$}}     &           &           \\
    \cline{2-2}\cline{8-16}
    &           x'_1        &           &           &           &
    &           &           \vline      &           \mc{7}{c}{\U'_2}
    &           &                                                           \\
    &           &           &           &           z'_1        &
    \cdots      &           z'_l        &           \vline      &
    z'_{l+1}    &           \vline      &           z'_{l+2}    &
    \vline      &           z'_{l+3}    &           \cdots      &
    z'_n        &           &                                               \\
    \end{array}.\label{SecondSynInterimMtrx}
\end{align}

The decomposition in \eqref{NewMatrixDecomp} is essentially a decomposition of the exceptional column submatrix $\Ev{\perm\TQ}\Zop'$ in \eqref{SmplIncnSynMtrx} and complies with the following principles similar to those for \eqref{MatrixDecomposition}.
The first one is the invariance of the latent primary exponents $\Al'_\star:=(\alpha_1,\alpha'_1)$ of the latent primary variables $x_1$ or $\ol x_1$ and $x'_1$ under the interim exponential matrix $\perm\TS'$ in \eqref{FirstSynInterimMtrx};
The second one is that the non-exceptional column submatrix $\Ec{\perm\TS}\Zap'$ as in Definition \ref{Def:ExceptionalSupport} should be identical to $\Ec{\perm\TQ}\Zap'$ in \eqref{SmplIncnSynMtrx};
The third one is that the difference between the exponents of the interim variables $\Ze'$ in \eqref{FirstSynInterimMtrx} and those of the exceptional variables $(\LEs z{l+3}n\prime\prime)$ in \eqref{SmplIncnSynMtrx} or \eqref{SecondSynInterimMtrx} should be a linear combination of the latent primary exponents $\alpha_1$ and $\alpha'_1$.

Let $p(\Z'_*,\Ze',\X'_\star)$ be the total transform of the interim transform $f(\XIp,x_1)$ in \eqref{InterimReorgFunc} under the interim monomial transformation $\CT{\per\TS'}$ as in \eqref{FirstSynInterimMtrx}.
Here the latent primary variables $\X'_\star:=(\ol x_1,x'_1)$.
After the interim preliminary and Weierstrass reductions as in Definition \ref{Def:InterimReductions}, the total transform $p(\Z'_*,\Ze',\X'_\star)$ is reduced to an interim transform $f(\XIq,\X'_\star)$ resembling $f(\XIp,x_1)$ in \eqref{InterimReorgFunc} with the synthetic variables $\XIq:=(\X''_*,\Ze')$.
That is,
\begin{equation}\label{Dgr2LtnGrdtn}
f(\XIq,\X'_\star)=\ol x_1^{\alpha_1}{x'_1}^{a'_1}{\Ze'}^{\A'\cdot\Ev{\per\TS}{\Ze'}'}
f_0(\XIq)+\phi(\XIq,\X'_\star),
\end{equation}
which is similar to the partial gradation in \eqref{InterimReorgFunc}.
Here the exponent $\alpha_1$ of the latent primary variable $\ol x_1$ is associated with the exceptional support $\es{\perm\TQ'}\Zop$ as in \eqref{SmplIncnSynMtrx} and $a'_1$ denotes the first component of the vertex $\A'$ associated with the reduced exponential matrix $\perm\TN'$ in \eqref{IncstnExpntlMtrx}.
The exponents $\Al'_\star$ of the latent primary variables $\X'_\star$ satisfy $\Al'_\star\ne (\alpha_1,a'_1)$ for $\forall (\Al_\XIq,\Al'_\star)\in\supp(\phi(\XIq,\X'_\star))$.

Similar to the identities in \eqref{VariableIdentity}, we have $\X''_0=\Z'_0$ and $\Al''_0=\Ga'_0$ as well, based on which we can define a new interim exponential matrix $\perm\wh\TT'$ as in \eqref{ModifiedInterimSyn} resembling the interim exponential matrix $\perm\TT'$ in \eqref{SecondSynInterimMtrx}.

\begin{equation}\label{ModifiedInterimSyn}\begin{aligned}
    \perm\wh\TT'&:=
    \begin{array}{r @{} c @{\hs{-0.5pt}} c @{\hs{-2.8pt}} p{4mm}
    @{\hs{-0.8mm}} *{3}c @{\hs{1mm}} c @{\hs{1mm}} c @{\hs{0.3mm}}
    c @{\hs{1mm}} c @{\hs{0.3mm}} c @{\hs{1mm}} *{3}c @{\hs{-0.8mm}}
    p{4mm} @{\hs{-2.5mm}} l}
    &           &           &           &           \E_1        &
    \cdots      &           \E_l        &           &           &
    &           &           &           &           &           &
    &                                                                       \\
    \mr{8}{$\rule{0mm}{24mm}$}          &           x''_1       &
    \mr{8}{$\rule{0mm}{24mm}$}          &
    \raisebox{2.5pt}[0pt]{\mr{9}{$\left[\rule{0mm}{66pt}\right.$}}
    &           \mc{3}{c}{\mr{3}{$\TE$}}            &           \vline
    &           \mc{7}{c}{\mr{3}{$\TO$}}            &
    \raisebox{2.5pt}[0pt]{\mr{9}{$\left.\rule{0mm}{66pt}\right]$}}
    &           \mr{8}{$\rule{0mm}{24mm}$}                                  \\
    &           \vdots      &           &           &           &
    &           &           \vline      &           &           &
    &           &           &           &           &           &           \\
    &           x''_l       &           &           &           &
    &           &           \vline      &           &           &
    &           &           &           &           &           &           \\
    \cline{2-16}
    &           \zeta'_{l+3}            &           &           &
    \mc{3}{c}{\mr{5}{$\TO$}}            &           \vline      &
    \mr{3}{$\TA'_1\bl'_2$}  &           \vline      &
    \mr{3}{$\bl'_1$}        &           \vline      &
    \mc{3}{c}{\mr{3}{$\TE$}}            &           &                       \\
    &           \vdots      &           &           &           &
    &           &           \vline      &           &           \vline
    &           &           \vline      &           &           &
    &           &                                                           \\
    &           \zeta'_n    &           &           &           &
    &           &           \vline      &           &           \vline
    &           &           \vline      &           &           &
    &           &                                                           \\
    \cline{2-2}\cline{8-16}
    &           \ol x_1     &           &           &           &
    &           &           \vline      &
    \mc{7}{c}{\raisebox{-1pt}[0pt]{$\ol\U'_1$}}     &           &           \\
    \cline{2-2}\cline{8-16}
    &           x'_1        &           &           &           &
    &           &           \vline      &           \mc{7}{c}{\U'_2}
    &           &                                                           \\
    &           &           &           &           x''_1       &
    \cdots      &           x''_l       &           \vline      &
    x''_{l+1}   &           \vline      &           x''_{l+2}   &
    \vline      &           x''_{l+3}   &           \cdots      &
    x''_n       &           &                                               \\
    \end{array}.
\end{aligned}\end{equation}

The interim exponential matrix $\perm\wh\TT'$ in \eqref{ModifiedInterimSyn} and reduced exponential matrix $\TN''$ in \eqref{MoreCstnExptlMtrx} or $\perm\TN''$ in \eqref{MoreIncnExptlMtrx} can be synthesized into a new exponential matrix $\TQ''$ in \eqref{FinalCnstSynMtrx} or $\perm\TQ''$ in \eqref{FinalSyntheticMatrix}, in a way similar to the construction of the synthetic exponential matrix $\TQ'$ in \eqref{CstnSyntheticMatrix} or $\perm\TQ'$ in \eqref{IncstnSyntheticMatrix}.
In \eqref{FinalCnstSynMtrx} we have $m>l$ and the submatrix $\TF'':=\bigl[\begin{smallmatrix}\TE\\ \TO\end{smallmatrix}\bigr]$ is of dimensions $(n-l)$ by $(m-l)$.
The submatrix $\TOm''_3:=\bigl[\begin{smallmatrix}\TOm''_2\\ \TE\end{smallmatrix}\bigr]$ is of dimensions $(n-l)$ by $(n-m)$ with $\TOm''_2$ being defined as in \eqref{MoreCstnExptlMtrx}.
The $(n-l-2)$ by $(n-l)$ submatrix $\TA'_2$ is defined as $\TA'_2:=[\TA'_1\bl'_2~\TA'_1]$ with $\TA'_1:=[\,\bl'_1~\TE]$ being the $(n-l-2)$ by $(n-l-1)$ submatrix of $\perm\wh\TT'$ in \eqref{ModifiedInterimSyn}.
The $(n-l)$-dimensional row vectors $\ol\U'_1$ and $\U'_2$ are as in \eqref{SecondSynInterimMtrx}.

In \eqref{FinalSyntheticMatrix} we have $m\ge l$ and the $(n-l)$ by $(m-l+1)$ submatrix $\TF'':=\bigl[\begin{smallmatrix}\TE\\ \TO\end{smallmatrix}\bigr]$.
The $(n-l)$ by $(n-m-1)$ submatrix $\TOm''_3:=\bigl[\begin{smallmatrix}\TOm''_2\\ \TE\end{smallmatrix}\bigr]$ with $\TOm''_2$ being defined as in \eqref{MoreIncnExptlMtrx}.
The $(n-m-1)$ by $(n-m)$ submatrix $\TB''$ denotes the submatrix $[\TD''\bl''~\TD'']$ in \eqref{MoreIncnExptlMtrx}.
The $(n-m)$-dimensional row vector $\N''_1$ is as in \eqref{MoreIncnExptlMtrx}.
The other notations are the same as in \eqref{FinalCnstSynMtrx}.

\begin{equation}\label{FinalCnstSynMtrx}\begin{aligned}
    \TQ'':=
    \begin{array}{r @{} c @{\hs{-0.5pt}} c @{\hs{-3.6pt}} p{4mm}
    @{\hs{-1mm}} *{3}c @{\hs{2pt}} c @{\hs{2pt}} *{3}c @{\hs{2pt}}
    c @{\hs{2pt}} *{3}c @{\hs{-0.8mm}} p{4mm} @{\hs{-2.5mm}} l}
    &           &           &           &           \E_1        &
    \cdots      &           \E_l    &           &           &
    &           &           &                                               \\
    \mr{12}{$\rule{0mm}{24mm}$}         &           x''_1       &
    \mr{12}{$\rule{0mm}{24mm}$}         &
    \raisebox{11.4mm}[0pt]{\mr{13}{$\left[\rule{0mm}{70pt}\right.$}}
    &           \mc{3}{c}{\mr{3}{$\TE$}}            &           \vline
    &           \mc{3}{c}{\mr{3}{$\TO$}}            &           \vline
    &           \mc{3}{c}{\mr{3}{$\TOm_1''\TD''$}}  &
    \raisebox{11.4mm}[0pt]{\mr{13}{$\left.\rule{0mm}{70pt}\right]$}}
    &           \mr{12}{$\rule{0mm}{24mm}$}                                 \\
    &           \vdots      &           &           &           &
    &           &           \vline      &           &           &
    &           \vline      &           &           &           &
    &                                                                       \\
    &           x''_l       &           &           &           &
    &           &           \vline      &           &           &
    &           \vline      &           &           &           &
    &                                                                       \\
    \cline{2-16}
    &           \zeta'_{l+3}            &           &           &
    \mc{3}{c}{\mr{5}{\raisebox{-3mm}[0pt]{$\TO$}}}  &           \vline
    &           \mc{3}{c}{\mr{3}{$\TA'_2\TF''$}}    &           \vline
    &           \mc{3}{c}{\mr{3}{$\TA'_2\TOm''_3\TD''$}}        &
    &                                                                       \\
    &           \vdots      &           &           &           &
    &           &           \vline      &           &           &
    &           \vline      &           &           &           &
    &                                                                       \\
    &           \zeta'_n    &           &           &           &
    &           &           \vline      &           &           &
    &           \vline      &           &           &           &
    &                                                                       \\
    \cline{2-2}\cline{8-16}
    &           x'_1        &           &           &           &
    &           &           \vline      &
    \mc{3}{c}{\raisebox{-1pt}[0pt]{$\U'_2\TF''$}}   &           \vline
    &           \mc{3}{c}{\raisebox{-1pt}[0pt]{$\U'_2\TOm''_3\TD''$}}
    &           &                                                           \\
    \cline{2-2}\cline{8-16}
    &           \ol x_1     &           &           &           &
    &           &           \vline      &
    \mc{3}{c}{\raisebox{-1pt}[0pt]{$\ol\U'_1\TF''$}}            &
    \vline      &
    \mc{3}{c}{\raisebox{-1pt}[0pt]{$\ol\U'_1\TOm''_3\TD''$}}    &
    &                                                                       \\
    &           &           &           &           z''_1       &
    \cdots      &           z''_l       &           \vline      &
    z''_{l+1}   &           \cdots      &           z''_m       &
    \vline      &           z''_{m+1}   &           \cdots      &
    z''_n       &           &                                               \\
    \end{array}
\end{aligned}\end{equation}

When a synthetic exponential matrix like $\TQ''$ in \eqref{FinalCnstSynMtrx} or $\perm\TQ''$ in \eqref{FinalSyntheticMatrix} has more than one latent primary variables like $\ol x_1$ and $x'_1$, it is important to examine the consistency of $\ol x_1$ and $x'_1$ with $\TQ''$ and $\iI\Zoq$ or $\perm\TQ''$ and $\iI\Zoq$ in the order of their nesting levels as in Definition \ref{Def:NestLatency}.
That is, the priority is always given to the latent primary variable with a lower nesting level.
In this case $\ol x_1$ has priority over $x'_1$ since the nesting level of $\ol x_1$ equals one whereas that of $x'_1$ equals two.

In the case when the latent primary variable $\ol x_1$ is consistent with the synthetic exponential matrix $\TQ''$ in \eqref{FinalCnstSynMtrx} or $\perm\TQ''$ in \eqref{FinalSyntheticMatrix} and is revived, the discussion on the canonical reduction $\TN_{\TQ'}$ of $\TQ'$ in \eqref{CstnNrmlzdMtrx} or $\TN_{\per\TQ'}$ of $\perm\TQ'$ in \eqref{IncstnNrmlzdMtrx} can be repeated almost verbatim here for $\TQ''$ in \eqref{FinalCnstSynMtrx} or $\perm\TQ''$ in \eqref{FinalSyntheticMatrix}. In particular, the conclusion is partly due to the fact that the latent primary exponent of $\ol x_1$ equals that of $x_1$ as per the interim exponential matrix $\perm\TS'$ in \eqref{FirstSynInterimMtrx}.

\begin{equation}\label{FinalSyntheticMatrix}\begin{aligned}
    \perm\TQ'':=
    \begin{array}{r @{} c @{\hs{-0.5pt}} c @{\hs{-3.6pt}} p{4mm}
    @{\hs{-1mm}} *{3}c @{\hs{2pt}} c @{\hs{2pt}} *{3}c @{\hs{2pt}}
    c @{\hs{2pt}} *{3}c @{\hs{-0.8mm}} p{4mm} @{\hs{-2.5mm}} l}
    &           &           &           &           \E_1        &
    \cdots      &           \E_{l-1}    &           &           &
    &           &           &                                               \\
    \mr{12}{$\rule{0mm}{24mm}$}         &           x''_2       &
    \mr{12}{$\rule{0mm}{24mm}$}         &
    \raisebox{8.3mm}[0pt]{\mr{13}{$\left[\rule{0mm}{75pt}\right.$}}
    &           \mc{3}{c}{\mr{3}{$\TE$}}            &           \vline
    &           \mc{3}{c}{\mr{3}{$\TO$}}            &           \vline
    &           \mc{3}{c}{\mr{3}{$\TOm_1''\TB''$}}  &
    \raisebox{8.3mm}[0pt]{\mr{13}{$\left.\rule{0mm}{75pt}\right]$}}
    &           \mr{12}{$\rule{0mm}{24mm}$}                                 \\
    &           \vdots      &           &           &           &
    &           &           \vline      &           &           &
    &           \vline      &           &           &           &
    &                                                                       \\
    &           x''_l       &           &           &           &
    &           &           \vline      &           &           &
    &           \vline      &           &           &           &
    &                                                                       \\
    \cline{2-16}
    &           \zeta'_{l+3}            &           &           &
    \mc{3}{c}{\mr{6}{\raisebox{-3mm}[0pt]{$\TO$}}}  &           \vline
    &           \mc{3}{c}{\mr{3}{$\TA'_2\TF''$}}    &           \vline
    &           \mc{3}{c}{\mr{3}{$\TA'_2\TOm''_3\TB''$}}        &
    &                                                                       \\
    &           \vdots      &           &           &           &
    &           &           \vline      &           &           &
    &           \vline      &           &           &           &
    &                                                                       \\
    &           \zeta'_n    &           &           &           &
    &           &           \vline      &           &           &
    &           \vline      &           &           &           &
    &                                                                       \\
    \cline{2-2}\cline{8-16}
    &           x'_1        &           &           &           &
    &           &           \vline      &
    \mc{3}{c}{\raisebox{-1pt}[0pt]{$\U'_2\TF''$}}   &           \vline
    &           \mc{3}{c}{\raisebox{-1pt}[0pt]{$\U'_2\TOm''_3\TB''$}}
    &           &                                                           \\
    \cline{2-2}\cline{8-16}
    &           \ol x_1     &           &           &           &
    &           &           \vline      &
    \mc{3}{c}{\raisebox{-1pt}[0pt]{$\ol\U'_1\TF''$}}            &
    \vline      &
    \mc{3}{c}{\raisebox{-1pt}[0pt]{$\ol\U'_1\TOm''_3\TB''$}}    &
    &                                                                       \\
    \cline{2-2}\cline{8-16}
    &           x''_1       &           &           &           &
    &           &           \vline      &
    \mc{3}{c}{\raisebox{-1.5pt}[0pt]{$\TO$}}        &           \vline
    &           \mc{3}{c}{\raisebox{-1pt}[0pt]{$\N''_1$}}       &
    &                                                                       \\
    &           &           &           &           z''_1       &
    \cdots      &           z''_{l-1}   &           \vline      &
    z''_l       &           \cdots      &           z''_m       &
    \vline      &           z''_{m+1}   &           \cdots      &
    z''_n       &           &                                               \\
    \end{array}
\end{aligned}\end{equation}

In the case when the latent primary variable $\ol x_1$ is inconsistent with the synthetic exponential matrix $\TQ''$ and $\iI\Zoq$ in \eqref{FinalCnstSynMtrx} or $\perm\TQ''$ and $\iI\Zoq$ in \eqref{FinalSyntheticMatrix} and thus sustains its latency, whereas $x'_1$ is consistent with $\TQ''$ and $\iI\Zoq$ or $\perm\TQ''$ and $\iI\Zoq$ and thus revived, a conclusion similar to Lemma \ref{Lemma:ConstantPrimaryExp} holds for the latent primary exponent $\alpha_1$ and exceptional exponents $\Ga''_0$ of the partial transform $q(\Zq)$ under the synthetic monomial transformation $\CT{\TQ''}$ or $\CT{\per\TQ''}$.
That is, a difference in $\alpha_1$ corresponds to one in $\Ga''_0$ under $\TQ''$ in \eqref{FinalCnstSynMtrx} or $\perm\TQ''$ in \eqref{FinalSyntheticMatrix}, based on which we have the following conclusion.

\begin{lemma}\label{Lemma:MiddleVariableRevive}
When the latent primary variable $\ol x_1$ is inconsistent with the synthetic exponential matrix $\TQ''$ in \eqref{FinalCnstSynMtrx} or $\perm\TQ''$ in \eqref{FinalSyntheticMatrix} whereas $x'_1$ is consistent and revived as a primary variable $t''_1$, the singularity height associated with $t''_1$ strictly decreases from the prior singularity height $d'$ in \eqref{ReorganizedFunc} and \eqref{InterimReorgFunc} at a regular reduced branch point.
\end{lemma}
\begin{proof}
When the exceptional support $\es{\TQ''}\Zoq$ or $\es{\perm\TQ''}\Zoq$ corresponds to the latent primary exponent $\alpha_1=a_1$ as in \eqref{Dgr2LtnGrdtn}, the conclusion follows from an argument similar to Lemma \ref{Lemma:SimpleDecrease} based on the Weierstrass form $f_0(\XIp)$ in \eqref{InterimWeierstrassForm} and the invariance of the latent primary exponents $\Al'_\star:=(\alpha_1,\alpha'_1)$ under $\perm\TS'$ in \eqref{FirstSynInterimMtrx} and after the ensuing interim preliminary and Weierstrass reductions.

Let $q(\Zq)$ denote the partial transform of the apex form $f(\XIq,\X'_\star)$ in \eqref{Dgr2LtnGrdtn} under $\CT{\TQ''}$ or $\CT{\per\TQ''}$.
For $\forall (\Al_\XIq,\Al'_\star)\in\supp (f(\XIq,\X'_\star))$, let $\Ga''((\Al_\XIq,\Al'_\star))$ denote the corresponding exponent in $\supp (q(\Zq))$.
When $\es{\TQ''}\Zoq$ or $\es{\perm\TQ''}\Zoq$ corresponds to a latent primary exponent $\alpha_1\ne a_1$, the inconsistency of $\ol x_1$ with $\TQ''$ and $\iI\Zoq$ in \eqref{FinalCnstSynMtrx} or $\perm\TQ''$ and $\iI\Zoq$ in \eqref{FinalSyntheticMatrix} indicates that $\Ga''(\Al_\XIq,a_1,d')\succx\bz$ for $\forall (\Al_\XIq,a_1,d')\in\supp (f(\XIq,\X'_\star))$ in \eqref{Dgr2LtnGrdtn} that corresponds to the apexes $\D'\in\supp (f_0(\XIp))$ in \eqref{InterimWeierstrassForm} and $\D'\in\supp (f(\Xp))$ in \eqref{ReorganizedFunc}.
It also indicates that for $\forall (\Al_\XIq,\Al'_\star)\in\supp (f(\XIq,\X'_\star))$ with $\alpha_1\ne a_1$ and $\alpha'_1\ge d'$, which corresponds to $(\Al_\XIp,\alpha_1)\in\supp (\phi(\XIp,x_1))$ in \eqref{InterimReorgFunc} with $\alpha'_1\ge d'$ and $\Al'\in\supp (\phi(\Xp))$ in \eqref{ReorganizedFunc} with $\alpha'_1\ge d'$, the exceptional exponents $\Ga''_0((\Al_\XIq,\Al'_\star))\ne\Ga''_0((\Al_\XIq,a_1,d'))$ for $\forall (\Al_\XIq,a_1,d')\in\supp (f(\XIq,\X'_\star))$ in \eqref{Dgr2LtnGrdtn}.
Thus it suffices to further prove that $\Ga''((\Al_\XIq,\Al'_\star))\succx\Ga''((\Al_\XIq,a_1,d'))$, which readily follows from $\Ga'(\Al')\succx\Ga'(\D')$ due to the inconsistency of $x'_1$ with $\perm\TN'$ and $\iI\Zop$ in the case of $\alpha'_1>d'$ and that of $x_1$ with $\perm\TQ'$ and $\iI\Zop$ in the case of $\alpha'_1=d'$.
Here $\D'\in\supp (f_0(\Xp))$ and $\Al'\in\supp (\phi(\Xp))$ as in \eqref{ReorganizedFunc} and $\Ga'$ denote the exponents of the partial transform $q(\Zp)$ of $f(\Xp)$ under $\CT{\per\TN'}$ as in \eqref{IncstnExpntlMtrx}.
Hence for $\forall (\Al_\XIq,\Al'_\star)\in\es{\TQ''}\Zoq$ or $\es{\perm\TQ''}\Zoq$, we have $\alpha'_1<d'$.
\end{proof}

\begin{equation}\label{GenericFinalSynMtrx}\begin{aligned}
    \perm\TQ^\ro:=
    \begin{array}{r @{} c @{\hs{-0.5pt}} c @{\hs{-3.6pt}} p{4mm}
    @{\hs{-0.8mm}} c @{\hs{0.3mm}} c @{\hs{1.3mm}} c @{\hs{0.3mm}}
    c @{\hs{1mm}} c @{\hs{0.5mm}} c @{\hs{1mm}} c @{\hs{0.5mm}} c
    @{\hs{1mm}} c @{\hs{0.5mm}} c @{\hs{1mm}} c @{\hs{0.5mm}} c
    @{\hs{0.7mm}} c @{\hs{0.5mm}} c @{\hs{1mm}} c @{} c
    @{\hs{-0.8mm}} p{4mm} @{\hs{-2.5mm}} l}
    &           &           &           &           \E_1        &
    \cdots      &           \E_{\kappa-1}           &           &
    &           &           &           &           &           &
    &           &           &           &           &           &
    &                                                                       \\
    \mr{12}{$\rule{0mm}{24mm}$}         &           x^\ro_2     &
    \mr{12}{$\rule{0mm}{24mm}$}         &
    \raisebox{1.6mm}[0pt]{\mr{13}{$\left[\rule{0mm}{96pt}\right.$}}
    &           \mc{3}{c}{\mr{4}{$\TE$}}            &           \vline
    &           \mr{4}{$\TO$}           &           \vline      &
    \mc{10}{c}{\mr{4}{$\TOm_\rho$}}     &
    \raisebox{1.6mm}[0pt]{\mr{13}{$\left.\rule{0mm}{96pt}\right]$}}
    &           \mr{12}{$\rule{0mm}{24mm}$}                                 \\
    &           \vdots      &           &           &           &
    &           &           \vline      &           &           \vline
    &           &           &           &           &           &
    &           &           &           &                                   \\
    &           x^\ro_\kappa            &           &           &
    &           &           &           \vline      &           &
    \vline      &           &           &           &           &
    &           &           &           &           &                       \\
    \cline{2-7}\cline{8-9}
    &           \ol x^{(\rho-1)}_1      &           &           &
    \mc{3}{c}{\mr{8}{$\TO$}}            &           \vline      &
    \mr{8}{$\W$}            &           \vline      &           &
    &           &           &           &           &           &
    &           &                                                           \\
    \cline{2-2}\cline{10-21}
    &           \ol x^{(\rho-2)}_1      &           &           &
    &           &           &           \vline      &           &
    \vline      &           \mr{6}{$\TG_\rho\bl_\rho$}          &
    \vline      &           \mr{5}{$\cdots$}        &           \vline
    &           \mr{3}{$\TG_4\bl_4$}    &           \vline      &
    \mr{2}{$\TG_3\bl_3$}    &           \vline      &
    \mc{2}{c}{\U_2}         &           &                                   \\
    \cline{2-2}\cline{18-21}
    &           \ol x^{(\rho-3)}_1      &           &           &
    &           &           &           \vline      &           &
    \vline      &           &           \vline      &           &
    \vline      &           &           \vline      &           &
    \vline      &           \mc{2}{c}{\U_3}         &           &           \\
    \cline{2-2}\cline{16-21}
    &           \ol x^{(\rho-4)}_1      &           &           &
    &           &           &           \vline      &           &
    \vline      &           &           \vline      &           &
    \vline      &           &           \vline      &
    \mc{4}{c}{\U_4}         &           &                                   \\
    \cline{2-2}\cline{14-21}
    &           \ol x^{(\rho-5)}_1      &           &           &
    &           &           &           \vline      &           &
    \vline      &           &           \vline      &           &
    \vline      &           \mc{6}{c}{\U_5}         &           &           \\
    \cline{2-2}\cline{14-21}
    &           \vdots      &           &           &           &
    &           &           \vline      &           &           \vline
    &           &           \vline      &           \mc{8}{c}{\vdots}
    &           &                                                           \\
    \cline{2-2}\cline{12-21}
    &           \ol x_1     &           &           &           &
    &           &           \vline      &           &           \vline
    &           &           \vline      &
    \mc{8}{c}{\U_\rho}      &           &                                   \\
    \cline{2-2}\cline{10-21}
    &           x^\ro_1     &           &           &           &
    &           &           \vline      &           &           \vline
    &           \mc{10}{c}{\U_0}        &           &                       \\
    &           &           &           &           z^\ro_1     &
    \cdots      &           z^\ro_{\kappa-1}        &           \vline
    &           z^\ro_\kappa            &           \vline      &
    z^\ro_{\kappa+1}        &           \vline      &           \cdots
    &           \vline      &           z^\ro_{n-3} &           \vline
    &           z^\ro_{n-2} &           \vline      &           z^\ro_{n-1}
    &           z^\ro_n     &           &                                   \\
    \end{array}
\end{aligned}\end{equation}

In the case of inconsistency of both the latent primary variables $x_1$ and $x'_1$ with $\TQ''$ and $\iI\Zoq$ in \eqref{FinalCnstSynMtrx} or $\perm\TQ''$ and $\iI\Zoq$ in \eqref{FinalSyntheticMatrix} such that $x_1$ and $x'_1$ sustain their latency, the synthetic exponential matrix $\TQ''$ or $\perm\TQ''$ has nesting degree $\fnd=2$ or $\fnd=3$.
A canonical reduction reduces $\TQ''$ to the scenario of $\perm\TQ'$ with $\fnd=2$ in \eqref{SmplIncnSynMtrx}.
Nevertheless the nesting degree $\fnd=3$ of $\perm\TQ''$ in \eqref{FinalSyntheticMatrix} is a strict increase from $\fnd=2$ of $\perm\TQ'$ in \eqref{SmplIncnSynMtrx}.
It is easy to see that if the nesting degree continues to increase strictly in this way, it shall lead to a scenario when there are no interim variables and the synthetic variables are composed of latent primary variables and non-exceptional variables, as the exemplar form $\perm\TQ^\ro$ in \eqref{GenericFinalSynMtrx} shows in which case $\rho+\kappa=n$ with the exceptional index set $\iI{\Z^\ro}=\dotc{\kappa+1}n$.
In particular, suppose that the variables $\X^\ro_\star:=(\ol x_1,\LEs{\ol x}11\prime{(\rho-1)},x^\ro_1)$ are the latent primary variables.
For $3<j\le\rho$, the submatrix $\TG_j$ in \eqref{GenericFinalSynMtrx} denotes the $(j-1)$ by $(j-1)$ submatrix $\bigl[\begin{smallmatrix}\TB_j\\ \U_j\end{smallmatrix}\bigr]$ with $\TB_j:=[\TG_{j-1}\bl_{j-1}~\TG_{j-1}]$.
And the submatrix $\TG_3$ in \eqref{GenericFinalSynMtrx} denotes the $2$ by $2$ submatrix $\bigl[\begin{smallmatrix}\U_2\\ \U_3\end{smallmatrix}\bigr]$.
Altogether we have the following nested non-degeneracies:
\[
\TG_0\cdot\prod_{j=3}^\rho\det\TG_j\ne 0,
\]
where the submatrix $\TG_0$ denotes the $\rho$ by $\rho$ submatrix $\bigl[\begin{smallmatrix}\TB_0\\ \U_0\end{smallmatrix}\bigr]$ with $\TB_0:=[\TG_\rho\bl_\rho~\TG_\rho]$.

For the latent primary variables $\X^\ro_\star$ as in \eqref{GenericFinalSynMtrx}, we examine their consistency with $\perm\TQ^\ro$ and $\iI{\Z^\ro}$ in the order of their nesting levels.
And the priority is given to the latent primary variables with lower nesting levels.
The exemplar form $\perm\TQ^\ro$ in \eqref{GenericFinalSynMtrx} shows that one of the latent primary variables in $\X^\ro_\star$ can be revived through a canonical reduction.
Similar to Lemma \ref{Lemma:MiddleVariableRevive}, we can prove that the strictly decreasing process of the singularity height associated with the revived primary variable continues.

Although all the above discussions are based on the assumptions that $|\iI\Zoq|=n-m\le |\iI\Zop|=n-l$ and $|\iI\Zop|=n-l\le |\iI\Zo|=n-k$, i.e., the codimensions of the exceptional index sets are increasing as $k\le l\le m$, it is easy to see that the discussions for the cases of $k>l$ and $l>m$ make no essential difference.
And in these cases the singularity heights strictly decrease ultimately as well.

\subsection{Resolution of heterogeneous and generic singularities}
\label{Section:Generic}

In the generic case when the irregular and inconsistent singularities are intermingled, there arises neither essential difference nor added complication in the resolution algorithm.

In the case when an irregular singularities is ensued by an inconsistent one, the latent gradations for both resolutions of singularities should be implemented together before the invocation of Weierstrass preparation theorem.
More specifically, let $\vfz f1(\Xp)$ be the apex form in \eqref{PreWeierstrass} whose latent reducible function $f(\Xp)$ bears the Weierstrass form in \eqref{MonoWeierstrass} and whose primary and latent variables are $x'_2$ and $z_1$ respectively.
Now suppose that the primary variable $x'_2$ is inconsistent with the exponential matrix $\TM'$ and exceptional index set $\iI\Yop$ in \eqref{VirtualMonomialTransform} whose canonical reduction $\perm\TN'$ resembles the form of $\perm\TN$ in \eqref{IncstnMnmlTrnsfrm}.
Let $\CT{\per\TN'}$ denote the reduced monomial transformation associated with $\perm\TN'$ and $\vfz q1(\Zp)$ the partial transform of the apex form $\vfz f1(\Xp)$ in \eqref{PreWeierstrass} under $\CT{\per\TN'}$.
Also let $\wt\alpha'_2$ denote the latent primary component associated with the latent primary variable $x'_2$.
Similar to \eqref{PreReorganize}, a latent gradation of the partial transform $\vfz q1(\Zp)$ can be implemented by $\wt\alpha'_2$ being zero or not.
This is followed by a localization and non-degenerate linear modification leading to the following apex form $\vfz f1(\Xq)$ and its gradation similar to \eqref{ReorganizedFunc}:
\begin{equation}\label{PreLatentGradation}
\vfz f1(\Xq)=f_{z_1,0}(\Xq)+\vfz\phi 1(\Xq),
\end{equation}
where the latent primary component $\wt\alpha'_2$ as in Definition \ref{Def:LatentPrimaryCompnt} of the exceptional exponents of the latent reducible function $f_{z_1,0}(\Xq)$ satisfies $\wt\alpha'_2=0$ whereas none of those of the latent remainder function $\vfz\phi 1(\Xq)$ equal zero.
The reduced exponential matrix $\perm\TN'$ can be decomposed as $\perm\TN'=\perm\TS'\cdot\perm\TT'$ similar to that in \eqref{MatrixDecomposition}.
Under $\CT{\per\TS'}$ and the subsequent interim preliminary and Weierstrass reductions, the apex form $\vfz f1(\Xq)$ in \eqref{PreLatentGradation} is transformed and then reduced to an interim transform $\vfz f1(\XIq,x'_2)$ as in Definition \ref{Def:InterimReductions} that is similar to \eqref{InterimReorgFunc}:
\begin{equation}\label{PreInterimGradation}
\vfz f1(\XIq,x'_2)={x'_2}^{a'_2}{\Ze'}^{\A'\cdot\Ev{\per\TS}{\Ze'}'}\cdot
f_{z_1,0}(\XIq)+\vfz\phi 1(\XIq,x'_2)
\end{equation}
with $\XIq:=(\X''_*,\Ze')$ and $\A':=(\LEs a2n\prime\prime)$ being the vertex associated with the reduced exponential matrix $\perm\TN'$.
Subsequently a latent gradation of the latent and interim reducible functions $f_{z_1,0}(\Xq)$ in \eqref{PreLatentGradation} and $f_{z_1,0}(\XIq)$ in \eqref{PreInterimGradation} by the latent variable $z_1$ can be implemented in a way similar to that in \eqref{PreWeierstrass}:
\begin{equation}\label{FinalVirtualGrad}\begin{aligned}
f_{z_1,0}(\Xq)&=z_1^{m'}f(\Xq)+\vfz\psi 1(\Xq);\\
f_{z_1,0}(\XIq)&=z_1^{m'}f(\XIq)+\vfz\psi 1(\XIq)
\end{aligned}\end{equation}
such that no coefficients in the latent remainder functions $\vfz\psi 1(\Xq)$ and $\vfz\psi 1(\XIq)$ contain a term with the monomial factor $z_1^{m'}$.
Here two functions with different variables are deemed as being different albeit they might share the same name.
The latent and interim reducible functions $f(\Xq)$ and $f(\XIq)$ in \eqref{FinalVirtualGrad} bear an apex form resembling the latent and interim reducible functions $f(\Xp)$ in \eqref{ParadigmFunction} and $f_0(\XIp)$ in \eqref{InterimWeierstrassForm} respectively such that Weierstrass preparation theorem and completion of perfect power can be invoked on $f(\Xq)$ and $f(\XIq)$ in \eqref{FinalVirtualGrad} in the same way as in \eqref{MonoWeierstrass} and \eqref{InterimWeierstrassForm}.

In Lemma \ref{Lemma:VirtualRecovery} when the primary variable $x'_2$ is inconsistent with the reduced exponential matrix $\TN'$ and exceptional index set $\iI\Zop$, that is, when the consistency assumption does not hold, the primary variable $x'_2$ becomes latent and the singularity height might have a temporary increase.
In this case we invoke the above procedure of simultaneous latent gradations as well as the procedure of resolution of inconsistent singularities in Section \ref{Section:Inconsistency}.
After the latent primary variable $x'_2$ is revived eventually, the singularity height is reduced to zero and we can show that the residue order of the latent variable $z_1$ strictly decreases from the singularity height $d$ of the Weierstrass polynomial $w(\X)$ in \eqref{PrimaryForm} after $z_1$ is revived.

When an inconsistent singularities is ensued by an irregular one immediately, the irregular one can be ignored.
In fact, in this case the inconsistent one already renders the primary variable latent and hence after the subsequent localization the singularity height might have a temporary increase anyway.
Thus the new primary variable is determined instead by a linear modification as per the order of the localized proper transform in the latent preliminary reduction.

Finally, we would like to point out that the resolution of irregular singularities in Section \ref{Section:SecondarySing} can be amalgamated with that for inconsistent singularities in Section \ref{Section:Inconsistency} to form a uniform algorithm for inconsistent singularities.
In fact, the linear modification $\CL{\X'_*}$ in \eqref{LinearModification} can be augmented to entail the trivial transformation $x'_1=z_1$.
As a result, the $(n-1)$ by $(n-1)$ reduced exponential matrix $\TN'$ in \eqref{VirtualResolution} is augmented into an $n$ by $n$ reduced exponential matrix $\perm\TN'$ bearing the same form as in \eqref{IncstnExpntlMtrx} except that the column vector $\bl'=\bz$ and $(n-l)$-dimensional row vector $\N'_1=\E_1=(1,\dotsc,0)$ in this case.
In this way it is obvious that the latent variable $z_1$ becomes a latent primary variable $x'_1$ that is inconsistent with $\perm\TN'$ and exceptional index set $\iI\Zop$.
Nevertheless it is unnecessary to invoke the method of synthetic exponential matrices as in Section \ref{Section:Inconsistency} in this case since the reduced exponential matrices after augmentations act on the exponents of the latent primary variables directly.

The resolution algorithm for the generic case of ideals can be reduced to the hypersurface case directly.
More specifically, suppose that an ideal $\CI$ is generated by functions $\{f_1(\X),\dotsc,f_m(\X)\}$.
In this case we can study the hypersurface generated by the product $f(\X):=f_1(\X)\cdots f_m(\X)$, which amounts to applying the resolution algorithm for the hypersurface case to each generator $f_j(\X)$ of $\CI$ for $1\le j\le m$ simultaneously.
This naive approach to the resolution algorithm is based on the fact that the Newton polyhedron can ``automatically" choose the appropriate resolution center in each resolution step, which is vindicated by the conclusions in \cite{MZ}.

\section{A finite partition of unity in a neighborhood\\
of a singular point}
\label{Section:Partition}

By default we regard the origin as the singular point of a function $f(\X)$ under investigation.
A couple of conclusions in this section are either well known in toric geometry or obtained already in \cite{MZ}.
They are presented here either for completeness or with a few improvements.
Let us start with the following well known result whose proof is omitted.

\begin{lemma}\label{Lemma:HalfSpaces}
Let $\PI\NP$ denote the normal vector set of a Newton polyhedron $\NP$ as in Definition \ref{Def:NewtonPolyhedron}.
Then $\NP=\bigcap_{\V\in\PI\NP}\sU (\V)$ with the upper half space $\sU (\V)$ associated with $\V$ being defined as $\sU (\V):=\{\Al\in\BB R^n\colon\langle\V,\Be\rangle\le\langle\V,\Al\rangle~\text{for}~\forall\Be\in\facet (\V)\}$, which is equivalent to the convex hull definition of Newton polyhedron in Definition \ref{Def:NewtonPolyhedron}.
\end{lemma}

\begin{lemma}\label{Lemma:IntersectLemma}
For $\forall\W\in\DR$, let $\face (\W):=\sF$ be a face of a Newton polyhedron $\NP$ associated with $\W$ as in \eqref{FaceEq}.
Then the following basic conclusions hold for $\face (\W)$.
\begin{inparaenum}[(a)]
\item\label{item:FaceIn}
For a finite vector set $V\subseteq\DR$, if $\W\in\CC^\circ(V)$ and $\bigcap_{\V\in V}\face (\V)\ne\emptyset$, then $\face (\W)=\bigcap_{\V\in V}\face (\V)$;
\item\label{item:Minimality}
The generator set $\gv\sF$ of the facial cone $\CV\sF$ as in Definition \ref{Def:VertexCone} is minimal, i.e., $\V\notin\CC (\gv\sF\setminus\{\V\})$ for $\forall\V\in\gv\sF$.
Moreover, $\CV\sF\cap\PI\NP=\gv\sF$;
\item\label{item:UbiquitousVertex}
There exists a vertex $\A$ of $\NP$ such that $\A\in\face (\W)$ and $\W\in\CV\A$;
\item\label{item:BasicConclusion}
$\face (\W)=\bigcap_{\V\in\gv\sF}\facet (\V)$;
\item\label{item:FacialAll}
The facial cone $\CV\sF$ satisfies $\CV\sF^\circ=\{\W\in\DR\colon\face (\W)=\sF\}$;
\item\label{item:Conjugacy}
Let $\SF 1$ and $\SF 2$ be two faces of $\NP$.
Then $\SF 1\subseteq\SF 2$ if and only if $\VF 1\supseteq\VF 2$;
\item\label{item:NoInteriInter} Let $\SF 1$ and $\SF 2$ be two faces of $\NP$.
Then $\VF 1^\circ\cap\VF 2^\circ\ne\emptyset$ if and only if $\SF 1=\SF 2$ if and only if $\VF 1=\VF 2$;
\item\label{item:VertexIn}
Let $\sF=\face (\W)$ as above and $\W'\in\DR$.
Then $\W'\in\CV\sF$ if and only if $\sF\subseteq\face (\W')$.
\end{inparaenum}
\end{lemma}
\begin{proof}
\begin{asparaenum}[(a)]
{\parindent=\saveparindent
\item For $\forall\X\in\bigcap_{\V\in V}\face (\V)\ne\emptyset$
and $\forall\Y\in\NP$, we have $\langle\V,\X\rangle\le\langle\V,\Y\rangle$ for $\forall\V\in V$ as per \eqref{FaceEq}.
Hence $\langle\W,\X\rangle\le\langle\W,\Y\rangle$.
Thus $\X\in\face (\W)$ and $\bigcap_{\V\in V}\face (\V)\subseteq\face (\W)$.
Conversely, for $\forall\Y\in\NP\setminus\bigcap_{\V\in V}\face (\V)$, evidently $\exists\U\in V$ such that $\Y\notin\face (\U)$. That is, for $\forall\X\in\bigcap_{\V\in V}\face (\V)\ne\emptyset$, we have $\langle\U,\X\rangle<\langle\U,\Y\rangle$ as per \eqref{FaceEq}.
Hence $\langle\W,\X\rangle<\langle\W,\Y\rangle$ and $\Y\notin\face (\W)$.
Thus $\face (\W)\subseteq\bigcap_{\V\in V}\face (\V)$.

\item For $\forall\V\in\gv\sF$, if $\V\in\CC (\gv\sF\setminus\{\V\})$, then there exists a generator subset $V\subseteq\gv\sF\setminus\{\V\}$ with $|V|>1$ such that $\V\in\CC^\circ(V)$.
As per \eqref{item:FaceIn}, $\facet (\V)=\bigcap_{\U\in V}\facet (\U)$ since $\sF\subseteq\bigcap_{\U\in V}\facet (\U)\ne\emptyset$. This contradicts the dimension of $\facet (\V)$.

Now let us prove that $\CV\sF\cap\PI\NP=\gv\sF$.
For $\forall\X\in\sF$, $\forall\Y\in\NP$, we have $\langle\V,\X\rangle\le\langle\V,\Y\rangle$ for $\forall\V\in\gv\sF$ as per \eqref{FaceEq} since $\sF\subseteq\facet (\V)$.
Hence $\langle\W,\X\rangle\le\langle\W,\Y\rangle$ for $\forall\W\in\CC (\gv\sF)\cap\PI\NP$ and thus $\sF\subseteq\facet (\W)$, from which we deduce that $\W\in\gv\sF$ as per Definition \ref{Def:VertexCone}.
This shows that $\CV\sF\cap\PI\NP\subseteq\gv\sF$ since $\CV\sF=\CC (\gv\sF)$.
The other direction is evident.

\item The conclusion $\exists\,\A\in\sF$ is evident when $\sF$ is a vertex of $\NP$.
Suppose $\dim (\sF)>0$ henceforth.
Let $\sH (\V)$ denote the hyperplane with a normal vector $\V\in\DR$ such that $\face (\V)\subseteq\sH (\V)$.
Let us also denote $\sP:=(\bigcap_{\V\in\gv\sF}\sH (\V))\cap\sH (\W)$.
It follows from \eqref{FaceEq} and Lemma \ref{Lemma:HalfSpaces} that
\[
\sF=\NP\cap\sH (\W)=\hs{-1mm}\bigcap_{\V\in\gv\sF}\hs{-1mm}(\sU (\V)\cap\sH (\W))
\cap\hs{-2mm}\bigcap_{\V\in\PI\NP\setminus\gv\sF}\hs{-2mm}\sU (\V)
=\sP\cap\hs{-2mm}\bigcap_{\V\in\PI\NP\setminus\gv\sF}\hs{-2mm}\sU (\V).
\]
Hence $\exists\U\in\PI\NP\setminus\gv\sF$ such that $\sF\cap\facet (\U)\ne\emptyset$ since $\sP\setminus\Rp^n\ne\emptyset$.
Let $\V\in\CC^\circ(\U,\W)$.
Then $\face (\V)=\face (\W)\cap\facet (\U)$ as per \eqref{item:FaceIn} and $\dim (\face(\V))<\dim (\sF)$ since $\U\notin\gv\sF$.
A decreasing induction on the dimensions of the intersecting faces such as $\face (\V)$ leads to the conclusion that there exists a vertex $\A\in\NP$ such that $\A\in\face (\W)$.

Now we are ready to prove that $\A\in\face (\W)$ yields $\W\in\CV\A$.
Without loss of generality, suppose that $\rank (\gv\A)=n=|\gv\A|$ and $\W=\sum_{\U\in\gv\A}\lambda_\U\U$.
If $\exists V\subseteq\gv\A$ such that $\lambda_\U<0$ for $\forall\U\in V$, then it is evident that $\gv\A\setminus V\ne\emptyset$.
For $\forall\Be\in (\bigcap_{\U\in\gv\A\setminus V}\facet (\U))\setminus (\bigcup_{\U\in V}\facet (\U))\ne\emptyset$, it follows that
\[
\langle\W,\A\rangle=\sumsepm\sum_{\U\in\gv\A\setminus V}\sumsep\lambda_\U\langle\U,\A\rangle
+\sum_{\U\in V}\lambda_\U\langle\U,\A\rangle
>\sumsepm\sum_{\U\in\gv\A\setminus V}\sumsep\lambda_\U\langle\U,\Be\rangle
+\sum_{\U\in V}\lambda_\U\langle\U,\Be\rangle=\langle\W,\Be\rangle
\]
since for $\forall\U\in V$, we have $\langle\U,\A\rangle<\langle\U,\Be\rangle$ as per \eqref{FaceEq} and $\lambda_\U<0$.
This contradicts $\A\in\face (\W)$ as per \eqref{FaceEq}.

\item First of all, it is obvious that $\sF\subseteq\bigcap_{\V\in\gv\sF}\facet (\V)$ as per Definition \ref{Def:VertexCone}.
According to \eqref{item:UbiquitousVertex}, there exists a vertex $\A\in\face (\W)$ such that $\W\in\CV\A$.
Hence $\exists V\subseteq\gv\A$ such that $\W\in\CC^\circ (V)$.
As per \eqref{item:FaceIn}, we have $\sF=\face (\W)=\bigcap_{\V\in V}\facet (\V)$.
This shows that $V\subseteq\gv\sF$ and thus $\bigcap_{\V\in\gv\sF}\facet (\V)\subseteq\sF$.

\item The conclusion is evident when $\sF$ is a facet of $\NP$ since $\gv\sF$ comprises a unique normal vector in this case.
When $\sF$ is not a facet of $\NP$, let us denote $U_\sF:=\{\U\in\DR\colon\face (\U)=\sF\}$.
It readily follows from \eqref{item:FaceIn} and \eqref{item:BasicConclusion} that $\CV\sF^\circ\subseteq U_\sF$.
Conversely for $\forall\U\in U_\sF$, the proof of \eqref{item:BasicConclusion} indicates that $\exists V\subseteq\gv\sF$ with $|V|>1$ such that $\U\in\CC^\circ(V)$ and $\sF=\bigcap_{\V\in V}\facet (\V)$.
Further, $\sF=\bigcap_{\V\in\gv\sF}\facet (\V)$ as per \eqref{item:BasicConclusion}.
Hence $\rank (V)=\rank (\gv\sF)$ and thus $\U\in\CC^\circ(\gv\sF)=\CV\sF^\circ$.

\item By Definition \ref{Def:VertexCone}, it is easy to see that $\SF 1\subseteq\SF 2$ indicates $\gv{\SF 1}\supseteq\gv{\SF 2}$ and hence $\VF 1\supseteq\VF 2$.
The converse readily follows from $\CV\sF\cap\PI\NP=\gv\sF$ in \eqref{item:Minimality} and \eqref{item:BasicConclusion}.

\item Suppose $\V\in\VF 1^\circ\cap\VF 2^\circ$.
It follows from \eqref{item:FacialAll} that $\face (\V)=\SF 1=\SF 2$.
The other conclusions readily follow from \eqref{item:Conjugacy} and \eqref{item:FacialAll}.

\item Suppose that $\W'\in\CV\sF$.
Then $\exists W\subseteq\gv\sF$ such that $\W'\in\CC^\circ(W)$.
Hence $\face (\W')=\bigcap_{\V\in W\subseteq\gv\sF}\facet ({\V})$ according to \eqref{item:FaceIn} and thus $\sF\subseteq\face (\W')$ as per \eqref{item:BasicConclusion}.
Conversely, it follows from $\sF\subseteq\face (\W'):=\sF'$ that $\CV\sF\supseteq\CV{\sF'}$ as per \eqref{item:Conjugacy}.
Moreover, $\W'\in\CV{\sF'}^\circ$ as per \eqref{item:FacialAll}.
Hence $\W'\in\CV\sF$.
\qedhere}
\end{asparaenum}
\end{proof}

\begin{definition}\label{Def:AmbientCone}
{\upshape ($\ON$; $\ONs$; $\FV\A$; $\gfv\A$; ambient cone $\DR$; $\partial\Rp^n$)}

The sets of all the vertex cones and refined vertex cones of a Newton polyhedron $\NP$ as in Definition \ref{Def:VertexCone} and Definition \ref{Def:ExponentialMatrix} are denoted as $\ON$ and $\ONs$ respectively.
A facet of a vertex cone $\CV\A$ is denoted as $\FV\A$ whose generator set is denoted as $\gfv\A:=\gv\A\cap\FV\A$.
With $\{\LE\E 1n\}$ denoting the standard basis of $\BB R^n$, the \emph{ambient} cone of $\NP$ refers to $\DR=\CC (\LE\E 1n)$ whose boundary is denoted as $\partial\Rp^n$.
\end{definition}

\begin{lemma}\label{Lemma:AllAreFans}
The following hold for a Newton polyhedron $\NP$.
\begin{inparaenum}[(a)]
\item\label{item:AmbientFan} The set of vertex cones $\ON$ constitutes the ambient cone $\DR=\bigcup_{\CV\A\in\ON}\CV\A$ such that $\DR$ is a polyhedral fan as in Definition \ref{Def:Simplicity};
\item\label{item:SimplicialFan} The set of refined vertex cones $\ONs$ constitutes the ambient cone $\DR=\bigcup_{\CV\A\in\ONs}\CV\A$ such that $\DR$ is a simplicial fan as in Definition \ref{Def:Simplicity}.
\end{inparaenum}
\end{lemma}
\begin{proof}
\begin{asparaenum}[(a)]
{\parindent=\saveparindent
\item For $\forall\W\in\DR$, there exists a vertex $\A\in\face (\W)$ such that $\W\in\CV\A$ as per Lemma \ref{Lemma:IntersectLemma} \eqref{item:UbiquitousVertex}.
Hence $\DR\subseteq\bigcup_{\CV\A\in\ON}\CV\A$.

Now let $\CV\A,\CV\B\in\ON$ with $\CV\A\cap\CV\B\ne\emptyset$ and the two vertices $\A\ne\B$.
For $\forall\W\in\CV\A\cap\CV\B$, Lemma \ref{Lemma:IntersectLemma} \eqref{item:VertexIn} and then \eqref{item:Conjugacy} indicate that the facial cone $\CV\sF\subseteq\CV\A\cap\CV\B$ with $\sF$ denoting $\face (\W)$.
As per $\CV\sF\cap\PI\NP=\gv\sF$ in Lemma \ref{Lemma:IntersectLemma} \eqref{item:Minimality}, we have $\gv\sF\subseteq\gv\A\cap\gv\B$.
Now consider the normal vector set $\CG{\A\B}:=\bigcup_{\W\in\CV\A\cap\CV\B}\{\gv\sF\colon\sF=\face ({\W})\}\subseteq\gv\A\cap\gv\B$.
It follows from Lemma \ref{Lemma:IntersectLemma} \eqref{item:NoInteriInter} that $\CC (\CG{\A\B})\cap(\CV\A^\circ\cup\CV\B^\circ)=\emptyset$.
In fact, if $\W\in\CC (\CG{\A\B})\cap\CV\A^\circ$, then Lemma \ref{Lemma:IntersectLemma} \eqref{item:FaceIn} or \eqref{item:FacialAll} indicates that $\sF:=\face (\W)=\A$ and hence $\CV\A=\CV\sF\subseteq\CV\A\cap\CV\B$ as above.
Thus $\CV\A\subseteq\CV\B$ and by Lemma \ref{Lemma:IntersectLemma} \eqref{item:NoInteriInter}, we have $\A=\B$ which constitutes a contradiction.
Similarly $\CC (\CG{\A\B})\cap\CV\B^\circ=\emptyset$.
Hence $\CC (\CG{\A\B})$ is a common facet or face of both $\CV\A$ and $\CV\B$ since $\CG{\A\B}\subseteq\gv\A\cap\gv\B$.
Moreover, for $\forall\W\in\CV\A\cap\CV\B$, it readily follows from Lemma \ref{Lemma:IntersectLemma} \eqref{item:FacialAll} that $\W\in\CV\sF^\circ$ with $\sF=\face (\W)$.
Since $\gv\sF\subseteq\CG{\A\B}$, we have $\W\in\CC (\CG{\A\B})$.
Thus $(\CV\A\cap\CV\B)\subseteq\CC (\CG{\A\B})$ and hence $(\CV\A\cap\CV\B)=\CC (\CG{\A\B})$.
This shows that $\ON$ constitutes a polyhedral fan.

\item The proof is evident based on \eqref{item:AmbientFan}.
\qedhere}
\end{asparaenum}
\end{proof}

\begin{definition}\label{Def:AdjointVector}
{\upshape (Dual vector of $\FV\A$; exterior product $\CP\U 1{n-1}$; adjoint variables $\YV\A$; adjoint vector $\V_{j_0}^*$ and adjoint variable $y_{j_0}$ of $\FV\A$ with respect to $\TM$; adjoint facet; $\Mnorm\T$; $\mnorm\T$)}

For a facet $\FV\A$ of $\CV\A$, if $\gv\A\setminus\FV\A=\{\V_{j_0}\}$ with $1\le j_0\le n$, then the vector $\V_{j_0}$ is called the \emph{dual} vector of $\FV\A$.

The \emph{exterior product} of $n-1$ vectors $\{\LE\U 1{n-1}\}$ with $\U_l\in\DN n$ for $1\le l<n$ is a multilinear and antisymmetric map from $(\BB N^n)^{n-1}$ to $\BB N^n$ defined in almost the same fashion as the algebraic cofactor elements of $\U_n$ in the matrix $[\ME\U 1{n-1}~\U_n]$, i.e.,
\begin{equation*}
\CP\U 1{n-1}:=(-\det\MP 1,\dotsc,(-1)^n\det\MP n),
\end{equation*}
where the $n$ by $(n-1)$ matrix $\TP:=[\ME\U 1{n-1}]$ has column vectors $\{\LE\U 1{n-1}\}$ and the matrix $\MP l$ for $1\le l\le n$ denotes the $(n-1)$ by $(n-1)$ submatrix of $\TP$ obtained by deleting the $l$-th row of $\TP$.

Let $\TM=[\ME\V 1n]$ be an exponential matrix associated with a refined vertex cone $\CV\A$ as in Definition \ref{Def:ExponentialMatrix} such that $\det\TM=1$.
Let $\X=\CT\TM(\Y):=\Y^\TM$ be the monomial transformation associated with $\TM$.
To discriminate the variables among different vertex cones, we also write the variables $\Y$ as $\YV\A$ if $\TM$ is associated with a refined vertex cone $\CV\A$ and call them the \emph{adjoint} variables of the vertex cone $\CV\A$ with respect to $\TM$.
If we denote the adjoint matrix of $\TM$ as $\TM^*:=\lt [\V_1^*~\cdots~\V_n^*]$ such that $\langle\V_l^*,\V_l\rangle=\det\TM=1$ for $1\le l\le n$, then $\V_{j_0}^*$ is defined as the \emph{adjoint} vector of $\FV\A$ with respect to $\TM$.
Moreover, the variable $y_{j_0}$  satisfying $y_{j_0}=\X^{\V_{j_0}^*}$ with $\X\in\DFs n$ is referred to as the \emph{adjoint} variable of $\V_{j_0}^*$ or $\FV\A$ with respect to $\TM$.
Reciprocally $\FV\A$ is called the \emph{adjoint} facet of $y_{j_0}$.

For a set of variables $\T=(\LE t1{\ell})$, its norm $\Mnorm\T:=\max\{|t_j|\colon 1\le j\le\ell\}$.
Further, we have a notation $\mnorm\T:=\min\{|t_j|\colon 1\le j\le\ell\}$.
\end{definition}

\begin{lemma}\label{Lemma:ConjugateVariables}
Suppose that $\CV\A$ and $\CV\B$ are two refined vertex cones of a Newton polyhedron $\NP$ sharing a common facet $\CV\A\cap\CV\B=\FV\A=\FV\B$.
Suppose also that $\TM_\A:=[\ME\V 1n]$ and $\TM_\B:=[\ME{\ol\V}1n]$ are the exponential matrices associated with $\CV\A$ and $\CV\B$ satisfying $\det\TM_\A=\det\TM_\B=1$ and whose associated monomial transformations are $\X=\CT{\TM_\A}(\Y)=\Y^{\TM_\A}$ and $\X=\CT{\TM_\B}(\ol\Y)=\ol\Y^{\TM_\B}$ respectively.
Let $\V_{j_0}$ and $\ol\V_{k_0}$ be the dual vectors of $\FV\A$ and $\FV\B$ whose adjoint variables with respect to $\TM_\A$ and $\TM_\B$ are $y_{j_0}$ and $\ol y_{k_0}$ respectively.
Then we have $\V_{j_0}^*=-\ol\V_{k_0}^*$ and hence $y_{j_0}=\ol y_{k_0}^{-1}$.
\end{lemma}
\begin{proof}
Evidently $\ol\V_{k_0}=\sum_{l=1}^n\lambda_l\V_l$ such that $\lambda_{j_0}<0$ since $\V_{j_0}$ and $\ol\V_{k_0}$ are dual vectors of $\FV\A=\FV\B$ respectively.
Further, $1=\det\TM_\B=\langle\ol\V_{k_0}^*,\ol\V_{k_0}\rangle
=\lambda_{j_0}\langle\ol\V_{k_0}^*,\V_{j_0}\rangle$ since $\gfv\A=\gfv\B$.
Hence $\langle\ol\V_{k_0}^*,\V_{j_0}\rangle<0$.
Therefore from $\langle\V_{j_0}^*,\V_{j_0}\rangle=\det\TM_\A=1>0$ and $\gfv\A=\gfv\B$, it readily follows that $\V_{j_0}^*=-\ol\V_{k_0}^*$.
Thus $y_{j_0}=\X^{\V_{j_0}^*}=\X^{-\ol\V_{k_0}^*}=\ol y_{k_0}^{-1}$.
\end{proof}

Please note that the above requirement on $\det\TM_\A=\det\TM_\B=1$ does not incur a global orientation problem for the refined vertex cones in $\ONs$ since the orders of the column vectors in $\TM_\A$ and $\TM_\B$ are independent of each other.

\begin{definition}\label{Def:EulerChar}
{\upshape (Set of facets $\delta\ONs$ and interior facets $\FFs$; conjugate variables; interior variable; simplex; simplicial complex; section $\SO$ and its boundary $\partial\SO$ and face; interior complex $\SO^\circ$; number of $j$-dimensional faces $\Bt\Gamma j$; Euler characteristic $\chi(\Gamma)$; adjoint sector $\SV\A$; exterior condition)}

For a Newton polyhedron $\NP$, the set of all the facets of the refined vertex cones in $\ONs$ is denoted as $\delta\ONs$.
In particular, its subset $\FFs:=\{\cF\in\delta\ONs\colon\cF\setminus\partial\Rp^n\ne\emptyset\}$ is called the set of \emph{interior} facets of $\ONs$.

The two adjoint variables of a common facet of two refined vertex cones such as $y_{j_0}$ and $\ol y_{k_0}$ in Lemma \ref{Lemma:ConjugateVariables} are called a pair of \emph{conjugate} variables associated with the common interior facet with respect to $\TM_\A$ and $\TM_\B$.
For simplicity both of them are also called the \emph{interior} variables.

Recall that for $n>1$, an $(n-1)$-dimensional \emph{simplex} refers to an $(n-1)$-dimensional polyhedron with $n$ vertices.
An $(n-1)$-dimensional \emph{simplicial complex} is a finite union of simplexes that only intersect at common facets or faces.

The \emph{section} of $\ONs$ is defined as the simplicial complex $\SO:=\bigcup_{\CV\A\in\ONs}\Sec{\CV\A}$ as per Lemma \ref{Lemma:AllAreFans} \eqref{item:SimplicialFan} with the simplex $\Sec{\CV\A}$ denoting the section of a refined vertex cone $\CV\A$ as in Definition \ref{Def:Simplicity}.
For an $n$-dimensional Newton polyhedron $\NP$, the \emph{boundary} of $\SO$ is an $(n-2)$-dimensional simplicial complex denoted as $\partial\SO$ and comprising the $j$-dimensional simplexes of $\SO$ for $0\le j<n-1$ that are contained in the boundary of the ambient cone $\partial\Rp^n$ as in Definition \ref{Def:AmbientCone}.
Moreover, a face of $\Sec{\CV\A}$ is also referred to as a \emph{face} of $\SO$ for $\forall\CV\A\in\ONs$ henceforth.
The \emph{interior complex} $\SO^\circ$ is defined as the $(n-1)$-dimensional simplicial complex $\SO\setminus\partial\SO$.

For an $m$-dimensional simplicial complex $\Gamma$, the number of its $j$-dimensional faces is denoted as $b_j(\Gamma)$ for $0\le j\le m$.
Its Euler characteristic number is denoted as $\chi(\Gamma)=\sum_{j=0}^m(-1)^j\Bt\Gamma j$.

The sector $\SV\A:=\{\X\in\DFs n\colon\Mnorm{\YV\A}\le 1\}$ is called the \emph{adjoint} sector associated with a refined vertex cone $\CV\A$ with $\YV\A$ being the adjoint variables of $\CV\A$ as in Definition \ref{Def:AdjointVector}.

If $y$ is an interior variable, the condition $|y|>1$ is called the \emph{exterior} condition imposed on $y$.
\end{definition}

Evidently for an $n$-dimensional Newton polyhedron $\NP$, the number of interior facets $|\FFs|=\Bti{n-2}$, which is the number of $(n-2)$-dimensional facets of the interior complex $\SO^\circ$.
The number of interior variables equals $2|\FFs|$ due to the conjugation as per Lemma \ref{Lemma:ConjugateVariables}.

\begin{lemma}
Let $\DE 0$ denote the set $\{\X\in\DFs n\colon\Mnorm\X\le 1\}$.
For every Newton polyhedron $\NP$, the following identity holds.
\begin{equation}\label{GlobalCovering}
\DE 0=\bigcup_{\CV\A\in\ONs}\SV\A
\end{equation}
with $\SV\A$ being the adjoint sector associated with $\CV\A$ as in Definition \ref{Def:EulerChar} and $\ONs$ the set of refined vertex cones as in Definition \ref{Def:AmbientCone}.
\end{lemma}
\begin{proof}
For $\forall\CV\A\in\ONs$, the inclusion $\SV\A\subseteq\DE 0$ readily follows from the monomial transformation $\X=\YV\A^\TM$ with $\TM$ being the exponential matrix associated with $\CV\A$.
Conversely let us prove that $\DE 0\subseteq\bigcup_{\CV\A\in\ONs}\SV\A$.

Consider an interior facet $\FV\A\in\FFs$ with $\gfv\A=\{\LE\V 1{n-1}\}$ and $\X=\YV\A^\TM$ with $\TM=[\ME\V 1n]$ satisfying $\det\TM=1$.
The adjoint vector of $\FV\A$ with respect to $\TM$ equals
\begin{equation}\label{AdjointProd}
\V_n^*=(-1)^n\CP\V 1{n-1}
\end{equation}
and the corresponding adjoint variable $y_n$ satisfies $y_n=\X^{\V_n^*}$.
When $n\ge 3$, the adjoint vectors of $\FFs$ such as $\V_n^*$ in \eqref{AdjointProd} are not mutually independent.
In fact, consider an $(n-3)$-dimensional face of the interior complex $\SO^\circ$ that is the intersection of at least three $(n-2)$-dimensional facets of $\SO^\circ$ and corresponds to the intersection of cones as $\CC (\LE\V 1{n-2})=\bigcap_{k=1}^3\CC (\LE\V 1{n-2},\V_{n-1,k})$.
Here we abuse the notations in \eqref{AdjointProd} for $\V_n^*$ a bit and consider three adjoint vectors
\begin{equation}\label{AdjointVectors}
\V_{n,k}^*=(-1)^{\delta_k}(\CP\V 1{n-2})\wedge\V_{n-1,k}
\end{equation}
with $\delta_k\in\BB Z/2\BB Z$ whose value is determined by the adjoint vector $\V_{n,k}^*$ for $1\le k\le 3$.
Suppose that we have an expansion
\begin{equation}\label{1stExpansion}
\V_1=\sum_{k=1}^3\lambda_k\V_{n-1,k}+\sum_{l=2}^{n-2}\bar\lambda_l\V_l
\end{equation}
with $\lambda_k,\bar\lambda_l\in\BB Q^*$.
Then we have an identity as follows based on \eqref{1stExpansion}:
\begin{equation}\label{1stConstraint}
\begin{aligned}
\sum_{k=1}^3(-1)^{\delta_k}\lambda_k\V_{n,k}^*
&=\sum_{k=1}^3\lambda_k(\CP\V 1{n-2})\wedge\V_{n-1,k}\\
&=(\CP\V 1{n-2})\wedge\V_1=\bz.
\end{aligned}\end{equation}
Let $y_k$ be the adjoint variable corresponding to the adjoint vector $\V_{n,k}^*$ as $y_k=\X^{\V_{n,k}^*}$.
We choose between $y_k$ and its conjugate variable $\ol y_k$ such that $(-1)^{\delta_k}\lambda_k=|\lambda_k|>0$ in \eqref{1stConstraint} since $\delta_k$ is determined by $\V_{n,k}^*$ as in \eqref{AdjointVectors}.
In this way \eqref{1stConstraint} amounts to a constraint for the exterior condition imposed on the interior variables as follows.
\begin{equation}\label{VariableConstraint}
\prod_{k=1}^3z_k^{|\lambda_k|}=1
\end{equation}
with $z_k=y_k$ or $\ol y_k$ as above.
That is, the exterior condition as in Definition \ref{Def:EulerChar} can be imposed on at most two variables among $\{\Le z\}$ in \eqref{VariableConstraint}.
It is easy to see that each $(n-3)$-dimensional face of the interior complex $\SO^\circ$ corresponds to a constraint identity like \eqref{VariableConstraint} and altogether there are $\Bti{n-3}$ such constraints for the number of interior variables upon which the exterior condition can be imposed.

Nonetheless when $n\ge 4$, neither are these constraint identities like \eqref{VariableConstraint} mutually independent.
In fact, each $(n-4)$-dimensional face of the interior complex $\SO^\circ$ is the intersection of at least four $(n-3)$-dimensional faces of $\SO^\circ$.
More specifically, suppose that the intersection corresponds to the intersection of cones as $\CC (\LE\V 1{n-3})=\bigcap_{j=1}^4\CC (\LE\V 1{n-3},\V_{n-2,j})$.
In this case we have four linear expansions resembling \eqref{1stExpansion}:
\begin{equation*}
\V_1=\sum_{k=1}^3\lambda_{k,j}\V_{n-1,k,j}+\sum_{l=2}^{n-3}\bar\lambda_{l,j}\V_l
+\bar\lambda_{n-2,j}\V_{n-2,j}.
\end{equation*}
for $1\le j\le 4$.
Accordingly we have four constraint identities resembling \eqref{1stConstraint}:
\begin{equation}\label{Pre2ndIdentity}
\begin{aligned}
\sum_{k=1}^3(-1)^{\delta_{k,j}}\lambda_{k,j}\V_{n,k,j}^*&=\sum_{k=1}^3\lambda_{k,j}(\CP\V 1{n-3}\wedge\V_{n-2,j})\wedge\V_{n-1,k,j}\\
&=(\CP\V 1{n-3}\wedge\V_{n-2,j})\wedge\V_1=\bz
\end{aligned}
\end{equation}
for $1\le j\le 4$ with $\delta_{k,j}$ being defined in the same way as $\delta_k$ in \eqref{AdjointVectors}.
Consequently there are four constraint identities on the interior variables resembling \eqref{VariableConstraint}:
\begin{equation}\label{2ndVariableConstraint}
\prod_{k=1}^3z_{k,j}^{|\lambda_{k,j}|}=1
\end{equation}
for $1\le j\le 4$.
Suppose that we have an expansion $\V_1=\sum_{j=1}^4\mu_j\V_{n-2,j}+\sum_{l=2}^{n-3}\bar\mu_l\V_l$ with $\mu_j,\bar\mu_l\in\BB Q^*$, from which we can easily deduce a dependency identity among the four constraint identities in \eqref{Pre2ndIdentity} for $1\le j\le 4$, that is,
\begin{equation}\label{2ndIdentity}
\begin{aligned}
\sum_{j=1}^4\sum_{k=1}^3(-1)^{\delta_{k,j}}\mu_j\lambda_{k,j}\V_{n,k,j}^*&=\sum_{j=1}^4\mu_j(\CP\V 1{n-3}\wedge\V_{n-2,j})\wedge\V_1\\
&=(\CP\V 1{n-3})\wedge\V_1\wedge\V_1=\bz.
\end{aligned}
\end{equation}
Accordingly the dependency identity among the four constraint identities in \eqref{2ndVariableConstraint} is as follows.
\begin{equation}\label{3rdVariableConstraint}
\prod_{j=1}^4\Bigl(\prod_{k=1}^3z_{k,j}^{|\lambda_{k,j}|}\Bigr)^{\mu_j}=1.
\end{equation}
Please note that we do not require that $\mu_j$ be positive for $1\le j\le 4$ in \eqref{3rdVariableConstraint}.
Altogether there are $\Bti{n-4}$ such dependency identities as \eqref{3rdVariableConstraint} among those constraint identities like \eqref{2ndVariableConstraint}.
Furthermore, neither are the dependency identities like \eqref{3rdVariableConstraint} mutually independent and we can proceed inductively in this way until we reach the dependency identities corresponding to vertices, i.e., the $0$-dimensional faces of the interior complex $\SO^\circ$.
There are altogether $\Bti 0$ such dependency identities corresponding to the vertices of $\SO^\circ$.

In summary, for $\forall\X\in\DFs n$, the number of independent constraints for the number of exterior conditions that can be imposed on the interior variables equals
\begin{equation}\label{ConstraintNumb}
\crC^\circ:=\sum_{j=3}^n(-1)^{j-1}\Bti{n-j}.
\end{equation}

Now we can prove $\DE 0\subseteq\bigcup_{\CV\A\in\ONs}\SV\A$ by contradiction.
Suppose that $\X\in\DE 0$ but $\X\notin\SV\A$ for $\forall\CV\A\in\ONs$.
Then for $\forall\CV\A\in\ONs$, its adjoint variables $\YV\A=(\LE y1n)$ have at least one component $y_j$ satisfying the exterior condition $|y_j|>1$ with $1\le j\le n$.
Moreover, $y_j$ has to be an interior variable of an interior facet of $\CV\A$.
In fact, suppose that $\cF$ is not an interior facet such that $\cF\in\delta\ONs\setminus\FFs$.
If $\cF\subseteq\CC (\LE\E 1{n-1})$, then its adjoint vector $\V_\cF^*=\E_n$ and hence its adjoint variable $y_\cF=\X^{\V_\cF^*}=x_n$.
Thus $\X\in\DE 0$ indicates that $|y_\cF|=|x_n|\le 1$ and the same holds for every adjoint variable of the facets contained in $\delta\ONs\setminus\FFs$.
Hence altogether $\X\notin\SV\A$ for $\forall\CV\A\in\ONs$ imposes $\Bti{n-1}$ exterior conditions on the interior variables since there are $\Bts{n-1}$ vertex cones in $\ONs$.
There are $2\Bti{n-2}$ interior variables on which $\Bti{n-2}$ exterior conditions can be imposed due to the conjugation as per Lemma \ref{Lemma:ConjugateVariables}.
Taking into account the number of constraints in \eqref{ConstraintNumb}, we can impose at most $\Bti{n-2}-\crC^\circ$ exterior conditions on the interior variables.
Thus the number of exterior conditions imposed minus the number that can be imposed equals:
\begin{equation}\label{ConstraintsDiff}
\Bti{n-1}-\bigl(\Bti{n-2}-\crC^\circ\bigr)
=\sum_{j=1}^n(-1)^{j-1}\Bti{n-j}:=\crD^\circ.
\end{equation}

Finally, the boundary $\partial\SO$ is an $(n-2)$-dimensional simplicial complex as per Lemma \ref{Lemma:AllAreFans} \eqref{item:SimplicialFan}.
It is easy to calculate its Euler characteristic as:
\begin{equation}\label{BoundaryChar}
\chi(\partial\SO)=\sum_{j=0}^{n-2}(-1)^j\Btb j=1+(-1)^n.
\end{equation}
Thus based on \eqref{BoundaryChar} and \eqref{ConstraintsDiff} we have
\[
1=\chi(\SO)=\chi(\partial\SO)+\sum_{j=0}^{n-1}(-1)^j\Bti j=1+(-1)^n+(-1)^{n-1}\crD^\circ,
\]
from which we can deduce that $\crD^\circ=1>0$.
This constitutes a contradiction.
\end{proof}

\begin{corollary}
For a Newton polyhedron $\NP$, let $\varepsilon\in (0,1)$ be an adjustable parameter and $d_j:=\max\bigl\{\sum_{\V\in\gv\A}\langle\V,\E_j\rangle\colon\CV\A\in\ONs\bigr\}>0$ for $1\le j\le n$.
Let us define $\DE\varepsilon:=\{\X\in\DFs n\colon |x_j|\le \varepsilon^{d_j},1\le j\le n\}$ and $\SV\A^\varepsilon:=\SV\A\setminus\SV\A^-$ with the adjoint sector $\SV\A$ being as in Definition \ref{Def:EulerChar} and $\SV\A^-:=\{\X\in\DFs n\colon\mnorm{\YV\A}>\varepsilon\}$.
Then we have:
\begin{equation}\label{DominantPartition}
\DE\varepsilon\subseteq\bigcup_{\CV\A\in\ONs}\SV\A^\varepsilon.
\end{equation}
\end{corollary}
\begin{proof}
According to \eqref{GlobalCovering}, it suffices to show that $(\DE\varepsilon\cap\SV\A)\subseteq\SV\A^\varepsilon$ for $\forall\CV\A\in\ONs$, which amounts to showing that $(\SV\A^-\cap\SV\A)\subseteq\SV\A\setminus\DE\varepsilon$.
This readily follows from the monomial transformation $\X=\YV\A^\TM$ with $\TM$ being the exponential matrix associated with the refined vertex cone $\CV\A\in\ONs$.
\end{proof}

According to \eqref{DominantPartition}, it is easy to construct a finite partition of unity
\begin{equation}\label{Partition}
\sum_{\CV\A\in\ONs}\psi_{\CV\A}(\YV\A)=1
\end{equation}
in a neighborhood of $\oDE\varepsilon$.
The smooth functions $\psi_{\CV\A}:=\varphi_{\CV\A}\big/\bigl(\sum_{\CV\A\in\ONs}\varphi_{\CV\A}\bigr)$ and $\varphi_{\CV\A}$ is supported in a neighborhood of the compact set $\oSV\A^\varepsilon$.
Here $\oDE\varepsilon$ and $\oSV\A^\varepsilon$ denote the closures of the sets $\DE\varepsilon$ and $\SV\A^\varepsilon$ in \eqref{DominantPartition} respectively.

It is easy to see that every singular point has a dominant neighborhood as in Definition \ref{Def:Deficiency}.
As a result, if $\oDE\varepsilon$ is a dominant neighborhood of the origin $\bz$, then according to the partition of unity in \eqref{Partition}, it suffices to consider the intersection $\EI\TM\Yo\cap\oSV\A^\varepsilon$ of the exceptional branch $\EI\TM\Yo$ as in Definition \ref{Def:ExceptionalSupport} and compact set $\oSV\A^\varepsilon$ for $\forall\iI\Yo$ and $\forall\CV\A\in\ONs$.
Therefore we only need to consider branch points $\R=(\R_*,\bz)$ satisfying $\Mnorm{\R_*}\le 1$ due to the definition of the adjoint sector $\SV\A$ in Definition \ref{Def:EulerChar}.
Every such branch point $\R$ has a neighborhood $U(\R)$ whose image $\wt\Z(U(\R))$ under the map $\wt\Z(\wt\Y)$ defined as in Lemma \ref{Lemma:NonSingMap} \eqref{item:NonDegMap} constitutes a dominant neighborhood of each reduced branch point $\bs=(\bs_*,\bz)$ of $\R$.
The compactness of $\oSV\A^\varepsilon$ indicates that a finite number of such neighborhoods $U(\R)$ can cover all the exceptional branches $\EI\TM\Yo$, from which a finite partition of unity can be constructed.
In each of such dominant neighborhoods as $\wt\Z(U(\R))$ of a reduced branch point $\bs=(\bs_*,\bz)$, we repeat the finite partition of unity as in \eqref{Partition} based on the Newton polyhedron $\NP(w(\Xp))$ of the Weierstrass polynomial $w(\Xp)$ as above \eqref{2ndMonomialTransform}, which is followed by a further step of partial resolution of singularities based on the reduced monomial transformation $\CT{\TN'}$ in \eqref{2ndCanonicalMonomialTransform}.

In the case of irregular singularities with a latent variable being implemented, a partition of unity like \eqref{Partition} in the lower dimensional space without the latent variable can be constructed based on the pertinent Newton polyhedron.
After the revived localization of the latent variable, the partition of unity in the higher dimensional space involving the latent variable continues.
We proceed in this way along with the resolution algorithm such that a finite partition of unity can be constructed in a neighborhood of the origin $\bz$.

In the case of inconsistent singularities, the finite partition of unity has no essential difference from that in the consistency case.
More specifically, suppose that the primary variable $x_1$ is inconsistent with the exponential matrix $\TM$ and exceptional index set $\iI\Yo$ such that a permutation and relabeling of the row vectors of $\TM$ and their corresponding variables $\X$ lead to the reduced exponential matrix $\perm\TN$ in \eqref{IncstnMnmlTrnsfrm} instead of $\TN$ in \eqref{CanonicalMatrix}.
Nonetheless it is evident that the permutation and relabeling as above have no impact on the identities \eqref{LocalConversion} and \eqref{ExceptionalRelation} and thus the conclusions of Lemma \ref{Lemma:NonSingMap} still hold in this case.
As a result, the above discussion on the dominant neighborhood like $\wt\Z(U(\R))$ of a reduced branch point $(\bs_*,\bz)$ still applies here for the inconsistency case.

\section{An example on analytic surfaces}\label{Section:SurfaceExample}

With $\BB F=\BB C$, $n=3$ and $\X=(\Le x)$, let us consider a surface that is defined by an analytic function $f(\X)=0$ represented in a Weierstrass form like \eqref{PrimaryForm} as follows.
\begin{equation}\label{3-WeierstrassPolynomial}
f(\X)=\biggl[x_1^d+\sum_{j=2}^dc_j(x_2,x_3)x_1^{d-j}\biggr]
(c+r(\X)):=w(\X)(c+r(\X))
\end{equation}
with singularity height $d$ and primary variable $x_1$.
As per a refined vertex cone $\CV\A=\CC (\Le\V)$ of the Newton polyhedron $\NP(w(\X))$ of the Weierstrass polynomial $w(\X)$ in \eqref{3-WeierstrassPolynomial}, the monomial transformation $\X=\CT\TM(\Y)=\Y^\TM$ associated with the exponential matrix $\TM=[\Me\V]$ satisfying $\det\TM=1$ transforms its vertex form as in \eqref{VertexForm} into a total transform as in \eqref{TotalTransform} with dimension $n=3$.
A finite partition of unity like \eqref{Partition} can be constructed in a dominant neighborhood of the origin $\X=\bz$.
Let us assume that the exceptional index set $\iI\Yo$ is in canonical form of either $\{2,3\}$ or $\{3\}$ as in Definition \ref{Def:ProperTransform} hereafter.

In the case when the exceptional index set $\iI\Yo=\{3\}$ and the primary variable $x_1$ is consistent with the exponential matrix $\TM$ and $\iI\Yo$ as in Definition \ref{Def:Consistency-1}, the canonical reduction $\TN$ of $\TM$ with respect to $\iI\Yo$ as in Definition \ref{Def:ReductionMatrix} bears the following form:
\begin{equation}\label{Cstn3ExpMtrx}\begin{aligned}
    \TN:=
    \begin{array}{r @{} c @{\hs{-0.5pt}} c @{} p{4mm} @{} c
    @{\hs{1.5mm}} c @{\hs{1.5mm}} c @{\hs{1.5mm}} c
    @{\hs{0.8mm}} p{4mm} @{\hs{-2.5mm}} l}
    &           &           &           &           \E_1
    &           \E_2        &           &           \V_3
    &           &                                                   \\
    \mr{4}{$\rule{0mm}{9mm}$}           &           x_1
    &           \mr{4}{$\rule{0mm}{9mm}$}           &
    \raisebox{2.3mm}[0pt]{\mr{4}{$\left[\rule{0mm}{9mm}\right.$}}
    &           1           &           0           &
    \vline      &           b           &
    \raisebox{2.3mm}[0pt]{\mr{4}{$\left.\rule{0mm}{9mm}\right]$}}
    &           \mr{6}{$\rule{0mm}{24mm}$}                          \\
    &           x_2         &           &           &
    0           &           1           &           \vline
    &           a           &           &                           \\
    \cline{2-9}
    &           x_3         &           &           &
    0           &           0           &           \vline
    &           g           &           &                           \\[-2.5pt]
    &           &           &           &           z_1
    &           z_2         &           \vline      &
    z_3         &           &                                       \\
    \end{array}
\end{aligned}\end{equation}
with $g\in\BB N^*$.
Let $q_*(\wt z_1,\wt z_2)$ denote the localized proper transform as in Definition \ref{Def:Localization} of the Weierstrass polynomial $w(\X)$ in \eqref{3-WeierstrassPolynomial} under the monomial transformation $\CT\TN$ in accordance with the reduced exponential matrix $\TN$ in \eqref{Cstn3ExpMtrx}.
Same as in Lemma \ref{Lemma:SimpleDecrease}, the exponential identity \eqref{ExponIdentity} indicates that $\ord (q_*)<d$ at a regular reduced branch point $(s_1,s_2,0)$ with $d$ being the singularity height of $w(\X)$ in \eqref{3-WeierstrassPolynomial}.
Here the reduced branch point $(s_1,s_2,0)$ being regular means that the localized proper transform $q_*(\wt z_1,\wt z_2)$ satisfies $q_*(\wt z_1,0)\ne 0$ when $\wt z_2=0$ as in Definition \ref{Def:Localization}.

When the above reduced branch point $(s_1,s_2,0)$ is irregular as in Definition \ref{Def:Localization}, i.e., $q_*(\wt z_1,0)=0$, let us make the variable $z_1$ latent and localize the non-exceptional variable $z_2$ so as to study a new function $\vfz q1(\wt z_2,z_3):=q(z_1,\wt z_2+s_2,z_3)$ like in \eqref{PreVirtualization} with $q(\Z)$ being the partial transform as in \eqref{NormalTransform} with $\Z=(\Le z)$.
For clarity, let us define the active and exceptional variables as $(x'_2,x'_3):=(\wt z_2,z_3)$ respectively and write the function $\vfz q1(\wt z_2,z_3)$ into an apex form $\vfz f1(\Xp)$ as in \eqref{InitialForm} with the variables $\Xp:=(x'_2,x'_3)$ and their exponents $\Al':=(\alpha'_2,\alpha'_3)$.
Here the singularity height $d'=\ord (\vfz q1^*(\wt z_2))$ with $\vfz q1^*(\wt z_2):=\vfz q1(\wt z_2,0)$ and the apex $\D':=(d',0)$.
The function $\vfz f1(\Xp)$ is treated as a function in the algebra $\BB C[z_1]\{\Xp\}$.
After a latent gradation of $\vfz f1(\Xp)$ into a reducible and remainder function like in \eqref{PreWeierstrass} as following:
\begin{equation}\label{3-PreWeierstrass}
\vfz f1(\Xp)=z_1^mf(\Xp)+\vfz\phi 1(\Xp)
\end{equation}
with $m:=\deg (c_{\D'}(z_1))\in\BB N$ such that the coefficients of $\vfz\phi 1(\Xp)$ do not contain a term in the form $cz_1^m$ with $c\in\BB C^*$, we invoke Weierstrass preparation theorem and complete perfect power like in \eqref{MonoWeierstrass} so that the reducible function $f(\Xp)$ in \eqref{3-PreWeierstrass} is reduced into a Weierstrass form as follows:
\begin{equation}\label{3-MonoWeierstrass}
f(\Xp)=\biggl[{x'_2}^{d'}+\sum_{j=2}^{d'}c_j(x'_3){x'_2}^{d'-j}\biggr]
(c+r(\Xp))
\end{equation}
with $c_j(0)=0$ for $1<j\le d'$ and $c\in\BB C^*$.
The redundant function $r(\Xp)$ satisfies $r(\bz)=0$.
Let us abuse the notations a bit and assume that the terms of the function $\vfz f1(\Xp)$ in \eqref{3-PreWeierstrass} consist with the Weierstrass form in \eqref{3-MonoWeierstrass}.
Let us resolve the singularity of $\vfz f1(\Xp)$ in \eqref{3-PreWeierstrass} by the monomial transformation $\CT{\TM'}$ as in \eqref{VirtualMonomialTransform} whose canonical reduction $\CT{\TN'}$ or $\CT{\per\TN'}$ as in \eqref{VirtualResolution} bears the following form depending on the consistency of the primary variable $x'_2$ with $\TM'$ and the exceptional index set $\iI\Yop=\{3\}$:
\begin{equation}\label{Reduced2XptlMtrx}
\begin{aligned}
    \TN':=
    \begin{array}{r @{} c @{\hs{-0.5pt}} c @{} p{4mm} @{} c
    @{\hs{1.5mm}} c @{\hs{1.5mm}} c @{\hs{0.8mm}}
    p{4mm} @{\hs{-2.5mm}} l}
    &           &           &           &           \E_1
    &           &           &           &                           \\
    \mr{4}{$\rule{0mm}{9mm}$}           &           x'_2
    &           \mr{4}{$\rule{0mm}{9mm}$}           &
    \raisebox{4.7mm}[0pt]{\mr{4}{$\left[\rule{0mm}{5.5mm}\right.$}}
    &           1           &           \vline      &
    a'          &
    \raisebox{4.7mm}[0pt]{\mr{4}{$\left.\rule{0mm}{5.5mm}\right]$}}
    &           \mr{6}{$\rule{0mm}{24mm}$}                          \\
    \cline{2-8}
    &           x'_3        &           &           &
    0\rule{0mm}{4.5mm}      &           \vline      &
    g'          &           &                                       \\[-2pt]
    &           &           &           &           z'_2
    &           \vline      &           z'_3        &
    &                                                               \\
    \end{array};
    &\hs{10mm}
    \perm\TN':=
    \begin{array}{r @{} c @{\hs{-0.5pt}} c @{} p{4mm} @{} c
    @{\hs{1.5mm}} c @{\hs{1.5mm}} c @{\hs{0.8mm}}
    p{4mm} @{\hs{-2.5mm}} l}
    &           &           &           &           \E_1
    &           &           \E_2        &           &               \\
    \mr{4}{$\rule{0mm}{9mm}$}           &           x'_3
    &           \mr{4}{$\rule{0mm}{9mm}$}           &
    \raisebox{4.7mm}[0pt]{\mr{4}{$\left[\rule{0mm}{5.5mm}\right.$}}
    &           1           &           \vline      &
    0           &
    \raisebox{4.7mm}[0pt]{\mr{4}{$\left.\rule{0mm}{5.5mm}\right]$}}
    &           \mr{6}{$\rule{0mm}{24mm}$}                          \\
    \cline{2-8}
    &           x'_2        &           &           &
    0\rule{0mm}{4.5mm}      &           \vline      &
    1           &           &                                       \\[-2pt]
    &           &           &           &           z'_2
    &           \vline      &           z'_3        &
    &                                                               \\
    \end{array}
\end{aligned}
\end{equation}
with $g'\in\BB N^*$.
Like in \eqref{ParadigmRemainder}, let $\vfz q1(\Zp)$ with $\Zp:=(z'_2,z'_3)$ denote the partial transform of the apex form $\vfz f1(\Xp)$ in \eqref{3-PreWeierstrass} under $\CT{\TN'}$ or $\CT{\per\TN'}$ as in \eqref{Reduced2XptlMtrx}.
After the partial resolution associated with $\CT{\TN'}$ as in \eqref{Reduced2XptlMtrx} with $a'\in\BB N^*$, every reduced branch point is regular since there is only a single non-exceptional variable $z'_2$ in the new variables $\Zp$.
Hence Lemma \ref{Lemma:SimpleDecrease} can be applied to the partial transform of the reducible function $f(\Xp)$ in \eqref{3-MonoWeierstrass} such that the singularity height of the localized partial transform $\vfz{\wt q}1(\wt z'_2,z'_3)$ at a reduced branch point $(s'_2,0)$ with $s'_2\in\BB C^*$ strictly decreases from the prior singularity height $d'$ of the apex form $\vfz f1(\Xp)$ in \eqref{3-PreWeierstrass} like in Lemma \ref{Lemma:TwoCases}.
A partition of unity on the $(x'_2,x'_3)$-plane based on the Newton polygon $\NP(\vfz f1(\Xp))$ can be constructed in the way as in \eqref{Partition}.
When the element $a'$ of $\TN'$ satisfies $a'=0$ in \eqref{Reduced2XptlMtrx}, $\TN'$ can be reduced to the unit matrix and it is easy to see that the exceptional branch
\begin{equation}\label{3-ExceptionalBranch}
\EI{\TN'}\Zop:=\{(z'_2,0)\colon\vfz q1(z'_2,0)=0,z'_2\in\BB C^*\}
\end{equation}
like in \eqref{ExceptionalBranch} is irrelevant to the irregular reduced branch point $(s_1,s_2,0)$ as above \eqref{3-PreWeierstrass} due to the identity $\Zp=\Xp=(\wt z_2,z_3)$ and can be disregarded since the singular points of $\vfz f1(\Xp)$ are isolated on $\BB C^2$.
Nevertheless we can still study part of the singularities of the exceptional branch $\EI{\TN'}\Zop$ in \eqref{3-ExceptionalBranch} when $a'=0$ via a deficient contraction as in Definition \ref{Def:DeficientContraction}.
More specifically, we make an appropriate contraction of the neighborhood of $\Xp=\bz$ such that it has no intersection with the deficient variety $\RI{\TN'}\Zp:=\{(x'_2,0)\in\BB C^2\colon c+r(x'_2,0)=0\}$ with $c+r(\Xp)$ being the redundant function as in \eqref{3-MonoWeierstrass}.
Then the discussion on the singularity height associated with the reduced branch point $(s'_2,0)$ in the contracted neighborhood of $\Xp=\bz$ is the same as above.

In the case of $\CT{\per\TN'}$ as in \eqref{Reduced2XptlMtrx}, the singularity height at $(s'_2,0)$ might temporarily increase from the prior singularity height $d'$ of $\vfz f1(\Xp)$ in \eqref{3-PreWeierstrass}.
Suppose that the localized partial transform $\vfz{\wt q}1(\wt z'_2,z'_3)$ is written into an apex form $\vfz f1(x''_2,x''_3)$ with singularity height $d''$ and the variables $\Xq:=(x''_2,x''_3):=(\wt z'_2,z'_3)$ as follows.
\begin{equation}\label{3-VirtualApexForm}
\vfz f1(\Xq):=c_{\D''}(z_1){x''_2}^{d''}
+\sum_{\Al''\in\supp(\vfz f1)\setminus\D''}c_{\Al''}(z_1)\Xq^{\Al''}
\end{equation}
such that the singularity height $d''=\ord (\vfz f1(x''_2,0))$.
A latent gradation similar to \eqref{3-PreWeierstrass} can be implemented as follows.
\begin{equation}\label{3-VirtualGradation}
\vfz f1(\Xq):=z_1^{m'}f(\Xq)+\vfz\phi 1(\Xq)
\end{equation}
with $m':=\deg (c_{\D''}(z_1))$.
Nonetheless in this special case we reorganize the terms of $f(\Xq)$ in \eqref{3-VirtualGradation} as follows.
\begin{equation}\label{3-NoWeierstrass}
f(\Xq):={x''_2}^{d''}(c+r(x''_2))+\ol f(\Xq)
\end{equation}
such that $c\in\BB C^*$, $r(0)=0$ and $\ol f(x''_2,0)=0$ instead of an invocation of Weierstrass preparation theorem and completion of perfect power as in \eqref{3-MonoWeierstrass}.
The reduced exponential matrix $\TN''$ or $\perm\TN''$ based on the Newton polygon of the apex form $\vfz f1(\Xq)$ in \eqref{3-VirtualApexForm} bears the following canonical form similar to those in \eqref{Reduced2XptlMtrx}, depending on the consistency of the primary variable $x''_2$ with $\TN''$ and the exceptional index set $\iI\Zoq=\{3\}$ or $\perm\TN''$ and $\iI\Zoq$.
\begin{equation}\label{Reduced2XptlMtrx-2}
\begin{aligned}
    \TN'':=
    \begin{array}{r @{} c @{\hs{-0.5pt}} c @{} p{4mm} @{} c
    @{\hs{1.5mm}} c @{\hs{1.5mm}} c @{\hs{0.8mm}}
    p{4mm} @{\hs{-2.5mm}} l}
    &           &           &           &           \E_1
    &           &           &           &                           \\
    \mr{4}{$\rule{0mm}{9mm}$}           &           x''_2
    &           \mr{4}{$\rule{0mm}{9mm}$}           &
    \raisebox{4.7mm}[0pt]{\mr{4}{$\left[\rule{0mm}{5.5mm}\right.$}}
    &           1           &           \vline      &
    a''         &
    \raisebox{4.7mm}[0pt]{\mr{4}{$\left.\rule{0mm}{5.5mm}\right]$}}
    &           \mr{6}{$\rule{0mm}{24mm}$}                          \\
    \cline{2-8}
    &           x''_3       &           &           &
    0\rule{0mm}{4.5mm}      &           \vline      &
    g''         &           &                                       \\[-2pt]
    &           &           &           &           z''_2
    &           \vline      &           z''_3       &
    &                                                               \\
    \end{array};
    &\hs{10mm}
    \perm\TN'':=
    \begin{array}{r @{} c @{\hs{-0.5pt}} c @{} p{4mm} @{} c
    @{\hs{1.5mm}} c @{\hs{1.5mm}} c @{\hs{0.8mm}}
    p{4mm} @{\hs{-2.5mm}} l}
    &           &           &           &           \E_1
    &           &           \E_2        &           &               \\
    \mr{4}{$\rule{0mm}{9mm}$}           &           x''_3
    &           \mr{4}{$\rule{0mm}{9mm}$}           &
    \raisebox{4.7mm}[0pt]{\mr{4}{$\left[\rule{0mm}{5.5mm}\right.$}}
    &           1           &           \vline      &
    0           &
    \raisebox{4.7mm}[0pt]{\mr{4}{$\left.\rule{0mm}{5.5mm}\right]$}}
    &           \mr{6}{$\rule{0mm}{24mm}$}                          \\
    \cline{2-8}
    &           x''_2       &           &           &
    0\rule{0mm}{4.5mm}      &           \vline      &
    1           &           &                                       \\[-2pt]
    &           &           &           &           z''_2
    &           \vline      &           z''_3       &
    &                                                               \\
    \end{array}
\end{aligned}
\end{equation}
with $g''\in\BB N^*$ and $x''_3=z'_3=x'_2$ as per $\perm\TN'$ in \eqref{Reduced2XptlMtrx}.
When the element $a''$ of $\TN''$ satisfies $a''=0$ in \eqref{Reduced2XptlMtrx-2}, $\TN''$ can be reduced to the unit matrix and the singularity height of the localized proper transform $\vfz q1^*(\wt z''_2)$ becomes zero at a reduced branch point $(s''_2,0)$ with $s''_2\in\BB C^*$ after an appropriate deficient contraction based on the redundant function $c+r(x''_2)$ in \eqref{3-NoWeierstrass}.
This amounts to $\vfz q1^*(\bz)$ being a nonzero univariate residual polynomial in $z_1$ as in Lemma \ref{Lemma:StarFalling} whose residual order strictly decreases from the singularity height $d$ of the Weierstrass form $f(\X)$ in \eqref{3-WeierstrassPolynomial}.
When the element $a''$ of $\TN''$ satisfies $a''\in\BB N^*$ in \eqref{Reduced2XptlMtrx-2}, $x'_2=x''_3$ as above is the latent primary variable and we can write $\TN''$ into a trivial synthetic exponential matrix $\TQ''$ as follows.
\begin{equation}\label{2-SyntheticMatrix}
\begin{aligned}
    \TQ'':=
    \begin{array}{r @{} c @{\hs{-0.5pt}} c @{} p{4mm} @{} c
    @{\hs{1.5mm}} c @{\hs{1.5mm}} c @{\hs{0.8mm}}
    p{4mm} @{\hs{-2.5mm}} l}
    &           &           &           &           \E_1
    &           &           &           &                           \\
    \mr{4}{$\rule{0mm}{9mm}$}           &           x''_2
    &           \mr{4}{$\rule{0mm}{9mm}$}           &
    \raisebox{4.7mm}[0pt]{\mr{4}{$\left[\rule{0mm}{5.5mm}\right.$}}
    &           1           &           \vline      &
    a''         &
    \raisebox{4.7mm}[0pt]{\mr{4}{$\left.\rule{0mm}{5.5mm}\right]$}}
    &           \mr{6}{$\rule{0mm}{24mm}$}                          \\
    \cline{2-8}
    &           x'_2        &           &           &
    0\rule{0mm}{4.5mm}      &           \vline      &
    g''         &           &                                       \\[-2pt]
    &           &           &           &           z''_2
    &           \vline      &           z''_3       &
    &                                                               \\
    \end{array};
    &\hs{10mm}
    \perm\TQ'':=
    \begin{array}{r @{} c @{\hs{-0.5pt}} c @{} p{4mm} @{} c
    @{\hs{1.5mm}} c @{\hs{1.5mm}} c @{\hs{0.8mm}}
    p{4mm} @{\hs{-2.5mm}} l}
    &           &           &           &           \E_1
    &           &           \E_2        &           &               \\
    \mr{4}{$\rule{0mm}{9mm}$}           &           x'_2
    &           \mr{4}{$\rule{0mm}{9mm}$}           &
    \raisebox{4.7mm}[0pt]{\mr{4}{$\left[\rule{0mm}{5.5mm}\right.$}}
    &           1           &           \vline      &
    0           &
    \raisebox{4.7mm}[0pt]{\mr{4}{$\left.\rule{0mm}{5.5mm}\right]$}}
    &           \mr{6}{$\rule{0mm}{24mm}$}                          \\
    \cline{2-8}
    &           x''_2       &           &           &
    0\rule{0mm}{4.5mm}      &           \vline      &
    1           &           &                                       \\[-2pt]
    &           &           &           &           z''_2
    &           \vline      &           z''_3       &
    &                                                               \\
    \end{array},
\end{aligned}
\end{equation}
which is similar to \eqref{CstnSyntheticMatrix}.
In this case we have the following identity on the exceptional support $\es{\TQ''}\Zoq$:
\begin{equation}\label{Identity2Dim}
a''\alpha''_2+g''\alpha'_2=p
\end{equation}
with $(\alpha''_2,\alpha'_2)$ being the exponents of the variables $(x''_2,x'_2)$ and the constant $p\in\BB N^*$.
The above identity \eqref{Identity2Dim} synchronizes the changes of the exponents $\alpha''_2$ and $\alpha'_2$ and shows that the latent primary variable $x'_2$ of $\TQ''$ in \eqref{2-SyntheticMatrix} is revived as $z''_2$ essentially as in Definition \ref{Def:Revival}.
Hence the singularity height of the localized partial transform $\vfz{\wt q}1(\wt z''_2,z''_3)$ strictly decreases from the prior singularity height $d'$ of the apex form $\vfz f1(\Xp)$ in \eqref{3-PreWeierstrass}.

In the inconsistency case when the reduced exponential matrix $\perm\TN''$ bears the form in \eqref{Reduced2XptlMtrx-2}, the synthetic exponential matrix $\perm\TQ''$ is as in \eqref{2-SyntheticMatrix} with the latent primary variable $x'_2$ being revived as $z''_2$ as well.

After the singularity height $d'$ of the apex form $\vfz f1(\Xp)$ in \eqref{3-PreWeierstrass} is reduced to $d'=1$, in the case of the reduced exponential matrix $\TN'$ in \eqref{Reduced2XptlMtrx} with $g'\in\BB N^*$, similar to the conclusions in Lemma \ref{Lemma:VirtualRecovery} and Lemma \ref{Lemma:StarFalling}, the apex $(1,0)$ is always on the exceptional support such that the residual order of the residual polynomial in $z_1$ as in Definition \ref{Def:ResidualOrder} is strictly less than the prior singularity height $d$ of the Weierstrass polynomial $w(\X)$ in \eqref{3-WeierstrassPolynomial}.
When the reduced exponential matrix bears the form of $\perm\TN'$ in \eqref{Reduced2XptlMtrx} in the case of inconsistency, the discussions on \eqref{3-VirtualApexForm} thenceforth can be repeated verbatim here to show that the residual order in $\wt z_1$ is strictly less than the prior singularity height $d$ of $w(\X)$ in \eqref{3-WeierstrassPolynomial} as well.

In the case when the exceptional index set $\iI\Yo=\{2,3\}$ rather than $\{3\}$ as above \eqref{Cstn3ExpMtrx}, let us study the exceptional branch $\EI\TM\Yo:=\{(y_1,0,0)\colon p(y_1,0,0)=0,y_1\in\BB C^*\}$ with $p(y_1,0,0)$ being the proper transform of $w(\X)$ in \eqref{3-WeierstrassPolynomial} under $\CT\TM$.
The finite partition of unity in \eqref{Partition} indicates that it suffices to study its compact subset $\EI\TM\Yo\cap\{(y_1,0,0)\colon |y_1|\le 1\}$.
That is, it suffices to consider the branch points in the form $(r_1,0,0)$ with $r_1\in\BB C^*$ satisfying $|r_1|\le 1$.

When the primary variable $x_1$ is consistent with the exponential matrix $\TM$ and exceptional index set $\iI\Yo=\{2,3\}$, the reduced exponential matrix in \eqref{Cstn3ExpMtrx} now bears the form:
\begin{equation}\label{Cstn3ExpMtrx-2}
\begin{aligned}
    \TN:=
    \begin{array}{r @{} c @{\hs{-0.5pt}} c @{} p{4mm} @{} c
    @{\hs{1.5mm}} c @{\hs{1.5mm}} c @{\hs{2.5mm}} c
    @{\hs{0.8mm}} p{4mm} @{\hs{-2.5mm}} l}
    &           &           &           &           \E_1
    &           &           \V_2        &           \V_3
    &           &                                                   \\
    \mr{4}{$\rule{0mm}{9mm}$}           &           x_1
    &           \mr{4}{$\rule{0mm}{9mm}$}           &
    \raisebox{2.9mm}[0pt]{\mr{4}{$\left[\rule{0mm}{9mm}\right.$}}
    &           1           &           \vline      &
    \mc{2}{c}{\La\TD}       &
    \raisebox{2.9mm}[0pt]{\mr{4}{$\left.\rule{0mm}{9mm}\right]$}}
    &           \mr{6}{$\rule{0mm}{24mm}$}                          \\
    \cline{2-9}
    &           x_2         &           &           &
    0           &           \vline      &
    \mc{2}{c}{\mr{2}{$\TD$}}            &           &               \\
    &           x_3         &           &           &
    0           &           \vline      &           &
    &           &                                                   \\[-2pt]
    &           &           &           &           z_1
    &           \vline      &           z_2         &
    z_3         &           &                                       \\
    \end{array}
\end{aligned}
\end{equation}
with $\det\TD\ne 0$ and the row vector $\La\in\BB Q^2$.
Its associated monomial transformation $\X=\Z^\TN$ transforms $w(\X)$ in \eqref{3-WeierstrassPolynomial} into the total transform in \eqref{NormalTransform}.
The reduction matrix $\TF$ in \eqref{ReductionMatrix} now equals $\det\TM=1$.
Hence a reduced branch point like $(s_1,0,0)$ with $s_1\in\BB C^*$ satisfies $s_1^{\det\TD}=r_1$ as in Definition \ref{Def:Localization}.
A dominant neighborhood of the branch point $(r_1,0,0)$ is mapped onto that of one of its reduced branch points like $(s_1,0,0)$ via the map $\wt\Z(\wt\Y)$ as in Lemma \ref{Lemma:NonSingMap} \eqref{item:NonDegMap}.
The proper transform $q(z_1,0,0)$ of $\CT\TN$ in \eqref{ProperTransform-N} is a univariate polynomial in $z_1$ since $\iI\Zo=\iI\Yo=\{2,3\}$, which ensures that every branch point is regular as in Definition \ref{Def:Localization}.
The exponential identity \eqref{ExponIdentity} and Weierstrass polynomial $w(\X)$ in \eqref{3-WeierstrassPolynomial} indicate that the localized proper transform $q_*(\wt z_1)$ at $(s_1,0,0)$ satisfies $\ord (q_*)<d$ like in Lemma \ref{Lemma:SimpleDecrease}, which is a strict decrease from the prior singularity height $d$ of $w(\X)$ in \eqref{3-WeierstrassPolynomial}.

When the primary variable $x_1$ is inconsistent with the exponential matrix $\TM$ and exceptional index set $\iI\Yo=\{2,3\}$, the canonical reduction $\perm\TN$ of $\TM$ with respect to $\iI\Yo$ bears the same form as in \eqref{IncstnMnmlTrnsfrm}:
\begin{equation}\label{Incstn3Matrix}\begin{aligned}
    \perm\TN:=
    \begin{array}{r @{} c @{\hs{-0.5pt}} c @{} p{4mm} @{} c
    @{\hs{1.5mm}} c @{\hs{1.5mm}} c @{\hs{2.5mm}} c
    @{\hs{0.8mm}} p{4mm} @{\hs{-2.5mm}} l}
    &           &           &           &           \E_1
    &           &           \V_2        &           \V_3
    &           &                                                   \\
    \mr{4}{$\rule{0mm}{9mm}$}           &           x_2
    &           \mr{4}{$\rule{0mm}{9mm}$}           &
    \raisebox{2.5mm}[0pt]{\mr{4}{$\left[\rule{0mm}{9mm}\right.$}}
    &           1           &           \vline      &
    \lambda c   &           \lambda g   &
    \raisebox{2.5mm}[0pt]{\mr{4}{$\left.\rule{0mm}{9mm}\right]$}}
    &           \mr{6}{$\rule{0mm}{24mm}$}                          \\
    \cline{2-9}
    &           x_3         &           &           &
    0           &           \vline      &           c
    &           g           &           &                           \\
    &           x_1         &           &           &
    0           &           \vline      &           a
    &           b           &           &                           \\[-2.5pt]
    &           &           &           &           z_1
    &           \vline      &           z_2         &
    z_3         &           &                                       \\
    \end{array}
\end{aligned}\end{equation}
such that $\lambda\in\BB Q_{\ge 0}$ and its exceptional submatrix $\Evv{\perm\TN}\Zo=\bigl[\begin{smallmatrix}c&g\\ a&b\end{smallmatrix}\bigr]:=\TG$ satisfies $\det\TG\ne 0$, which amounts to the condition in \eqref{PrimaryVector}.
Hence we assume that $g\in\BB N^*$ up to a permutation and relabeling of the exceptional variables $\Z_0=(z_2,z_3)$ and their corresponding column vectors $\V_2$ and $\V_3$ in \eqref{Incstn3Matrix}.
The reduced exponential matrix $\perm\TN$ in \eqref{Incstn3Matrix} yields the following exponential identity resembling \eqref{LatentIdentity}:
\begin{equation}\label{Latent3Identity}
\Ga_0:=(\gamma_2,\gamma_3)=\alpha_1(a,b)+(\alpha_2,\alpha_3)\cdot
\begin{bmatrix}\lambda\\ 1\end{bmatrix}\cdot (c,g),
\end{equation}
where $\Ga_0$ denotes the exponents of the exceptional variables $\Z_0=(z_2,z_3)$ and $(\Le\alpha)$ those of the variables $(\Le x)$.
The other notations are the same as those of $\perm\TN$ in \eqref{Incstn3Matrix}.

From \eqref{Latent3Identity} as well as $\det\TG=\det\bigl[\begin{smallmatrix}c&g\\ a&b\end{smallmatrix}\bigr]\ne 0$, it is easy to deduce that the conclusion of Lemma \ref{Lemma:ConstantPrimaryExp} still holds, that is, for $\forall\pAl=(\alpha_2,\alpha_3,\alpha_1)\in\es{\perm\TN}\Zo$, the latent primary exponent $\alpha_1$ in \eqref{Latent3Identity} is a constant.
Here the definition of the exceptional support $\es{\perm\TN}\Zo$ is as in Definition \ref{Def:ExceptionalSupport}.
Hence we can implement a latent gradation as
\[
q(\Z)=q_0(\Z)+q_1(\Z)
\]
according to the latent primary component $\wt\alpha_1$ as in Definition \ref{Def:LatentPrimaryCompnt} being zero or not like in \eqref{PreReorganize}.
Here $q(\Z)$ denotes the partial transform of the Weierstrass polynomial $w(\X)$ in \eqref{3-WeierstrassPolynomial} under the reduced monomial transformation $\CT{\per\TN}$.
Since there is only one non-exceptional variable $z_1$ here, there arises no irregular singularities after the localization of $q(\Z)$ and the linear modification $\CL\Xp$ is also trivial as $\X'=\wt\Z$.
After the latent preliminary and Weierstrass reductions as in Definition \ref{Def:LatentReductions}, we obtain an apex form:
\begin{equation}\label{Reorganized3Func}
f(\Xp)=f_0(\Xp)+\phi(\Xp)
\end{equation}
like in \eqref{ReorganizedFunc} with the latent reducible function $f_0(\Xp)$ being in Weierstrass form as in \eqref{LatentWeierstrassForm} with singularity height $d'$.
In particular, all the latent primary components $\wt\alpha_1$ of $f_0(\Xp)$ are zero whereas none of those of the latent remainder function $\phi(\Xp)$ zero.

After the above resolution step based on the canonical reduction $\perm\TN$ in \eqref{Incstn3Matrix}, the exceptional index set $\iI\Zop$ is either $\{2,3\}$ or $\{3\}$ in the next resolution step ensued.
In the case when $\iI\Zop=\{2,3\}$, the reduced exponential matrix bears the following form similar to $\TN$ in \eqref{Cstn3ExpMtrx-2} or $\perm\TN$ in \eqref{Incstn3Matrix} depending on the consistency of the primary variable $x'_1$.
\begin{equation}\label{Reduced3XptlMtrx}
\begin{aligned}
    \TN':=
    \begin{array}{r @{} c @{\hs{-0.5pt}} c @{} p{4mm} @{} c
    @{\hs{1.5mm}} c @{\hs{1.5mm}} c @{\hs{2.5mm}} c
    @{\hs{0.8mm}} p{4mm} @{\hs{-2.5mm}} l}
    &           &           &           &           \E_1
    &           &           \V'_2       &           \V'_3
    &           &                                                   \\
    \mr{4}{$\rule{0mm}{9mm}$}           &           x'_1
    &           \mr{4}{$\rule{0mm}{9mm}$}           &
    \raisebox{2.9mm}[0pt]{\mr{4}{$\left[\rule{0mm}{9mm}\right.$}}
    &           1           &           \vline      &
    \mc{2}{c}{\La'\TD'}     &
    \raisebox{2.9mm}[0pt]{\mr{4}{$\left.\rule{0mm}{9mm}\right]$}}
    &           \mr{6}{$\rule{0mm}{24mm}$}                          \\
    \cline{2-9}
    &           x'_2        &           &           &
    0           &           \vline      &
    \mc{2}{c}{\mr{2}{$\TD'$}}           &           &               \\
    &           x'_3        &           &           &
    0           &           \vline      &           &
    &           &                                                   \\[-2pt]
    &           &           &           &           z'_1
    &           \vline      &           z'_2        &
    z'_3        &           &                                       \\
    \end{array};
    &\hs{10mm}
    \perm\TN':=
    \begin{array}{r @{} c @{\hs{-0.5pt}} c @{} p{4mm} @{} c
    @{\hs{1.5mm}} c @{\hs{1.5mm}} c @{\hs{1.5mm}} c
    @{\hs{0.8mm}} p{4mm} @{\hs{-2.5mm}} l}
    &           &           &           &           \E_1
    &           &           \V'_2       &           \V'_3
    &           &                                                   \\
    \mr{4}{$\rule{0mm}{9mm}$}           &           x'_2
    &           \mr{4}{$\rule{0mm}{9mm}$}           &
    \raisebox{2.9mm}[0pt]{\mr{4}{$\left[\rule{0mm}{9mm}\right.$}}
    &           1           &           \vline      &
    \lambda'c'  &           \lambda'g'  &
    \raisebox{2.9mm}[0pt]{\mr{4}{$\left.\rule{0mm}{9mm}\right]$}}
    &           \mr{6}{$\rule{0mm}{24mm}$}                          \\
    \cline{2-9}
    &           x'_3        &           &           &
    0           &           \vline      &           c'
    &           g'          &           &                           \\
    &           x'_1        &           &           &
    0           &           \vline      &           a'
    &           b'          &           &                           \\[-2pt]
    &           &           &           &           z'_1
    &           \vline      &           z'_2        &
    z'_3        &           &                                       \\
    \end{array},
\end{aligned}
\end{equation}
where we have non-degenerate exceptional submatrices $\det\TD'\cdot\det\TG'\ne 0$ with $\TG':=\bigl[\begin{smallmatrix}c'&g'\\ a'&b'\end{smallmatrix}\bigr]$.
Let us assume that $g'\in\BB N^*$ in $\perm\TN'$ up to a permutation of the exceptional variables $\Z'_0=(z'_2,z'_3)$ and their corresponding column vectors $\V'_2$ and $\V'_3$.
Moreover, $\La'$ is a $2$-dimensional row vector in $\BB Q^2$ and $\lambda'\in\BB Q_{\ge 0}$.

From the perspective of synthetic monomial transformations, the reduced exponential matrix $\perm\TN$ in \eqref{Incstn3Matrix} can be decomposed as $\perm\TN=\perm\TS\cdot\perm\TT$ like in \eqref{MatrixDecomposition} such that
\begin{equation}\label{Interim3Matrix}\begin{aligned}
    \perm\TS:=
    \begin{array}{r @{} c @{\hs{-0.5pt}} c @{} p{4mm} @{} c
    @{\hs{1.5mm}} c @{\hs{1.5mm}} c @{\hs{2.5mm}} c
    @{\hs{0.8mm}} p{4mm} @{\hs{-2.5mm}} l}
    &           &           &           &           \E_1
    &           &           \U          &           \E_3
    &           &                                                   \\
    \mr{4}{$\rule{0mm}{9mm}$}           &           x_2
    &           \mr{4}{$\rule{0mm}{9mm}$}           &
    \raisebox{2.9mm}[0pt]{\mr{4}{$\left[\rule{0mm}{9mm}\right.$}}
    &           1           &           \vline      &
    \lambda g   &           0           &
    \raisebox{2.9mm}[0pt]{\mr{4}{$\left.\rule{0mm}{9mm}\right]$}}
    &           \mr{6}{$\rule{0mm}{24mm}$}                          \\
    \cline{2-9}
    &           x_3         &           &           &
    0           &           \vline      &           g
    &           0           &           &                           \\
    &           x_1         &           &           &
    0           &           \vline      &           0
    &           1           &           &                           \\[-2pt]
    &           &           &           &           z_1
    &           \vline      &           \zeta       &
    x_1         &           &                                       \\
    \end{array};
    &\hs{10mm}
    \perm\TT:=
    \begin{array}{r @{} c @{\hs{-0.5pt}} c @{} p{4mm} @{} c
    @{\hs{1.5mm}} c @{\hs{1.5mm}} c @{\hs{2.5mm}} c
    @{\hs{0.8mm}} p{4mm} @{\hs{-2.5mm}} l}
    &           &           &           &           \E_1
    &           &           &           &           &               \\
    \mr{4}{$\rule{0mm}{9mm}$}           &           z_1
    &           \mr{4}{$\rule{0mm}{9mm}$}           &
    \raisebox{2.9mm}[0pt]{\mr{4}{$\left[\rule{0mm}{9mm}\right.$}}
    &           1           &           \vline      &
    0           &           0           &
    \raisebox{2.9mm}[0pt]{\mr{4}{$\left.\rule{0mm}{9mm}\right]$}}
    &           \mr{6}{$\rule{0mm}{24mm}$}                          \\
    \cline{2-9}
    &           \zeta       &           &           &
    0           &           \vline      &           \frac cg
    &           1           &           &                           \\
    &           x_1         &           &           &
    0           &           \vline      &           a
    &           b           &           &                           \\[-2pt]
    &           &           &           &           z_1
    &           \vline      &           z_2         &
    z_3         &           &                                       \\
    \end{array}
\end{aligned}\end{equation}
with $\zeta$ being the interim variable.
The decomposition adheres to the same principles as those for the one in \eqref{MatrixDecomposition} including the invariance of the latent primary exponent $\alpha_1$ under $\perm\TS$ and coincidence of the non-exceptional column submatrices of $\perm\TS$ and $\perm\TN$ in \eqref{Incstn3Matrix}, as well as the requirement that the difference between $\alpha_\zeta$ and $\gamma_3$ be an integral multiple of $\alpha_1$, which are the respective exponents of the variables $\zeta$, $z_3$ and $x_1$ in \eqref{Interim3Matrix}.
Let $p(z_1,\zeta,x_1)$ be the total transform of the Weierstrass polynomial $w(\X)$ in \eqref{3-WeierstrassPolynomial} under $\CT{\per\TS}$ like in \eqref{SyntheticTransform}.
A latent gradation of $p(z_1,\zeta,x_1)$ by the latent primary variable $x_1$, which is ensued by a localization of $z_1$, yields an apex form:
\begin{equation}\label{3-InterimReorgFunc}
f(x'_1,\zeta,x_1)=x_1^{a_1}\zeta^{\A\cdot\U}f_0(x'_1,\zeta)+\phi(x'_1,\zeta,x_1)
\end{equation}
that is similar to \eqref{InterimReorgFunc}.
Here we have $x'_1=\wt z_1$ and $\A=(\Le a)$ is the vertex associated with the exponential matrix $\TM$ that was reduced to $\perm\TN$ in \eqref{Incstn3Matrix}.
The vector $\U$ is the second column vector of $\perm\TS$ in \eqref{Interim3Matrix}.
Weierstrass preparation theorem and a completion of perfect power can be invoked on $f_0(x'_1,\zeta)$ in \eqref{3-InterimReorgFunc} such that it bears a Weierstrass form with singularity height $d'$, similar to $f(x'_2,x'_3)$ in \eqref{3-MonoWeierstrass}.

Based on $\perm\TT$ in \eqref{Interim3Matrix}, a new interim exponential matrix $\perm\wh\TT$ like in \eqref{ModifiedInterimMatrix} can be defined as follows.
\begin{equation}\label{3-NewInterimMtrx}
    \perm\wh\TT:=
    \begin{array}{r @{} c @{\hs{-0.5pt}} c @{} p{4mm} @{} c
    @{\hs{1.5mm}} c @{\hs{1.5mm}} c @{\hs{2.5mm}} c
    @{\hs{0.8mm}} p{4mm} @{\hs{-2.5mm}} l}
    &           &           &           &           \E_1
    &           &           &           &           &               \\
    \mr{4}{$\rule{0mm}{9mm}$}           &           x'_1
    &           \mr{4}{$\rule{0mm}{9mm}$}           &
    \raisebox{2.9mm}[0pt]{\mr{4}{$\left[\rule{0mm}{9mm}\right.$}}
    &           1           &           \vline      &
    0           &           0           &
    \raisebox{2.9mm}[0pt]{\mr{4}{$\left.\rule{0mm}{9mm}\right]$}}
    &           \mr{6}{$\rule{0mm}{24mm}$}                          \\
    \cline{2-9}
    &           \zeta       &           &           &
    0           &           \vline      &           \frac cg
    &           1           &           &                           \\
    &           x_1         &           &           &
    0           &           \vline      &           a
    &           b           &           &                           \\[-2pt]
    &           &           &           &           x'_1
    &           \vline      &           x'_2        &
    x'_3        &           &                                       \\
    \end{array}
\end{equation}
such that the commutative diagram in \eqref{InterimCommut} holds with $\XIp=(x'_1,\zeta)$ here.
As in \eqref{Incstn3Matrix} we have $c\in\BB N$ and $g\in\BB N^*$.
The new interim exponential matrix $\perm\wh\TT$ in \eqref{3-NewInterimMtrx} and reduced exponential matrix $\TN'$ or $\perm\TN'$ in \eqref{Reduced3XptlMtrx} can be synthesized into a synthetic exponential matrix $\TQ'$ or $\perm\TQ'$ similar to \eqref{CstnSyntheticMatrix} or \eqref{IncstnSyntheticMatrix} as follows.
\begin{equation}\label{Syn3XptlMtrx}
\begin{aligned}
    \TQ':=
    \begin{array}{r @{} c @{\hs{-0.5pt}} c @{} p{4mm} @{} c
    @{\hs{1.5mm}} c @{\hs{1.5mm}} c @{\hs{2.5mm}} c
    @{\hs{0.8mm}} p{4mm} @{\hs{-2.5mm}} l}
    &           &           &           &           \E_1
    &           &           &           &           &               \\
    \mr{4}{$\rule{0mm}{9mm}$}           &           x'_1
    &           \mr{4}{$\rule{0mm}{9mm}$}           &
    \raisebox{2.9mm}[0pt]{\mr{4}{$\left[\rule{0mm}{9mm}\right.$}}
    &           1           &           \vline      &
    \mc{2}{c}{\La'\hs{-0.8mm}\TD'}      &
    \raisebox{2.9mm}[0pt]{\mr{4}{$\left.\rule{0mm}{9mm}\right]$}}
    &           \mr{6}{$\rule{0mm}{24mm}$}                          \\
    \cline{2-9}
    &           \zeta       &           &           &
    0           &           \vline      &
    \mc{2}{c}{\mr{2}{$\Th\hs{-0.4mm}\TD'$}}         &
    &                                                               \\
    &           x_1         &           &           &
    0           &           \vline      &           &
    &           &                                                   \\[-2pt]
    &           &           &           &           z'_1
    &           \vline      &           z'_2        &
    z'_3        &           &                                       \\
    \end{array};
    &\hs{10mm}
    \perm\TQ':=
    \begin{array}{r @{} c @{\hs{-0.5pt}} c @{} p{4mm} @{} c
    @{\hs{1.5mm}} c @{\hs{1.5mm}} c @{\hs{2.5mm}} c
    @{\hs{0.8mm}} p{4mm} @{\hs{-2.5mm}} l}
    &           &           &           &           &
    &           &           &           &                           \\
    \mr{4}{$\rule{0mm}{9mm}$}           &           \zeta
    &           \mr{4}{$\rule{0mm}{9mm}$}           &
    \raisebox{2.9mm}[0pt]{\mr{4}{$\left[\rule{0mm}{9mm}\right.$}}
    &           \frac cg    &           \vline      &
    \mu c'      &           \mu g'      &
    \raisebox{2.9mm}[0pt]{\mr{4}{$\left.\rule{0mm}{9mm}\right]$}}
    &           \mr{6}{$\rule{0mm}{24mm}$}                          \\
    \cline{2-9}
    &           x_1         &           &           &
    a           &           \vline      &           \nu c'
    &           \nu g'      &           &                           \\
    &           x'_1        &           &           &
    0           &           \vline      &           a'
    &           b'          &           &                           \\[-2pt]
    &           &           &           &           z'_1
    &           \vline      &           z'_2        &
    z'_3        &           &                                       \\
    \end{array},
\end{aligned}
\end{equation}
where the submatrix $\TD'$ is as in $\TN'$ in \eqref{Reduced3XptlMtrx} and $\Th$ denotes the submatrix $\bigl[\begin{smallmatrix}c/g&1\\ a&b\end{smallmatrix}\bigr]$ of $\perm\wh\TT$ in \eqref{3-NewInterimMtrx}.
The vector $(\mu,\nu)$ in $\perm\TQ'$ is defined as:
\begin{equation}\label{CoeffAnalysis}
\begin{bmatrix}\mu\\ \nu\end{bmatrix}
=\Th\cdot\begin{bmatrix}\lambda'\\ 1\end{bmatrix}
=\begin{bmatrix}\frac cg&1\\ a&b\end{bmatrix}\cdot\begin{bmatrix}\lambda'\\ 1\end{bmatrix}
\end{equation}
with $\lambda'$ being as in $\perm\TN'$ in \eqref{Reduced3XptlMtrx}.
Let $f_{\TN'}(\Zp)$ or $f_{\per\TN'}(\Zp)$ denote the total transform of the apex form $f(\Xp)$ in \eqref{Reorganized3Func} under the reduced monomial transformation $\CT{\TN'}$ or $\CT{\per\TN'}$ associated with $\TN'$ or $\perm\TN'$ in \eqref{Reduced3XptlMtrx} respectively.
It is easy to see that the commutative diagram in \eqref{SynCommut} holds here.

When the latent primary variable $x_1$ is consistent with the synthetic exponential matrix $\TQ'$ in \eqref{Syn3XptlMtrx} and exceptional index set $\iI\Zop$, similar to \eqref{CstnNrmlzdMtrx}, $\TQ'$ has a canonical reduction as follows.
\begin{equation}\label{3-NrmlzdSynMtrx}
    \TN_{\TQ'}:=
    \begin{array}{r @{} c @{\hs{-0.5pt}} c @{} p{4mm} @{} c
    @{\hs{1.5mm}} c @{\hs{1.5mm}} c @{\hs{2.5mm}} c
    @{\hs{0.8mm}} p{4mm} @{\hs{-2.5mm}} l}
    &           &           &           &           \E_1
    &           &           &           &           &               \\
    \mr{4}{$\rule{0mm}{9mm}$}           &           x_1
    &           \mr{4}{$\rule{0mm}{9mm}$}           &
    \raisebox{2.9mm}[0pt]{\mr{4}{$\left[\rule{0mm}{9mm}\right.$}}
    &           1           &           \vline      &
    \mc{2}{c}{\U'}          &
    \raisebox{2.9mm}[0pt]{\mr{4}{$\left.\rule{0mm}{9mm}\right]$}}
    &           \mr{6}{$\rule{0mm}{24mm}$}                          \\
    \cline{2-9}
    &           \zeta       &           &           &
    0           &           \vline      &
    \mc{2}{c}{\mr{2}{$\Th'$}}           &           &               \\
    &           x'_1        &           &           &
    0           &           \vline      &           &
    &           &                                                   \\[-2pt]
    &           &           &           &           t'_1
    &           \vline      &           t'_2        &
    t'_3        &           &                                       \\
    \end{array},
\end{equation}
where $\bigl[\begin{smallmatrix}\U'\\ \Th'\end{smallmatrix}\bigr]$ is the exceptional column submatrix $\Ev\TQ\Zop'$ of $\TQ'$ in \eqref{Syn3XptlMtrx} as in Definition \ref{Def:ExceptionalSupport} with $\det\Th'\ne 0$.
It is easy to see that the latent primary variable $x_1$ is revived as $t'_1$ in this case and the singularity height strictly decreases from the prior singularity height $d$ of the Weierstrass polynomial $w(\X)$ in \eqref{3-WeierstrassPolynomial}, which is similar to the discussion in Lemma \ref{Lemma:SimpleDecrease}.

In the case of inconsistency of the latent primary variable $x_1$ with synthetic exponential matrix $\TQ'$ and $\iI\Zop$ in \eqref{Syn3XptlMtrx}, $\TQ'$ bears a form similar to $\perm\TN'$ in \eqref{Reduced3XptlMtrx} as follows.
\begin{equation}\label{3-SpecialSynMtrx}
    \TQ':=
    \begin{array}{r @{} c @{\hs{-0.5pt}} c @{} p{4mm} @{} c
    @{\hs{1mm}} c @{\hs{1mm}} c @{\hs{1mm}} c @{\hs{0.8mm}}
    p{4mm} @{\hs{-2.5mm}} l}
    &           &           &           &           \E_1
    &           &           &           &           &               \\
    \mr{4}{$\rule{0mm}{9mm}$}           &           x'_1
    &           \mr{4}{$\rule{0mm}{9mm}$}           &
    \raisebox{2.9mm}[0pt]{\mr{4}{$\left[\rule{0mm}{9mm}\right.$}}
    &           1           &           \vline      &
    \lambda'c'  &           \lambda'g'  &
    \raisebox{2.9mm}[0pt]{\mr{4}{$\left.\rule{0mm}{9mm}\right]$}}
    &           \mr{6}{$\rule{0mm}{24mm}$}                          \\
    \cline{2-9}
    &           \zeta       &           &           &
    0           &           \vline      &           c'
    &           g'          &           &                           \\
    &           x_1         &           &           &
    0           &           \vline      &           a'
    &           b'          &           &                           \\[-2pt]
    &           &           &           &           z'_1
    &           \vline      &           z'_2        &
    z'_3        &           &                                       \\
    \end{array},
\end{equation}
where we abuse the notations a bit and use the elements of $\perm\TN'$ in \eqref{Reduced3XptlMtrx} for those of $\TQ'$.
Similar to Lemma \ref{Lemma:ConstantPrimaryExp}, the latent primary exponent $\alpha_1$ of $x_1$ is a constant on the exceptional support $\es{\TQ'}\Zop$.
Hence based upon the latent gradation and Weierstrass form $f_0(x'_1,\zeta)$ in \eqref{3-InterimReorgFunc}, we can discuss like in Lemma \ref{Lemma:MiddleReviveDecrease} to show that the singularity height strictly decreases from the singularity height $d'$ of $f_0(x'_1,\zeta)$.

It follows from the assumptions $c\in\BB N$ and $g\in\BB N^*$ in \eqref{3-NewInterimMtrx} and $\lambda'\in\BB Q_{\ge 0}$ in \eqref{Reduced3XptlMtrx} that $\mu>0$ in \eqref{CoeffAnalysis}.
Thus the latent primary variable $x_1$ is always consistent with the synthetic exponential matrix $\perm\TQ'$ and $\iI\Zop$ in \eqref{Syn3XptlMtrx} which has a canonical reduction as follows.
\begin{equation}\label{3-NormalSynMtrx}
    \TN_{\per\TQ'}:=
    \begin{array}{r @{} c @{\hs{-0.5pt}} c @{} p{4mm} @{} c
    @{\hs{1.5mm}} c @{\hs{1.5mm}} c @{\hs{2.5mm}} c
    @{\hs{0.8mm}} p{4mm} @{\hs{-2.5mm}} l}
    &           &           &           &           &
    &           &           &           &                           \\
    \mr{4}{$\rule{0mm}{9mm}$}           &           x_1
    &           \mr{4}{$\rule{0mm}{9mm}$}           &
    \raisebox{2.9mm}[0pt]{\mr{4}{$\left[\rule{0mm}{9mm}\right.$}}
    &           1           &           \vline      &
    \nu c'      &           \nu g'      &
    \raisebox{2.9mm}[0pt]{\mr{4}{$\left.\rule{0mm}{9mm}\right]$}}
    &           \mr{6}{$\rule{0mm}{24mm}$}                          \\
    \cline{2-9}
    &           \zeta       &           &           &
    0           &           \vline      &           \mu c'
    &           \mu g'      &           &                           \\
    &           x'_1        &           &           &
    0           &           \vline      &           a'
    &           b'          &           &                           \\[-2pt]
    &           &           &           &           t'_1
    &           \vline      &           t'_2        &
    t'_3        &           &                                       \\
    \end{array}.
\end{equation}
In this case the latent primary variable $x_1$ is revived as the primary variable $t'_1$ as above with a strict decrease of the singularity height from $d$ of $w(\X)$ in \eqref{3-WeierstrassPolynomial}.

Now consider the case when the exceptional index set $\iI\Zop=\{3\}$ rather than $\{2,3\}$ in \eqref{Reduced3XptlMtrx}.
In this case the reduced exponential matrix bears the form of either $\TN'$ or $\perm\TN'$ in \eqref{Reduced3XptlMtrx-2} with $g'\in\BB N^*$, depending on the consistency of the primary variable $x'_1$ with $\TN'$ and $\iI\Zop$ or $\perm\TN'$ and $\iI\Zop$ as in \eqref{Reduced3XptlMtrx-2}.
\begin{equation}\label{Reduced3XptlMtrx-2}
\begin{aligned}
    \TN':=
    \begin{array}{r @{} c @{\hs{-0.5pt}} c @{} p{4mm} @{} c
    @{\hs{1.5mm}} c @{\hs{1mm}} c @{\hs{1.7mm}} c @{\hs{0.8mm}}
    p{4mm} @{\hs{-2.5mm}} l}
    &           &           &           &           \E_1
    &           \E_2        &           &           \V'_3
    &           &                                                   \\
    \mr{4}{$\rule{0mm}{9mm}$}           &           x'_1
    &           \mr{4}{$\rule{0mm}{9mm}$}           &
    \raisebox{2.9mm}[0pt]{\mr{4}{$\left[\rule{0mm}{9mm}\right.$}}
    &           1           &           0           &
    \vline      &           b'          &
    \raisebox{2.9mm}[0pt]{\mr{4}{$\left.\rule{0mm}{9mm}\right]$}}
    &           \mr{6}{$\rule{0mm}{24mm}$}                          \\
    &           x'_2        &           &           &
    0           &           1           &           \vline
    &           a'          &           &                           \\
    \cline{2-9}
    &           x'_3        &           &           &
    0           &           0           &           \vline
    &           g'          &           &                           \\[-2pt]
    &           &           &           &           z'_1
    &           z'_2        &           \vline      &
    z'_3        &           &                                       \\
    \end{array};
    &\hs{10mm}
    \perm\TN':=
    \begin{array}{r @{} c @{\hs{-0.5pt}} c @{} p{4mm} @{} c
    @{\hs{1.5mm}} c @{\hs{1mm}} c @{\hs{1.5mm}} c
    @{\hs{0.8mm}} p{4mm} @{\hs{-2.5mm}} l}
    &           &           &           &           \E_1
    &           \E_2        &           &           \V'_3
    &           &                                                   \\
    \mr{4}{$\rule{0mm}{9mm}$}           &           x'_2
    &           \mr{4}{$\rule{0mm}{9mm}$}           &
    \raisebox{2.9mm}[0pt]{\mr{4}{$\left[\rule{0mm}{9mm}\right.$}}
    &           1           &           0           &
    \vline      &           0           &
    \raisebox{2.9mm}[0pt]{\mr{4}{$\left.\rule{0mm}{9mm}\right]$}}
    &           \mr{6}{$\rule{0mm}{24mm}$}                          \\
    &           x'_3        &           &           &
    0           &           1           &           \vline
    &           0           &           &                           \\
    \cline{2-9}
    &           x'_1        &           &           &
    0           &           0           &           \vline
    &           g'          &           &                           \\[-2pt]
    &           &           &           &           z'_1
    &           z'_2        &           \vline      &
    z'_3        &           &                                       \\
    \end{array}.
\end{aligned}
\end{equation}

The interim exponential matrix $\perm\wh\TT$ in \eqref{3-NewInterimMtrx} and $\TN'$ or $\perm\TN'$ in \eqref{Reduced3XptlMtrx-2} can be synthesized into a synthetic exponential matrix $\TQ'$ or $\perm\TQ'$ as following:
\begin{equation}\label{Syn3XptlMtrx-2}
\begin{aligned}
    \TQ':=
    \begin{array}{r @{} c @{\hs{-0.5pt}} c @{} p{4mm} @{} c
    @{\hs{1.5mm}} c @{\hs{1mm}} c @{\hs{1.5mm}} c
    @{\hs{0.8mm}} p{4mm} @{\hs{-2.5mm}} l}
    &           &           &           &           \E_1
    &           &           &           &           &               \\
    \mr{4}{$\rule{0mm}{9mm}$}           &           x'_1
    &           \mr{4}{$\rule{0mm}{9mm}$}           &
    \raisebox{2.9mm}[0pt]{\mr{4}{$\left[\rule{0mm}{9mm}\right.$}}
    &           1           &           0           &
    \vline      &           b'          &
    \raisebox{2.9mm}[0pt]{\mr{4}{$\left.\rule{0mm}{9mm}\right]$}}
    &           \mr{6}{$\rule{0mm}{24mm}$}                          \\
    &           \zeta       &           &           &
    0           &           \frac cg    &           \vline
    &           \wt g       &           &                           \\
    \cline{2-9}
    &           x_1         &           &           &
    0           &           a           &           \vline
    &           h           &           &                           \\[-2pt]
    &           &           &           &           z'_1
    &           z'_2        &           \vline      &
    z'_3        &           &                                       \\
    \end{array};
    &\hs{10mm}
    \perm\TQ':=
    \begin{array}{r @{} c @{\hs{-0.5pt}} c @{} p{4mm} @{} c
    @{\hs{1.5mm}} c @{\hs{1mm}} c @{\hs{1.5mm}} c
    @{\hs{0.8mm}} p{4mm} @{\hs{-2.5mm}} l}
    &           &           &           &           &
    &           &           &           &                           \\
    \mr{4}{$\rule{0mm}{9mm}$}           &           \zeta
    &           \mr{4}{$\rule{0mm}{9mm}$}           &
    \raisebox{2.9mm}[0pt]{\mr{4}{$\left[\rule{0mm}{9mm}\right.$}}
    &           \frac cg    &           1           &
    \vline      &           0           &
    \raisebox{2.9mm}[0pt]{\mr{4}{$\left.\rule{0mm}{9mm}\right]$}}
    &           \mr{6}{$\rule{0mm}{24mm}$}                          \\
    &           x_1         &           &           &
    a           &           b           &           \vline
    &           0           &           &                           \\
    \cline{2-9}
    &           x'_1        &           &           &
    0           &           0           &           \vline
    &           g'          &           &                           \\[-2pt]
    &           &           &           &           z'_1
    &           z'_2        &           \vline      &
    z'_3        &           &                                       \\
    \end{array},
\end{aligned}
\end{equation}
where the $2$ by $2$ submatrix $\bigl[\begin{smallmatrix}c/g&\wt g\\ a&h\end{smallmatrix}\bigr]$ of $\TQ'$ is defined as $\TG\cdot\bigl[\begin{smallmatrix}1&a'\\ 0&g'\end{smallmatrix}\bigr]$ with $\TG=\bigl[\begin{smallmatrix}c/g&1\\ a&b\end{smallmatrix}\bigr]$ being the $2$ by $2$ submatrix of $\perm\wh\TT$ in \eqref{3-NewInterimMtrx}.
In particular, the element $\wt g=\frac{ca'}g+g'>0$ since $c,a'\in\BB N$ and $g,g'\in\BB N^*$.
Thus the canonical reduction of $\TQ'$ or $\perm\TQ'$ in \eqref{Syn3XptlMtrx-2} bears the following form:
\begin{equation}\label{RedSyn3XptlMtrx}
\begin{aligned}
    \TN_{\TQ'}:=
    \begin{array}{r @{} c @{\hs{-0.5pt}} c @{} p{4mm} @{} c
    @{\hs{1.5mm}} c @{\hs{0.7mm}} c @{\hs{1.5mm}} c
    @{\hs{0.8mm}} p{4mm} @{\hs{-2.5mm}} l}
    &           &           &           &           \E_1
    &           \E_2        &           &           &
    &                                                               \\
    \mr{4}{$\rule{0mm}{9mm}$}           &           x'_1
    &           \mr{4}{$\rule{0mm}{9mm}$}           &
    \raisebox{2.9mm}[0pt]{\mr{4}{$\left[\rule{0mm}{9mm}\right.$}}
    &           1           &           0           &
    \vline      &           b'          &
    \raisebox{2.9mm}[0pt]{\mr{4}{$\left.\rule{0mm}{9mm}\right]$}}
    &           \mr{6}{$\rule{0mm}{24mm}$}                          \\
    &           x_1         &           &           &
    0           &           1           &           \vline
    &           h           &           &                           \\
    \cline{2-9}
    &           \zeta       &           &           &
    0           &           0           &           \vline
    &           \wt g       &           &                           \\
    &           &           &           &           t'_1
    &           t'_2        &           \vline      &
    t'_3        &           &                                       \\
    \end{array};
    &\hs{10mm}
    \TN_{\per\TQ'}:=
    \begin{array}{r @{} c @{\hs{-0.5pt}} c @{} p{4mm} @{} c
    @{\hs{1.5mm}} c @{\hs{1mm}} c @{\hs{1.5mm}} c
    @{\hs{0.8mm}} p{4mm} @{\hs{-2.5mm}} l}
    &           &           &           &           \E_1
    &           \E_2        &           &           &
    &                                                               \\
    \mr{4}{$\rule{0mm}{9mm}$}           &           \zeta
    &           \mr{4}{$\rule{0mm}{9mm}$}           &
    \raisebox{2.9mm}[0pt]{\mr{4}{$\left[\rule{0mm}{9mm}\right.$}}
    &           1           &           0           &
    \vline      &           0           &
    \raisebox{2.9mm}[0pt]{\mr{4}{$\left.\rule{0mm}{9mm}\right]$}}
    &           \mr{6}{$\rule{0mm}{24mm}$}                          \\
    &           x_1         &           &           &
    0           &           1           &           \vline
    &           0           &           &                           \\
    \cline{2-9}
    &           x'_1        &           &           &
    0           &           0           &           \vline
    &           g'          &           &                           \\
    &           &           &           &           t'_1
    &           t'_2        &           \vline      &
    t'_3        &           &                                       \\
    \end{array}.
\end{aligned}
\end{equation}
It is evident that we can take $t'_2$ as the revived primary variable in both the above cases in \eqref{RedSyn3XptlMtrx} such that the latent primary variable $x_1$ is revived.
In this way the singularity height strictly decreases from the prior singularity height $d$ of the Weierstrass polynomial $w(\X)$ in \eqref{3-WeierstrassPolynomial} at every regular reduced branch point.
When the reduced branch point is irregular, the resolution of irregular singularities of $\vfz f1(\Xp)$ in \eqref{3-PreWeierstrass} like in \eqref{Reduced2XptlMtrx} thenceforth can be repeated almost verbatim here to show that eventually the singularity height strictly decreases from $d$ as well.

When the primary variable $x_1$ is inconsistent with the exponential matrix $\TM$ and exceptional index set $\iI\Yo=\{3\}$, the canonical reduction $\perm\TN$ of $\TM$ bears the following form instead of that in \eqref{Incstn3Matrix}:
\begin{equation}\label{Incstn3ExpMtrx-2}\begin{aligned}
    \perm\TN:=
    \begin{array}{r @{} c @{\hs{-0.5pt}} c @{} p{4mm} @{} c
    @{\hs{1.5mm}} c @{\hs{1.5mm}} c @{\hs{1.5mm}} c
    @{\hs{0.8mm}} p{4mm} @{\hs{-2.5mm}} l}
    &           &           &           &           \E_1
    &           \E_2        &           &           \V_3
    &           &                                                   \\
    \mr{4}{$\rule{0mm}{9mm}$}           &           x_2
    &           \mr{4}{$\rule{0mm}{9mm}$}           &
    \raisebox{2.4mm}[0pt]{\mr{4}{$\left[\rule{0mm}{9mm}\right.$}}
    &           1           &           0           &
    \vline      &           0           &
    \raisebox{2.4mm}[0pt]{\mr{4}{$\left.\rule{0mm}{9mm}\right]$}}
    &           \mr{6}{$\rule{0mm}{24mm}$}                          \\
    &           x_3         &           &           &
    0           &           1           &           \vline
    &           0           &           &                           \\
    \cline{2-9}
    &           x_1         &           &           &
    0           &           0           &           \vline
    &           g           &           &                           \\[-2.5pt]
    &           &           &           &           z_1
    &           z_2         &           \vline      &
    z_3         &           &                                       \\
    \end{array}
\end{aligned}\end{equation}
with $g\in\BB N^*$.
The canonical reduction $\perm\TN$ can be further reduced to the unit matrix since its column vector $\V_3$ can be reduced to the unit vector $\E_3$.
Hence we just take $\perm\TN=\TE$ henceforth rather than the form in \eqref{Incstn3ExpMtrx-2}.
We make latent preliminary and Weierstrass reductions like in \eqref{Reorganized3Func} as in Definition \ref{Def:LatentReductions} with $(\wt z_1,\wt z_2,z_3)=(x'_1,x'_2,x'_3)$ so as to obtain an apex form $f(\Xp)$ with singularity height $d'$ as in \eqref{Reorganized3Func}.

In the next resolution step ensued, the exceptional index set $\iI\Zop$ is either $\{2,3\}$ or $\{3\}$.
In the case when $\iI\Zop=\{2,3\}$, the reduced exponential matrix bears the form $\TN'$ or $\perm\TN'$ in \eqref{Reduced3XptlMtrx}, depending on the consistency of the primary variable $x'_1$.
Similar to \eqref{Interim3Matrix}, the reduced exponential matrix $\perm\TN=\TE$ acquired from the form in \eqref{Incstn3ExpMtrx-2} can be decomposed as $\perm\TN=\perm\TS\cdot\perm\TT$ with

\begin{equation}\label{Interim3Matrix-2}
\begin{aligned}
    \perm\TS:=
    \begin{array}{r @{} c @{\hs{-0.5pt}} c @{} p{4mm} @{} c
    @{\hs{1.5mm}} c @{\hs{1mm}} c @{\hs{1.5mm}} c
    @{\hs{0.8mm}} p{4mm} @{\hs{-2.5mm}} l}
    &           &           &           &           \E_1
    &           \E_2        &           &           \E_3
    &           &                                                   \\
    \mr{4}{$\rule{0mm}{9mm}$}           &           x_2
    &           \mr{4}{$\rule{0mm}{9mm}$}           &
    \raisebox{2.9mm}[0pt]{\mr{4}{$\left[\rule{0mm}{9mm}\right.$}}
    &           1           &           0           &
    \vline      &           0           &
    \raisebox{2.9mm}[0pt]{\mr{4}{$\left.\rule{0mm}{9mm}\right]$}}
    &           \mr{6}{$\rule{0mm}{24mm}$}                          \\
    &           x_3         &           &           &
    0           &           1           &           \vline
    &           0           &           &                           \\
    \cline{2-9}
    &           x_1         &           &           &
    0           &           0           &           \vline
    &           1           &           &                           \\[-2pt]
    &           &           &           &           z_1
    &           z_2         &           \vline      &
    x_1         &           &                                       \\
    \end{array};
    &\hs{10mm}
    \perm\TT:=
    \begin{array}{r @{} c @{\hs{-0.5pt}} c @{} p{4mm} @{} c
    @{\hs{1.5mm}} c @{\hs{1mm}} c @{\hs{1.5mm}} c
    @{\hs{0.8mm}} p{4mm} @{\hs{-2.5mm}} l}
    &           &           &           &           \E_1
    &           \E_2        &           &           \E_3
    &           &                                                   \\
    \mr{4}{$\rule{0mm}{9mm}$}           &           z_1
    &           \mr{4}{$\rule{0mm}{9mm}$}           &
    \raisebox{2.9mm}[0pt]{\mr{4}{$\left[\rule{0mm}{9mm}\right.$}}
    &           1           &           0           &
    \vline      &           0           &
    \raisebox{2.9mm}[0pt]{\mr{4}{$\left.\rule{0mm}{9mm}\right]$}}
    &           \mr{6}{$\rule{0mm}{24mm}$}                          \\
    &           z_2         &           &           &
    0           &           1           &           \vline
    &           0           &           &                           \\
    \cline{2-9}
    &           x_1         &           &           &
    0           &           0           &           \vline
    &           1           &           &                           \\[-2pt]
    &           &           &           &           z_1
    &           z_2         &           \vline      &
    z_3         &           &                                       \\
    \end{array}.
\end{aligned}
\end{equation}

Let $p(z_1,z_2,x_1)$ be the total transform of the Weierstrass polynomial $w(\X)$ in \eqref{3-WeierstrassPolynomial} under $\CT{\per\TS}$ as defined in \eqref{SyntheticTransform}.
Similar to \eqref{3-InterimReorgFunc}, the interim preliminary and Weierstrass reductions as in Definition \ref{Def:InterimReductions} reduce $p(z_1,z_2,x_1)$ into an apex form:
\begin{equation}\label{3-Apexform}
f(x'_1,x'_2,x_1)=x_1^{a_1}f_0(x'_1,x'_2)+\phi(x'_1,x'_2,x_1),
\end{equation}
where $f_0(x'_1,x'_2)$ is in Weierstrass form with singularity height $d'$ like $f(x'_2,x'_3)$ in \eqref{3-MonoWeierstrass}.
Similar to $\perm\wh\TT$ in \eqref{3-NewInterimMtrx}, a new interim exponential matrix can be defined as follows according to $\perm\TT$ in \eqref{Interim3Matrix-2}:
\begin{equation}\label{3-NewInterimMtrx-2}
    \perm\wh\TT:=
    \begin{array}{r @{} c @{\hs{-0.5pt}} c @{} p{4mm} @{} c
    @{\hs{1.5mm}} c @{\hs{1mm}} c @{\hs{1.5mm}} c
    @{\hs{0.8mm}} p{4mm} @{\hs{-2.5mm}} l}
    &           &           &           &           \E_1
    &           \E_2        &           &           \E_3
    &           &                                                   \\
    \mr{4}{$\rule{0mm}{9mm}$}           &           x'_1
    &           \mr{4}{$\rule{0mm}{9mm}$}           &
    \raisebox{2.9mm}[0pt]{\mr{4}{$\left[\rule{0mm}{9mm}\right.$}}
    &           1           &           0           &
    \vline      &           0           &
    \raisebox{2.9mm}[0pt]{\mr{4}{$\left.\rule{0mm}{9mm}\right]$}}
    &           \mr{6}{$\rule{0mm}{24mm}$}                          \\
    &           x'_2        &           &           &
    0           &           1           &           \vline
    &           0           &           &                           \\
    \cline{2-9}
    &           x_1         &           &           &
    0           &           0           &           \vline
    &           1           &           &                           \\
    &           &           &           &           x'_1
    &           x'_2        &           \vline      &
    x'_3        &           &                                       \\
    \end{array}.
\end{equation}

The interim exponential matrix $\perm\wh\TT$ in \eqref{3-NewInterimMtrx-2} and reduced exponential matrix $\TN'$ or $\perm\TN'$ in \eqref{Reduced3XptlMtrx} can be synthesized into the following synthetic exponential matrix $\TQ'$ or $\perm\TQ'$ similar to \eqref{Syn3XptlMtrx} respectively.
\begin{equation}\label{Syn3XptlMtrx-3}
\begin{aligned}
    \TQ':=
    \begin{array}{r @{} c @{\hs{-0.5pt}} c @{} p{4mm} @{} c
    @{\hs{1.5mm}} c @{\hs{1.5mm}} c @{\hs{2.5mm}} c
    @{\hs{0.8mm}} p{4mm} @{\hs{-2.5mm}} l}
    &           &           &           &           \E_1
    &           &           &           &           &               \\
    \mr{4}{$\rule{0mm}{9mm}$}           &           x'_1
    &           \mr{4}{$\rule{0mm}{9mm}$}           &
    \raisebox{2.9mm}[0pt]{\mr{4}{$\left[\rule{0mm}{9mm}\right.$}}
    &           1           &           \vline      &
    \mc{2}{c}{\La'\hs{-0.7mm}\TD'}      &
    \raisebox{2.9mm}[0pt]{\mr{4}{$\left.\rule{0mm}{9mm}\right]$}}
    &           \mr{6}{$\rule{0mm}{24mm}$}                          \\
    \cline{2-9}
    &           x'_2        &           &           &
    0           &           \vline      &
    \mc{2}{c}{\mr{2}{$\TD'$}}           &           &               \\
    &           x_1         &           &           &
    0           &           \vline      &           &
    &           &                                                   \\[-2pt]
    &           &           &           &           z'_1
    &           \vline      &           z'_2        &
    z'_3        &           &                                       \\
    \end{array};
    &\hs{10mm}
    \perm\TQ':=
    \begin{array}{r @{} c @{\hs{-0.5pt}} c @{} p{4mm} @{} c
    @{\hs{1.5mm}} c @{\hs{1.5mm}} c @{\hs{1.5mm}} c
    @{\hs{0.8mm}} p{4mm} @{\hs{-2.5mm}} l}
    &           &           &           &           \E_1
    &           &           &           &           &               \\
    \mr{4}{$\rule{0mm}{9mm}$}           &           x'_2
    &           \mr{4}{$\rule{0mm}{9mm}$}           &
    \raisebox{2.9mm}[0pt]{\mr{4}{$\left[\rule{0mm}{9mm}\right.$}}
    &           1           &           \vline      &
    \lambda'c'  &           \lambda'g'  &
    \raisebox{2.9mm}[0pt]{\mr{4}{$\left.\rule{0mm}{9mm}\right]$}}
    &           \mr{6}{$\rule{0mm}{24mm}$}                          \\
    \cline{2-9}
    &           x_1         &           &           &
    0           &           \vline      &           c'
    &           g'          &           &                           \\
    &           x'_1        &           &           &
    0           &           \vline      &           a'
    &           b'          &           &                           \\[-2pt]
    &           &           &           &           z'_1
    &           \vline      &           z'_2        &
    z'_3        &           &                                       \\
    \end{array}.
\end{aligned}
\end{equation}
When the latent primary variable $x_1$ is consistent with $\TQ'$ and $\iI\Zop$ or $\perm\TQ'$ and $\iI\Zop$ in \eqref{Syn3XptlMtrx-3}, which amounts to $\lambda'>0$ for $\perm\TQ'$, through a canonical reduction of $\TQ'$ or $\perm\TQ'$ similar to \eqref{3-NrmlzdSynMtrx} and \eqref{3-NormalSynMtrx}, $x_1$ is revived and the singularity height is strictly less than that of $w(\X)$ in \eqref{3-WeierstrassPolynomial} which equals $d$.

In the case of inconsistency of the latent primary variable $x_1$ with the synthetic exponential matrix $\TQ'$ and $\iI\Zop$ in \eqref{Syn3XptlMtrx-3}, it is easy to see that $\TQ'$ bears a form resembling the one in \eqref{3-SpecialSynMtrx} with $x'_1$ being the primary variable.
Thus based on the Weierstrass form $f_0(x'_1,x'_2)$ in \eqref{3-Apexform} and similar to the discussion in Lemma \ref{Lemma:MiddleReviveDecrease}, we can show that the singularity height strictly decreases from the singularity height $d'$ of the apex form $f(x'_1,x'_2,x_1)$ in \eqref{3-Apexform}.

When $\lambda'=0$ in the synthetic exponential matrix $\perm\TQ'$ in \eqref{Syn3XptlMtrx-3}, the singularity height might have a temporary increase from the singularity height $d'$ of the apex form $f(x'_1,x'_2,x_1)$ in \eqref{3-Apexform}.
Nonetheless the nesting degree as in Definition \ref{Def:NestLatency} increases to $\fnd=2$.
Suppose that prior to the next resolution step, the latent preliminary and Weierstrass reductions yield an apex form $f(\Xq)$ with singularity height $d''$ that resembles the apex form $f(\Xq)$ in \eqref{NewApexForm}.
In the ensuing resolution step the exceptional index set $\iI\Zoq$ is either $\{2,3\}$ or $\{3\}$.

In the case when $\iI\Zoq=\{2,3\}$, the reduced exponential matrix bears the following form $\TN''$ or $\perm\TN''$, depending on the consistency of the primary variable $x''_1$, which is similar to the scenario in  \eqref{Reduced3XptlMtrx}.
\begin{equation}\label{Reduced3XptlMtrx-3}
\begin{aligned}
    \TN'':=
    \begin{array}{r @{} c @{\hs{-0.5pt}} c @{} p{4mm} @{} c
    @{\hs{1.5mm}} c @{\hs{1.5mm}} c @{\hs{2.5mm}} c
    @{\hs{0.8mm}} p{4mm} @{\hs{-2.5mm}} l}
    &           &           &           &           \E_1
    &           &           \V''_2      &           \V''_3
    &           &                                                   \\
    \mr{4}{$\rule{0mm}{9mm}$}           &           x''_1
    &           \mr{4}{$\rule{0mm}{9mm}$}           &
    \raisebox{2.9mm}[0pt]{\mr{4}{$\left[\rule{0mm}{9mm}\right.$}}
    &           1           &           \vline      &
    \mc{2}{c}{\La''\hs{-1.8pt}\TD''}    &
    \raisebox{2.9mm}[0pt]{\mr{4}{$\left.\rule{0mm}{9mm}\right]$}}
    &           \mr{6}{$\rule{0mm}{24mm}$}                          \\
    \cline{2-9}
    &           x''_2       &           &           &
    0           &           \vline      &
    \mc{2}{c}{\mr{2}{$\TD''$}}          &           &               \\
    &           x''_3       &           &           &
    0           &           \vline      &           &
    &           &                                                   \\[-2pt]
    &           &           &           &           z''_1
    &           \vline      &           z''_2        &
    z''_3       &           &                                       \\
    \end{array};
    &\hs{10mm}
    \perm\TN'':=
    \begin{array}{r @{} c @{\hs{-0.5pt}} c @{} p{4mm} @{} c
    @{\hs{1.5mm}} c @{\hs{1.5mm}} c @{\hs{2.5mm}} c
    @{\hs{0.8mm}} p{4mm} @{\hs{-2.5mm}} l}
    &           &           &           &           \E_1
    &           &           \V''_2      &           \V''_3
    &           &                                                   \\
    \mr{4}{$\rule{0mm}{9mm}$}           &           x''_2
    &           \mr{4}{$\rule{0mm}{9mm}$}           &
    \raisebox{2.9mm}[0pt]{\mr{4}{$\left[\rule{0mm}{9mm}\right.$}}
    &           1           &           \vline      &
    \lambda''\hs{-1.2pt}c'' &           \lambda''\hs{-1.2pt}g''
    &
    \raisebox{2.9mm}[0pt]{\mr{4}{$\left.\rule{0mm}{9mm}\right]$}}
    &           \mr{6}{$\rule{0mm}{24mm}$}                          \\
    \cline{2-9}
    &           x''_3       &           &           &
    0           &           \vline      &           c''
    &           g''         &           &                           \\
    &           x''_1       &           &           &
    0           &           \vline      &           a''
    &           b''         &           &                           \\[-2pt]
    &           &           &           &           z''_1
    &           \vline      &           z''_2       &
    z''_3       &           &                                       \\
    \end{array},
\end{aligned}
\end{equation}
where we have non-degeneracies $\det\TD''\cdot\det\TG''\ne 0$ with $\TG''$ denoting the $2$ by $2$ submatrix $\bigl[\begin{smallmatrix}c''&g''\\ a''&b''\end{smallmatrix}\bigr]$ of $\perm\TN''$ in \eqref{Reduced3XptlMtrx-3}.
From the perspective of synthetic monomial transformations and similar to \eqref{Interim3Matrix}, the synthetic exponential matrix $\perm\TQ'$ in \eqref{Syn3XptlMtrx-3} with $\lambda'=0$ can be decomposed as $\perm\TQ'=\perm\TS'\cdot\perm\TT'$ with the interim exponential matrices $\perm\TS'$ and $\perm\TT'$ being as follows.
\begin{equation}\label{3-InterimMtrx}
\begin{aligned}
    \perm\TS':=
    \begin{array}{r @{} c @{\hs{-0.5pt}} c @{} p{4mm} @{} c
    @{\hs{1mm}} c @{\hs{1mm}} c @{\hs{1mm}} c @{\hs{0.8mm}}
    p{4mm} @{\hs{-2.5mm}} l}
    &           &           &           &           \E_1
    &           &           \E_2        &           \E_3
    &           &                                                   \\
    \mr{4}{$\rule{0mm}{9mm}$}           &           x'_2
    &           \mr{4}{$\rule{0mm}{9mm}$}           &
    \raisebox{2.9mm}[0pt]{\mr{4}{$\left[\rule{0mm}{9mm}\right.$}}
    &           1           &           \vline      &
    0           &           0           &
    \raisebox{2.9mm}[0pt]{\mr{4}{$\left.\rule{0mm}{9mm}\right]$}}
    &           \mr{6}{$\rule{0mm}{24mm}$}                          \\
    \cline{2-9}
    &           x_1         &           &           &
    0           &           \vline      &           1
    &           0           &           &                           \\
    &           x'_1        &           &           &
    0           &           \vline      &           0
    &           1           &           &                           \\[-2pt]
    &           &           &           &           z'_1
    &           \vline      &           x_1         &
    x'_1        &           &                                       \\
    \end{array};
    &\hs{10mm}
    \perm\TT':=
    \begin{array}{r @{} c @{\hs{-0.5pt}} c @{} p{4mm} @{} c
    @{\hs{1mm}} c @{\hs{1mm}} c @{\hs{1mm}} c @{\hs{0.8mm}}
    p{4mm} @{\hs{-2.5mm}} l}
    &           &           &           &           \E_1
    &           &           &           &           &               \\
    \mr{4}{$\rule{0mm}{9mm}$}           &           z'_1
    &           \mr{4}{$\rule{0mm}{9mm}$}           &
    \raisebox{2.9mm}[0pt]{\mr{4}{$\left[\rule{0mm}{9mm}\right.$}}
    &           1           &           \vline      &
    0           &           0           &
    \raisebox{2.9mm}[0pt]{\mr{4}{$\left.\rule{0mm}{9mm}\right]$}}
    &           \mr{6}{$\rule{0mm}{24mm}$}                          \\
    \cline{2-9}
    &           x_1         &           &           &
    0           &           \vline      &           c'
    &           g'          &           &                           \\
    &           x'_1        &           &           &
    0           &           \vline      &           a'
    &           b'          &           &                           \\[-2pt]
    &           &           &           &           z'_1
    &           \vline      &           z'_2        &
    z'_3        &           &                                       \\
    \end{array}.
\end{aligned}
\end{equation}

The interim monomial transformation $\CT{\per\TS'}$ in \eqref{3-InterimMtrx} ensued by a localization of the variable $z'_1$ as well as the identity $x''_1=\wt z'_1$ reduce $f(x'_1,x'_2,x_1)$ in \eqref{3-Apexform} into an apex form:
\begin{equation}\label{3-UnivariableForm}
f(x''_1,x_1,x'_1)=x_1^{\alpha_1}{x'_1}^{a'_1}f_0(x''_1)+\phi (x''_1,x_1,x'_1)
\end{equation}
similar to \eqref{Dgr2LtnGrdtn} with $a'_1$ being the first component of the vertex $\A'$ associated with $\perm\TN'$ in \eqref{Reduced3XptlMtrx} and $\alpha_1$ the latent primary exponent associated with the exceptional support $\es{\perm\TQ'}\Zop$ in \eqref{Syn3XptlMtrx-3} as per Lemma \ref{Lemma:ConstantPrimaryExp}.
According to $\perm\TT'$ in \eqref{3-InterimMtrx}, a new interim exponential matrix can be defined as follows.
\begin{equation}\label{3-ModInterimMtrx}
    \perm\wh\TT':=
    \begin{array}{r @{} c @{\hs{-0.5pt}} c @{} p{4mm} @{} c
    @{\hs{1mm}} c @{\hs{1mm}} c @{\hs{1mm}} c @{\hs{0.8mm}}
    p{4mm} @{\hs{-2.5mm}} l}
    &           &           &           &           \E_1
    &           &           &           &           &               \\
    \mr{4}{$\rule{0mm}{9mm}$}           &           x''_1
    &           \mr{4}{$\rule{0mm}{9mm}$}           &
    \raisebox{2.9mm}[0pt]{\mr{4}{$\left[\rule{0mm}{9mm}\right.$}}
    &           1           &           \vline      &
    0           &           0           &
    \raisebox{2.9mm}[0pt]{\mr{4}{$\left.\rule{0mm}{9mm}\right]$}}
    &           \mr{6}{$\rule{0mm}{24mm}$}                          \\
    \cline{2-9}
    &           x_1         &           &           &
    0           &           \vline      &           c'
    &           g'          &           &                           \\
    &           x'_1        &           &           &
    0           &           \vline      &           a'
    &           b'          &           &                           \\[-2pt]
    &           &           &           &           x''_1
    &           \vline      &           x''_2       &
    x''_3       &           &                                       \\
    \end{array}.
\end{equation}

The above new interim exponential matrix $\perm\wh\TT'$ and reduced exponential matrix $\TN''$ or $\perm\TN''$ in \eqref{Reduced3XptlMtrx-3} can be synthesized into a synthetic exponential matrix $\TQ''$ or $\perm\TQ''$ as follows.

\begin{equation}\label{Syn3XptlMtrx-4}
\begin{aligned}
    \TQ'':=
    \begin{array}{r @{} c @{\hs{-0.5pt}} c @{} p{4mm} @{} c
    @{\hs{1.5mm}} c @{\hs{2mm}} c @{\hs{1.5mm}} c
    @{} p{4mm} @{\hs{-2.5mm}} l}
    &           &           &           &           \E_1
    &           &           &           &           &               \\
    \mr{4}{$\rule{0mm}{9mm}$}           &           x''_1
    &           \mr{4}{$\rule{0mm}{9mm}$}           &
    \raisebox{2.9mm}[0pt]{\mr{4}{$\left[\rule{0mm}{9mm}\right.$}}
    &           1           &           \vline      &
    \mc{2}{c}{\La''\hs{-2pt}\TD''}      &
    \raisebox{2.9mm}[0pt]{\mr{4}{$\left.\rule{0mm}{9mm}\right]$}}
    &           \mr{6}{$\rule{0mm}{24mm}$}                          \\
    \cline{2-9}
    &           x_1         &           &           &
    0           &           \vline      &
    \mc{2}{c}{\mr{2}{$\TG'\hs{-2pt}\TD''$}}         &
    &                                                               \\
    &           x'_1        &           &           &
    0           &           \vline      &           &
    &           &                                                   \\[-2pt]
    &           &           &           &           z''_1
    &           \vline      &           z''_2       &
    z''_3       &           &                                       \\
    \end{array};
    &\hs{10mm}
    \perm\TQ'':=
    \begin{array}{r @{} c @{\hs{-0.5pt}} c @{} p{4mm} @{} c
    @{\hs{1.5mm}} c @{\hs{2mm}} c @{\hs{1.5mm}} c
    @{} p{4mm} @{\hs{-2.5mm}} l}
    &           &           &           &           \E_1
    &           &           &           &           &               \\
    \mr{4}{$\rule{0mm}{9mm}$}           &           x''_1
    &           \mr{4}{$\rule{0mm}{9mm}$}           &
    \raisebox{2.9mm}[0pt]{\mr{4}{$\left[\rule{0mm}{9mm}\right.$}}
    &           0           &           \vline      &
    a''         &           b''         &
    \raisebox{2.9mm}[0pt]{\mr{4}{$\left.\rule{0mm}{9mm}\right]$}}
    &           \mr{6}{$\rule{0mm}{24mm}$}                          \\
    \cline{2-9}
    &           x_1         &           &           &
    c'          &           \vline      &           \mu'c''
    &           \mu'g''     &           &                           \\
    &           x'_1        &           &           &
    a'          &           \vline      &           \nu'c''
    &           \nu'g''     &           &                           \\
    &           &           &           &           z''_1
    &           \vline      &           z''_2       &
    z''_3       &           &                                       \\
    \end{array},
\end{aligned}
\end{equation}
where $\TD''$ is the submatrix of $\TN''$ in \eqref{Reduced3XptlMtrx-3} and $\TG'$ denotes the submatrix $\bigl[\begin{smallmatrix}c'&g'\\ a'&b'\end{smallmatrix}\bigr]$ of $\perm\wh\TT'$ in \eqref{3-ModInterimMtrx}.
Moreover,
\begin{equation}\label{NewCoeffAnalysis}
\begin{bmatrix}\mu'\\ \nu'\end{bmatrix}
=\TG'\cdot\begin{bmatrix}\lambda''\\ 1\end{bmatrix}
=\begin{bmatrix}c'&g'\\ a'&b'\end{bmatrix}\cdot\begin{bmatrix}\lambda''\\ 1\end{bmatrix},
\end{equation}
from which it follows that $\mu'>0$ since $g'\in\BB N^*$ as in \eqref{Reduced3XptlMtrx}.

Since the primary variable $x''_1$ is inconsistent with $\perm\TQ''$ and $\iI\Zoq$ in \eqref{Syn3XptlMtrx-4}, it follows that the latent primary variable $x_1$ can be revived when $\nu'>0$ such that the singularity height strictly decreases from the prior singularity height $d$ of $w(\X)$ in \eqref{3-WeierstrassPolynomial}.
In the case of $\nu'=0$ in $\perm\TQ''$ in \eqref{Syn3XptlMtrx-4}, the latent primary variable $x_1$ sustains its latency as in Definition \ref{Def:NestLatency} whereas $x'_1$ is revived.
Based on the Weierstrass form $f_0(x'_1,x'_2)$ in \eqref{3-Apexform}, we can argue as in Lemma \ref{Lemma:MiddleVariableRevive} to show that the new singularity height strictly decreases from the prior singularity height $d'$ of $f(x'_1,x'_2,x_1)$ in \eqref{3-Apexform}.
Moreover, it is easy to see that the same conclusions follow when the latent primary variable $x_1$ or $x'_1$ is consistent with $\TQ''$ and $\iI\Zoq$ in \eqref{Syn3XptlMtrx-4}.

When $\La''=\bz$ in $\TQ''$ in \eqref{Syn3XptlMtrx-4}, the latent primary variables $x_1$ and $x'_1$ are inconsistent with $\TQ''$ and $\iI\Zoq$ such that they cannot be revived.
When the exceptional support $\es{\TQ''}\Zoq$ in \eqref{Syn3XptlMtrx-4} corresponds to the latent primary exponents $\alpha_1$ and $a'_1$ of $x_1$ and $x'_1$ as in \eqref{3-UnivariableForm} according to Lemma \ref{Lemma:ConstantPrimaryExp}, we write $f_0(x''_1)={x''_1}^{d''}(c+r(x''_1))$ in \eqref{3-UnivariableForm} with $c\in\BB C^*$ and $r(0)=0$, which is similar to \eqref{3-NoWeierstrass}.
After a localization of the non-exceptional variable $z''_1$, the singularity height is reduced to zero after an appropriate deficient contraction as in Definition \ref{Def:DeficientContraction} based on the identity $z''_1=x''_1$ from $\TQ''$ in \eqref{Syn3XptlMtrx-4} with $\La''=\bz$ as well as the above redundant function $c+r(x''_1)$.
And the same holds when $\es{\TQ''}\Zoq$ in \eqref{Syn3XptlMtrx-4} corresponds to other latent primary exponents of $x_1$ and $x'_1$ in \eqref{3-UnivariableForm}.

In the case when the exceptional index set $\iI\Zoq=\{3\}$, the reduced exponential matrix $\TN''$ or $\perm\TN''$ bears the following form instead of the one in \eqref{Reduced3XptlMtrx-3}:
\begin{equation}\label{Reduced3XptlMtrx-4}
\begin{aligned}
    \TN'':=
    \begin{array}{r @{} c @{\hs{-0.5pt}} c @{} p{4mm} @{} c
    @{\hs{1.5mm}} c @{\hs{1mm}} c @{\hs{1.7mm}} c @{\hs{0.8mm}}
    p{4mm} @{\hs{-2.5mm}} l}
    &           &           &           &           \E_1
    &           \E_2        &           &           \V''_3
    &           &                                                   \\
    \mr{4}{$\rule{0mm}{9mm}$}           &           x''_1
    &           \mr{4}{$\rule{0mm}{9mm}$}           &
    \raisebox{2.9mm}[0pt]{\mr{4}{$\left[\rule{0mm}{9mm}\right.$}}
    &           1           &           0           &
    \vline      &           b''         &
    \raisebox{2.9mm}[0pt]{\mr{4}{$\left.\rule{0mm}{9mm}\right]$}}
    &           \mr{6}{$\rule{0mm}{24mm}$}                          \\
    &           x''_2       &           &           &
    0           &           1           &           \vline
    &           a''         &           &                           \\
    \cline{2-9}
    &           x''_3       &           &           &
    0           &           0           &           \vline
    &           g''         &           &                           \\[-2pt]
    &           &           &           &           z''_1
    &           z''_2       &           \vline      &
    z''_3       &           &                                       \\
    \end{array};
    &\hs{10mm}
    \perm\TN'':=
    \begin{array}{r @{} c @{\hs{-0.5pt}} c @{} p{4mm} @{} c
    @{\hs{1.5mm}} c @{\hs{1mm}} c @{\hs{1.5mm}} c
    @{\hs{0.8mm}} p{4mm} @{\hs{-2.5mm}} l}
    &           &           &           &           \E_1
    &           \E_2        &           &           \V''_3
    &           &                                                   \\
    \mr{4}{$\rule{0mm}{9mm}$}           &           x''_2
    &           \mr{4}{$\rule{0mm}{9mm}$}           &
    \raisebox{2.9mm}[0pt]{\mr{4}{$\left[\rule{0mm}{9mm}\right.$}}
    &           1           &           0           &
    \vline      &           0           &
    \raisebox{2.9mm}[0pt]{\mr{4}{$\left.\rule{0mm}{9mm}\right]$}}
    &           \mr{6}{$\rule{0mm}{24mm}$}                          \\
    &           x''_3       &           &           &
    0           &           1           &           \vline
    &           0           &           &                           \\
    \cline{2-9}
    &           x''_1       &           &           &
    0           &           0           &           \vline
    &           g''         &           &                           \\[-2pt]
    &           &           &           &           z''_1
    &           z''_2       &           \vline      &
    z''_3       &           &                                       \\
    \end{array}
\end{aligned}
\end{equation}
with $g''\in\BB N^*$.
In particular, $\perm\TN''$ can be further reduced to the unix matrix like in \eqref{Incstn3ExpMtrx-2} since its column vector $\V''_3$ can be reduced to the unit vector $\E_3$.
Henceforth we just take $\perm\TN''=\TE$ in \eqref{Reduced3XptlMtrx-4} as well.
The interim exponential matrix $\perm\wh\TT'$ in \eqref{3-ModInterimMtrx} and reduced exponential matrix $\TN''$ or $\perm\TN''$ in \eqref{Reduced3XptlMtrx-4} can be synthesized into a synthetic exponential matrix $\TQ''$ or $\perm\TQ''$ as follows.
\begin{equation}\label{3-SynMatrix}
\begin{aligned}
    \TQ'':=
    \begin{array}{r @{} c @{\hs{-0.5pt}} c @{} p{4mm} @{} c
    @{\hs{1.5mm}} c @{\hs{1mm}} c @{\hs{1.7mm}} c @{\hs{0.8mm}}
    p{4mm} @{\hs{-2.5mm}} l}
    &           &           &           &           \E_1
    &           &           &           &           &               \\
    \mr{4}{$\rule{0mm}{9mm}$}           &           x''_1
    &           \mr{4}{$\rule{0mm}{9mm}$}           &
    \raisebox{2.9mm}[0pt]{\mr{4}{$\left[\rule{0mm}{9mm}\right.$}}
    &           1           &           0           &
    \vline      &           b''         &
    \raisebox{2.9mm}[0pt]{\mr{4}{$\left.\rule{0mm}{9mm}\right]$}}
    &           \mr{6}{$\rule{0mm}{24mm}$}                          \\
    &           x_1         &           &           &
    0           &           c'          &           \vline
    &           \mu'        &           &                           \\
    \cline{2-9}
    &           x'_1        &           &           &
    0           &           a'          &           \vline
    &           \nu'        &           &                           \\
    &           &           &           &           z''_1
    &           z''_2       &           \vline      &
    z''_3       &           &                                       \\
    \end{array};
    &\hs{10mm}
    \perm\TQ'':=
    \begin{array}{r @{} c @{\hs{-0.5pt}} c @{} p{4mm} @{} c
    @{\hs{1.5mm}} c @{\hs{1mm}} c @{\hs{1.5mm}} c
    @{\hs{0.8mm}} p{4mm} @{\hs{-2.5mm}} l}
    &           &           &           &           &
    &           &           &           &                           \\
    \mr{4}{$\rule{0mm}{9mm}$}           &           x''_1
    &           \mr{4}{$\rule{0mm}{9mm}$}           &
    \raisebox{2.9mm}[0pt]{\mr{4}{$\left[\rule{0mm}{9mm}\right.$}}
    &           0           &           0           &
    \vline      &           1           &
    \raisebox{2.9mm}[0pt]{\mr{4}{$\left.\rule{0mm}{9mm}\right]$}}
    &           \mr{6}{$\rule{0mm}{24mm}$}                          \\
    &           x_1         &           &           &
    c'          &           g'          &           \vline
    &           0           &           &                           \\
    \cline{2-9}
    &           x'_1        &           &           &
    a'          &           b'          &           \vline
    &           0           &           &                           \\
    &           &           &           &           z''_1
    &           z''_2       &           \vline      &
    z''_3       &           &                                       \\
    \end{array},
\end{aligned}
\end{equation}
where we have:
\[\begin{bmatrix}\mu'\\ \nu'\end{bmatrix}
=\TG'\cdot\begin{bmatrix}a''\\ g''\end{bmatrix}
=\begin{bmatrix}c'&g'\\ a'&b'\end{bmatrix}\cdot\begin{bmatrix}a''\\ g''\end{bmatrix}\]
in $\TQ''$ and it is easy to verify that we always have $\mu'\in\BB N^*$.
The best scenario is when $b''\in\BB N^*$ or $\nu'\in\BB N^*$ in which case the latent primary variable $x_1$ is revived and the singularity height strictly decreases from the prior singularity height $d$ of the Weierstrass polynomial $w(\X)$ in \eqref{3-WeierstrassPolynomial}.
When $\nu'=b''=0$, the canonical reduction of $\TQ''$ bears the following form similar to \eqref{RedSyn3XptlMtrx}:
\begin{equation}\label{3-NormalizedSynMtrx}
    \TN_{\TQ''}:=
    \begin{array}{r @{} c @{\hs{-0.5pt}} c @{} p{4mm} @{} c
    @{\hs{1.5mm}} c @{\hs{1mm}} c @{\hs{1.7mm}} c @{\hs{0.8mm}}
    p{4mm} @{\hs{-2.5mm}} l}
    &           &           &           &           \E_1
    &           \E_2        &           &           &
    &                                                               \\
    \mr{4}{$\rule{0mm}{9mm}$}           &           x'_1
    &           \mr{4}{$\rule{0mm}{9mm}$}           &
    \raisebox{2.9mm}[0pt]{\mr{4}{$\left[\rule{0mm}{9mm}\right.$}}
    &           1           &           0           &
    \vline      &           0           &
    \raisebox{2.9mm}[0pt]{\mr{4}{$\left.\rule{0mm}{9mm}\right]$}}
    &           \mr{6}{$\rule{0mm}{24mm}$}                          \\
    &           x''_1       &           &           &
    0           &           1           &           \vline
    &           0           &           &                           \\
    \cline{2-9}
    &           x_1         &           &           &
    0           &           0           &           \vline
    &           \mu'        &           &                           \\
    &           &           &           &           t''_1
    &           t''_2       &           \vline      &
    t''_3       &           &                                       \\
    \end{array}.
\end{equation}
In \eqref{3-NormalizedSynMtrx} the latent primary variable $x'_1$ is revived as $t''_1$ such that the new singularity height strictly decreases from the prior singularity height $d'$ of the Weierstrass form $f_0(x'_1,x'_2)$ in \eqref{3-Apexform}.
The argument is similar to that in Lemma \ref{Lemma:MiddleVariableRevive}, same as the case of $\nu'=0$ in $\perm\TQ''$ in \eqref{Syn3XptlMtrx-4} when $x'_1$ is revived.

From $g'\in\BB N^*$, the synthetic exponential matrix $\perm\TQ''$ in \eqref{3-SynMatrix} can be reduced to the following canonical form:
\begin{equation}
    \TN_{\per\TQ''}:=
    \begin{array}{r @{} c @{\hs{-0.5pt}} c @{} p{4mm} @{} c
    @{\hs{1.5mm}} c @{\hs{1mm}} c @{\hs{1.5mm}} c
    @{\hs{0.8mm}} p{4mm} @{\hs{-2.5mm}} l}
    &           &           &           &           \E_1
    &           \E_2        &           &           &
    &                                                               \\
    \mr{4}{$\rule{0mm}{9mm}$}           &           x_1
    &           \mr{4}{$\rule{0mm}{9mm}$}           &
    \raisebox{2.9mm}[0pt]{\mr{4}{$\left[\rule{0mm}{9mm}\right.$}}
    &           1           &           0           &
    \vline      &           0           &
    \raisebox{2.9mm}[0pt]{\mr{4}{$\left.\rule{0mm}{9mm}\right]$}}
    &           \mr{6}{$\rule{0mm}{24mm}$}                          \\
    &           x'_1        &           &           &
    0           &           1           &           \vline
    &           0           &           &                           \\
    \cline{2-9}
    &           x''_1       &           &           &
    0           &           0           &           \vline
    &           g'          &           &                           \\
    &           &           &           &           t''_1
    &           t''_2       &           \vline      &
    t''_3       &           &                                       \\
    \end{array}
\end{equation}
such that the latent primary variable $x_1$ is revived and the new singularity height strictly decreases from the singularity height $d$ of the Weierstrass polynomial $w(\X)$ in \eqref{3-WeierstrassPolynomial}.

After the resolution step associated with the canonical reduction $\perm\TN$ in \eqref{Incstn3ExpMtrx-2} with $\iI\Zo=\{3\}$, now consider the case when the reduced exponential matrix bears the form of either $\TN'$ or $\perm\TN'$ in \eqref{Reduced3XptlMtrx-2} with the exceptional index set $\iI\Zop=\{3\}$, depending on the consistency of the primary variable $x'_1$.
The interim exponential matrix $\perm\wh\TT$ in \eqref{3-NewInterimMtrx-2} and canonical reduction $\TN'$ or $\perm\TN'$ in \eqref{Reduced3XptlMtrx-2} can be synthesized into a synthetic exponential matrix $\TQ'$ or $\perm\TQ'$ as follows.
\begin{equation}\label{Syn3XptlMtrx-5}
\begin{aligned}
    \TQ':=
    \begin{array}{r @{} c @{\hs{-0.5pt}} c @{} p{4mm} @{} c
    @{\hs{1.5mm}} c @{\hs{1mm}} c @{\hs{1.7mm}} c @{\hs{0.8mm}}
    p{4mm} @{\hs{-2.5mm}} l}
    &           &           &           &           \E_1
    &           \E_2        &           &           &
    &                                                               \\
    \mr{4}{$\rule{0mm}{9mm}$}           &           x'_1
    &           \mr{4}{$\rule{0mm}{9mm}$}           &
    \raisebox{2.9mm}[0pt]{\mr{4}{$\left[\rule{0mm}{9mm}\right.$}}
    &           1           &           0           &
    \vline      &           b'          &
    \raisebox{2.9mm}[0pt]{\mr{4}{$\left.\rule{0mm}{9mm}\right]$}}
    &           \mr{6}{$\rule{0mm}{24mm}$}                          \\
    &           x'_2        &           &           &
    0           &           1           &           \vline
    &           a'          &           &                           \\
    \cline{2-9}
    &           x_1         &           &           &
    0           &           0           &           \vline
    &           g'          &           &                           \\[-1pt]
    &           &           &           &           z'_1
    &           z'_2        &           \vline      &
    z'_3        &           &                                       \\
    \end{array};
    &\hs{10mm}
    \perm\TQ':=
    \begin{array}{r @{} c @{\hs{-0.5pt}} c @{} p{4mm} @{} c
    @{\hs{1.5mm}} c @{\hs{1mm}} c @{\hs{1.5mm}} c
    @{\hs{0.8mm}} p{4mm} @{\hs{-2.5mm}} l}
    &           &           &           &           \E_1
    &           \E_2        &           &           \V'_3
    &           &                                                   \\
    \mr{4}{$\rule{0mm}{9mm}$}           &           x'_2
    &           \mr{4}{$\rule{0mm}{9mm}$}           &
    \raisebox{2.9mm}[0pt]{\mr{4}{$\left[\rule{0mm}{9mm}\right.$}}
    &           1           &           0           &
    \vline      &           0           &
    \raisebox{2.9mm}[0pt]{\mr{4}{$\left.\rule{0mm}{9mm}\right]$}}
    &           \mr{6}{$\rule{0mm}{24mm}$}                          \\
    &           x_1         &           &           &
    0           &           1           &           \vline
    &           0           &           &                           \\
    \cline{2-9}
    &           x'_1        &           &           &
    0           &           0           &           \vline
    &           g'          &           &                           \\[-1pt]
    &           &           &           &           z'_1
    &           z'_2        &           \vline      &
    z'_3        &           &                                       \\
    \end{array}
\end{aligned}
\end{equation}
with $g'\in\BB N^*$ and $a',b'\in\BB N$.
When $a'^2+b'^2\in\BB N^*$ in $\TQ'$ in \eqref{Syn3XptlMtrx-5}, it is easy to see that the latent primary variable $x_1$ is revived and the singularity height strictly decreases from that of $w(\X)$ in \eqref{3-WeierstrassPolynomial}, which is equal to $d$ and associated with $x_1$.
When $a'=b'=0$, it is evident that $z'_1$ of $\TQ'$ in \eqref{Syn3XptlMtrx-5} serves as the primary variable in this case such that the singularity height strictly decreases from the prior singularity height $d'$ of the Weierstrass form $f_0(x'_1,x'_2)$ in \eqref{3-Apexform}.
The argument is similar to that in the case of $\TQ'$ in \eqref{3-SpecialSynMtrx} when the latent primary variable $x_1$ is inconsistent with $\TQ'$ and $\iI\Zop$.

The scenario of the synthetic exponential matrix $\perm\TQ'$ in \eqref{Syn3XptlMtrx-5} is simple since the latent primary variable $x_1$ is revived via the revived primary variable $z'_2$ in this case.
It is evident that the singularity height strictly decreases from the prior singularity height $d$ of the Weierstrass polynomial $w(\X)$ in \eqref{3-WeierstrassPolynomial}.
All the above strict decreases of singularity heights starting from \eqref{3-NormalizedSynMtrx} hold at regular reduced branch points.
In the case when a reduced branch point is irregular, the resolution of irregular singularities of $\vfz f1(\Xp)$ in \eqref{3-PreWeierstrass} like in \eqref{Reduced2XptlMtrx} thenceforth can be repeated almost verbatim here to show that eventually the singularity height strictly decreases from the prior singularity height as well.

\end{document}